\documentclass[a4paper,10pt]{article}
\usepackage{amssymb,amsmath,amsthm}
\usepackage{graphicx}

\usepackage{amsfonts}
\usepackage{latexsym}
\usepackage[english]{babel}
\usepackage{hyperref}
\usepackage{verbatim} 
\usepackage{bbm} 
\usepackage{color}
\numberwithin{equation}{section}
\usepackage{enumitem}
\usepackage{relsize}
\usepackage{xcolor}
\usepackage{todonotes}

\numberwithin{equation}{section}

\newtheorem{theorem}{Theorem}[section]
\newtheorem{proposition}{Proposition}[section]
\newtheorem{lemma}{Lemma}[section]
\newtheorem{corollary}{Corollary}[section]

\newtheorem{remark}{Remark}[section]
\newtheorem{remarks}{Remark}[section]
\newtheorem{definition}{Definition}[section]
\newcommand{\be}{\begin{equation}}
\newcommand{\ee}{\end{equation}}

\newcommand{\e}{\varepsilon}

\newcommand{\R}{\mathbb R}
\newcommand{\C}{\mathbb C}

\newcommand{\Z}{\mathbb Z}

\newcommand{\N}{\mathbb N}
\newcommand{\T}{\mathbb T}

\newcommand{\s }{\sigma }
\newcommand{\ii }{{\rm i} }

\newcommand{\vphi}{\varphi }

\newcommand{\pa}{\partial}

\begin{document}

\title{{\bf On the stability of periodic multi-solitons of the KdV equation
}}


\author{
Thomas Kappeler\footnote{Supported in part by the Swiss National Science Foundation.}  ,
Riccardo Montalto\footnote{Supported in part by the Swiss National Science Foundation and INDAM-GNFM.} 
}

\date{March 28, 2021}

\maketitle

\noindent
{\bf Abstract.}
In this paper we obtain the following stability result for periodic multi-solitons of the KdV equation: We prove that under any given semilinear Hamiltonian perturbation 
of small size $\varepsilon > 0$, a large class of periodic multi-solitons of the KdV equation, including ones of large amplitude, 
are orbitally stable for a time interval of length at least $O(\varepsilon^{-2})$. To the best of our knowledge, this is the first stability result of such type 
for periodic multi-solitons of large size of an integrable PDE.

\medskip


\medskip


\medskip



\tableofcontents

\section{Introduction}\label{introduzione paper}
\label{1. Introduction}

The Korteweg-de Vries (KdV) equation
 \begin{equation} \label{KdV} 
   \partial_t u = -  \partial_x^3 u + 6 u \partial_x u  
   \end{equation}
   is one of the most important model equations for describing dispersive phenomena.
   It is named after the two Dutch mathematician Korteweg and de Vries \cite{KdV} 
   (cf. also Boussinesq \cite{Bouss}, Raleigh \cite{Ra})
   and originally was proposed as a model equation in one space dimension for long surface waves of water in a narrow and shallow channel.
   Today it is used in many
   branches of physics as well as in the engineering sciences. The seminal discovery in the late sixties
   that \eqref{KdV} 
   admits infinitely many conservation laws (\cite{Lax}, \cite{MGK}), and the development 
   of the inverse scattering transform method (\cite{GGKM}) led to the modern theory of 
   integrable systems of finite and infinite dimension (see e.g. \cite{DKN}, \cite{FT}, and references therein).
   More recently, as one of the most prominent examples among dispersive equations, \eqref{KdV} played a major role 
   in the development of the theory of dispersive PDEs to which many of the leading analysts of our times contributed.
   In particular, the (globally in time) well-posedness theory of \eqref{KdV} has been established in various setups in great detail -- see \cite{Wiki}.
   
A distinguished feature of equation \eqref{KdV} is the existence of sharply localized traveling wave solutions
of arbitrarily large amplitude and particle like properties. Kruskal and Zabusky, who studied them in numerical experiments 
in the early sixties
(cf. \cite{KZ}), coined the name 'soliton' for them. More generally, they found solutions, which are localized 
near finitely many points in space, referred to as multi-solitons.
In the periodic setup, these solutions often are referred to as {\em periodic multi-solitons} or {\em finite gap solutions}.
Due to their importance in applications, various stability aspects have been considered such as 
the long time asymptotics of solutions with initial data near (periodic) multi-solitons
(orbital stability, soliton resolution conjecture). 
Two major questions arise in connection with the {\em structural stability} of \eqref{KdV}. 
One of them concerns the persistence of the (periodic) multi-solitons under perturbations of \eqref{KdV},
and the other one concerns the long time asymptotics of solutions of perturbations of \eqref{KdV}
with initial data close to a (periodic) multi-soliton. 
In the periodic setup, the first question has been studied quite extensively by developing KAM methods, pioneered 
by Kolmogorov, Arnold, and Moser to treat perturbations of finite dimensional integrable system, for PDEs 
(cf. \cite{K87}, \cite{K2-KdV}, \cite{K}, \cite{Wayne}, \cite{B-KAM}, \cite{KP}, \cite{CY}, \cite{LY}, \cite{PP}, \cite{BBM3}, \cite{BKM}, and references therein), 
whereas the second one turned out to be quite challenging and little is known so far. Our goal is to
address this longstanding open problem.

\medskip

The aim of this paper is to study in the periodic setup the long time asymptotics of the solutions of Hamiltonian perturbations of \eqref{KdV}
with initial data close to a periodic multi-soliton of arbitrary large amplitude.
To describe the class of perturbations considered, recall that \eqref{KdV} with the space periodic
variable $x \in \T_1:= \R / \Z$ can be written in Hamiltonian form
\begin{equation}\label{Ham-KdV}
\partial_t u = \partial_x \nabla H^{kdv}(u)\ , \qquad H^{kdv}(u) := \int_0^1 \big( \frac{1}{2} (\partial_x u)^2 + u^3  \big) d x \ ,
\end{equation}
where $\nabla H^{kdv}(u)$ denotes the $L^2-$gradient of $H^{kdv}$ and where $\partial_x$ is the Poisson structure,
corresponding to the Poisson bracket defined for functionals $F$, $G$ by
\begin{equation}\label{Poisson}
\{ F, G \} (u) = \int_0^1 \nabla F \partial_x \nabla G d x.
\end{equation}
We consider semilinear Hamiltonian perturbations of \eqref{KdV} of the form
 \begin{equation} \label{1.1} 
   \partial_t u = -  \partial_x^3 u + 6 u \partial_x u + \e F(u)  
   \end{equation}
     where $0 < \e <1$ is a small parameter and $F$ is a semilinear Hamiltonian vector field
     \begin{equation}\label{cal N def intro}
     F(u) =  \partial_x \nabla {P_f}(u).
     \end{equation}
     Here $P_f$ is a Hamiltonian of the form
     \begin{equation}\label{cal K def intro}
     { P_f}(u) := \int_0^1 f(x, u(x))\, d x 
     \end{equation}
     and  $f$ a $C^{\infty}-$smooth density
     \begin{equation}\label{f def intro}
     f : \T_1 \times \R \to \R, \quad (x, \zeta) \mapsto f(x, \zeta)  ,
     \end{equation}
     so that with $f'(x, \zeta) := \partial_\zeta f(x, \zeta)$ and $f''(x, \zeta) := \partial^2_\zeta f(x, \zeta)$,
     $$
     F(u)(x) = \partial_x \nabla P_f(u)(x) = \partial_x f'(x, u(x)) + f''(x, u(x))\partial_x u(x).
     $$
     
%
 To state our main results, we first need to introduce some more notations. Since $u \mapsto \langle u \rangle_x := \int_0^1 u \, d x$ is 
 a Casimir for the Poisson bracket \eqref{Poisson} and hence a prime integral of \eqref{1.1}, 
 we restrict our attention to spaces of functions with zero mean (cf. \cite{KP}, Section 13) and choose as phase spaces of \eqref{1.1}
 the scale of Sobolev spaces $H^s_0(\T_1)$, $s \in \Z_{\ge 0}$,
 $$
 H^s_0(\T_1) := \{ q \in H^s(\T_1) \ : \ \int_0^1 q(x) d x = 0  \} , 
 \qquad L^2_0 (\T_1) \equiv H^0_0(\T_1) , 
 $$
 where
     \begin{equation}\label{def Hs intro}
   H^s(\T_1) \equiv H^s(\T_1, \R) : = \big\{ q = \sum_{n \in \Z} q_n e^{2 \pi \ii n x} \ :  \,
  q_n \in \C, \ q_{-n} = \overline q_n \ \forall n \in \Z, \   \| q \|_s < \infty \big\} \,,
  \end{equation}
  and 
  $$
   \| q \|_s = \big( \sum_{n \in \Z} \langle n \rangle^{2 s} |q_n|^2 \big)^{\frac12} \ , \qquad 
  \langle n \rangle := {\rm max}\{1, |n| \} \ , \quad \ \forall \ n \in \Z \ . 
  $$
  On $L^2_0(\mathbb T_1)$, the Poisson structure $\partial_x$ is nondegenerate and the corresponding symplectic form is given by 
\be\label{KdV symplectic}
{\mathcal W}_{L^2_0} (u, v) := 
\int_0^1 (\partial_x^{- 1} u ) v\, d x \, , \qquad 
\partial_x^{- 1} u = 
\sum_{n \ne 0} \frac{1}{\ii n} u_n e^{\ii 2 \pi n x}\, ,
\qquad
\forall u, v \in L^2_0(\T_1)\, .
\ee
Note that the Hamiltonian vector field $ X_H (u) = \pa_x \nabla H (u) $, associated with the Hamiltonian $ H $, satisfies
$ d H (u)[ \cdot ] = {\mathcal W}_{L^2_0} ( X_H , \cdot ) $. 

%

\smallskip

Our results can informally be stated as follows:  for any $f \in C^{\infty}(\T_1 \times \R)$, 
$s$ sufficiently large, $ \varepsilon > 0$ sufficiently small, and for most of the finite gap solutions $q: t \mapsto q(t, \cdot)$ of \eqref{KdV}, the following holds:
for any initial data  $u_0 \in H^s_0(\T_1)$, which is $\varepsilon$-close in $H^s_0(\T_1)$ to the orbit $\mathcal O_q := \{ q(t, \cdot) : \, t \in \R \}$ of $q$,
the perturbed equation \eqref{1.1} admits a unique solution $t \mapsto u(t, \cdot)$ in $H^s_0(\T_1)$ with initial data $u(0, \cdot) = u_0$ and life span at least
$[- T, T]$, $T = O(\varepsilon^{-2})$. The solution $u(t, \cdot)$ stays $\varepsilon$-close in $H^s_0(\T_1)$
to the orbit $\mathcal O_q $.

\smallskip
%
%

To state our results in precise terms, we need to define the notion of finite gap solution and the invariant tori, on which they evolve, 
and explain for which of these solutions the above
stability results hold.  Since these finite gap solutions are not small, we need to introduce coordinates to describe them.
Most conveniently, this can be done in terms of a Euclidean version of
action angle coordinates, referred to as Birkhoff coordinates. Let us now explain this in detail.

  According to \cite{KP}, the KdV equation \eqref{Ham-KdV} on $\T_1$ is an integrable PDE in the strongest possible sense,
  meaning that it admits globally defined canonical coordinates on $L^2_0(\T_1)$ so that 
  when expressed in these coordinates, \eqref{Ham-KdV} can be solved by quadrature. 
  To describe these coordinates in more detail, we introduce for any $s \in \Z_{\ge 0}$ the weighted $\ell^2-$sequence spaces
   $$ 
   h_0^s := \big\{ (w_n)_{n \neq 0} \in h^s_{0, c} \ : \ w_{- n} = \overline{w}_n \,\, \forall n \geq 1 \big\} , \qquad  \ell^2_0 \equiv h^0_0,
   $$ 
   where $  h^s_{0, c} \equiv h^s(\Z \setminus \{ 0 \}, \C) $ is given by
   $$
   h^s_{0,c} : = \big\{ w = (w_n)_{n \neq 0} : w_n \in \C \,\,  \forall n \ne 0, \ \| w \|_s < \infty \big\}  , 
   \qquad 
     \| w \|_s := \big( \sum_{n \neq 0} |n|^{2 s} |w_n|^2 \big)^{\frac12} .
   $$
   By \cite{KP} there exists a real analytic diffeomorphism, referred to as (complex) Birkhoff map,
   $$
   \Phi^{kdv} : L^2_0(\T_1)  \to \ell^2_0, \quad q \mapsto w(q) := (w_n(q))_{n \neq 0}\,,
   $$
  which is canonical in the sense that 
   \begin{equation}\label{w_n canonical}
   \{ w_n, w_{- n} \} = \int_0^1 \nabla w_n \partial_x \nabla w_{- n}\, d x = 2 \pi \ii n, \quad \forall n \neq 0\,,
   \end{equation}
   whereas the brackets between all other coordinate functions vanish, and which has the property that for any $s \in \N$, 
   the restriction of $\Phi^{kdv}$ to $H^s_0(\T_1)$ is a real analytic diffeomorphism with range $h^s_0$, $\Phi^{kdv} : H^s_0(\T_1)  \to h^s_0$,
   so that the KdV Hamiltonian, when expressed in the coordinates $w_n,$ $n \ne 0,$ is in normal form. More precisely, 
   $$
    {H}^{kdv} \circ \Psi^{kdv} : h^1_0 \to \R\,, \qquad \Psi^{kdv} := (\Phi^{kdv})^{- 1}\,,
   $$
   is a real analytic function $ {\cal H}^{kdv }$ of the actions $I(w)= (I_n(w))_{n \ge1}$ alone,
   $$
   {\cal H}^{kdv } :\ell^{1, 3}_+ \to \R, \ I \mapsto  {\cal H}^{kdv }(I), \qquad
   I_n(w) := 2 \pi n w_n w_{- n}, \   \forall n \ge 1,
   $$
   where $\ell^{1,3}_+$ denotes the positive quadrant of the weighted $\ell^1-$sequence space,
   $$
   \ell^{1,3} \equiv \ell^{1,3}(\N, \R) := \{ I=(I_n)_{n \geq 1} \subset \R \,  : \, \sum_{n =1}^\infty n^3 |I_n| < \infty \} \, , \qquad  \N:= \Z_{\ge1}\,  .
   $$
   Equation \eqref{Ham-KdV}, when expressed in the coordinates $w_n$, $n \ne 0$, then takes the form
   \begin{equation}\label{KdV in Birkhoff}
   \dot w_n = \ii \omega_n^{kdv}(I)  w_n\,, \qquad \forall n \ne 0,
   \end{equation}
   where $\omega_n^{kdv}(I)$, $n \ne 0$, denote the KdV frequencies 
   \begin{equation}\label{kdv frequencies}
   \omega_n^{kdv}(I) := \partial_{I_n} {\cal H}^{kdv}(I) \ , \quad  \omega_{-n}^{kdv}(I)  := - \omega_n^{kdv}(I) , \qquad \forall n \ge 1.
   \end{equation}
   Since by \eqref{w_n canonical}  the action variables Poisson commute, $\{ I_n, I_m \}$, $\forall n, m \ge 1$, it follows
   that they are prime integrals of \eqref{Ham-KdV} and so are the frequencies $  \omega_n^{kdv}(I)$, $n \ne 0$.
   As a consequence, \eqref{KdV in Birkhoff} can be solved by quadrature.
   Finally,  the differential 
   $d_0 \Phi^{kdv} : L^2_0(\T_1)  \to \ell^2_0$ of $\Phi^{kdv}$ at $q = 0$  is the Fourier transform (cf. \cite{KP}, Theorem 9.8)
   $$
   {\cal F} : L^2_0(\T_1)  \to \ell^2_0,  \quad q \mapsto (q_n)_{n \neq 0}, \quad q_n := \int_0^1 q(x) e^{- 2 \pi \ii n x}\, d x , 
   $$
   and hence $d_0 \Psi^{kdv}$ is given by the inverse Fourier transform ${\cal F}^{- 1}$. 
   We remark that the coordinates $w_{\pm n} \equiv w_{\pm n}(q)$, referred to as (complex) Birkhoff coordinates,
   are related to the (real) Birkhoff coordinates  $x_n,$ $y_n$, $n \ge 1$, introduced in \cite{KP}, by
   $$
   x_n = \frac{w_n + w_{-n}}{2 \sqrt{n \pi}} ,   \quad   y_n = \ii \frac{w_n - w_{-n}}{2 \sqrt{n \pi}} , \qquad \forall \ n \geq 1\,, 
   $$
   where $\sqrt{\cdot}$ denotes the principal branch of the square root, $\sqrt{\cdot} \equiv \sqrt[+]{\cdot}\,$.
   
   The Birkhoff coordinates are well suited to describe the finite gap solutions of \eqref{Ham-KdV}.
   For any {\em finite} subset $S_+ \subseteq \N $, let
  
   $$
    S := S_+ \cup (- S_+)\, ,  \qquad S^\bot := \Z \setminus (S \cup  \{ 0 \})\,.
   $$
   We denote by $M_S$ the submanifold of $L^2_0(\T_1) $, given by
   $$
   M_S := \big\{ q = \Psi^{kdv}(w) \  : \ w_n(q) = 0 \  \ \forall \, n \in S^\bot  \big\}, \quad
   $$
   whose elements are referred to as $S$-gap potentials, 
   and by $M_S^o$ the open subset of $M_S$, consisting of the so called proper $S$-gap potentials,
   $$
   M_S^o := \{ q \in M_S \  :  \, w_n(q)  \ne 0 \  \ \forall \, n \in S \}\,.
   $$
   Note that $M_S$ is contained in $\cap_{s \geq 0} H^s_0(\T_1)$ and hence consists of $C^\infty$-smooth potentials  
   and that $M_S^o$ can be parametrized by the action-angle coordinates 
   $ \theta = (\theta_k)_{k \in S_+} \in  \T^{S_+},$ and $I = (I_k)_{k \in S_+} \in \R^{S_+}_{> 0}$,
   $$
  \Psi_{S_+}  : {\cal M}_S^o := \T^{S_+} \times \R^{S_+}_{> 0} \to M_S^o, \,\, (\theta, I) \mapsto  \Psi_{S_+}(\theta, I) := \Psi^{kdv}(w(\theta, I))
   $$
   where $\T := \R / 2 \pi \Z$ and $w(\theta, I) = (w_n(\theta, I))_{n \ne 0}$ is defined by 
   \begin{equation}\label{coo for S_+}
    w_{\pm n} := \sqrt{ I_n / (2 \pi n)} e^{\mp \ii \theta_n}, \quad \forall n \in S_+ ,  \qquad  \quad
   w_n : = 0, \quad \forall n \in S^\bot \,.
   \end{equation}
 Introduce 
   $$
   h^s_\bot := \big\{ w \in h^s_{\bot c} : w_{- n} = \overline{w}_n \,\, \forall n \in S^\bot \big\}, \qquad h^s_{\bot c} := h^s(S^\bot, \C)\,.
   $$
   For notational convenience, we view ${\cal M}_S^o \times h^s_\bot$ as a subset of $h^s_0$.
   Its elements are denoted by
   $$
\quad \theta = (\theta_n)_{n \in S_+}, \,\,\, I = (I_n)_{n \in S_+}, \,\,\, w = (w_n)_{n \in S^\bot}
   $$
  and it is endowed with the canonical Poisson bracket, given by 
   $$
   \{ I_n, \theta_n \} = 1, \quad \forall n \in S_+, \qquad \{ w_n, w_{- n} \} =  \ii 2 \pi n, \quad \forall n \in S^\bot_+ := S^\bot \cap \N \, , 
   $$
   whereas the brackets between all other coordinate functions vanish.
 It is convenient to introduce the frequency vector $\omega(I)$ (cf. \eqref{kdv frequencies}),
\begin{equation}\label{normal frequencies}
\omega(I):= (\omega^{kdv}_n(I , 0))_{n \in S_+}\,.  
\end{equation} 
By \cite{BK}, the action to frequency map $\omega: \R^{S_+}_{> 0} \to \R^{S_+},$ $I \mapsto \omega(I)$, is a local diffeomorphism.
Throughout the paper, we denote by $\Xi \subset \R^{S_+}_{> 0}$ the closure of a bounded, open, nonempty set so that the restriction
of $\omega$ to $\Xi$ is a diffeomorphism onto its image $\Pi := \omega(\Xi)$ and so that for some $\delta > 0,$
$$
\Xi + B_{S_+}(\delta) \subset \R^{S_+}_{> 0} ,
$$ 
where $B_{S_+}(\delta)$ is the ball in $\R^{S_+}$ of radius $\delta > 0$, centered at the origin. 
We remark that for any $I \in \Xi + B_{S_+}(\delta)$,  the nth action $I_n = I_n(w)$, $n \in S_+$, is of the form 
$ I_n(w) =  I_n^{(0)} +y$ where $I_n^{(0)}:= 2 \pi n w^{(0)}_n w^{(0)}_{- n} \in \Xi$ and
\begin{equation}\label{formula y}
y_n = (w_n - w_n^{(0)})w_{-n}^{(0)} + w_n^{(0)} (w_{-n} - w_{-n}^{(0)}) + (w_n - w_n^{(0)})(w_{-n} - w_{-n}^{(0)}) \, .
\end{equation}
The inverse of $\omega : \Xi \to \Pi$ is denoted by $\mu$,
$$
\mu: \Pi \to \Xi, \quad \omega \mapsto \mu(\omega)\,.
$$
 In what follows, we will consider the frequency vector  $\omega$ as a parameter.
For any $\omega \in \Pi,$ a $S-$gap solution of \eqref{Ham-KdV} is defined as a solution of the form
\begin{equation}\label{finite gap in coordinate originarie}
q(t, x; \omega ) = \Psi_{S_+}(\theta^{(0)} + \omega  t, \mu(\omega))(x) \ , \qquad  \theta^{(0)} \in \T^{S_+} , 
\end{equation}
whereas a finite gap solution of \eqref{Ham-KdV} is a solution of the form \eqref{finite gap in coordinate originarie}
for some  $S = S_+ \cup (- S_+)$ with $S_+ \subset \N$ finite.
The $S-$gap solution $t \mapsto q(t, x; \omega )$ is a curve on the $|S_+|-$dimensional torus 
$$
{\frak T}_{\mu(\omega)} := \Psi_{S_+}\big(\T^{S_+} \times \{\mu(\omega)\} \big).
$$
We note that ${\frak T}_{\mu(\omega)}$ is invariant under \eqref{Ham-KdV} and Lyapunov stable 
in $H^s_0(\T_1)$ for any $ s \ge 0$. More precisely, 
for any $\e > 0$ there exists $\delta > 0$, depending on $s$, so that for any initial data
$u_0 \in H^s_0(\T_1)$ with 
\begin{equation}\label{distance}
{\rm dist}_{H^s} \big(u_0, {\frak T}_{\mu(\omega)} \big)  \leq \delta \ ,
\qquad {\rm dist}_{H^s} \big(u_0, {\frak T}_{\mu(\omega)} \big) := \inf_{q \in {\frak T}_{\mu(\omega)} } \| u_0 - q\|_s \ , 
\end{equation}
the solution $u(t, \cdot)$ of \eqref{Ham-KdV} with  $u(0, \cdot) = u_0$ satisfies
$$
 {\rm dist}_{H^s} \big(u(t, \cdot), {\frak T}_{\mu(\omega)} \big)   \leq \e \ , \qquad \forall \ t \in \R .
$$
Finally, we introduce the so called normal frequencies, 
\begin{equation}\label{def notation normal frequencies}
\Omega_j(\omega) := \omega^{kdv}_j(\mu(\omega), 0), \quad j \in S^\bot , \  \omega \in \Pi \, ,
\end{equation}
and for any given  $\tau > |S_+| $, the subsets $\Pi_\gamma$ of $\Pi$, 
\begin{equation}\label{non resonant set tot}
\Pi_\gamma := \cap_{i = 0}^3 \Pi_\gamma^{(i)} , \qquad  0 < \gamma <1 \ ,
\end{equation}
where $\Pi_\gamma^{(i)}$, $0 \le i \le 3$,  are given by
\begin{equation}\label{condizioni forma normale}
\begin{aligned}
\Pi_\gamma^{(0)} & := \big\{ \omega \in \Pi \ : \ |\omega \cdot \ell| \geq \frac{\gamma}{\langle \ell \rangle^\tau} \ \  \forall \ell \in \Z^{S_+} \setminus \{0\}\big\}\,, \\
\Pi^{(1)}_\gamma & := \big\{ \omega \in \Pi \ : \ |\omega  \cdot \ell + \Omega_j(\omega)| \geq \frac{\gamma}{\langle \ell \rangle^\tau} \ \  \forall (\ell, j)  \in \Z^{S_+} \times S^\bot \big\}\,, \\
\Pi_\gamma^{(2)} & := \big\{ \omega \in \Pi \ : \ |\omega \cdot \ell + \Omega_{j_1}(\omega) + \Omega_{j_2}(\omega)| 
\geq \frac{\gamma}{\langle \ell \rangle^\tau}  \\
& \ \  \forall (\ell, j_1, j_2) \in \Z^{S_+} \times S^\bot \times S^\bot \ \text{with} \ (\ell, j_1, j_2) \neq (0, j_1, - j_1) \big\}\,, \\
\Pi^{(3)}_\gamma & := \big\{ \omega \in \Pi  \ : \ |\omega  \cdot \ell + \Omega_{j_1}(\omega) + \Omega_{j_2}(\omega) + \Omega_{j_3}(\omega)| \geq \frac{\gamma}{\langle \ell \rangle^\tau \langle j_1 \rangle^2 \langle j_2 \rangle^2 \langle j_3 \rangle^2}  \\
& \ \  \forall (\ell, j_1, j_2, j_3) \in \Z^{S_+} \times S^\bot \times S^\bot \times S^\bot \ \text{with} \  j_k + j_m \neq 0 \ \  \forall k, m \in \{1,2,3\}  \big\}\,. 
\end{aligned}
\end{equation}
Here we used the standard notation for vectors $y$ in $\R^n$,
\begin{equation}\label{def langle y rangle}
\langle  y \rangle := \max\{1, | y |  \} , \quad | y | := (\sum_{j=1}^n |y_j|^2)^{1/2} ,  \qquad  \forall \, y \in \R^n \, .
\end{equation}
We refer to $\Pi^{(j)}_\gamma$, $0 \le j \le 3$, as the {\em jth Melnikov conditions} and note that the third Melnikov conditions allow for 'a loss of derivatives in space'
-- see item $(ii)$ in {\em Comments on Theorem \ref{stability theorem}} below.

The goal of this paper is to prove a long time stability result of finite gap solutions \eqref{finite gap in coordinate originarie}
of the Korteweg-de Vries equation on $\T_1$. To state it, we denote
for any Banach space $X$ with norm $\| \cdot \|_X$, integer $m \ge 0$, and interval $J \subset \R$, by $C^m(J, X)$ the Banach space of
functions $f: J \to X$, which are $m$ times continuously differentiable, endowed with the supremum norm, 
$\|f\|_{C^m_t} := \max_{0 \le j \le m} \sup\{ \| \partial_t^j f(t) \|_X \, : \, t \in J; 0 \le j \le m\}$. 

\begin{theorem}\label{stability theorem}
Let  $f$ be a function in $ C^{\infty}(\T_1 \times \R)$ (cf. \eqref{cal K def intro}), $S_+$ be a finite subset of $\N$, 
and $\tau$ be a number with $\tau > |S_+| $ (cf. \eqref{condizioni forma normale}).
Then for any integer $s$ sufficiently large and any $\omega \in \Pi_\gamma$, $0 < \gamma < 1$, there exists  
$0 < \varepsilon_0 \equiv \varepsilon_0(s, \gamma) < 1$
 with the following properties:
for any  $0 < \varepsilon \le \varepsilon_0$
and  any initial data  $u_0 \in H^s_0(\T_1)$, satisfying 
\begin{equation}\label{initial value}
{\rm dist}_{H^s} \big(u_0, {\frak T}_{\mu(\omega)} \big) \leq \e \ ,
\end{equation}
equation \eqref{1.1} admits a unique solution  $t \mapsto u(t, \cdot)$ in 
$C^0([-T, T], H^s_0(\T_1)) \cap C^1([-T, T], H^{s-3}_0(\T_1))$ with initial data $u(0, x) = u_0(x)$
and $T \equiv T_{\e, s, \gamma} = O(\e^{- 2})$. 
Moreover, $u$ satisfies the estimate 
$$
 {\rm dist}_{H^s} \big(u(t, \cdot), {\frak T}_{\mu(\omega)} \big) \lesssim_{s, \gamma} 
\  \e\, , \qquad \forall \ - T \le  t \le T \ ,
$$
where the distance function ${\rm dist}_{H^s}$ is defined in \eqref{distance}.
Furthermore, there exists  $0 < \mathtt a < 1$ so that for any $0 < \gamma < 1,$
 the Lebesgue measure  $|\Pi \setminus \Pi_\gamma |$  
of $ \Pi \setminus \Pi_\gamma$ satisfies 
\begin{equation}\label{main measure estimate}
|\Pi \setminus \Pi_\gamma| \lesssim \gamma^{\mathtt a} \ , \quad {\rm implying \  that } \quad
\lim_{\gamma \to 0}|\Pi_\gamma| = |\Pi| \ . 
\end{equation}
Here and in the sequel, the notation $ h \lesssim_{\alpha, \ldots} g$ means that the real valued function $h$, depending on various variables,
satisfies an estimate of the form $h \le C g$ where $g$ is also a real valued function, typically small, and the constant $C>0$ only depends on the parameters $\alpha, \ldots$.
For notational convenience, the dependence of the constant $C$ on $f$, $S_+$, and $\tau$ is not indicated.
\end{theorem}

\smallskip

\noindent
{\bf Comments on Theorem \ref{stability theorem}}\label{comments}

\smallskip

\noindent
$(i)$ {\em Initial data.} Note that the size of the distance of the initial value $u_0$
to  the considered $S-$gap solution of the KdV equation (cf. \eqref{initial value})  
is assumed to be of the same order of magnitude as the size of the perturbation $\e F(u)$ in \eqref{1.1}.

\smallskip

\noindent
$(ii)$ {\em Measure estimate \eqref{main measure estimate}.}
The proof of the measure estimates \eqref{main measure estimate} requires that the third Melnikov conditions $\Pi^{(3)}_\gamma $ in \eqref{condizioni forma normale}
allow for a loss of derivatives in space. Furthermore, a key ingredient into the proof of \eqref{main measure estimate} is the case $n=3$ 
of Fermat's Last Theorem, proved by Euler \cite{Euler} (cf. Lemma \ref{Lemma three wave res}). 

\smallskip

\noindent
$(iii)$ {\em Assumptions in Theorem  \ref{stability theorem}.} The results of Theorem \ref{stability theorem} hold for any density $f(x, \zeta)$ of class $\mathcal C^{\sigma}$ with $\sigma$ sufficiently large. 
Furthermore, corresponding results hold for (invariant tori of) finite gap solutions of the KdV equation in the affine spaces $c + H^s_0(\mathbb T_1)$, $c \in \R$. 
We assume in this paper that $f$ is $\mathcal C^\infty-$smooth and that $c=0$ merely to simplify the exposition. 

In order to limit the size of the paper, we assume the perturbation $\e F(u)$ to be semilinear (cf. \eqref{cal N def intro}),  leaving the case of a quasilinear one for future work. 
Most likely, the elaborate method designed in \cite{Feola-Iandoli} will allow to transform quasilinear perturbations into normal form while preserving the Hamiltonian structure
of the equation.

\smallskip

\noindent
$(iv)$ {\em Time of stability.}  It seems unlikely that the stability results of Theorem \ref{stability theorem} in the generality stated are valid for time intervals of size larger than $O(\e^{-2})$
since the conditions, required to hold for the frequencies $\Omega_j$, $j \in S^\bot$, so that the normal form procedure could be implemented, are too strong.
See Remark \ref{remark on Pi_gamma^4} at the end of Section \ref{measure estimates}.
Actually, it might be possible that the (almost) resonances of the KdV frequencies of degree four can be used to prove instability results for solutions of the perturbed equation \eqref{1.1} 
-- see \cite{CKSTT}, \cite{GHHMP} and references therein for related results for Schr\"odinger equations in two space dimension.

\smallskip

\noindent
$(v)$ {\em Conservation of momentum.} If the density $f$ of the perturbation ${P_f}(u) = \int_0^1 f(x, u(x))\, d x$
does not explicitely depend on $x$, then the momentum $M(u) := \frac12 \int_{\T_1} u^2\, d x$ is a prime integral of equation \eqref{1.1}. 
We plan to prove in future work that the stability time can be improved in such a case. 

\smallskip

\noindent 
$(v)$ {\em Integrable PDEs.} The method of proof of Theorem \ref{stability theorem} is quite general. We expect that for any integrable PDE,
admitting coordinates of the type constructed in \cite{Kap-Mon-2}, a corresponding version of Theorem \ref{stability theorem} holds,
up to the measure estimates related to the nonresonance conditions for the frequencies of the integrable PDE considered.
These estimates might require specific arithmetic properties of the frequencies -- see item $(ii)$ above.


\medskip

To explain the main ideas of the proof, we first need to introduce some terminology and additional notations.
They will be used throughout the paper.

\smallskip

\noindent
{\em Notations and terminology.} 
 For any finite subset $S_+  \subset \N$, $L^2_\bot(\T_1)$ is the subspace, given by
\begin{equation}\label{Def:L2bot}
L^2_\bot(\T_1) := \big\{ w = \sum_{n \in S^\bot }  w_n e^{\ii 2\pi n x} \in L^2_0 (\T_1) \big\}\, , \qquad 
S^\bot  =  \Z \setminus \big( S_+ \cup (- S_+) \cup \{ 0 \} \big) \, ,
\end{equation}
and $\Pi_\bot$ denotes the $L^2-$orthogonal projector onto the subspace $ L^2_\bot (\T_1) $. For any $s > 0$, we set
\be\label{Hsbot}
H^s_\bot(\T_1) := H^s(\T_1) \cap L^2_\bot(\T_1), \qquad  H_\bot^0(\T_1) := L^2_\bot( \T_1)\, .
\ee
By $ {\mathcal E}_s$ we denote the phase space and by $E_s$ the corresponding tangent space, given by
\be\label{EsEs}
 {\mathcal E}_s := \T^{S_+} \times \R^{S_+} \times H^s_\bot(\T_1)\,, \quad 
{\mathcal E} \equiv {\mathcal E}_0 \, , \qquad \quad
E_s := \R^{S_+} \times \R^{S_+} \times H^s_\bot(\T_1)\,, \quad E \equiv E_0\, ,
\ee  
where $\T_1 = \R / \Z$ and $\T = \R / 2\pi \Z$.
Elements of ${\mathcal E}$ are denoted by $\frak x = (\theta, y , w)$
and the ones of its tangent space $E$ by 
$ \widehat{ \frak x} = (\widehat \theta, \widehat y,\widehat w)$.
For $s > 0$, $ H^{s}_\bot (\T_1)^*$ denotes the dual space of $ H^{s}_\bot (\T_1)$,
which is canonically identified with the Sobolev space $ H^{-s}_\bot (\T_1)$  of distributions.
The spaces
$ {\mathcal E}_{-s} $ and $ E_{-s}$  are then defined as in \eqref{EsEs}. 
On $E$, we denote by 
$ \langle \cdot, \cdot \rangle_E$ the inner  product defined  
by  
\be\label{bi-form}
\big\langle (\widehat \theta_1, \widehat y_1, \widehat w_1), (\widehat \theta_2, \widehat y_2, \widehat w_2)  \big\rangle_E := \widehat \theta_1 \cdot \widehat \theta_2 + \widehat y_1 \cdot \widehat y_2  + \big\langle \widehat w_1, \widehat w_2 \big\rangle  \, 
\ee
where $ \langle \cdot , \cdot \rangle $ is the standard real scalar product on $L^2_\bot$. For notational convenience, $\Pi_\bot$ 
also denotes the projector of $E_s$  onto its third component,
$$
\Pi_\bot : E_s  \to H^s_\bot(\T_1) \, , \, \quad 
(\widehat \theta, \widehat y, \widehat w) \mapsto \widehat w\, .
$$
For any $0 < \delta < 1$, we denote by
$B_{S_+}(\delta)$ the open ball in $\R^{S_+}$ of radius $\delta$  centered at $0$ and by $B_\bot^s(\delta)$, $s \ge 0$, the corresponding one in $H^s_\bot(\T_1)$.
For $s=0$, we also write  $B_\bot(\delta)$ instead of $B^0_\bot(\delta)$. These balls are used
to define the following open neighborhoods in $\mathcal E_s$, 
$s \ge 0$,
\be\label{Vns}
{\cal V}^s(\delta) := \T^{S_+}_1 \times  B_{S_+}(\delta) \times 
B_\bot^s(\delta)  \,, \qquad  {\cal V}(\delta) \equiv {\cal V}^0(\delta) \, , \qquad 0 < \delta < 1\,  .
\ee
For notational convenience, often without stating it explicitly, $\delta > 0$ will take on different values
in the course of our arguments. In particular, $\delta > 0$ typically will depend on $s$. 
(Note that by \eqref{formula y}, the coordinates $y= (y_n)_{n \in S_+}$ are of the same order as the coordinates $w= (w_n)_{n \in S^\bot}$.) 

For any $k \ge 1, $ $\partial_x^{-k} : L^2(\T_1) \to L^2_0(\T_1)$ is  the linear operator, defined by
$$
\partial_x^{-k}[e^{2\pi \ii nx}] = \frac{1}{(2\pi \ii n)^k} e^{2\pi \ii nx}\, , \quad \forall n \ne 0\,,  
\qquad \mbox{and} \qquad   \partial_x^{-k}[1] = 0\,.
$$   
The space $ {\cal V}^s(\delta)$ is endowed with the symplectic form 
 \begin{equation}\label{2form}
{\cal W} := \big( {\mathop \sum}_{j \in S_+} d y_j \wedge d \theta_j \big)  \oplus 
{\cal W}_\bot  
\ee
where $ {\cal W}_\bot $ is the restriction to  $ L^2_\bot(\T_1) $ of the symplectic form 
$ {\cal W}_{L^2_0} $ defined  in \eqref{KdV symplectic}. Throughout the paper, the Hamiltonians considered depend 
on the small parameter $\e \in [0, \e_0]$, $0 < \varepsilon_0 < 1$, and are $C^\infty$-smooth maps,
${\cal V}^s(\delta) \times [0, \e_0] \to \R$. Given such a Hamiltonian $H$, we often do not indicate the dependence of $H$ on the parameter $\e$.
The Hamiltonian vector field of $H$ is denoted by $X_H$. It is given by
\begin{equation}\label{campo hamiltoniano notazioni}
X_H (\frak x) = {\cal J} \nabla H (\frak x) = 
\big( - \nabla_y H(\frak x), \,  \nabla_\theta H(\frak x), \, \partial_x \nabla_\bot H(\frak x) \big)
\end{equation}
where $\mathcal J$ is the Poisson structure, associated to the symplectic form $\cal W$,
\begin{equation}\label{Poisson struture}
\mathcal J : E_{s } \to E_{s - 1} \, , \quad 
(\widehat \theta,  \widehat y,  \widehat w) \mapsto 
(- \widehat y,   \widehat \theta,  \partial_x \widehat w) \, 
\end{equation}
and where $\nabla_\bot H(\frak x) \equiv \nabla_w H(\frak x)$ denotes the $L^2-$gradient of $H$ with respect to the variable $w$.
For notational convenience, we denote by $ \{ F, G \}$ the Poisson bracket corresponding to $\mathcal J,$
\begin{equation}\label{def poisson action-angle}
 \{ F, G \} = {\cal W}(X_F, X_G) =  \big\langle \nabla F \,,\, {\cal J} \nabla G \big\rangle_E 
 = - \nabla_\theta F \cdot \nabla_y G + \nabla_y F \cdot \nabla_\theta G + \big\langle \nabla_\bot F\,,\, \partial_x \nabla_\bot G \big\rangle\,. 
 \end{equation}
 Given a Hamiltonian vector field $X_F :  {\cal V}^s(\delta) \times [0, \e_0]  \to E_s$ with Hamiltonian $F$, 
 we denote by $\Phi_F(\tau, \cdot)$  or $\Phi_{X_F}(\tau, \cdot) $
 the flow generated by $X_F$. For the vector fields $X_F$ considered in this paper, there exists $0 < \delta' < \delta $ so that 
 for any $\tau \in [- 1, 1]$,  the flow map $ \mathcal V^s(\delta') \to  \mathcal V^s(\delta)$, $\frak x \mapsto  \Phi_F(\tau, \frak x)$ is well defined.
The Taylor expansion of $\tau \mapsto H \circ \Phi_F(\tau, \frak x)$ at $\tau = 0$ can be computed as
\begin{equation}\label{Lie expansion Hamiltonian}
H \circ \Phi_F(\tau, \frak x) =  H(\frak x) + \tau \{ H, F \}(\frak x) + \tau^2 \int_0^1 (1 - t) \{ \{H, F \}, F \} \circ \Phi_F(t\tau, \frak x)\, d t\,. 
\end{equation}
We will also need to consider $C^\infty$-smooth vector fields, which are not necessarily Hamiltonian,
$$
X = (X^{(\theta)}, \, X^{(y)}, \, X^\bot)  :  {\cal V}^s(\delta) \times [0, \e_0]  \to E_s\, , 
$$
where $X^{(\theta)}$, $X^{(y)}$, and $X^\bot$ are the components of $X$,
$$
X^{(\theta)}, \ X^{(y)} : {\cal V}^s(\delta) \times [0, \e_0]  \to \R^{S_+} \, , \qquad  X^\bot : {\cal V}^s(\delta) \times [0, \e_0] \to H^s_\bot(\T_1)\,. 
$$
The corresponding flow is denoted by $\Phi_X(\tau, \cdot)$. Again we will only consider vector fields $X$ with the property that there exists
$0 < \delta' < \delta$ so that for any $\tau \in [- 1, 1]$, $\Phi_X(\tau, \cdot)$  is well defined on ${\cal V}^s(\delta')$. Given two 
$C^\infty$-smooth vector fields 
$X, Y : {\cal V}^s(\delta) \times  [0, \e_0]  \to E_s$, the commutator $[X, Y]$ is defined as 
\begin{equation}\label{definition notations nonlinear commutators}
[X, Y](\frak x) := d X(\frak x)[Y(\frak x)] - d Y(\frak x)[X(\frak x)]\,. 
\end{equation}
The pull-back of a vector field $X : {\cal V}^s(\delta) \to E_s$ by a $C^\infty$-smooth diffeomorphism $\Phi : {\cal V}^s(\delta') \to {\cal V}^s(\delta)$ 
is defined as,
\begin{equation}\label{def pushforward notations}
\Phi^* X(\frak x) := d \Phi(\frak x)^{ - 1} X(\Phi(\frak x))\, , \qquad \forall \frak x \in  {\cal V}^s(\delta')\, .
\end{equation}
If $ \Phi_\tau(\cdot) \equiv \Phi_Y(\tau, \cdot)$ is the flow of a vector field $Y$, then the Taylor expansion 
of $\tau \mapsto \Phi_\tau^* X (\frak x)$ at $\tau = 0$ reads
\begin{align}\label{espansione commutatori notazioni}
\Phi_\tau^* X (\frak x)  & =   X(\frak x) + \tau \int_0^1 (d \Phi(t\tau, \frak x))^{- 1}[X, Y] (\Phi(t \tau, \frak x))\, d t  \nonumber \\
& = X(\frak x) + \tau [X, Y](\frak x) + \tau^2 \int_0^1 (1 - t) (d \Phi(t\tau, \frak x))^{- 1}[[X, Y], Y] (\Phi(t \tau, \frak x))\, d t \,.
\end{align}
In the case $\tau =1$, we will often write $\Phi_Y^*X$ instead of $\Phi_1^* X$.
Clearly if $X = X_H$, $Y = Y_F$ are Hamiltonian vector fields, then 
$$
[X, Y] = X_{\{ H, F\}}, \quad (\Phi_Y(\tau, \cdot))^* X = X_{H \circ \Phi_Y(\tau, \cdot)}\,. 
$$
Given two linear operators $A, B$, acting on $L^2(\T_1)$ (or $L^2_\bot(\T_1)$),  their commutator is conveniently denoted by $[A, B]_{lin}$,
\begin{equation}\label{def linear commutator}
[A, B]_{lin} = A B - B A\, . 
\end{equation}
Moreover, given a densely defined linear operator $A : L^2_\bot(T_1) \to L^2_\bot(\T_1)$, whose domain contains 
the elements of the Fourier basis $e^{\ii 2 \pi j x}$, $j \in S^\bot$,
we denote by $A_j^{j'}$ or $[A]_j^{j'}$  the (Fourier) matrix coefficients of $A$,
$$
A_j^{j'} := \int_0^1 A[e^{\ii 2 \pi j' x}] e^{- \ii 2 \pi j x}\, dx , \qquad j, j' \in S^\bot\,. 
$$
Given a Banach space $(X, \| \cdot \|_X)$, we denote by $C^\infty_b ({\cal V}^s(\delta) \times [0, \e_0], X)$ the space of $C^\infty$ functions ${\cal V}^s(\delta) \times [0, \e_0] \to X$ with all derivatives bounded. 

\noindent
In our normal form procedure, we need to take into account the order of vanishing with respect to the variables $y$, $w$ and the small parameter $\e$. 
The following definition turns out to be convenient. 
 \begin{definition}\label{def map order p}
 Let $(B, \| \cdot \|_B)$ be a Banach space and $p \in \Z_{\geq 0}$. A $C^\infty$-smooth map 
 $$
 g :  {\cal V}^{s}(\delta) \times  [0, \e_0]  \to B, \ (\frak x, \e) \mapsto g(\frak x, \e)
 $$ 
  is said to be small of order $p$ if for any $ \beta \in \Z_{\ge 0}^{S_+}$ and $k_1, k_2 \in \Z_{\ge 0}$ with $|\beta| + k_1 + k_2 \leq p - 1$
   \begin{equation}\label{estimates map small order p}
  d^{k_2}_\bot\partial_y^\beta  \partial_\e^{k_1} g(\theta, 0, 0, 0)= 0\, , \qquad  \forall \, \theta \in \T^{S_+}\, .
  \end{equation}
  \end{definition}
  Note that if $g$ is small of order $p$, then 
 $$
  \| g ( \frak x, \e) \|_B \lesssim_{g}  ( |y| + \| w \|_s + \e)^p\, , 
  \qquad \forall \,  \frak x = (\theta, y, w) \in {\cal V}^s(\delta), \  \forall \, \e \in [0, \e_0]\, ,
 $$  
 and for any $\alpha \in \Z_{\ge 0}^{S_+}$, $\partial_\theta^\alpha g$ is small of order $p$ as well. 
  
  \smallskip
  
  Given two Banach spaces $(X, \| \cdot \|_X)$, $(Y, \| \cdot \|_Y)$, we denote by ${\cal B}(X, Y)$ the space of bounded linear operators $X \to Y$. 
  If $X = Y$, we write ${\cal B}(X)$ instead of ${\cal B}(X, X)$. Moreover for any integer $p \geq 2$, we denote by ${\cal B}_p(X, Y)$, 
  the space of bounded, $p$-multilinear maps $M : X^p \to Y$, equipped with the standard norm,
  \begin{equation}\label{forme multilineari}
  \| M\|_{{\cal B}_p(X, Y)} := \sup_{\| u_1 \|_{X}, \ldots, \| u_p \|_X \leq 1} \| M[u_1, \ldots, u_p] \|_Y\,, \quad M \in {\cal B}_p(X, Y)\,.
  \end{equation}
  If $X = Y$, we write ${\cal B}_p(X)$ instead of ${\cal B}_p(X, X)$. Furthermore, given open sets $U \subset X$ and $V \subset Y$, we denote
  by $C^\infty_b \big(U, V \big)$ the space of maps $f: U \to V$ which are $C^\infty$-smooth and together with each of its derivatives, bounded.
  
  \medskip

\noindent
{\em Overview of the proof of Theorem  \ref{stability theorem}.}
We prove Theorem \ref{stability theorem} by the means of a normal form procedure. 
A key ingredient are canonical coordinates near a torus ${\frak T}_{\mu(\omega)}$ of arbitrary size, constructed in \cite{Kap-Mon-2}. 
They are obtained by first linearizing the Birkhoff map $\Phi^{kdv}$ at ${\frak T}_{\mu(\omega)}$ and then constructing a symplectic corrector.
The new coordinates yield a family of canonical transformations $\Phi^{kdv}_{\mu}$, parametrized by $\mu \equiv \mu(\omega)$, $\omega \in \Pi$. 
One of the main features of these transformations is that they admit expansions in terms of pseudo-differential operators up to 
a remainder of arbitrary negative order. To prove Theorem  \ref{stability theorem} we then follow a strategy developed in \cite{Berti-Delort}
in the context of water waves.

In a first step, referred to as Step 1, we write the perturbed Hamiltonian $H^{kdv} + \e P_f$  in the new coordinates (cf. Theorem \ref{modified Birkhoff map}). 
More precisely, in Theorem \ref{modified Birkhoff map}, we rephrase \cite[Theorem 1.1]{Kap-Mon-2} in a form taylored to our needs
and in Corollary \ref{espansion hom}, we compute for any given $\mu \equiv \mu(\omega)$, $\omega \in \Pi$, and $\frak x = (\theta, y, w) \in \mathcal V^1(\delta)$ 
the Taylor expansion of ${\cal H}_{\e, \mu} := (H^{kdv} + \e P_f) \circ \Phi^{kdv}_\mu$ 
at $(\theta, 0, 0)$ up to order three in the variables $y$, $w$, and $\e$,
\begin{align}\label{Taylor 1}
& {\cal H}_{\e, \mu}(\theta, y, w)  =  {\cal N}_\mu(y, w)  + {\cal P}_{\e, \mu}(\theta, y, w) \,, \\
& {\cal N}_\mu( y, w)  := \omega \cdot y  + \frac12  \Omega_{S_+}(\omega) [y] \cdot y  + \frac12 \big\langle D_\bot^{- 1} \Omega_\bot(\omega) w\,,\, w \big\rangle  \, ,
\end{align}
where $\Omega_{S_+}(\omega)$ is given by the $S_+ \times S_+$ matrix
$(\partial_{I_j} \omega_i^{kdv}(\mu, 0))_{i, j \in S_+}$  and where
$D^{- 1}_\bot : L^2_\bot(\T_1) \to L^2_\bot(\T_1)$ and  
$\Omega_\bot(\omega) \equiv \Omega_{S^\bot}(\omega): \, L^2_\bot(\T_1) \to L^2_\bot(\T_1)$ are Fourier multipliers in diagonal form, 
\begin{equation}\label{definition Omega bot}
D^{- 1}_\bot [w] : = \sum_{n \in S^\bot}  \frac{1}{2 \pi n} w_n e^{\ii 2\pi n x} \, ,  \qquad 
\Omega_\bot(\omega)[w] := \sum_{n \in S^\bot} \Omega_n(\omega) w_n e^{\ii 2\pi n x}  \, ,
\end{equation}
with $ \Omega_n(\omega)$ given by \eqref{def notation normal frequencies}. 
In order to simplify notation, in the sequel, we often will not indicate the dependence of quantities such as $\mathcal H_{\e, \mu}$, $\mathcal P_{\e, \mu}$, $\Omega_\bot(\omega)$, $\ldots$ \
on $\e$, $\mu \equiv \mu(\omega)$, and $\omega$.

We note that $\Omega_\bot$ is an unbounded operator.
For any $\frak x = (\theta, y, w)$, ${\cal P}(\frak x) $ can be expanded as
\begin{equation}\label{Taylor 2}
{\cal P}(\frak x) = \e \big( {\cal P}_{00}(\theta) + {\cal P}_{1 0}(\theta) \cdot y 
+ \langle {\cal P}_{0 1}(\theta), w \rangle \big) + {\cal P}_{e}(\frak x) \, ,
\end{equation}
where ${\cal P}_{e}(\frak x)$ is small of order three (cf. Definition \eqref{def map order p}). 
The Hamiltonian vector field $X_{\mathcal H}$, associated to ${\cal H}$, is given at any point $\frak x = (\theta, y, w)$ by
\begin{equation}\label{X cal H intro}
X_{\cal H} (\frak x)= \begin{pmatrix}
- \nabla_y {\cal H} (\frak x)\\
 \nabla_\theta {\cal H}(\frak x) \\
\partial_x \nabla_\bot {\cal H}(\frak x)
\end{pmatrix} = \begin{pmatrix}
-  \omega - \Omega_{S_+ }[ y] - \e {\cal P}_{1 0}(\theta) - \nabla_y {\cal P}_{e}(\frak x)  \\
 \e \nabla_\theta \big( {\cal P}_{00}(\theta) + {\cal P}_{1 0}(\theta) \cdot y + \langle {\cal P}_{0 1}(\theta), w \rangle \big)  + \nabla_\theta {\cal P}_{e}(\frak x) \\
\ii \Omega_\bot w +  \e \partial_x {\cal P}_{0 1}(\theta) + \partial_x \nabla_\bot {\cal P}_{e} (\frak x)
\end{pmatrix} \, .
\end{equation}
We also show that the normal component $\partial_x \nabla_\bot {\cal P}_{e}$ of the Hamiltonian vector field $X_{\mathcal P_{e}}$
is the sum of a para-differential vector field of order one (cf. Definition \ref{paradiff vector fields} in Section \ref{section paradiff nonlinear}) 
and a smoothing vector field (cf. Definition \ref{def smoothing vector fields} in Section \ref{section paradiff nonlinear}), i.e.,  for $\frak x = (\theta, y, w)$,  
\begin{equation}\label{expansion P e}
\partial_x \nabla_\bot {\cal P}_{e}(\frak x) = \Pi_\bot \sum_{k = 0}^{N+1} T_{a_{1 - k}(\frak x)} \partial_x^{1 - k} w + 
{\cal R}^\bot_N(\frak x) \, ,
\end{equation}
where for any $0 \le k \le N + 1$, $T_{a_{1 - k}(\frak x)}$ is the operator of para-multiplication 
with $a_{1 - k}(\frak x) \in H^s(\T_1)$ (cf. \eqref{definizione paraprodotto} in Section \ref{para-differential calculus}),
which is small of order one, and 
where $ {\cal R}^\bot_{N}(\frak x)$  is a regularizing vector field, which is small of order two.

In Step 2, we apply a regularization procedure, which conjugates the vector field \eqref{X cal H intro} to another one,
which is a smoothing perturbation of a vector field in diagonal form. 
Since the torus ${\frak T}_{\mu(\omega)}$ in the coordinates $(\theta, y, w)$ is described by 
$\{ y = 0, w = 0 \}$, the variables $y$, $w$ can be used to measure the distance of a solution of the equation 
\begin{equation}\label{main system intro}
\begin{cases}
\partial_t \theta = - \nabla_y {\cal H} \\
\partial_t y =  \nabla_\theta {\cal H} \\
\partial_t w = \partial_x \nabla_\bot {\cal H}
\end{cases}
\end{equation}
from ${\frak T}_{\mu(\omega)}$. Theorem \ref{stability theorem} follows from Theorem \ref{long time ex action-angle} in Section \ref{normal form theorem},
which states that for $\mu$ in a large subset of $\Xi$ and for any initial data $\frak x_0 = (\theta_0, y_0, w_0)$, 
satisfying $|y_0|, \| w_0 \|_s \leq \e$ with $s > 0$ large enough, 
the solution $t \mapsto \frak x(t) = (\theta(t), y(t), w(t))$ of \eqref{main system intro} exists on a time interval of the form $[- T, T]$ with $T \equiv T_{\e, s, \gamma} = O(\e^{- 2})$ and 
$$
|y(t)|, \, \| w(t) \|_s \lesssim_{s, \gamma} \e,  \quad \forall t \in [- T, T]\,. 
$$
We deduce Theorem \ref{long time ex action-angle} from Theorem \ref{teorema totale forma normale} 
and a local existence Theorem (cf. Appendix \ref{appendix loc well posed}), using energy estimates (cf. Section \ref{conclusioni forma normale}).
Theorem \ref{teorema totale forma normale} provides coordinates having the property that the vector field in \eqref{main system intro}, 
when expressed in these coordinates, is a vector field $X = (X^{(\theta)}, X^{(y)}, X^{\bot})$ with the following two features: 
(F1) The $y$-component $X^{(y)}$ of $X$ is small of order three. 
(F2) The normal component $X^\bot(\frak x)$ of $X(\frak x)$ at $\frak x = (\theta, y, w)$ reads 
\begin{equation}\label{campo vettoriale finalissimo intro}
X^\bot(\frak x) = \ii \Omega_\bot w + {\mathtt D}^\bot(\frak x)[w] + \Pi_\bot T_{a(\frak x)} \partial_x w + {\cal R}^\bot(\frak x)\, ,
\end{equation}
where ${\mathtt D}^\bot(\frak x)$ is a skew-adjoint Fourier multiplier of order one (depending nonlinearly on $\frak x$), 
$a(\frak x) \in H^s(\T_1)$ is small of order two, 
and the remainder ${\cal R}^\bot(\frak x)$ is small of order three.
In broad terms, our normal form procedure {\it diagonalizes} the normal component $X^\bot$ of the vector field $X$ up to a term,
which is small of order three and which can be controlled by energy estimates. 
The procedure consists in {\it eliminating}/{\it normalizing} the terms of the Taylor expansion \eqref{Taylor 1} - \eqref{Taylor 2}
of $X_{\cal H}$, which are $p$-homogeneous in $y$, $w$, $\e$ with $0 \le p \le 2$ (cf. Definition \ref{def map order p}). 

Based on the normal form procedure, developed in Section \ref{forma normale smoothing standard} and Section \ref{normalization II},
Theorem \ref{teorema totale forma normale} is proved in Section \ref{conclusioni forma normale}.
In Section \ref{measure estimates} we show that the Lebesgue measure   $|\Pi \setminus \Pi_\gamma |$  
of $ \Pi \setminus \Pi_\gamma$ (cf. \eqref{condizioni forma normale}) satisfies
$|\Pi \setminus \Pi_\gamma| \lesssim \gamma^{\mathtt a} $  for some $0 < \mathtt a < 1$. 
As already mentioned in item (ii) of {\em Comments on Theorem \ref{stability theorem}}, a key ingredient of the proof is the case $n=3$ 
of Fermat's Last Theorem, proved by Euler \cite{Euler} (cf. Lemma \ref{Lemma three wave res}). 
Section \ref{para-differential calculus} and Section \ref{section paradiff nonlinear} are prelimimary where para-differential calculus and 
para-differential vector fields are discussed to the extent needed in the paper.

\smallskip

We finish our overview of the proof of Theorem  \ref{stability theorem} by describing in some more detail the normal form procedure, developed 
in Sections \ref{forma normale smoothing standard} - \ref{normalization II}, to prove Theorem \ref{teorema totale forma normale}. 
In order to setup such a procedure in an effective way, we introduce, in the spirit of \cite{Delort}, \cite{Berti-Delort}, \cite{Feola-Iandoli}, 
various classes of para-differential and smoothing vector fields, 
which possibly depend in a nonlinear fashion on $\frak x = (\theta, y, w)$, and develop a symbolic calculus for them - see Section \ref{section paradiff nonlinear}. 
The order of homogeneity in our symbol classes is computed with respect to $y$, $w$, $\e$ 
where we recall that $y$, $w$ (together with $\theta$) are phase space variables and $\e$ is the perturbation parameter appearing in \eqref{1.1} and \eqref{initial value}. 
Our normal form procedure is split into two steps which we now describe.

\medskip

\noindent
In a first step, presented in Section \ref{forma normale smoothing standard}, we normalize the terms in the Taylor expansion of the Hamiltonian ${\cal H}$,
which are linear with respect to the normal variable $w$ and homogeneous of order at most three in $(y, w, \e)$. 
Equivalently, this means that we normalize the terms in the Taylor expansion of the Hamiltonian vector field $X_{{\cal H}}$ 
which do not contain $w$ and are homogeneous of order at most two. 
This is achieved by a {\it standard normal form procedure} which consists in
 constructing a canonical transformation, given by the time one flow map $\Phi_\mathcal F$ of a Hamiltonian vector field $X_{\mathcal F}$
 with a Hamiltonian $\mathcal F$ of the form
\begin{equation}\label{generatrici pezzi lineari intro}
{\cal F}(\theta, y, w) := {\cal F}_{0}(\theta, y) + \big\langle {\cal F}_{1}(\theta, y), w \big\rangle \, ,
\end{equation}
with the property that $X_{\mathcal F }$
is a  smoothing Hamiltonian vector field (cf. Lemma \ref{prop astratte campi vettoriali smoothing NF}).
Hence its flow is a smoothing perturbation of the identity, 
implying that the Hamiltonian vector field of the Hamiltonian ${\cal H} \circ \Phi_{\cal F}$ 
has a normal component, which is again of the form \eqref{expansion P e}
(cf.  Lemma \ref{push forward smoothing 1}). To construct $\mathcal F$, 
we only need to impose zeroth and first  Melnikov conditions on $\omega$,
i.e., $\omega \in \Pi_\gamma^{(0)} \cap \Pi_\gamma^{(1)}$  (cf.  \eqref{condizioni forma normale}). 
For notational convenience, the Hamiltonian vector field obtained in this way is again denoted by 
$X = (X^{(\theta)}, X^{(y)}, X^\bot)$. The $y-$component 
$X^{(y)}$ is small of order three and the normal component $X^\bot$ of $X$ at $\frak x = (\theta, y, w)$ has the form 
\begin{equation}\label{forma X bot intro}
X^\bot (\frak x) = \ii \Omega_\bot [w] +  X^\bot_1(\theta, y)[w] + X^\bot_2(\theta)[w, w] + \text{term small of order three}
\end{equation}
where 
\begin{equation}\label{X1 X2 intro}
\begin{aligned}
& X^\bot_1(\theta, y)[w]  = \Pi_\bot \sum_{k = 0}^{N+1} T_{a_{1 - k}(\theta, y)} \partial_x^{1 - k}w + {\cal R}^\bot_{N, 1}(\theta, y)[w]\,, \\
& X^\bot_2(\theta)[w, w] = \Pi_\bot  \sum_{k = 0}^{N+1} T_{A_{1 - k}(\theta)[w]} \partial_x^{1 - k} w + {\cal R}^\bot_{N, 2}(\theta)[w, w] \, ,
\end{aligned}
\end{equation}
and for any $0 \le k \le N + 1$, $a_{1 - k} (\theta, y)$ is small of order one, $w \mapsto A_{1 - k}(\theta)[w]$ is a linear operator, whereas
$w \mapsto {\cal R}^\bot_{N, 1}(\theta, y)[w]$ is a linear smoothing operator (smoothing of order $N+1$), 
and $w \mapsto {\cal R}^\bot_{N, 2}(\theta)[w,w]$ is a quadratic smoothing operator (smoothing of order $N+1$). 
The term in \eqref{forma X bot intro}, which is small of order three, is the sum of a para-differential vector field of order one and a smoothing vector field. 

\smallskip

\noindent
The second step  of our normal form procedure is developed in Section \ref{normalization II}. 
Since $\Pi_\gamma^{(3)}$ (cf. \eqref{condizioni forma normale}) allows for a {\it loss of derivatives in space}, 
we first need to reduce the terms in the Taylor expansion of the normal component $X^\bot$ of $X$, which are linear and quadratic in $w$,
to constant coefficients up to smoothing terms - see Subsection \ref{sec regolarizzazione w w2}.
This regularization procedure is achieved by constructing a transformation which is {\em not canonical}, but
nevertheless preserves the following important property, needed for the energy estimates: 
the linearization of $X^\bot$ at $w = 0$ equals $X^\bot_1(\theta, y)$ and hence {\em is Hamiltonian}. 
In particular, the diagonal elements of the Fourier matrix representation of the linear operator $X_1^\bot(\theta, y)$ are purely imaginary,
\begin{equation}\label{prop parte diagonale intro}
[X^\bot_1(\theta, y)]_j^j \in \ii \R, \qquad \forall j \in S^\bot\,. 
\end{equation}
We remark that in the spirit of \cite{Feola-Iandoli}, 
one could construct a canonical transformation, but the construction of the one in Subsection \ref{sec regolarizzazione w w2}
is technically simpler and due to \eqref{prop parte diagonale intro} suffices for our purposes. 

We now describe the second step of our normal form procedure in more detail. 
We begin by normalizing the operator
$$
\Pi_\bot T_{a_1(\theta, y)} \partial_x + \Pi_\bot T_{ A_1(\theta)[w]} \partial_x =  \Pi_\bot T_{a_1(\theta, y) + A_1(\theta)[w]} \partial_x 
$$ 
in the expansion of the vector field $X^\bot_1(\theta, y)[w] + X^\bot_2(\theta)[w, w]$ (cf. \eqref{forma X bot intro}, \eqref{X1 X2 intro}). 
We transform the vector field in \eqref{forma X bot intro} by the means of the time one flow map $\Phi_{Y}$
of the vector field
$$
Y(\theta, y, w) = \big(0, \, 0, \,  \Pi_\bot T_{b(\theta, y) + B(\theta)[w]} \partial_x^{- 1} w \big)
$$
with $b$ and $B$ given by 
\begin{equation}\label{def gn intro}
b(\theta, y) := \frac13 \partial_x^{- 1}\big(\langle a_{1}(\theta, y) \rangle_x  - a_{1}(\theta, y) \big), 
\qquad  B(\theta)[w] := \frac13 \partial_x^{- 1}\big(\langle  A_{1}(\theta)[w] \rangle_x - A_{1}(\theta)[w] \big).
\end{equation} 
(Recall that for $a \in L^2(\T_1)$, $\langle a \rangle_x = \int_0^1 a\, d x$.) Note that $b$ and $B$ satisfy
\begin{equation}\label{omologica lin in w intro}
3 \partial_x b(\theta, y) + a_{1 }(\theta, y) = \langle a_{1 }(\theta, y) \rangle_x, \qquad
 3 \partial_x B(\theta)[w] + A_{1 }(\theta)[w] = \langle A_{1 }(\theta)[w] \rangle_x.
\ee
For notational convenience, we denote the transformed vector field also by $X_{1} = (X_{1}^{(\theta)}, \, X_{1}^{(y)}, \, X_{1}^\bot)$.
We show that $X_{1}^{(y)}$ is small of order three and that $ X_{1}^\bot(\theta, y, w)$ has the form 
\begin{equation}\label{intro bla bla 0}
 \ii \Omega_\bot w + {\cal D}^{\bot}_{1, 1}(\theta, y)[w] + {\cal D}^{\bot}_{1, 2}(\theta, w)[w] + 
 X_{1, 1}^{\bot}(\theta, y)[w] + X_{1, 2}^{\bot}(\theta)[w, w] + \text{term small of order three} 
 \end{equation}
 with
$$ 
{\cal D}^{\bot}_{1, 1}(\theta, y) := \langle a_1(\theta, y) \rangle_x \partial_x , 
\qquad {\cal D}^{\bot}_{1, 2}(\theta, w) :=   \langle A_1(\theta)[w] \rangle_x \partial_x\, , \quad
$$
$$
X_{1, 1}^{\bot}(\theta, y)[w] : = \Pi_\bot \sum_{k = 1}^{N+1} T_{a_{1,1 - k}(\theta, y)} \partial_x^{ 1- k} w + {\cal R}_{N, 1}^{\bot}(\theta, y)[w]\, ,  \qquad 
$$
$$
X_{1, 2}^{\bot}(\theta)[w, w] :=  \Pi_\bot \sum_{k = 1}^{N + 1} T_{A_{1, 1 - k}(\theta)[w]} \partial_x^{1 - k} w + {\cal R}_{N, 2}^{\bot}(\theta)[w, w]\, , \qquad 
$$
where for any $1 \le k \le N+1 $, $a_{1, 1 - k}(\theta, y)$ is small of order one and 
$w \mapsto A_{1, 1 - k}(\theta)[w]$ is a linear operator.
Furthermore,  ${\cal R}_{N, 1}^{\bot}(\theta, y)$ is a smoothing linear operator and ${\cal R}_{N, 2}^{\bot}(\theta)$ is a smoothing bilinear operator. 
The term in \eqref{intro bla bla 0}, which is small of order three, is the sum of a para-differential vector field of order one and a smoothing vector field. 
We also show that the linear vector field $X_{1, 1}^{\bot}(\theta, y)[w] $ in \eqref{intro bla bla 0} satisfies the property \eqref{prop parte diagonale intro}, 
i.e., $[X_{1, 1}^{\bot}(\theta, y)]_j^j \in \ii \R$ for any $j \in S^\bot$, and that the Fourier multiplier ${\cal D}^{\bot}_{1, 1}(\theta, y)$ is skew-adjoint. 
By iterating this procedure $N + 2$ times, one gets a vector field, which we denote by $X_{4} = (X_{4}^{(\theta)}, X_{4}^{(y)}, X_{4}^\bot)$ 
(cf. Proposition \ref{prop total regularization}),
with the following properties:  $X_{4}^{(y)}$ is small of order three and $X_{4}^\bot(\theta, y, w) $ has the form 
\begin{equation}\label{forma X4 intro}
\begin{aligned}
 \ii \Omega_\bot w + {\cal D}_{4, 1}^{\bot}(\theta, y)[w] + {\cal D}_{4, 2}^{\bot}(\theta, w)[w]   
 + {\cal R}_{N, 1}^{\bot}(\theta, y)[w] + {\cal R}_{N, 2}^{\bot}(\theta)[w, w] + \text{term small of order three} \, .
\end{aligned}
\end{equation}
Here ${\cal D}_{4, 1}^{\bot}(\theta, y)$ and ${\cal D}_{ 4, 2}^{\bot}(\theta, w)$ are Fourier multipliers of the form 
\begin{equation}\label{forma cal D (4) intro}
\begin{aligned}
{\cal D}_{4, 1}^{\bot}(\theta, y) = \sum_{k = 0}^{N+1} \lambda_{1 - k}(\theta, y) \partial_x^{1 - k} \, , \qquad 
{\cal D}_{4, 2}^{\bot}(\theta, w) := \sum_{k = 0}^{N+1} \Lambda^\bot_{1 - k}(\theta) [w] \partial_x^{1 - k} \, ,
\end{aligned}
\end{equation}
where for any $0 \le k \le N +1$, $\lambda_{1 - k}(\theta, y) \in \R$ is small of order one and 
$w \mapsto \Lambda^\bot_{1 - k}(\theta)[w] \in \R$ is a linear operator. 
The remainder ${\cal R}_{N, 1}^{\bot}(\theta, y)$ is a smoothing linear operator and ${\cal R}_{N, 2}^{\bot}(\theta)$ is a smoothing bilinear operator. 
In addition, the Fourier multiplier ${\cal D}_{4, 1}^{\bot}(\theta, y)$ is skew-adjoint. Moreover we show that
\begin{equation}\label{prop cal R 4 (1) intro}
[{\cal R}_{N, 1}^{\bot}(\theta, y)]_j^j \in \ii \R, \quad \forall j \in S^\bot.
\end{equation} 
Since the  transformation $\Phi_Y$ and the subsequent transformations constructed in the interative procedure are not canonical, the linear operator $ {\cal D}_{4, 2}^{\bot}(\theta, w)$ is not necessarily skew-adjoint. 
However the leading order term $\Lambda^\bot_1(\theta)[w] \partial_x$ of $ {\cal D}_{4, 2}^{\bot}(\theta, w)$ is skew-adjoint since $\Lambda^\bot_1(\theta)[w]\in \R$. 

In Subsection \ref{normalizzazione Fourier multipliers} we design a normal form procedure to remove
\begin{equation}\label{ordine zero intro}
\sum_{k = 1}^{N+1} \Lambda^\bot_{1 - k}(\theta) [w] \partial_x^{1 - k}
\end{equation} 
from $ {\cal D}_{4, 2}^{\bot}(\theta, w)$ which requires to impose 
first Melnikov conditions on $\omega$ (cf. definition \eqref{condizioni forma normale} of $\Pi_\gamma^{(1)}$). 
We transform the vector field $X_{4}$ (cf.  \eqref{forma X4 intro}) by the means of the time one flow map of a vector field,
which in view of \eqref{ordine zero intro}  is chosen to be of the form
\begin{equation}\label{generatrice multiplier intro}
\big( 0, \, 0, \, \sum_{k = 1}^{N+1} \Xi^\bot_{1 - k}(\theta)[w] \partial_x^{1 - k}w \big) \, 
\end{equation}
where for any $1 \le k \le N +1$, the linear functional $w \mapsto \Xi^\bot_{1 - k}(\theta)[w]$ is a solution of
\begin{equation}\label{eq homologica fourier multiplier intro}
\omega \cdot \partial_\theta \,  \Xi^\bot_{1 - k}(\theta)[w]  - \Xi^\bot_{1 - k}(\theta)[ \ii \Omega_\bot w] + \Lambda^\bot_{1 - k}(\theta)[w] = 0 \, .
\end{equation}
The latter equation can be solved if $\omega \in \Pi_\gamma^{(1)}$ (first Melnikov conditions). 
The transformed vector field is denoted by $X_{5} = (X_{5}^{(\theta)}, X_{5}^{(y)}, X_{5}^\bot)$.
We show that $X_{5}^{(y)}$ is small of order three and that $X_{5}^\bot(\theta, y, w)$ has the form  
\begin{equation}\label{X5 intro}
\begin{aligned}
\ii \Omega_\bot w + {\cal D}^\bot_{5}(\theta, y, w)[w] + {\cal R}^\bot_{N, 1}(\theta, y)[w] + {\cal R}_{N, 2}^\bot(\theta)[w, w] + \text{ term small of order three} ,
\end{aligned}
\end{equation}
where 
\begin{equation}\label{cal D 5 intro}
{\cal D}^\bot_{5}(\frak x) := {\cal D}^\bot_{4, 1}(\theta, y) + \Lambda^\bot_1(\theta)[w] \partial_x
\end{equation}
and ${\cal R}_{N, 1}^\bot$, ${\cal R}_{N, 2}^\bot$ are as in \eqref{forma X4 intro}. Clearly, the Fourier multiplier ${\cal D}^\bot_{5}(\frak x)$ is skew-adjoint. \\
Finally in Section \ref{section smoothing remainders} we normalize the term in the Taylor expansion of the $\theta$-component $X_{5}^{(\theta)}$ of $X_{5}$, 
which is quadratic in $w$, and normalize the smoothing vector fields ${\cal R}_{N, 1}^\bot$ and ${\cal R}_{N, 2}^\bot$ in $X^\bot_5$. 
Let us explain in more detail how to achieve the latter.
We transform the vector field $X_{5}$ by the time one flow map generated by the vector field 
\begin{equation}\label{smoothing vector field intro}
\big( 0, \, 0, \, \, {\cal S}^\bot_1(\theta, y)[w] + {\cal S}^\bot_2(\theta)[w, w] \big)
\end{equation}
where ${\cal S}^\bot_1(\theta, y)$ is a smoothing linear operator and ${\cal S}^\bot_2(\theta)$ is a smoothing bilinear operator. 
They are chosen to be solutions of
\begin{equation}\label{eq omologica 1 smoothing intro}
 - \omega \cdot \partial_\theta \, {\cal S}^\bot_1(\theta, y) + [\ii \Omega_\bot , \, {\cal S}^\bot_1(\theta, y)]_{lin} + {\cal R}^\bot_{N, 1}(\theta, y) = {\cal Z}^\bot (y) \, \qquad
 \end{equation}
 and, respectively,
 \begin{equation}\label{eq omologica 2 smoothing intro}
 -  \omega \cdot \partial_\theta \, {\cal S}^\bot_2(\theta)[w, w] +\ii \Omega_\bot {\cal S}^\bot_2(\theta)[w, w] - 
 {\cal S}^\bot_2(\theta)\big( [\ii  \Omega_\bot w, w] + [w, \ii \Omega_\bot w] \big) + {\cal R}^\bot_{N, 2}(\theta)[ w, w] = 0 \, ,
\end{equation}
 where
\begin{equation}\label{def Z}
 {\cal Z}^\bot(y) := {\rm diag}_{j \in S^\bot} [\widehat{\cal R}^\bot_{N, 1}(0, y)]_j^j  \, , \qquad 
 [\widehat{\cal R}^\bot_{N, 1}(0, y)]_j^j  := \frac{1}{(2 \pi)^{S_+}} \int_{\T^{S_+}} [{\cal R}^\bot_{N, 1}(\theta, y)]_j^j\, d \theta\, .
\end{equation}

\noindent
Equation  \eqref{eq omologica 1 smoothing intro} can be solved by imposing the {\em second} Melnikov conditions
on  $ \omega$, i.e., $\omega \in \Pi_\gamma^{(2)}$, 
and equation  \eqref{eq omologica 2 smoothing intro} by imposing the {\em third} Melnikov conditions,
$\omega \in \Pi_\gamma^{(3)}$ - see Lemma \ref{lemma equazioni omologiche smoothing}.
Note that in equation \eqref{eq omologica 2 smoothing intro}, the right hand side vanishes, meaning that the left hand side does not contain any resonant terms.
Finally we get a vector field 
$X_{6} =  (X_{6}^{(\theta)}, X_{6}^{(y)}, X_{6}^\bot)$ where $X_{6}^{(y)}$ 
is small of order three and $X_{6}^\bot(\frak x)$ has the form 
\begin{equation}\label{X6 bot}
\begin{aligned}
X_{6}^\bot(\frak x) = \ii \Omega_\bot w 
+  {\cal D}^\bot_{5}(\frak x)[w] + {\cal Z}^\bot (y)[w]
+ \text{term small of order three}\,. 
\end{aligned}
\end{equation}
By the property \eqref{prop cal R 4 (1) intro} and  the definition  \eqref{def Z} of ${\cal Z}^\bot(y)$, 
it follows that ${\cal Z}^\bot(y)$ and hence $ {\cal D}^\bot_{5}(\frak x) + {\cal Z}^\bot (y)$ are skew-adjoint Fourier multiplier. 
Finally one shows that $X_{6}^\bot$ in \eqref{X6 bot} has the form stated in \eqref{campo vettoriale finalissimo intro}.

%
%
%
%
%
%

\medskip
   
\noindent
{\em Related work.} 
Prior to our work, 
no results have been obtained on the long time asymptotics of the solutions of Hamiltonian perturbations of
integrable PDEs such as the KdV or the nonlinear Schr\"odinger equation on $\T_1$ with initial data close to a {\em periodic multi-soliton of possibly large amplitude}.
For Hamiltonian perturbations of linear integrable PDEs on $\T_1$, which satisfy nonresonance conditions, 
a by now standard normal form method has been developed allowing to prove the stability of the {\em equilibrium solution $u\equiv 0$} of (Hamiltonian) perturbations 
for time intervals of large size -- see e.g. \cite{Bam}, \cite{Bam2}, \cite{BamG}, \cite{Berti-Delort}, \cite{Bou}, \cite{CLSY}, \cite{Delort}, \cite{Feola-Iandoli}
and references therein.
More recently, these techniques have been refined so that in specific cases, such results can also be proved for Hamiltonian perturbations 
of resonant linear integrable PDEs by approximating the perturbed equation by nonlinear integrable systems, satisfying
nonresonance conditions -- see \cite{Bou}, \cite{BFG} for Hamiltonian perturbations of the linear Schr\"odinger equation and \cite{BG} for such perturbations
of the Airy equation as well as the linearized Benjamin-Ono equation. We remark that for the Airy equation, the Hamiltonian perturbations considered  in \cite{BG} 
are of the form $\partial_x \nabla P_f$ (cf. \eqref{cal K def intro} - \eqref{f def intro}) with the density $f(u(x))$ not explicitly depending on $x$ and $f(z)$ 
being analytic in a neighborhood of $z=0$ in $\C$.\\
Finally, we mention the recent paper  \cite{BKM} where it is proved by KAM methods that many periodic multi-solitons persist 
under quasi-linear perturbations of the KdV equation. As in this paper,
a key ingredient are the normal form coordinates, constructed in \cite{Kap-Mon-2}.

\medskip

\noindent
{\em Acknowledgments:} We would like to thank Michela Procesi for the example in Remark \ref{remark on Pi_gamma^4}  and very valuable feedback
and Massimiliano Berti for insightful discussions.

T. K. is supported by Swiss National Foundation. R. M. is supported by INDAM-GNFM.


\section{Para-differential calculus}\label{para-differential calculus}

In this section we review some standard notions and results of the para-differential calculus, needed throughout the paper. For details we refer to \cite{Metivier}. 

We begin with reviewing the notion of para-product. To this end we need the following
\begin{definition}\label{cut-off function}
A function $\psi \in C^\infty(\R \times \R)$ is said to be an admissible cut-off function, if there exist $0 < \e' < \e <1$ so that 
$$
{\rm supp}(\psi) \subseteq \{ (\eta, \xi) \in \R \times \R : |\eta| \leq \e \langle \xi \rangle\}\,, \qquad \quad
\psi(\eta, \xi) = 1 \,, \quad \forall (\eta, \xi) \in \R \times \R \ \text{with} \  |\eta| \leq \e' \langle \xi\rangle\,,
$$
and
$$
|\partial_\eta^{\alpha} \partial_\xi^\beta \psi(\eta, \xi)| \lesssim_{\alpha, \beta} \langle \xi \rangle^{- \alpha - \beta}\,, \qquad \forall (\alpha, \beta) \in \Z_{\geq 0} \times \Z_{\geq 0}
$$
where by  \eqref{def langle y rangle} $\langle \xi \rangle = \max\{1, |\xi | \}$. 
\end{definition}
Given a cut-off function $\psi$ as in Definition \ref{cut-off function}, 
the para-product $T_a u$ of a function $a\in H^{1}(\T_1)$ with a function $u \in H^s(\T_1)$, $s \ge 1$, is defined as 
\begin{equation}\label{definizione paraprodotto}
T_a u(x) := \sigma_a(x, D)u(x) = \sum_{\xi \in \Z} \sigma_a(x, \xi) \widehat u(\xi) e^{\ii 2 \pi \xi x}\,,\qquad \sigma_a(x, \xi) := \sum_{\eta \in \Z} \psi(\eta, \xi) \widehat a(\eta) e^{\ii 2 \pi \eta x }\, ,
\end{equation}
where $\widehat a(\eta)$, also denoted by $a_\eta$, is the $\eta$th Fourier coefficient of $a$,
$$
\widehat a(\eta) = \int_0^1 a(x) e^{- \ii 2\pi \eta x} d x \, .
$$

\begin{lemma}\label{prop 1 paraproduct}
For any $a \in H^{1}(\T_1)$ and $s \geq 1$,  $T_a$ is in ${\cal B}(H^s(\T_1), H^s(\T_1))$ and 
\begin{equation}\label{stima elementare paraproduct}
\| T_a \|_{{\cal B}(H^s, H^s)} \lesssim_s \| a \|_{1}\,. 
\end{equation}
Furthermore, for any $s \ge 1$, the map $H^1(\T_1) \to {\cal B}(H^s(\T_1), H^s(\T_1)), \, a \mapsto T_a$, is linear.
\end{lemma}
Given two functions $a, u \in H^s(\T_1)$ with $s \geq  1$, their product can be split as
\begin{equation}\label{paraprodotto}
a u = T_a u + T_u a + {\cal R}^{(B)}(a, u)\,, 
\end{equation}
where the remainder ${\cal R}^{(B)}(a, u)$ is given by
\begin{equation}\label{resto paraprodotto}
{\cal R}^{(B)}(a, u) (x)= \sum_{\eta, \xi \in \Z} \omega(\eta, \xi)\widehat a(\eta) \widehat u(\xi) e^{\ii 2 \pi (\eta + \xi) x}\,,
\qquad \omega(\eta, \xi) := 1- \psi(\eta, \xi)- \psi(\xi, \eta)\,.
\end{equation}
Note that the support ${\rm supp}(\omega)$ of $\omega : \Z \times \Z \to \R$ satisfies
\begin{equation}\label{supporto di omega}
 \big\{ (\eta, \xi) \in \Z^{2 } : \e \langle \xi \rangle  <  |\eta| < \frac{\langle \xi \rangle}{\e}   \big\} \cup \{ (0,0) \}
 \subseteq {\rm supp}(\omega) \subseteq \big\{ (\eta, \xi) \in \Z^{2 } : \e' \langle \xi \rangle <  |\eta| <  \frac{ \langle \xi \rangle }{\e'} \big\} \cup \{ (0,0) \} \,.
\end{equation}
The main feature of ${\cal R}^{(B)}(a, u)$ is that it is a regularizing bilinear operator in the following sense.
\begin{lemma}\label{lemma remainder paraprod}
For any $s_1, s_2 \geq 0$,  
$$
\mathcal R^{(B)} : H^{s_1 + 1}(\T_1) \times H^{s_2}(\T_1)  \to H^{s_1 + s_2}(\T_1), \, (a, u) \mapsto {\cal R}^{(B)}(a, u)
$$ 
is a bilinear map, satisfying 
\begin{equation}\label{stima resto paraprodotto}
\|{\cal R}^{(B)}(a, u) \|_{s_1 + s_2} \lesssim_{s_1, s_2}  \| a \|_{s_1 + 1} \| u \|_{s_2}\, \qquad \forall \, a \in H^{s_1 + 1}(\T_1), \, u \in H^{s_2}(\T_1)\, .
\end{equation}
\end{lemma}
Next, we discuss the standard symbolic calculus for para-differential operators to the extent needed in this paper. 
It suffices to consider operators of the form 
\begin{equation}\label{nostri paradiff}
 T_a \partial_x^m, \qquad a \in H^1(\T_1), \ m \in \Z\,, 
\end{equation}
where we recall that for any $m \in \Z$,  the Fourier multiplier $\partial_x^m$ is defined by 
$$
\partial_x^m [e^{\ii 2 \pi j x}] := ( \ii 2 \pi  j)^m e^{\ii 2 \pi j x}\,, \ \ \forall \, j \neq 0\,, \qquad \partial_x^m[1] := 0\,.
$$
Alternatively, $\partial_x^m$ can be written as the pseudo-differential operator ${\rm Op}( ( \ii 2 \pi \xi)^m \chi(\xi))$ with symbol $( \ii 2 \pi \xi)^m \chi(\xi)$
where $\chi: \R \to \R$ is a $C^\infty-$smooth cut-off function, satisfying 
\begin{equation}\label{definizione cut off partial x m generale}
\chi(\xi) = 1\,, \ \  \forall \, |\xi| \geq \frac23\,, \qquad \chi(\xi) = 0\,, \ \  \forall \, |\xi| \leq \frac13\,. 
\end{equation}
The symbol of an operator of the form \eqref{nostri paradiff} is given by 
$$
\sigma_{a}(x, \xi) = \sum_{\eta \in \Z} \psi(\eta, \xi) \widehat a(\eta) ( \ii 2 \pi \xi )^m e^{\ii 2 \pi \eta x}\,. 
$$
\begin{lemma}\label{lemma paraprodotto fine}
Let $a, b \in H^{N+3}(\T_1)$ with $N \in \N$. Then
$$
T_a \circ T_b = T_{ab}  + {\cal R}_N(a, b)
$$
where for any $s \geq 0$, 
$$
\mathcal R_N : H^{N+3}(\T_1) \times H^{N+3} (\T_1) \to {\cal B}\big( H^s(\T_1), H^{s + N+1}(\T_1)\big), \, (a, b) \mapsto  {\cal R}_N(a, b)\, ,
$$  
is a bilinear map, satisfying 
$$
\| {\cal R}_N(a, b)\|_{{\cal B}(H^s, H^{s + N+1})} \lesssim_{s, N} \| a \|_{N+3} \| b \|_{N+3} \, , \qquad  \forall  \, a, b \in H^{N+3}(\T_1). 
$$
\end{lemma}
\begin{lemma}\label{composizione paraprodotto e derivate}
Let $m \in \Z$, $N \in \N$. Then there exist an integer $\sigma_N > N + m$ and combinatorial constants $(K_{n, m})_{1 \le n  \le N + m}$, with 
$K_{1, m} = m$ so that  for any $a \in H^{\sigma_N}(\T_1)$
$$
\partial_x^m \circ T_a =   T_{ a} \partial_x^{m } + \sum_{n = 1}^{N+m} K_{n, m} T_{\partial_x^n a} \partial_x^{m - n} + {\cal R}_{N, m}(a)\, ,
$$
where for any $s \geq 0$, the map 
$$
{\cal R}_{N, m} : H^{\sigma_N}(\T_1) \to {\cal B}(H^s(\T_1), H^{s + N + 1}(\T_1)), \, a \mapsto {\cal R}_{N, m}(a)
$$ 
is linear and  satisfies the estimate 
$$
\| {\cal R}_{N, m}(a)\|_{ {\cal B}(H^s, H^{s + N + 1})} \lesssim_{s, m, N} \| a \|_{\sigma_N}\,, \quad \forall a \in H^{\sigma_N}(\T_1) ,
$$
and where we use the customary convention that the sum $\sum_{n = 1}^{N+m}$ equals $0$ if $N+m < 1$.
\end{lemma}
Combining Lemma \ref{lemma paraprodotto fine} and Lemma \ref{composizione paraprodotto e derivate} yields the following 
\begin{lemma}\label{lemma composizione nostri simboli}
Let $m, m' \in \Z$, $N \in \N$. Then there exists an integer $\sigma_N > N + m$ so that for any $a, b \in H^{\sigma_N}(\T_1)$,
\begin{equation}\label{formula espansione composizione}
T_a \partial_x^m \circ T_b \partial_x^{m'} =  T_{a b} \partial_x^{m + m' } + 
 \sum_{n = 1}^{N+m+m'} K_{n, m} T_{a \partial_x^n b} \partial_x^{m + m' - n} + {\cal R}_{N, m, m'}(a, b) \, ,
\end{equation}
where $K_{n, m}$ are the combinatorial constants of Lemma \ref{composizione paraprodotto e derivate} and where
for any $s \geq 0$, the map 
$$
{\cal R}_{N, m, m'} : H^{\sigma_N}(\T_1) \times H^{\sigma_N}(\T_1) \to {\cal B}(H^s(\T_1), \, H^{s + N + 1}(\T_1)), \, (a, b) \mapsto {\cal R}_{N, m, m'}(a, b)
$$
 is bilinear and satisfies the estimate
 $$
 \| {\cal R}_{N, m, m'}(a, b)\|_{{\cal B}(H^s, H^{s + N + 1})} \lesssim_{s, m, N} \| a \|_{\sigma_N} \| b \|_{\sigma_N} \, , \qquad \forall \, a, b \in H^{\sigma_N}(\T_1).
 $$ 
 According to Lemma \ref{lemma paraprodotto fine},  in the case $m=0$, a possible choice is $\sigma_N= N+ 3$,  $K_{n, 0} = 0$ for $1 \le n \le N+m'$.
\end{lemma}
Using that $K_{1,m} = m$, one infers from Lemma \ref{lemma composizione nostri simboli} an expansion of the commutator 
$[T_a \partial_x^m \,,\, T_b \partial_x^{m'} ]_{lin}$.
\begin{corollary}[\bf Commutator expansion]\label{corollario espansione commutatore}
Let $m, m' \in \Z$, $N \in \N$. Then there exists $\sigma_N > N +m +m'$ so that for any $a, b \in H^{\sigma_N}(\T_1)$, 
$[T_a \partial_x^m \,,\, T_b \partial_x^{m'} ]_{lin} $
has an expansion of the form
\begin{equation}\label{formula espansione commutatore}
  T_{m a \partial_x b - m' b \partial_x a  } \partial_x^{m + m' - 1} 
+ \sum_{n = 2}^{N+m+m'} \big( K_{n, m} T_{a \partial_x^n b} - K_{n,m'} T_{b  \partial_x^n a } \big) \partial_x^{m + m' - n} + {\cal R}^{\mathcal C}_{N, m, m'}(a, b)\, ,
\end{equation}
where for any $s \geq 0$, the map 
$$
{\cal R}^{\mathcal C}_{N, m, m'} : 
H^{\sigma_N}(\T_1) \times H^{\sigma_N}(\T_1) \to {\cal B}\big( H^s(\T_1), H^{s + N + 1}(\T_1) \big), \, (a, b) \mapsto {\cal R}^\mathcal C_{N, m, m'}(a, b)
$$ 
is bilinear and satisfies
$$
\| {\cal R}^\mathcal C_{N, m, m'}(a, b)\|_{{\cal B}(H^s, H^{s + N + 1})} \lesssim_{s, m, m', N} \| a \|_{\sigma_N} \| b \|_{\sigma_N}\, , \qquad \forall \, a, b \in H^{\sigma_N}(\T_1) \,.
$$   
 According to Lemma \ref{lemma paraprodotto fine},  in the case $m=0, m'=0$, 
 $[T_a \, ,\, T_b ]_{lin} =  {\cal R}_N(a, b) -  {\cal R}_N(b, a)$.  Hence a possible choice is $\sigma_N= N+ 3$,  $K_{n, 0} = 0$ for $1 \le n \le N$.
\end{corollary}
Finally, we discuss the adjoint $T_a^\top$ of $T_a$ with respect to the standard $L^2-$inner product.
\begin{lemma}\label{lemma aggiunto paraprodotto}
Let $a \in H^{N+1}(\T_1)$ with $N \in \N$. Then $T_a^\top= T_{ a} + {\cal R}_\top( a)$ where for any $s \geq 0$, the map 
$$
\mathcal R_\top : H^{N+1}(\T_1) \to {\cal B}\big( H^s(\T_1), H^{s + N + 1}(\T_1) \big), \, a \mapsto {\cal R}_\top(a) \, ,
$$ 
is linear and  for any $a \in H^{N+1}(\T_1)$ satisfies $\| {\cal R}_\top(a)\|_{{\cal B}(H^s, H^{s + N+1})} \lesssim_{s, N} \| a \|_{N+1}$.  
\end{lemma}
Combining Lemma \ref{composizione paraprodotto e derivate} and Lemma \ref{lemma aggiunto paraprodotto} yields the following
\begin{corollary}\label{aggiunto nostri operatori}
Let $m \in \Z$, $N \in \N$. Then there exists an integer $\sigma_N > N + m$ so that for any  $a \in H^{\sigma_N}(\T_1)$,  
$(T_a \partial_x^m)^\top$ admits the expansion 
$$
(T_a \partial_x^m)^\top = (- 1)^m T_{  a} \partial_x^{m } +  (- 1)^m\sum_{n = 1}^{N+m} K_{n, m}T_{\partial_x^n  a} \partial_x^{m - n} + {\cal R}_{\top, N, m}(a)  ,
$$
where $K_{n, m}$ are the combinatorial constants of Lemma \ref{composizione paraprodotto e derivate}, and
where for any $s \geq 0$, the map 
$$
{\cal R}_{\top, N, m} : H^{\sigma_N}(\T_1) \to {\cal B}(H^s(\T_1), H^{s + N + 1}(\T_1)), \, a \mapsto {\cal R}_{\top, N, m}(a),
$$ 
is linear and for any $a \in H^{\sigma_N}(\T_1)$ satisfies $\| {\cal R}_{\top, N, m}(a)\|_{{\cal B}(H^s, H^{s + N + 1})} \lesssim_{s, N} \| a \|_{\sigma_N}$.   
\end{corollary}

\section{Para-differential vector fields}\label{section paradiff nonlinear}
In this section we introduce several classes of vector fields, compute the commutators between vector fields from these classes and study their flows.
As part of the proof of Theorem \ref{stability theorem} ,
these vector fields are used to transform equation \eqref{1.1} into normal form.

\noindent
\subsection{Definitions}

\begin{definition}[\bf Para-differential vector fields]\label{paradiff vector fields}
Let $N$, $p \in \N$ and $m \in \Z$. A vector field $X^\bot$ in normal direction, defined on a subset of $\mathcal E$ and depending on the parameters $\e$ and $\mu$, 
is said to be of class ${\cal O B}(m, N)$, $X^\bot \in {\cal O B}(m, N)$, if it is of the form
\begin{equation}\label{formula generale paradiff vector field}
X^\bot(\frak x) =\Pi_\bot  \sum_{k = 0}^{N + m} T_{a_{m - k}(\frak x)} \partial_x^{m - k} w
\end{equation}
and has the following property: there are integers $\sigma_N, s_N \geq 0$ so that for any $s \geq s_N$ there exist $0 < \delta \equiv \delta(s, N) < 1$ 
and $0 < \e_0 \equiv \e_0(s, N) < 1$ so that for any $0 \le k \le N + m$ 
$$
a_{m - k} :  {\cal V}^{s  + \sigma_N}(\delta) \times [0, \e_0]   \to H^s(\T_1), \, (\frak x, \e) \mapsto a_{m - k}(\frak x) \equiv a_{m - k}(\frak x, \e)
$$
is $C^\infty-$smooth and together with each of its derivatives bounded.
$X^{\bot}$ is said to be of class ${\cal OB}^p(m, N)$ if it is in ${\cal OB}(m, N)$ and in addition, the functions $a_{m -k}$ are small of order $p - 1$.
\end{definition}
\begin{remark}
$(i)$ If $N+ m < 0$ in \eqref{formula generale paradiff vector field}, the sum is defined to be the zero vector field.
As a consequence, ${\cal OB}(m, N) = \{ 0 \}$ if $N + m  <  0$. Throughout the paper, the same convention holds for any sum of terms, indexed by an empty set,
and for any of the used classes of vector fields. 

\noindent
$(ii)$ We point out that 
the bounds are  uniform in the parameter $\mu$, but no regularity assumptions with respect to $\mu$ are required.
Throughout the paper, the same convention holds. 
\end{remark}
\begin{definition}[\bf Fourier multiplier vector fields]\label{classe fourier multipliers}
Let $N$, $p \in \N$ and $m \in \Z$. A vector field $\mathcal M^\bot$ in normal direction, defined on a subset of $\mathcal E$ and depending on the parameters $\e$ and $\mu$, 
is said to be of class ${\cal O F}(m, N)$, $\mathcal M^\bot \in {\cal O F}(m, N)$, if it is of the form
\begin{equation}\label{fourier multiplier vector field}
\mathcal M^\bot(\frak x) = \sum_{k = 0}^{N + m} \lambda_{m - k}(\frak x) \partial_x^{m - k} w
\end{equation}
and has the following property: there exist an integer $\sigma_N \geq 0$, $0 < \delta \equiv \delta(N) <1$, and $0 < \e_0 \equiv \e_0(N) < 1$  
so that for any $0 \le k \le N + m$,
$$
\lambda_{m - k} :  {\cal V}^{\sigma_N}(\delta) \times [0, \e_0]  \to \R, \, (\frak x, \e) \mapsto \lambda_{m - k}(\frak x) \equiv \lambda_{m - k}(\frak x, \e)
$$
is $C^\infty$-smooth and together with each of its derivatives bounded. 
 $\mathcal M^{\bot}$ is said to be of class  ${\cal OF}^p(m, N)$ if it is in ${\cal OF}(m, N)$ 
and in addition, the functions $\lambda_{m -k}$ are small of order $p - 1$.
\end{definition}
\begin{definition}[\bf Smoothing vector fields]\label{def smoothing vector fields}
Let $N$, $p \in \N$. A vector field $\mathcal R$, defined on a subset of $\mathcal E$ and depending on the parameters $\e$ and $\mu$, 
is said to be of class ${\cal OS}( N)$, $\mathcal R \in {\cal OS}(N)$, if 
there exist $s_N \geq 0$ so that for any $s \geq s_N$, there exist $0 < \delta \equiv \delta(s, N) <1$ and $0 < \e_0 \equiv \e_0(s, N) < 1$ with the property that 
$$
{\cal R} :  {\cal V}^s(\delta) \times [0, \e_0] \to E_{s + N + 1}, \, (\frak x, \e) \mapsto {\cal R}(\frak x) \equiv {\cal R}(\frak x, \e)
$$
is $C^\infty$-smooth and together with each of its dervatives bounded.
${\cal R}$ is said to be of class ${\cal OS}^p(N)$ if it is in ${\cal OS}(N)$ and in addition is small of order $p$.
\end{definition}
\begin{remark}\label{smooth and bounded}
For notational convenience, in the sequel, we refer to a function, which is $C^\infty$-smooth and together with each of its derivatives bounded,
as a function which is $C^\infty$-smooth and bounded.
\end{remark}
Next we introduce  special classes of vector fields which are small of order $2$ with respect to $y$, $w$, $\e$. 
\begin{definition}\label{paradiff omogenei di ordine 2}
Let $N \in \N$ and $m \in \Z$.\\
$(i)$ Assume that $X^\bot(\frak x) = \Pi_\bot \sum_{k = 0}^{m + N} T_{a_{m - k}(\frak x)} \partial_x^{m - k} w$ is of class ${\cal OB}^2(m, N)$. \\
$(i1)$  $X^\bot$ is said to be of class ${\cal OB}^2_{w}(m , N)$ if it is linear with respect to $w$. As a consequence,
for any $0 \le k \le  m + N$, the coefficient $a_{m - k}$ is small of order one and independent of $w$.
More precisely, there is an integer $s_N \ge 0$ so that for any $s \geq s_N$, there exist $0 < \delta \equiv \delta(s, N) < 1$ 
and $\e_0 \equiv \e_0(s, N) > 0$ with the property that 
$$
a_{m - k} : \T^{S_+} \times B_{S_+}(\delta) \times [0, \e_0] \to  H^s(\T_1), \, (\theta, y, \e) \mapsto a_{m - k}(\theta, y) \equiv a_{m - k}(\theta, y, \e)
$$
is $C^\infty$-smooth and bounded (cf. Remark \ref{smooth and bounded}).
In this case, we often write $X^\bot(\theta, y)[w]$ instead of $X^\bot(\frak x)$ where 
$$
X^\bot(\theta, y) := \Pi_\bot \sum_{k = 0}^{N+m} T_{a_{m - k}(\theta, y)} \partial_x^{m - k}\, .
$$
\noindent
$(i2)$ $X^\bot$ is said be of class ${\cal OB}^2_{ww}(m, N)$ if it is quadratic with respect to $w$ and independent of $y$. As a consequence,
for any $0 \le k \le  m + N$, the coefficient $a_{m - k}$ is linear with respect to $w$ and independent of $y$. 
More precisely,  there are integers $s_N \ge 0$, $\sigma_N \ge 0$ so that for any $s \geq s_N$ there exist $0 < \delta \equiv \delta(s, N) < 1$ 
and $0 < \e_0 \equiv \e_0(s, N) < 1$  with the property that 
$$
a_{m - k} : \T^{S_+} \times H_\bot^{s + \sigma_N} \times [0, \e_0] \to  H^s(\T_1), \, (\theta, w, \e) \mapsto a_{m - k}(\theta, w) \equiv A_{m - k}(\theta)[w] \, ,
$$
 with 
 $$
 A_{m - k}: \T^{S_+} \times [0, \e_0] \to {\cal B}(H^{s + \sigma_N}_\bot(\T_1), H^s(\T_1)), \, (\theta, \e)  \mapsto A_{m - k}(\theta) \equiv A_{m - k}(\theta, \e) \ 
 $$
 being $C^\infty$-smooth and bounded. In this case we often write $X^\bot(\theta, w)[w]$ instead of $X^\bot(\frak x)$  where 
$$
X^\bot(\theta, w) = \Pi_\bot \sum_{k = 0}^{N+m} T_{A_{m - k}(\theta)[w]} \partial_x^{m - k} \, .
$$
$(ii)$ Assume that $\mathcal M^\bot(\frak x) = \sum_{k = 0}^{N + m} \lambda_{m - k}(\frak x) \partial_x^{m - k} w$ is of class ${\cal OF}^2(m, N)$. \\
$(ii1)$ $\mathcal M^\bot$ is said to be of class ${\cal OF}^2_w(m, N)$ if it is linear with respect to $w$. More precisely, 
there exist $0 < \delta \equiv \delta(N) < 1$ and $0 < \e_0 \equiv \e_0(N) < 1$ with the property that 
for any $0 \le k \le  m + N$, 
$$
\lambda_{m - k} : \T^{S_+} \times B_{S_+}(\delta) \times [0, \e_0] \to \R, \, (\theta, y, \e) \mapsto  \lambda_{m-k}(\theta, y) \equiv  \lambda_{m-k}(\theta, y, \e)
$$
is $C^\infty$-smooth and bounded. \\
$(ii2)$ $\mathcal M^\bot$ is said to be of class ${\cal OF}^2_{ww}(m, N)$ if it is quadratic with respect to $w$ and independent of $y$. More precisely,
there exist an integer $\sigma_N \ge 0$, $0 < \e_0 \equiv \e_0(N) < 1$,  and 
 for any $0 \le k \le  m + N$ a $C^\infty-$smooth map
 $$
\Lambda_{m -k} : \T^{S_+} \times [0, \e_0] \to {\cal B}(H^{\sigma_N}_\bot(\T_1), \R), \, \theta \mapsto \Lambda_{m -k} (\theta) \equiv \Lambda_{m -k} (\theta, \e),
$$
so that $\lambda_{m - k}(\frak x) = \Lambda_{m - k}(\theta)[w]$.\\
$(iii)$ Assume that ${\cal R}$ is a smoothing vector field of class ${\cal OS}^2(N)$. \\
$(iii1)$ ${\cal R}$ is said to be of class ${\cal OS}_{w}^2(N)$ if  ${\cal R}(\frak x)$ of the form ${\frak R}(\theta, y)[w]$
with $\frak R$ having the following property: there is an integer $s_N \ge 0$ so that for any $s \ge s_N$, there exist $0 < \delta \equiv \delta(s, N) < 1 $ 
and $0 < \e_0 \equiv \e_0(s, N) < 1$ with the property that
$$
{\frak R} : \T^{S_+} \times B_{S_+}(\delta) \times [0, \e_0] \to {\cal B}(H^s(\T_1), H^{s + N + 1}(\T_1)), \, (\theta, y, \e) \mapsto {\frak R}(\theta, y) \equiv {\frak R}(\theta, y; \e)
$$ 
is $C^\infty$-smooth, bounded, and small of order one. In the sequel, we will also write ${\cal R}(\theta, y)[w]$ for ${\frak R}(\theta, y)[w]$.\\
$(iii2)$  ${\cal R}$ is said to be of class ${\cal OS}^2_{ww}(N)$ if ${\cal R}$ is quadratic with respect to $w$ and independent of $y$. 
More precisely,  ${\cal R}(\frak x)$ is of the form ${\frak R}(\theta)[w, w]$ with $\frak R$ having the following property: there is an integer $s_N \ge  0$ so that for any $s \ge s_N$
there exists $0 < \e_0 \equiv \e_0(s, N) < 1$ with the property that 
$$
{\frak R} : \T^{S_+} \times [0, \e_0] \to {\cal B}_2\big(H^s_\bot(\T_1) , H^{s + N + 1}_\bot(\T_1)\big), \, (\theta, \e) \mapsto {\frak R}(\theta) \equiv {\frak R}(\theta, \e)
$$
is $C^\infty$-smooth and bounded. 
In the sequel, we will often write ${\cal R}(\theta)[w, w]$ instead of ${\frak R}(\theta)[w, w]$.
\end{definition}
\begin{remark}\label{remark fourier multiplier paradiff}
For any $N \in \N$ and $m \in \Z$, the following inclusions between the classes of vector fields introduced above hold:
$$
\begin{aligned}
& {\cal OF}(m, N) \subseteq {\cal OB}(m, N), \qquad \  {\cal OF}^p(m, N) \subseteq {\cal OB}^p(m, N)\,, \\
& {\cal OF}^2_w(m, N) \subseteq {\cal OB}^2_{w}(m, N), \quad {\cal OF}^2_{ww}(m, N) \subseteq {\cal OB}^2_{ww}(m, N)\,.  \quad
\end{aligned}
$$
These inclusions hold since by \eqref{definizione paraprodotto} the operator 
$T_\lambda$ of para-multiplication with any constant $\lambda \in \R$ satisfies $\Pi_\bot T_\lambda = \lambda \Pi_\bot$. 
\end{remark}
For notational convenience, we will often not distinguish between a vector field $X$ of the form $(0, 0, X^\bot)$ and its normal component $X^\bot$.
Given two vector fields $X$ and $Y$, defined on a subset of $\mathcal E$ and depending on the parameters $\e$ and $\mu$,  we write
$$
X = Y + {\cal O}_1 + \cdots + {\cal O}_n
$$
if for any $1 \le j \le n,$ there exists a vector field $X_j  \in {\cal O}_j$ so that
$X = Y + X_1 + \cdots + X_n$.
Here ${\cal O}_j$ denotes any of the  classes of vector fields introduced above.

\subsection{Commutators}

\begin{lemma}[\bf Commutators I]\label{Commutator of smoothing vector fields}
Let $N$, $p$, and $q$ be in $\N$.\\
(i) For any smoothing vector fields ${\cal R}$, ${\cal Q} \in {\cal OS}(N)$, the commutator $[{\cal R}, {\cal Q}]$ is also in ${\cal OS}(N)$. \\
$(ii)$ For any vector fields ${\cal R} \in {\cal OS}^p(N)$ and ${\cal Q} \in {\cal OS}^q(N)$, one has $[{\cal R}, {\cal Q}] \in {\cal OS}^{p + q - 1}(N)$
\end{lemma}
\begin{proof}
The two items follow from  Definition \ref{def smoothing vector fields} (smoothing vector fields) 
and the definition \eqref{definition notations nonlinear commutators} of the commutator. 
\end{proof}
\begin{lemma}[\bf Commutators II]\label{comm smoothing paradiff}
Let $N$, $p$, $q \in \N$ and $m \in \Z$.\\
If  $X= (0, 0, X^\bot)$ with $X^\bot \in {\cal O B}(m, N)$ and  ${\cal R} = (\mathcal R^{(\theta)}, \mathcal R^{(y)}, \mathcal R^\bot) \in {\cal OS}(N)$, then
 \begin{equation}\label{formula II}
 [(0, 0, X^\bot), \, {\cal R}] = (0,0, \,  \mathcal C^\bot_{[X, \mathcal R]}) + {\cal R}_{[X, \mathcal R]}\, , \qquad  
 \mathcal C^\bot_{[X, \mathcal R]} \in {\cal OB}(m, N)\, , \quad {\cal R}_{[X, \mathcal R]} \in {\cal OS}(N - m)\, .
 \end{equation}
If  $X^\bot \in {\cal O B}^p(m, N)$ and  ${\cal R} \in {\cal OS}^q(N)$, then
 $  \mathcal C^\bot_{[X, \mathcal R]} \in {\cal OB}^{p + q  -1}(m, N)$ and ${\cal R}_{[X, \mathcal R]}  \in {\cal OS}^{p + q - 1}(N - m)$.
\end{lemma}
\begin{proof}
By \eqref{formula generale paradiff vector field},  $X$ can be written as $X(\frak x) := \sum_{k = 0}^{N + m} X_k(\frak x)$ where
$$
X_k(\frak x) = \big(0,0,  \, \Pi_\bot T_{a_{m -k}(\frak x)} \partial_x^{m - k} w \big), 
\qquad  \forall \, 0 \le k \le N+m\, .
$$ 
For any $0 \le k \le N+m$, the commutator $ [X_k, {\cal R}] (\frak x)= d X_k(\frak x)[{\cal R}(\frak x)] - d {\cal R}(\frak x) [X_k(\frak x)]$ can be computed as
$$
 [X_k, {\cal R}](\frak x) = (0 ,0,  \Pi_\bot T_{a_{m - k}(\frak x)} \partial_x^{m - k} {\cal R}^\bot (\frak x)) 
 + (0,0,  \Pi_\bot T_{d a_{m - k}(\frak x)[{\cal R}(\frak x)]} \partial_x^{m - k} w) - d {\cal R}(\frak x) [X_k(\frak x)] ,
$$
where $\mathcal R = (\mathcal R^{(\theta)}, \mathcal R^{(y)}, \mathcal R^\bot)$. 
Note that $d {\cal R}(\frak x) [X_k(\frak x)] \in {\cal OS}(N - m)$ and that for any $0 \le k \le N+m$,
$$
 \big(0 ,0, \, \Pi_\bot T_{a_{m - x}(\frak x)} \partial_x^{m - k} {\cal R}^\bot(\frak x) \big)  \in {\cal OS}(N - (m - k)) \subseteq {\cal OS}(N - m).
$$ 
Formula \eqref{formula II} then follows by setting 
$ \mathcal C^\bot_{[X, \mathcal R]}(\frak x) := \Pi_\bot  \sum_{k = 0}^{m + N} T_{d a_{m - k}(\frak x)[{\cal R}(\frak x)]} \partial_x^{m - k} w$, and 
$$
{\cal R}_{[X, \mathcal R]}(\frak x) := \sum_{k = 0}^{m + N}\big(0, 0, \, \Pi_\bot T_{a_{m - k}(\frak x)} \partial_x^{m - k} {\cal R}^\bot(\frak x) \big) - d {\cal R}(\frak x) [X_k(\frak x)]\,. 
$$
The remaining part of the lemma is proved by using similar arguments.
\end{proof}
\begin{lemma}[\bf Commutators III]\label{commutatore due paradiff nonlin}
Let $N$, $p$, $q \in \N$,  $m$, $m' \in \Z$, and  let $m_* := \max\{ m + m' - 1, \, m, \, m' \}$.
For any $X^\bot \in {\cal OB}(m, N)$ and $Y^\bot \in {\cal OB}(m', N)$, one has
$$
[X^\bot , Y^\bot] = \mathcal C^\bot_{[X^\bot, Y^\bot]}+ {\cal R}^\bot_{[X^\bot, Y^\bot]} \, , \qquad {\cal C}^\bot_{[X^\bot, Y^\bot]} \in {\cal OB}(m_*, N),
\quad {\cal R}^\bot_{[X^\bot, Y^\bot]} \in {\cal OS}(N).
$$
If in fact  $X^\bot \in {\cal OB}^p(m, N)$ and $Y^\bot \in {\cal OB}^q(m', N)$, then
$$
{\cal C}^\bot_{[X^\bot, Y^\bot]} \in {\cal OB}^{p + q - 1}(m_*, N), \qquad {\cal R}^\bot_{[X^\bot, Y^\bot]} \in {\cal OS}^{p + q - 1}(N).
$$ 
\end{lemma}
\begin{proof}
By formula \eqref{formula generale paradiff vector field}, $X^\bot \in {\cal OB}(m, N)$ and $Y^\bot \in {\cal OB}(m', N)$ are of the form
$$
X^\bot(\frak x) =\Pi_\bot  \sum_{k = 0}^{N + m} T_{a_{m - k}(\frak x)} \partial_x^{m - k} w\, , \qquad
Y^\bot(\frak x) =\Pi_\bot  \sum_{k = 0}^{N + m'} T_{b_{m' - k}(\frak x)} \partial_x^{m' - k} w\, .
$$
With $X^\bot = \sum_{k = 0}^{N + m} X^\bot_k$ and $Y^\bot = \sum_{j = 0}^{N + m'} Y^\bot_j$ one gets 
$[X^\bot, Y^\bot] =  \sum_{k = 0}^{N + m}\sum_{j = 0}^{N + m'}  [X^\bot_k, Y^\bot_j] $ where 
$$
X^\bot_k(\frak x) = \Pi_\bot T_{a_{m - k}(\frak x)} \partial_x^{m - k} w,  \ \  \forall \, 0 \le k \le  N+m\, ,
\qquad Y_j^\bot(\frak x) :=  \Pi_\bot T_{b_{m' - j}(\frak x)} \partial_x^{m' - j} w, \ \ \forall \, 0 \le j \le  N+m'.
$$
To 
compute $[X^\bot_k, Y^\bot_j] $ for $k$, $j$ in the corresponding ranges, for notational convenience we let
$$
X^\bot_* := X^\bot_k, \quad Y^\bot_* := Y^\bot_j, \quad 
a(\frak x) := a_{m - k}(\frak x), \quad b(\frak x) := b_{m' - j}(\frak x), \quad n := m - k, \quad n' := m' - j.
$$  
One computes
$$
[ X^\bot_*, Y^\bot_* ]  = [\Pi_\bot T_{a} \partial_x^{n}\,,\, \Pi_\bot T_{b} \partial_x^{n'}]_{lin} w + 
\Pi_\bot 
T_{d_\bot  a(\frak x)[Y_*^\bot(\frak x)]} \partial_x^{n} w - \Pi_\bot T_{d_\bot  b(\frak x)[X_*^\bot(\frak x)]} \partial_x^{n'} w .
$$
Using the formula
$$
\Pi_\bot T_{a} \partial_x^{n} \circ \Pi_\bot T_{b} \partial_x^{n'} = 
\Pi_\bot \circ \big( T_{a} \partial_x^{n} \circ  T_{b} \partial_x^{n'}  + 
T_{a} \partial_x^{n} \circ ( \Pi_\bot -  \text{Id})  T_{b} \partial_x^{n'} \big),
$$
and the corresponding one  for $\Pi_\bot T_{b} \partial_x^{n'} \circ \Pi_\bot T_{a} \partial_x^{n}$, one obtains
$[X^\bot_*, \, Y^\bot_*] = {\cal C}^\bot_1 + {\cal R}^\bot_1$ where
$$
 {\cal C}^\bot_1 (\frak x)  := \Pi_\bot [T_{a} \partial_x^{n}, \, T_b \partial_x^{n'}]_{lin} w + 
\Pi_\bot T_{d_\bot  a(\frak x)[Y_*^\bot
(\frak x)]} \partial_x^{n} w - \Pi_\bot T_{d_\bot  b(\frak x)
[X_*^\bot(\frak x)]} \partial_x^{n'} w 
$$
and
$$
{\cal R}^\bot_1(\frak x) := \Pi_\bot T_{a(\frak x)} \partial_x^n \circ (\Pi_\bot - {\rm Id}) T_{b(\frak x)} \partial_x^{n'} w \, - \, 
\Pi_\bot T_{b(\frak x)} \partial_x^{n'} \circ (\Pi_\bot - {\rm Id}) T_{a(\frak x)} \partial_x^{n} w\,.  
$$
Since by assumption, there exist integers $s_N \ge 0$, $\sigma_N \ge 0$, so that for any $s \geq s_N$ there is $0 < \delta \equiv \delta(s, N) < 1$
and $0 < \e_0 \equiv \e_0(s, N) < 1$
with the property that
$a, b: {\cal V}^{s + \sigma_N}(\delta) \times  [0, \e_0]  \to  H^s(\T_1)$ are $C^\infty$-smooth and bounded, 
it then follows that
$$
\Pi_\bot T_{d_\bot  a(\frak x)[
Y_*
(\frak x)]} \partial_x^n w \in {\cal OB}(n, N), \qquad \Pi_\bot T_{d_\bot  b(\frak x)[
X_*
(\frak x)]} \partial_x^{n'} w \in {\cal OB}(n', N),
$$
and, in view of Corollary \ref{corollario espansione commutatore}, that
$$
\Pi_\bot [T_a \partial_x^n, \, T_b \partial_x^{n'}]_{lin} w = {\cal OB}(n + n' - 1, N) + {\cal OS}(N) \, .
$$
Furthermore, since $\Pi_\bot - {\rm Id}$ is a smoothing operator, one concludes that $ {\cal R}^\bot_1 \in {\cal OS}(N)$.
Altogether, we have proved that $[X^\bot_*, \, Y^\bot_*]$ is of the form  ${\cal C}^\bot_{[X^\bot_*, \, Y^\bot_*]} + {\cal R}^\bot_{[X^\bot_*, \, Y^\bot_*]}$
where 
$$
{\cal C}^\bot_{[X^\bot_*, \, Y^\bot_*]}  \in  {\cal OB}(n_*, N) , 
\quad n_* = \max\{n + n' - 1, n , n' \} \le m_*\, , 
\qquad {\cal R}^\bot_{[X^\bot_*, \, Y^\bot_*]} \in {\cal OS}(N).
$$
If in fact $X_*^\bot \in {\cal OB}^p(m, N)$ and $Y_*^\bot \in {\cal OB}^q(m', N)$, then $a$ is small of order $p - 1$, $b$ is small of order $q - 1$
and it follows that
$$
 \Pi_\bot T_{d_\bot  a(\frak x)[
 Y_*^\bot
 (\frak x)]} \partial_x^n w \in {\cal OB}^{p + q - 1}(n, N), \qquad \Pi_\bot T_{d_\bot  b(\frak x)[
 X_*^\bot
 (\frak x)]} \partial_x^{n'} w \in {\cal OB}^{p + q - 1}(n', N)\,, 
$$
$$
\Pi_\bot [T_a \partial_x^n, \, T_b \partial_x^{n'}]_{lin} w =  {\cal OB}^{p + q - 1}(n + n'- 1, N) + {\cal OS}^{p + q - 1}(N) \,, \qquad  {\cal R}^\bot_1 \in {\cal OS}^{p + q - 1}(N)\,. 
$$
One then infers that ${\cal C}^\bot_{[X^\bot_*, \, Y^\bot_*]} \in  {\cal OB}^{p + q - 1}(n_*, N)$ and ${\cal R}^\bot_{[X^\bot_*, \, Y^\bot_*]} \in {\cal OS}^{p + q - 1}(N)$.
\end{proof}
\begin{lemma}[{\bf Commutators IV}]\label{lemma commutatori Fourier multiplier}
Let $N$, $p$, $q \in \N$,  $m$, $m' \in \Z$, and  let $m_* := \max\{ m + m' - 1, \, m, \, m' \}$.

\smallskip
\noindent
$(i)$ For any $\mathcal M^\bot \in {\cal OF}(m, N)$ and $\mathcal M'^\bot \in {\cal OF}(m', N)$
$$
[\mathcal M^\bot, \, \mathcal M'^\bot]   \in {\cal OF}\big(m\lor m' , N \big).
$$  
If in fact $ \mathcal M^\bot \in {\cal OF}^p(m, N)$ and  $\mathcal M'^\bot \in {\cal OF}^q(m', N)$, 
then $[\mathcal M^\bot , \mathcal M'^\bot] \in {\cal OF}^{p + q - 1}(m \lor m' , N )$. 

\smallskip
\noindent
$(ii)$ For any $X^\bot \in {\cal OB}(m, N)$ and $\mathcal M^\bot \in {\cal OF}(m', N)$, 
$$
[X^\bot, \mathcal M^\bot] = {\cal C}^\bot_{[X^\bot, \mathcal M^\bot]} + {\cal R}^\bot_{[X^\bot, \mathcal M^\bot]}, \qquad  
{\cal C}^\bot_{[X^\bot, \mathcal M^\bot]} \in {\cal OB}(m_* ,N), \qquad  {\cal R}^\bot_{[X^\bot, \mathcal M^\bot]} \in {\cal OS}(N).
$$ 
If $X^\bot \in {\cal OB}^p(m, N)$ and $\mathcal M^\bot \in {\cal OF}^q(m', N)$, then 
$$
{\cal C}^\bot_{[X^\bot, \mathcal M^\bot]} \in {\cal OB}^{p + q - 1}(m_*, N),\qquad  
{\cal R}^\bot_{[X^\bot, \mathcal M^\bot]} \in {\cal OS}^{p + q - 1}(N).
$$ 

\smallskip
\noindent
$(iii)$ For any $\mathcal M= (0, 0, \mathcal M^\bot)$ with $\mathcal M^\bot \in {\cal OF}(m, N)$ and 
$\mathcal R = \big( \mathcal R^{(\theta)}, \mathcal R^{(y)}, \mathcal R^\bot \big) \in {\cal OS}(N)$
$$
[\mathcal M, \mathcal R] = (0,0,  {\cal C}^\bot_{[\mathcal M, \mathcal R^]}) + {\cal R}_{[\mathcal M, \mathcal R]}, \qquad  
{\cal C}^\bot_{[\mathcal M, \mathcal R]} \in {\cal OF}(m, N), \qquad {\cal R}_{[\mathcal M, \mathcal R]} \in {\cal OS}(N - m)
$$ 
If $\mathcal M^\bot \in {\cal OF}^p(m, N)$ and $\mathcal R \in {\cal OS}^q(N)$, then ${\cal C}^\bot_{[\mathcal M, \mathcal R]} \in {\cal OF}^{p + q - 1}(m, N)$ 
and ${\cal R}_{[M, \mathcal R]} \in {\cal OS}^{p + q - 1}(N)$.  
\end{lemma}
\begin{proof}
Since the claims of the lemma follow by arguing as in the proofs of  Lemma \ref{comm smoothing paradiff} and Lemma \ref{commutatore due paradiff nonlin}, the details of the proofs are omitted. 
\end{proof}

\subsection{Flows of para-differential vector fields}

In this subsection we study the flow of para-differential vector fields of the form $Y = (0, 0, \, Y^\bot)$  with
\begin{equation}\label{campo vettoriale astratto flussi paradiff}
\begin{aligned}
Y^\bot(\frak x) = \Pi_\bot T_{a_{m}(\frak x)} \partial_x^{ m} w \in {\cal OB}^p(m, N),  \qquad  N, \, p \ge 1\,, \ m \le 0\, . 
\end{aligned}
\end{equation}
By Definition  \ref{formula generale paradiff vector field},  there are integers $s_N \geq 0$, $\sigma_N\ge 0$ so that for any $s \geq s_N$ 
there exist $0 < \delta \equiv \delta(s, N) < 1$  and $0 < \e_0 \equiv \e_0(s, N) < 1$ with the property that
$$
a_{m} :  {\cal V}^{s  + \sigma_N}(\delta) \times [0, \e_0]   \to H^s(\T_1), \, (\frak x, \e) \mapsto a_{m}(\frak x) \equiv a_{m}(\frak x, \e)
$$
is $C^\infty-$smooth and bounded. In the sequel, we will often {\em tacitly} increase $s_N$, $\sigma_N$ 
and decrease $\delta \equiv \delta(s, N)$, $\e_0\equiv \e_0(s, N)$, whenever needed.

Denote by $\Phi_Y (\tau, \cdot)$
the flow associated with $Y$. 
By the standard ODE theorem in Banach spaces, for any $s \ge s_N$, there exist $0 < \delta \equiv \delta(s, N) < 1$, and 
$0 < \e_0 \equiv \e_0(s, N) \ll  \delta$,
so that for any $-1 \le \tau \le 1$, 
\begin{equation}\label{domain flow}
\Phi_Y(\tau, \cdot) \in C^\infty_b \big({\cal V}^s(\delta) \times [0, \e_0], \, {\cal V}^s(2 \delta) \big)\, .
\end{equation}
It then follows that for any $-1 \le  \tau \le 1$ and any $\frak x \in \mathcal V^s(\delta),$ one has
$\Phi_Y(- \tau, \Phi_Y(\tau, \frak x)) = \frak x$.
\begin{remark}\label{inverse of flow}
For notational convenience,
$\Phi_Y(- \tau, \cdot)$ is referred to as the inverse of $\Phi_Y(\tau, \cdot)$ and we write $\Phi_Y(\tau, \cdot)^{-1} = \Phi_Y(- \tau, \cdot)$. In particular, $\Phi_Y(1, \cdot)^{-1} = \Phi_Y(- 1, \cdot)$. 
Using our convention of tacitly decreasing $\delta$ and $\e_0$, if needed, $\Phi_Y(\tau, \cdot)^{-1}$ is defined  for $(\frak x, \e)  \in {\cal V}^s(\delta) \times [0, \e_0]$.
More generally,  a similar convention is used for diffeomorphisms between neighborhoods of $\T^{S_+} \times 0 \times 0$ in $\mathcal E_s$ throughout the paper. 
\end{remark}

The following lemma provides a para-differential expansion of the flow $\Phi_Y (\tau, \cdot)$. 
\begin{lemma}\label{expansion flow paradiff} 
Let $N$, $p \in \N$ and assume that the normal component $Y^\bot$ of $Y = (0, 0, \, Y^\bot)$ satisfies \eqref{campo vettoriale astratto flussi paradiff}. 
Then for any $-1 \le \tau \le 1$, $\Phi_Y(\tau, \frak x)$ admits an expansion of the form
$$
\Phi_Y(\tau, \frak x) = \frak x + \big(0, 0, \, {\Upsilon}^\bot(\tau, \frak x) + {\cal R}^\bot_{N}(\tau, \frak x) \big)  
$$
where 
$${\Upsilon}^\bot(\tau, \frak x) = \Pi_\bot \sum_{k = 0 }^{N +m} T_{b_{m - k}(\tau, \frak x)}\partial_x^{m - k} w  \in {\cal OB}^p(m, N)\, , \qquad
{\cal R}^\bot_{N}(\tau, \frak x) \in {\cal OS}^{2 p - 1}(N) \, . 
$$
\end{lemma}
\begin{proof}
The normal component $\Phi_Y^\bot(\tau, \frak x)$ of the flow $\Phi_Y(\tau, \frak x)$ satisfies the integral equation
\begin{equation}\label{integral equation}
\Phi_{Y}^\bot(\tau, \frak x) = w  + \int_0^\tau Y^\bot(\Phi_Y(t, \frak x))\, d t\, , \qquad \forall \, -1 \le \tau \le 1\,.
\end{equation}
To solve it, we make the ansatz that $\Phi_Y^\bot(\tau, \frak x)$ admits an expansion of the form
\begin{equation}\label{ansatz}
\Phi_Y^\bot(\tau, \frak x) = w +  {\Upsilon}^\bot(\tau, \frak x)  + {\cal R}^\bot_{N}(\tau, \frak x)\, , \qquad
{\Upsilon}^\bot(\tau, \frak x) = \Pi_\bot \sum_{k = 0 }^{N +m} T_{b_{m - k}(\tau, \frak x)} \partial_x^{m - k} w \, ,
\end{equation}
with the property that there exist $s_N \ge 0$, $\sigma_N \ge 0$ so that the following holds: for any $s \geq s_N$, 
there exist $0 < \delta \equiv \delta(s, N) < 1$ and $0 < \e_0 \equiv \e_0(s, N) < 1$ so that 
for any $-1 \le \tau \le 1$ and $ 0 \le k \le N+m$,
\begin{equation}\label{ansatz b - j cal RN}
b_{m-k}(\tau, \cdot) \in C^\infty_b \big( {\cal V}^{s + \sigma_N}(\delta) \times [0, \e_0] , \, H^s(\T_1) \big),   \ \  b_{m - k}  \ \text{small of order } p - 1\,, 
\qquad {\cal R}^\bot_{N}(\tau, \cdot ) \in {\cal OS}^p(N)\,.
\end{equation}
To determine $\big( b_{m-k}\big)_{0 \le k \le N+m}$ and $ {\cal R}^\bot_{N}$, 
in terms of the coefficient $a_{m}$ of $Y^\bot$ in \eqref{campo vettoriale astratto flussi paradiff},
we compute the expansion of the right hand side of the equation \eqref{integral equation} by 
substituting the ansatz \eqref{ansatz} into the integrand $Y^\bot (\Phi_{Y}(t, \frak x))$. In view of definition \eqref{campo vettoriale astratto flussi paradiff}
of $Y^\bot$, one gets for any $-1 \le t \le 1$,
\begin{equation}\label{marmelade 0}
\begin{aligned}
 Y^\bot(\Phi_Y(t, \frak x)) & = \Pi_\bot  T_{a_{m}(\Phi_Y(t, \frak x))} \partial_x^{ m} \Phi_Y^\bot (t, \frak x) \\
& = \Pi_\bot  T_{a_{m}(\Phi_Y(t, \frak x))} \partial_x^{m} \Big(w +  \Pi_\bot \sum_{k = 0  }^{N  +m} T_{b_{m-k}(t, \frak x)} \partial_x^{m-k} w 
+ {\cal R}^\bot_{N}(t, \frak x) \Big) \,. \\
\end{aligned}
\end{equation} 
Using that $\Pi_\bot - {\rm Id}$ is a smoothing operator
and that $\Phi_Y(\tau, \cdot) \in C^\infty_b\big( {\cal V}^s(\delta)\times [0, \e_0] , \, {\cal V}^s(2 \delta) \big)$ one gets
\begin{equation}\label{romolo 0}
\begin{aligned}
&  \Pi_\bot  T_{a_{m}(\Phi_Y(t, \frak x))} \partial_x^{m}  (\Pi_\bot - {\rm Id}) 
\sum_{k = 0}^{N  +m} T_{b_{m-k}(t, \frak x)} \partial_x^{m-k} w \, \in {\cal OS}^{2 p - 1}(N) \stackrel{p \geq 1}{\subseteq} {\cal OS}^p(N)\,, \\
&  \Pi_\bot \sum_{k = N+1+2m}^{N +m}  T_{a_{m}(\Phi_Y(t, \frak x))} \partial_x^{m} T_{b_{m-k}(t, \frak x)} \partial_x^{m-k} w  \in {\cal OS}^{2 p - 1}(N)
 \stackrel{p \geq 1}{\subseteq} {\cal OS}^p(N)
\end{aligned}
\end{equation}
where we recall that $m \le 0$ and that by our convention, a sum of terms over an empty index set equals $0$. 
Moreover, by increasing $s_N, \sigma_N$ if needed, it follows that for any $s \ge s_N$ and $-1 \le t \le 1,$
the map $A(t, \frak x) := \Pi_\bot T_{a_{m}(\Phi_Y(t, \frak x))} \partial_x^{m}$ satisfies (after decreasing $\delta$ and $\e_0$ if necessary)
$$
A(t, \cdot)  \in  C^\infty_b\big( {\cal V}^{s+ \sigma_N}(\delta) \times [0, \e_0], \ {\cal B}(H_\bot^{s + N + 1}(\T_1)) \big) 
$$
and hence in view of \eqref{ansatz b - j cal RN},
\begin{equation}\label{def op cal A lemma flusso}
 A(t, \cdot)[{\cal R}^\bot_{N}(t, \cdot)] \in  {\cal OS}^{2 p - 1}(N)  \stackrel{p \geq 1}{\subseteq} {\cal OS}^p(N) \, .
\end{equation}
In view of \eqref{romolo 0} - \eqref{def op cal A lemma flusso}, we rewrite \eqref{marmelade 0} as 
\begin{equation}\label{marmelade 100}
\begin{aligned}
 Y^\bot(\Phi_Y(t, \frak x)) & =\Pi_\bot T_{a_{m}(\Phi_Y(t, \frak x))} \partial_x^{m}w 
 +  \Pi_\bot \sum_{ k= 0}^{N + 2m} T_{a_{m}(\Phi_Y(t, \frak x))} \partial_x^{m} T_{b_{m-k}(t, \frak x)} \partial_x^{m-k} w 
 + {\cal OS}^p(N) \, .
\end{aligned}
\end{equation}
Since $a_{m}$ and $b_{m-k}$ are small of order $p - 1$ (cf. \eqref{ansatz b - j cal RN}),
it follows from Lemma \ref{lemma composizione nostri simboli} that for any $0 \le k \le N + 2m$,
the term $T_{a_{m}(\Phi_Y(t, \frak x))} \partial_x^{m} T_{b_{m-k}(t, \frak x)} \partial_x^{m-k} w $ has an expansion of the form 
\begin{equation}\label{marmelade 1}
T_{a_{m}(\Phi_{Y}(t, \frak x)) b_{m-k}(t, \frak x)} \partial_x^{ 2m - k } w + 
 \sum_{j = 1}^{N + 2m - k} K(j, m) T_{a_{m}(\Phi_{Y}(t, \frak x)) \partial_x^j b_{m-k}(t, \frak x)} \partial_x^{ 2m - k - j} w + {\cal OS}^{2p -1}(N) 
\end{equation}
with the constants $K(j, m)$ given as in Lemma \ref{lemma composizione nostri simboli},  implying that 
\begin{equation}\label{marmelade 2}
\begin{aligned}
& \Pi_\bot \sum_{ k=0}^{N + 2m}  T_{a_{m}(\Phi_Y(t, \frak x))} \partial_x^{m} T_{b_{m - k}(t, \frak x)} \partial_x^{m-k} w 
 = \Pi_\bot \sum_{k=0}^{N +2m} T_{a_{m}(\Phi_{Y}(t, \frak x)) b_{m-k}(t, \frak x)} \partial_x^{ 2m - k} w \\
& + \Pi_\bot \sum_{k=0}^{N +2m} \sum_{j=1}^{N+2m - k} K(j, m) T_{a_{m}(\Phi_{Y}(t, \frak x)) \partial_x^j b_{m - k}(t, \frak x)} \partial_x^{ 2m - k - j} w 
   + {\cal OS}^{2 p - 1}(N)  \\
& =  \Pi_\bot \sum_{i = 0}^{N +2m} T_{g_{2m - i}(t, \frak x)} \partial_x^{2m - i} w + {\cal OS}^{2 p - 1}(N)\, ,
\end{aligned}
\end{equation}
where $g_{2m}(t, \frak x) = a_{m}(\Phi_{Y}(t, \frak x)) b_{m}(t, \frak x)$
and for any $1 \le i \le N + 2m$,
\begin{equation}\label{def g eta}
g_{2m-i}(t, \frak x) = a_{m}(\Phi_{Y}(t, \frak x)) b_{m-i}(t, \frak x) + 
\sum_{k= 1}^{i-1} K(i-k, m) a_{m}(\Phi_{Y}(t, \frak x)) \partial_x^{i-k} b_{m - k}(t, \frak x) \,.
\end{equation}
Combining \eqref{integral equation}-\eqref{def g eta} then yields the following identity,
$$
\begin{aligned}
\Pi_\bot \sum_{k = 0 }^{N +m} T_{b_{m - k}(\tau, \frak x)} \partial_x^{m - k} w & = 
 \Pi_\bot \big( \int_{0}^\tau T_{a_{m}(\Phi_Y(t, \frak x))} d t \big) \,  \partial_x^{m}w 
 +  \Pi_\bot \big( \int_{0}^\tau T_{a_{m}(\Phi_Y(t, \frak x)) b_m(t, \frak x)} d t \big) \, \partial_x^{2m}w \\
&  + \Pi_\bot  \sum_{i = 1}^{N +2m} \big(\int_0^\tau T_{g_{2m - i}(t, \frak x)}  d t\big) \, \partial_x^{2m - i} w  + {\cal OS}^{2 p - 1}(N)\, .
\end{aligned}
$$
Let us first consider the case where $m \le -1$. 
We then require that the coefficients $b_{m-k}$, $0 \le k \le N+m$, satisfy the following system of equations,
\begin{equation}\label{marmelade 3}
\begin{aligned}
b_{m}(\tau, \frak x) & = \int_0^\tau a_{m}(\Phi_{Y}(t, \frak x))\, d t, \qquad \qquad \ \ \  b_{m - k}(\tau, \frak x)  = 0, \quad  \forall \, 1 \le k \le |m| -1, \\
b_{2m}(\tau, \frak x) & = \int_0^\tau  a_{m}(\Phi_Y(t, \frak x)) b_m(t, \frak x) \, dt,
\quad \  \  b_{m - k}(\tau, \frak x)  = 
\int_0^\tau  g_{m - k}(t, \frak x)\, d t, \ \ \forall \,  |m| + 1 \le k \le N + 2m . \\
\end{aligned}
\end{equation}
Since for  any $ |m| + 1 \le k \le N + 2m$,  $g_{m-k}$ only depends on  $b_{m-k'}$ with $k' \le k + m \le k -1$ (cf. \eqref{def g eta}),
the coefficients $b_{m-k}$ are determined inductively in terms of $a_{m}$. 
One then verifies that the properties of the coefficients $b_{m-k}$, stated in ansatz \eqref{ansatz b - j cal RN}, are satisfied. 
The remainder ${\cal R}^\bot_{N}$ then satisfies the following integral equation 
\begin{equation}\label{eq integrale resto}
{\cal R}^\bot_{N}(\tau, \frak x)  = {\cal Q}^\bot_N(\tau, \frak x) + \int_0^\tau A(t, \frak x)[{\cal R}^\bot_{N}(t, \frak x)] d t\, ,
\end{equation}
where ${\cal Q}^\bot_{N}(\tau, \cdot) \in {\cal OS}^{2 p- 1}(N)$ is given by the sum of the two terms in \eqref{romolo 0}
and the operator $A(t, \frak x)$ is defined in \eqref{def op cal A lemma flusso}. 
By increasing $s_N$ if needed, it follows that for any $s \geq s_N$, 
$$
\| {\cal R}^\bot_N(\tau, \frak x) \|_{s + N + 1} \leq 
\sup_{\tau \in [- 1, 1]} \| {\cal Q}^\bot_N(\tau, \frak x) \|_{s + N +1} + \int_0^\tau \| A(t, \frak x) \|_{{\cal B}(H_\bot^{s + N + 1}(\T_1))} \| {\cal R}^\bot_N(t, \frak x)\|_{s + N +1}\, d t
$$
and hence by the Gronwall Lemma, one infers that ${\cal R}^\bot_N$ satisfies
$$
\| {\cal R}^\bot_N(\tau, \frak x) \|_{s + N +1} 
\lesssim_{s, N} {\rm exp}\big( \int_{-1}^1 \| A(t, \frak x) \|_{{\cal B}(H_\bot^{s + N + 1}(\T_1))}\, d t \big)\sup_{t \in [- 1, 1]} \| {\cal Q}^\bot_N(t, \frak x) \|_{s + N +1}\, ,
$$
implying that $ \| {\cal R}^\bot_N(\tau, \frak x) \|_{s + N + 1} \lesssim_{s, N} (\e + \| y \| + \| w \|_s)^{2 p - 1}$. 
Similar estimates hold for the derivatives of ${\cal R}^\bot_N$. Altogether we have shown that ${\cal R}^\bot_N \in {\cal OS}^{2 p - 1}(N)$. 

Finally let us consider case $m=0$. We then require that the coefficients $b_{-k}$, $0 \le k \le N$, satisfy the following system of equations,
\begin{equation}\label{marmelade 3}
\begin{aligned}
b_{0}(\tau, \frak x)  = \int_0^\tau a_{0}(\Phi_{Y}(t, \frak x))\, d t + \int_0^\tau  a_{0}(\Phi_Y(t, \frak x)) b_0(t, \frak x) \, dt,  \qquad
 b_{ - k}(\tau, \frak x)   = 
\int_0^\tau  g_{ - k}(t, \frak x)\, d t, \ \forall \,    1 \le k \le N . \nonumber
\end{aligned}
\end{equation}
The solution $b_0$ then reads  $b_{0}(\tau, \frak x)  = e^{\int_0^\tau  a_{0}(\Phi_Y(t, \frak x)) \, dt}  -1$. 
The remaining part of the proof then follows as in the case $m \le -1$.
\end{proof}
\begin{lemma}\label{Corollario d Phi F}
Let $N$, $p \in \N$ and let $\Phi_Y(\tau, \frak x)$ denote the flow map considered in Lemma \ref{expansion flow paradiff},
corresponding to the vector field $Y = (0, 0, \, Y^\bot)$, with $Y^\bot(\frak x) = \Pi_\bot T_{a_{m}(\frak x)} \partial_x^{ m} w$ 
and $m \le 0$, satisfying \eqref{campo vettoriale astratto flussi paradiff}.
Then for any $-1\le  \tau \le 1,$
$d \Phi_Y( \tau, \frak x)^{- 1}[\widehat{\frak x}]$  admits an expansion of the form
\begin{equation}\label{espansion d Phi tau inverse}
d \Phi_Y( \tau, \frak x)^{- 1}[\widehat{\frak x}] = \widehat{\frak x}  
+ \big(0, 0, \, {\Upsilon}^\bot(\tau, \frak x)[\widehat{\frak x} ] 
+ {\cal R}^\bot_{N}(\tau, \frak x)[\widehat{\frak x} ]  \big) \, ,
\end{equation}
$$
{\Upsilon}^\bot(\tau, \frak x) [\widehat{\frak x} ]  := \Pi_\bot \sum_{k = 0 }^{N +m} T_{b_{m - k}(\tau, \frak x)}\partial_x^{m - k} [\widehat w]
+ \Pi_\bot \sum_{k = 0}^{N +m} T_{B_{m - k}(\tau, \frak x)[\widehat{\frak x}]} \partial_x^{m - k} w 
$$
with the following properties: there exist $s_N$, $\sigma_N \geq N$ so that for any $s \geq s_N$, 
there exist $\delta \equiv \delta(s, N) > 0$ and $0 < \e_0 \equiv \e_0(s, N) < 1$ so that the following holds: 
for any $0 \le k \le N + m$ and $-1 \le \tau \le 1$,
$$
\begin{aligned}
 b_{m - k}(\tau, \cdot) \in C^\infty_b( {\cal V}^{s + \sigma_N}(\delta) & \times [0, \e_0] , \, H^s(\T_1)), 
\qquad B_{m - k}(\tau, \cdot) \in C^\infty_b\big( {\cal V}^{s + \sigma_N}(\delta) \times [0, \e_0], \, {\cal B}(E_{s + \sigma_N}, H^s(\T_1))\big)\,, \\
& {\cal R}^\bot_N(\tau, \cdot) \in C^\infty_b \big(  {\cal V}^s(\delta) \times [0, \e_0], \, {\cal B}(H^s(\T_1), H^{s + N +1}_\bot(\T_1)) \big)
\end{aligned}
$$
with $b_{m- k}(\tau, \cdot)$, $B_{m- k}(\tau, \cdot)$, and ${\cal R}^\bot_N(\tau, \cdot)$ being small of order $p - 1$,
and the expansion above holds for any $\frak x \in \mathcal V^{s+\sigma_N}(\delta)$ and $\widehat{\frak x} \in E_{s+\sigma_N}$.
\end{lemma}
\begin{proof}
First we note that for any $-1 \le \tau \le 1$, $d \Phi_Y(\tau, \frak x)^{- 1} = d \Phi_Y(- \tau, \Phi_Y(\tau, \frak x)) $ and that by Lemma \ref{expansion flow paradiff},
$$
\Phi_Y(\tau, \frak x) = \frak x + \Big( 0, 0, \, \Pi_\bot \sum_{k = 0}^{N +m} T_{b_{m - k}(\tau, \frak x; \Phi_Y)} \partial_x^{m - k} w + {\cal R}^\bot_N(\tau, \frak x; \Phi_Y) \Big)
$$
with $b_{m-k}(\tau, \cdot; \Phi_Y) \in C^\infty_b\big( {\cal V}^{s + \sigma_N}(\delta) \times [0, \e_0], \, H^s(\T_1) \big)$ being small of order $p - 1$ 
and ${\cal R}^\bot_N(\tau, \cdot; \Phi_Y) \in {\cal OS}^p(N)$. 
To simplify notation, let $\widetilde b_{m-k}(\tau, \frak x):= b_{m-k}(\tau, \frak x; \Phi_Y)$ and $\widetilde {\cal R}^\bot_N(\tau, \frak x) := {\cal R}^\bot_N(\tau, \frak x; \Phi_Y)$.
Then the normal component of 
$d \Phi_Y( \tau, \frak x)^{- 1}[\widehat{\frak x}] - \widehat{\frak x}$ can be computed as follows
$$
\Pi_\bot \sum_{k = 0}^{N +m} T_{\widetilde b_{m-k}(- \tau, \Phi_Y(\tau, \frak x))} \partial_x^{m-k} \widehat w 
+ \Pi_\bot \sum_{k = 0}^{N +m} T_{d \widetilde b_{m-k}(- \tau, \Phi_Y(\tau, \frak x))[\widehat{\frak x}]} \partial_x^{m-k} \Phi_Y^\bot(\tau, \frak x) 
+ d \widetilde{{\cal R}}^\bot_N(- \tau, \Phi_Y(\tau, \frak x)) [\widehat{\frak x}] \, .
$$ 
By expanding the terms 
$T_{d \widetilde b_{m-k}(- \tau, \Phi_Y(\tau, \frak x))[\widehat{\frak x}]} \partial_x^{m-k} \Phi_Y^\bot(\tau, \frak x)$ with the help of Lemma \ref{lemma composizione nostri simboli},
one is led to define 
$b_{m-k}(\tau, \frak x)$, $B_{m-k}(\tau, \frak x)$, and ${\cal R}^\bot_N(\tau, \frak x)$
with the claimed properties.
\end{proof}
Combining Lemma \ref{expansion flow paradiff} and Lemma \ref{Corollario d Phi F}, one obtains an expansion of the
pullback of various types of vector fields by the time one flow map $\Phi_Y(1, \cdot)$:
\begin{lemma}\label{lemma coniugazioni}
Let $N$, $p$, $q \in \N$ and let $\Phi_Y(1, \frak x)$ denote the time one flow map,
corresponding to the vector field $Y = (0, 0, \, Y^\bot)$, with $Y^\bot(\frak x) = \Pi_\bot T_{a_{m}(\frak x)} \partial_x^{ m} w$ 
and $m \le 0$, satisfying \eqref{campo vettoriale astratto flussi paradiff} (cf.  Lemma \ref{expansion flow paradiff}). 
Then the following holds:\\
$(i)$ For any $X := (0,0, X^\bot)$ with $X^\bot \in {\cal OB}^q(n, N)$ and $n \geq 0$,  the pullback 
$\Phi_Y^* X(\frak x) = d\Phi_Y(1, \frak x)^{-1} X(\Phi_Y(1, \frak x))$ of $X$ by $\Phi_Y(1, \cdot)$ admits an expansion of the form
$$
\Phi_Y^* X (\frak x)=    \big(0, 0, \, X^\bot(\frak x) + \Upsilon^\bot(\frak x) + {\cal R}_N^\bot(\frak x) \big)
$$
where $\Upsilon^\bot \in {\cal OB}^{p + q - 1}(n, N)$ and ${\cal R}_N^\bot \in {\cal OS}^{p + q - 1}(N)$. 

%
\noindent
$(ii)$ For any $X$ in ${\cal OS}^q(N)$, the pullback $\Phi_Y^* X$ of $X$ by $\Phi_Y(1, \cdot)$ admits an expansion of the form
$$
\Phi_Y^* X(\frak x) =  X(\frak x) + \big(0, 0, \, \Upsilon^\bot(\frak x) \big) + {\cal R}_N(\frak x)
$$ 
where $ \Upsilon^\bot  \in {\cal OB}^{p + q - 1}(m, N)$ and ${\cal R}_N \in {\cal OS}^{p + q - 1}(N)$. 
\end{lemma}
\begin{proof}
We only prove item $(i)$ since item $(ii)$ can be proved by similar arguments. Since by \eqref{espansione commutatori notazioni} with $\tau =1$
$$
\Phi_Y^* X( \frak x) = X(\frak x) + \int_0^1 (d\Phi_Y(t, \frak x))^{- 1}[X, Y](\Phi_Y(t, \frak x))\, d t,
$$
we analyze for any $t \in [0, 1]$ the vector field 
\begin{equation}\label{campo vettoriale nella proof}
Z(t, \frak x) := (d\Phi_Y(t, \frak x) )^{- 1}[X, Y](\Phi_Y(t, \frak x)).
\end{equation} 
Recall that $Y^\bot(\frak x) = \Pi_\bot T_{a_{m}(\frak x)} \partial_x^{m} w \in {\cal OB}^p(m, N)$.
Taking into account that $m_* = \max\{n+m-1, m, n\} = n$ (since $n \geq 0 \ge m$), it follows from Lemma \ref{commutatore due paradiff nonlin} that
$[X, Y] = \big(0, 0, [X^\bot , Y^\bot]\big)$ satisfies
$$
 [X^\bot , Y^\bot] =  \mathcal C^\bot_{[X^\bot, Y^\bot]}+ {\cal R}^\bot_{[X^\bot, Y^\bot]} , \qquad
  \mathcal C^\bot_{[X^\bot, Y^\bot]} \in {\cal OB}^{p + q - 1}(n, N), \quad  {\cal R}^\bot_{[X^\bot, Y^\bot]}  \in {\cal OS}^{p + q - 1}(N).
$$ 
By Definitions \ref{paradiff vector fields} - \ref{def smoothing vector fields},  and  Lemma \ref{expansion flow paradiff}, Lemma \ref{Corollario d Phi F},
as well as Lemma \ref{lemma composizione nostri simboli}, one obtains 
\begin{equation}\label{scrittura esplicita comm X Y}
\int_0^1 Z(t, \frak x)\, d t = \big(0, 0, \, \Upsilon^\bot(\frak x) + {\cal R}_N^\bot(\frak x) \big)
\end{equation}
with $\Upsilon^\bot(\frak x) \in {\cal OB}^{p + q - 1}(n, N)$ and ${\cal R}_N^\bot(\frak x) \in {\cal OS}^{p + q - 1}(N)$. 
\end{proof}
Next we  analyze the pullback $\Phi_Y^* X_{\cal N}$ of the Hamiltonian vector field $X_{{\cal N}}(\frak x)$ with $\mathcal N$ being the following Hamiltonian
in normal form (cf. \eqref{hamiltoniana totale}),
\begin{equation}\label{cal N Q Omega bot}
 {\cal N}(\frak x) := (\omega + \e \widehat \omega) \cdot y  + Q(y) +  \frac12 \big\langle D^{- 1}_\bot \Omega_\bot w, w \big\rangle\,, \qquad  \omega \in \Pi \, , \ \ \widehat \omega \in \R^{S_+}\, , 
 \end{equation}
 where the Fourier multipliers $D^{- 1}_\bot$ and  $\Omega_\bot \equiv  \Omega_\bot(\omega)$ are given by \eqref{definition Omega bot}
and $Q$ is assumed to be a map in $C^\infty_b(B_{S_+}(\delta) \times [0, \e_0], \, \R)$ with $Q(0) = 0$ and $\nabla_y Q(0) = 0$.
Since $\partial_x D_\bot^{- 1} \Omega_\bot = \ii \Omega_\bot$, the vector field $X_{{\cal N}}(\frak x)$ then reads
\begin{equation}\label{forma campo vettoriale forma normale astratto}
X_{\cal N}(\frak x) = \begin{pmatrix}
-  \nabla_y {\cal N}(\frak x) \\
 \nabla_\theta {\cal N}(\frak x) \\
\partial_x \nabla_\bot {\cal N}(\frak x)
\end{pmatrix} =  \begin{pmatrix}
- \omega -  \e \widehat \omega - \nabla_y Q(y) \\
0 \\
\ii \Omega_\bot w 
\end{pmatrix}
\end{equation}
and its differential is given by 
\begin{equation}\label{forma campo vettoriale forma normale astratto differenziale}
d X_{{\cal N}}(\frak x) = \begin{pmatrix}
0 & - d_y \nabla_y Q(y) & 0 \\
0 & 0 & 0 \\
0 & 0 & \ii \Omega_\bot
\end{pmatrix} \, .
\end{equation}
Note that  ${\cal N}(\frak x)$ does not depend on $\theta$, but only on $y$, $w$, and $\e$. For notational convenience, we will
often write ${\cal N}(y, w)$ instead of ${\cal N}(\frak x)$.
The following result on the expansion of $\ii \Omega_\bot$ can be found in \cite{Kap-Mon-2}.
\begin{lemma}[{\cite[Lemma C.7]{Kap-Mon-2}}]\label{lemma espansione Omega bot}
For any $N \in \N$, the Fourier multiplier $\ii \Omega_\bot$ has an expansion of the form 
$$
\ii \Omega_\bot = - \partial_x^3 + \sum_{k = 1}^N c_{- k} \partial_x^{- k} + {\cal R}^\bot_N\, ,
$$
where $c_{- k}\equiv c_{-k}(\omega)$ are real constants, depending only on the parameter $\omega \in \Pi$, and 
 ${\cal R}^\bot_N \equiv \mathcal R_N^\bot(\omega)$ is in ${\cal B}(H^s_\bot(\T_1), H^{s + N + 1}_\bot(\T_1))$ for any 
 $s \in \R$. 
\end{lemma}
\begin{lemma}\label{corollario coniugazione Omega bot}
Let $X_{\mathcal N}$ be the vector field given by \eqref{forma campo vettoriale forma normale astratto} and
$Y = (0, 0, \, Y^\bot)$ be the vector field with
$Y^\bot(\frak x) = \Pi_\bot T_{a_{m}(\frak x)} \partial_x^{ m} w$ 
and $m \le 0$, satisfying \eqref{campo vettoriale astratto flussi paradiff} with $p, N \in \N$. 
Furthermore let $\Phi_Y(1, \frak x)$ be the time one flow map corresponding to the vector field $Y$
 (cf.  Lemma \ref{expansion flow paradiff}).
Then the following holds:\\
 \noindent
$(i)$ If in addition $Y^\bot(\frak x) =  \Pi_\bot T_{a_{m}(\frak x)} \partial_x^{ m} w$ is in ${\cal OB}^2_{w}(m, N)$, hence $a_m(\frak x) \equiv a_m(\theta, y)$ 
independent of $w$, and  if $\langle a_{m}(\frak x) \rangle_x = 0$,
then $[X_{\cal N}, Y]$ is of the form  $\big( 0, 0 , \, [X_{\cal N}, Y]^\bot  \big)$ with  $[X_{\cal N}, Y]^\bot \in {\cal OB}^2(2 + m, N)$ and admits an expansion of the form
$$
[X_{\cal N}, Y]^\bot (\frak x) =   \Pi_\bot T_{-3\partial_x a_{m}(\frak x)} \partial_x^{2 + m} w + 
\mathcal C^\bot( \frak x)  + {\cal R}^\bot_{N}( \frak x)  + {\cal OB}^3(m, N),
$$
where  $\mathcal C^\bot( \frak x)  \in {\cal OB}^2_{w}(1 + m, N)$ and 
${\cal R}^\bot_{N}( \frak x) \in {\cal OS}_{w}^2(N)$.
Moreover $\mathcal C^\bot( \frak x)$ and ${\cal R}^\bot_{N}( \frak x)$
are of the form $\mathcal C^\bot( \frak x)  = \mathcal C^\bot( \theta, y) [w]$ and, respectively,
${\cal R}^\bot_{N}( \frak x) = {\cal R}^\bot_{N}(\theta, y)[w] $, and the  diagonal matrix elements of $\mathcal C^\bot( \theta, y)$ and ${\cal R}^\bot_{N}(\theta, y)$ vanish,
$$
[\mathcal C^\bot( \theta, y)]_j^j = 0, \qquad [\mathcal R^\bot_N( \theta, y)]_j^j = 0, \ \ \qquad \forall j \in S^\bot\,. 
$$ 

\noindent
$(ii)$ If in addition $Y^\bot(\frak x)$ is in ${\cal OB}^2_{ww}(m, N)$, hence $a_m(\frak x)$ of the form $A_{m}(\theta)[w]$,
then $[X_{\cal N}, Y](\frak x)$ is of the form $(0, 0, [X_{\cal N}^\bot, Y^\bot](\frak x))$ 
with $[X_{\cal N}, Y]^\bot  \in {\cal OB}^2(2 + m, N)$ and admits an expansion of the form
$$
[X_{\cal N}, Y]^\bot  (\frak x)=  \Pi_\bot  T_{- 3\partial_x A_{m}(\theta)[w]} \partial_x^{2 + m} w  +{\cal C}^\bot(\frak x)
+ {\cal R}^\bot_{N}( \frak x)  + {\cal OB}^3(m, N) \, ,
$$
where $\mathcal C^\bot( \frak x)  \in {\cal OB}^2_{ww}(1 + m, N)$ and 
${\cal R}^\bot_{N}( \frak x) \in {\cal OS}_{ww}^2(N)$. 
\end{lemma}
\begin{proof}
$(i)$ Since $Y^\bot(\frak x) =  \Pi_\bot T_{a_{m}(\frak x)} \partial_x^{ m} w$ is in ${\cal OB}^2_{w}(m, N)$, $a_m$ is independent of $w$ and
for any $s \ge s_N$,
\begin{equation}\label{tailandia 1}
 a_{m} \in C^\infty_b \big( {\cal V}^{s + \sigma_N}(\delta) \times [0, \e_0], \, H^s(\T_1) \big) \ \  \text{small of order one.}
 \end{equation}  
For notational convenience, we write  $Y^\bot(\theta, y)[w]$ instead of  $Y^\bot(\frak x)$ (similarly as we write $a_m(\theta, y)$ instead of $a_m(\frak x)$).
Then  $ [X_{\cal N}, Y](\frak x)  = d X_{\cal N}(y, w)[Y(\frak x)] - d Y(\frak x)[X_{\cal N}(y, w)]$ can be computed as
\begin{equation}\label{formula generale commutatore paradiff forma normale}
\begin{aligned}
& [X_{\cal N}, Y](\frak x)   \stackrel{\eqref{forma campo vettoriale forma normale astratto}, \eqref{forma campo vettoriale forma normale astratto differenziale}}{=} \begin{pmatrix}
0 & - d_y (\nabla_y Q(y)) & 0 \\
0 & 0 & 0 \\
0 & 0 & \ii \Omega_\bot
\end{pmatrix} \begin{pmatrix}
0 \\
0 \\
Y^\bot(\frak x)
\end{pmatrix} \\
& \quad - \begin{pmatrix}
0 & 0 & 0 \\
0 & 0 & 0 \\
\partial_\theta Y^\bot(\frak x)& \partial_y Y^\bot(\frak x) & Y^\bot(\theta, y)
\end{pmatrix} \begin{pmatrix} 
- \omega -  \e \widehat \omega - \nabla_y Q(y) \\
0 \\
\ii \Omega_\bot w
\end{pmatrix} 
 = \begin{pmatrix}
0 \\
0 \\
[X_{\cal N}, Y]^\bot (\frak x)
\end{pmatrix} 
\end{aligned}
\end{equation}
where
$$
[X_{\cal N}, Y]^\bot(\frak x) :=\Big(  [\ii \Omega_\bot , \, Y^\bot(\theta, y)]_{lin} 
+ (\omega +  \e \widehat \omega) \cdot \partial_\theta \, Y^\bot(\theta, y)  + \nabla_y Q(y) \cdot \partial_\theta \, Y^\bot(\theta, y)\Big)[w]  
$$
By \eqref{cal N Q Omega bot}, $\nabla_y Q(y)$ is small of order one and hence
$$
\begin{aligned}
& \omega \cdot \partial_\theta \, Y^\bot(\theta, y)[w] =  \Pi_\bot T_{\omega \cdot \partial_\theta a_{m}(\theta, y)} \partial_x^{m} w \in {\cal OB}^2_{w}(m, N)\,, \\
& \e \widehat \omega \cdot \partial_\theta \, Y^\bot(\theta, y)[w] = \e \Pi_\bot T_{\widehat \omega \cdot \partial_\theta a_{m}(\theta, y)} \partial_x^{ m} w \in {\cal OB}^3(m, N)\,, \\
& \nabla_y Q(y) \cdot \partial_\theta \, Y^\bot(\theta, y)[w] = \Pi_\bot T_{\nabla_y Q(y) \cdot \partial_\theta \, a_{m}(\theta, y)} \partial_x^{m} w \in {\cal OB}^3( m, N)\, .
\end{aligned}
$$
Furthermore by \eqref{tailandia 1}, Corollary \ref{corollario espansione commutatore}, and  Lemma \ref{lemma espansione Omega bot}, one sees that  
\begin{equation}\label{tailandia 2}
\begin{aligned}
& \big[ \ii \Omega_\bot , Y^\bot(\theta, y) \big]_{lin} w =  \Pi_\bot T_{- 3\partial_x a_{m}(\theta, y)} \partial_x^{2 + m} w 
+ {\cal C}^{(1)}(\theta, y)[w] + {\cal R}^\bot_N(\theta, y)[w]\,, \\
& {\cal C}^{(1)}(\theta, y)[w] \in {\cal OB}^2_{w}(1 + m, N), \qquad {\cal R}^\bot_N(\theta, y)[w] \in {\cal OS}_{w}^2(N)\,. 
\end{aligned}
\end{equation}
Altogether we have shown that
$$
\begin{aligned}
&[X_{\cal N}, Y]^\bot(\frak x) =  \Pi_\bot  T_{-3\partial_x a_{m}(\theta, y)} \partial_x^{2 + m} w + {\cal C}^\bot(\theta, y)[w] + {\cal R}^\bot_N(\theta, y)[w] + {\cal OB}^3( m, N) \,,  \\
& \mathcal C^\bot( \theta, y) [w] := {\cal C}^{(1)}(\theta, y)[w]  + \omega \cdot \partial_\theta Y^\bot(\theta, y)[w] \in {\cal OB}^2_{w}(1 + m, N)\, .
\end{aligned}
$$
For any $j \in S^\bot$,  the diagonal matrix element $[\omega \cdot \partial_\theta \, Y^\bot(\theta, y)]_j^j$ vanishes,
$$
[\omega \cdot \partial_\theta \, Y^\bot(\theta, y)]_j^j = \omega \cdot \partial_\theta \langle a_{m}(\theta, y) \big\rangle_x   (\ii 2 \pi j)^{m} = 0,
$$
since by assumption $\langle a_{m}(\theta, y) \rangle_x = 0$, and so does the diagonal matrix element 
$\big[ [\ii \Omega_\bot\,,\, Y^\bot(\theta, y)]_{lin} \big]_j^j$, implying together with \eqref{tailandia 2}
$$
[ \mathcal C^\bot( \theta, y)]_j^j = 0, \qquad [{\cal R}^\bot_{N}( \theta, y)]_j^j = 0\, , \qquad \ \ \forall j \in S^\bot \, .
$$

\noindent
$(ii)$ 
Since $Y^\bot(\frak x) =  \Pi_\bot T_{a_{m}(\frak x)} \partial_x^{ m} w$ is in ${\cal OB}^2_{ww}(m, N)$, 
it follows from Definition \ref{paradiff omogenei di ordine 2} that
 $a_m(\frak x)$ is of the form $a_m(\frak x) = A_{m}(\theta)[w]$ and that for any $s \geq s_N$,
 \begin{equation}\label{tailandia 1 1}
 A_{m} \in C^\infty_b\big(\T^{S_+}, \, {\cal B}(H^{s + \sigma_N}_\bot(\T_1), H^s(\T_1) \big) \, .
 \end{equation}  
For notational convenience, we write $a_m(\theta, w)$ instead of $a_m(\frak x)$.
 Arguing as in \eqref{formula generale commutatore paradiff forma normale}, one sees that 
$ [X_{\cal N}, Y](\frak x)  = d X_{\cal N}(y, w)[Y(\frak x)] - d Y(\frak x)[X_{\cal N}(y, w)]$ can be computed as
\begin{equation}\label{formula generale commutatore paradiff forma normale 1}
\begin{aligned}
& [X_{\cal N}, Y](\frak x)  = \begin{pmatrix}
0 & - d_y( \nabla_y Q(y)) & 0 \\
0 & 0 & 0 \\
0 & 0 & \ii \Omega_\bot
\end{pmatrix} \begin{pmatrix}
0 \\
0 \\
Y^\bot(\frak x)
\end{pmatrix} \\
& \quad - \begin{pmatrix}
0 & 0 & 0 \\
0 & 0 & 0 \\
d_\theta Y^\bot(\frak x)& 0& d_\bot Y^\bot(\frak x)
\end{pmatrix} \begin{pmatrix} 
- \omega -  \e \widehat \omega - \nabla_y Q(y) \\
0 \\
\ii \Omega_\bot w
\end{pmatrix}
 = \begin{pmatrix}
0 \\
0 \\
[X_{\cal N}, Y]^\bot (\frak x)
\end{pmatrix}   
\end{aligned}
\end{equation}
where
$$
\begin{aligned}
& [X_{\cal N}, Y]^\bot (\frak x) =  \ii \Omega_\bot [Y^\bot(\frak x)] - d_\bot Y^\bot(\frak x)[\ii \Omega_\bot w] + ( \omega +  \e \widehat \omega) \cdot \partial_\theta \, Y^\bot(\frak x)  
 + \nabla_y Q(y) \cdot \partial_\theta \, Y^\bot(\frak x) \, .
\end{aligned}
$$
Since by \eqref{cal N Q Omega bot}, $\nabla_y Q(y)$ is small of order one, one infers that
\begin{equation}\label{tailandia 100}
\begin{aligned}
& \omega \cdot \partial_\theta \, Y^\bot(\frak x) =  \Pi_\bot T_{\omega \cdot \partial_\theta \, A_{m}(\theta)[w]} \partial_x^{m} w \in {\cal OB}^2_{ww}(m, N)\,, \\
& \e \widehat \omega \cdot \partial_\theta \, Y^\bot(\frak x) = \e \Pi_\bot T_{\widehat \omega \cdot \partial_\theta \, A_{m}(\theta)[w]} \partial_x^{m} w \in {\cal OB}^3(m, N)\,, \\
& \nabla_y Q(y) \cdot \partial_\theta \, Y^\bot(\frak x) = \Pi_\bot T_{\nabla_y Q(y) \cdot \partial_\theta \, A_{m}(\theta)[w]} \partial_x^{m} w \in {\cal OB}^3(m, N)\,.
\end{aligned}
\end{equation}
Furthermore, $ \ii \Omega_\bot [Y^\bot(\frak x)] - d_\bot Y(\frak x)[\ii \Omega_\bot w]$ can be computed as 
\begin{equation}\label{tailandia 101}
\begin{aligned}
& \ii \Omega_\bot \Pi_\bot T_{A_{m}(\theta)[w]} \partial_x^{m} w 
- \Pi_\bot T_{A_{m}(\theta)[w]} \partial_x^{m} \ii \Omega_\bot w - \Pi_\bot T_{A_{m}(\theta)[\ii \Omega_\bot w]} \partial_x^{m} w  \\
& = \Pi_\bot \big[\ii \Omega_\bot\,,\, T_{A_{m}(\theta)[w]} \partial_x^{m} \big]_{lin} w - \Pi_\bot T_{A_{m}(\theta)[\ii \Omega_\bot w]} \partial_x^{ m} w\,. 
\end{aligned}
\end{equation}
By \eqref{tailandia 1 1}, Corollary \ref{corollario espansione commutatore}, and Lemma \ref{lemma espansione Omega bot} one has 
\begin{equation}\label{tailandia 110}
\begin{aligned}
& \Pi_\bot T_{A_{m}(\theta)[\ii \Omega_\bot w]} \partial_x^{ m} w \in {\cal OB}_{ww}^2( m, N)\, ,\\
& \Pi_\bot \big[\ii \Omega_\bot\,,\, T_{A_{m}(\theta)[w]} \partial_x^{m} \big]_{lin} w = \Pi_\bot T_{- 3\partial_x A_{m}(\theta)[w]} \partial_x^{2 + m}w 
+ {\cal C}^{(1)}(\frak x) + {\cal R}^\bot_N(\frak x) + {\cal OB}^3( m, N) \,,  \\
& \mathcal C^{(1)}( \frak x) \in {\cal OB}_{ww}^2( 1 + m, N)\, , \qquad {\cal R}^\bot_N(\frak x) \in {\cal OS}_{ww}^2(N)\, .
\end{aligned}
\end{equation}
Altogether, the identities \eqref{tailandia 100}-\eqref{tailandia 110} yield 
$$
\begin{aligned}
&[X_{\cal N}, Y]^\bot(\frak x) =  \Pi_\bot  T_{-3\partial_x A_{m}(\theta)[w]} \partial_x^{2 + m} w 
+  \mathcal C^\bot( \frak x) + {\cal R}^\bot_N(\frak x) + {\cal OB}^3( m, N) \,,  \\
& \mathcal C^\bot( \frak x)  := {\cal C}^{(1)}(\frak x) 
+  \Pi_\bot T_{\omega \cdot \partial_\theta \, A_{m}(\theta)[w]} \partial_x^{m} w  -  \Pi_\bot T_{A_{m}(\theta)[\ii \Omega_\bot w]} \partial_x^{ m} w \in {\cal OB}^2_{ww}(1 + m, N)
\end{aligned}
$$
and hence item $(ii)$ is proved.
\end{proof}
\begin{lemma}\label{corollario flusso con Omega bot}
Let $X_{\mathcal N}$ be the vector field given by \eqref{forma campo vettoriale forma normale astratto} and let 
$Y(\frak x) = (0, 0, \, Y^\bot(\frak x))$ where $Y^\bot(\frak x)= \big(0, 0, Y^\bot_{0}(\frak x) + Y^\bot_{1}(\frak x) \big)$ and
 \begin{equation}\label{ucraina 0}
  Y^\bot_{0}(\frak x) \equiv Y^\bot_{0}(\theta, y)[w] = \Pi_\bot T_{a_{m}(\theta, y)} \partial_x^{m} w \in {\cal OB}^2_{w}(m, N), \ \
  Y^\bot_{1}(\frak x)= \Pi_\bot T_{A_{ m}(\theta)[w]} \partial_x^{m} w \in {\cal OB}_{ww}^2(m, N),
 \end{equation}
with  $N \in \N$ and $m \le 0$. If in addition $\langle a_{m} (\theta, y) \rangle_x = 0$, 
then the pullback $X_{\mathcal N, \Phi} \equiv \Phi_Y^* X_{\cal N}$ of the vector field $X_{\cal N}$ 
by  be the time one flow map $\Phi_Y(1, \cdot)$
corresponding to $Y$ has an expansion of the form 
$$
X_{\mathcal N, \Phi}(\frak x) = \big( - \omega -  \e \widehat \omega -  \nabla_y Q(y), \ 0, \ X_{\mathcal N, \Phi}^\bot(\frak x) \big)
$$
where 
$$
\begin{aligned}
 X_{\mathcal N, \Phi}^\bot(\frak x) & = \ii \Omega_\bot w + \Pi_\bot  T_{- 3\partial_x (a_{ m}(\theta, y) + A_{m}(\theta)[w] )} \partial_x^{2 + m} w 
+ {\cal C}_0^\bot(\theta, y)[w] + {\cal C}_1^\bot(\frak x) \\
&  \qquad + {\cal R}^\bot_{N, 0}(\theta, y )[w] + {\cal R}^\bot_{N, 1}(\frak x)
+ {\cal OB}^3(2 + m, N) + {\cal OS}^3(N)
\end{aligned}
$$
and  ${\cal C}_0^\bot(\theta, y)$,  $ {\cal R}^\bot_{N, 0}(\theta, y)$,  
and ${\cal C}_1^\bot(\frak x) $, ${\cal R}^\bot_{N, 1}(\frak x)$ 
are given by Lemma \ref{corollario coniugazione Omega bot}. Hence these terms satisfy
$$
\begin{aligned}
& {\cal C}^\bot_0(\theta, y)[w]  \in {\cal OB}^2_{w}(1 + m, N), \qquad {\cal R}^\bot_{N, 0}(\theta, y)[w] \in {\cal OS}_{w}^2(N)\,, \\
& {\cal C}_1^\bot(\frak x) \in {\cal OB}^2_{ww}(1 + m, N), \  \qquad  \quad \ {\cal R}^\bot_{N, 1}(\frak x) \in {\cal OS}^2_{ww}(N)\, ,
\end{aligned}
$$
and the diagonal matrix elements of ${\cal C}_0^\bot(\theta, y)$ and ${\cal R}^\bot_{N, 0}(\theta, y)$ vanish,
$$
[{\cal C}_0^\bot(\theta, y)]_j^j = 0, \qquad [{\cal R}^\bot_{N, 0}(\theta, y)]_j^j = 0, \quad \qquad \forall j \in S^\bot\,. 
$$
\end{lemma}
\begin{proof}
By \eqref{espansione commutatori notazioni}, $X_{\mathcal N, \Phi}$ can be expanded as
$$
X_{\mathcal N, \Phi} = \Phi_Y^* X_{\cal N} = X_{\cal N} + [X_{\cal N}, Y] + Z, \qquad  
Z(\frak x) :=  \int_0^1 (1 - t) (d \Phi_Y(t, \frak x))^{- 1}[[X_{\cal N}, Y], Y](\Phi_Y(t, \frak x))\, d t .
$$ 
By Lemma \ref{corollario coniugazione Omega bot}, one has $ [X_{\cal N}, Y] = \big( 0, \, 0, \, [X_{\cal N}, Y]^\bot \big)$
with $  [X_{\cal N}, Y]^\bot \in {\cal OB}^2(2 + m, N)$ given by
\begin{equation}\label{libellula 0}
   \Pi_\bot  T_{-3\partial_x ( a_{m}(\theta, y) + A_{m}(\theta)[w] )} \partial_x^{2 + m} w 
+ {\cal C}_0^\bot(\theta, y)[w] + {\cal R}^\bot_{N, 0}(\theta, y)[w] 
 +  {\cal C}_1^\bot(\frak x)+ {\cal R}^\bot_{N, 1}(\frak x) + {\cal OB}^3(m, N) \, ,
\end{equation}
where ${\cal C}_0^\bot(\theta, y)$,  $ {\cal R}^\bot_{N, 0}(\theta, y)$,  
and ${\cal C}_1^\bot(\frak x) $, ${\cal R}^\bot_{N, 1}(\frak x)$ 
are given as in Lemma \ref{corollario coniugazione Omega bot}. In particular, the diagonal matrix elements of 
${\cal C}_0^\bot(\theta, y)$ 
and ${\cal R}^\bot_{N, 0}(\theta, y)$ vanish. 
Furthermore, by Lemmata \ref{comm smoothing paradiff}, \ref{commutatore due paradiff nonlin}, one infers that
\begin{equation}\label{libellula 1}
[[X_{\cal N}, Y], Y] (\frak x)= \big( 0, \, 0, \, {\cal C}_2^\bot(\frak x) + {\cal R}^\bot_{N, 2}(\frak x) \big) , 
\qquad {\cal C}_2^\bot \in {\cal OB}^{2}(2 + m, N), \quad {\cal R}^\bot_{N, 2} \in {\cal OS}^{3}(N )  ,
\end{equation}
and hence concludes by Lemma \ref{lemma coniugazioni} that 
\begin{equation}\label{libellula 2}
Z(\frak x) = \big( 0, \, 0, \, {\cal C}^\bot_3(\frak x) + {\cal R}^\bot_{N, 3}(\frak x) \big), \qquad {\cal C}^\bot_3 \in {\cal OB}^{3}(2 + m, N), \quad {\cal R}^\bot_{N, 3} \in {\cal OS}^{3}(N)\,. 
\end{equation}
The claimed statement then follows by \eqref{libellula 0}-\eqref{libellula 2}. 
\end{proof}

\subsection{Flows of Fourier multiplier vector fields and smoothing vector fields}
In this subsection we discuss additional properties of Fourier multiplier vector fields and smooth vector fields and their flows, 
needed in  Subsection \ref{normalizzazione Fourier multipliers}.

We begin by considering the flows corresponding to Fourier multiplier vector fields.
Let ${\cal M}$ be a vector field of the form $(0, 0, {\cal M}^\bot)$ with ${\cal M}^\bot \in  {\cal OF}^p(0, N)$ and $N, p \in \N$ 
(cf. Definition \ref{classe fourier multipliers}).
Then $\mathcal M^\bot(\frak x)$ has an expansion of the form
$\mathcal M^\bot(\frak x) = \sum_{k = 0}^{N} \lambda_{ - k}(\frak x) \partial_x^{ - k} w$ with the property that
there exist $\sigma_N \geq 0$, $0 < \delta \equiv \delta(N)  <1$, and $0 < \e_0 \equiv \e_0(N) < 1$, so that for any $0 \le k \le N$,
$$
\lambda_{- k} :  {\cal V}^{\sigma_N}(\delta) \times [0, \e_0]  \to \R, \, (\frak x, \e) \mapsto \lambda_{m - k}(\frak x) \equiv \lambda_{- k}(\frak x, \e)
$$
is $C^\infty$-smooth and bounded. 
We denote by $\Phi_{\cal M} (\tau, \cdot)$ the flow corresponding to the vector field ${\cal M}$. 
By the standard ODE theorem in Banach spaces, there exist $s_N \ge 0$ so that for any $s \ge s_N$, there exist $0 < \delta \equiv \delta(s, N) < 1$, and 
$0 < \e_0 \equiv\e_ 0(s, N) \ll  \delta$, so that 
$$
\Phi_\mathcal M(\tau, \cdot) \in C^\infty_b \big({\cal V}^s(\delta) \times [0, \e_0], \, {\cal V}^s(2 \delta) \big)\, ,  \qquad   \forall \, -1  \le \tau \le 1 \, .
$$
The following lemma can be proved arguing as in the proof of Lemma \ref{expansion flow paradiff} (actually, the proof is simpler). 
\begin{lemma}\label{expansion flow multiplier}
For any $\tau \in [- 1, 1]$, the flow map $\Phi_{\cal M}(\tau, \cdot)$ admits an expansion of the form
$$
\Phi_{\cal M}(\tau, \frak x) = \frak x + (0, 0, \Upsilon^\bot(\tau, \frak x) + {\cal R}^\bot_N(\tau, \frak x))  
$$
where $\Upsilon^\bot(\tau, \cdot) \in {\cal OF}^p(0, N)$ and ${\cal R}^\bot_N \in {\cal OS}^{2 p - 1}(N)$. 
\end{lemma}
The following lemma can be proved arguing as in the proof of Lemma \ref{Corollario d Phi F}. 
\begin{lemma}\label{Corollario d Phi cal M}
Let $\Phi_\mathcal M(\tau, \frak x)$ denote the flow map considered in Lemma \ref{expansion flow multiplier},
corresponding to the vector field $\mathcal M = (0, 0, \, \mathcal M^\bot)$ with ${\cal M}^\bot \in  {\cal OF}^p(0, N)$ and $N, p \in \N$.
Then 
$d \Phi_\mathcal M( \tau, \frak x)^{- 1}[\widehat{\frak x}]$  admits an expansion of the form
\begin{equation}\label{espansion d Phi tau inverse 2}
d \Phi_\mathcal M( \tau, \frak x)^{- 1}[\widehat{\frak x}] = \widehat{\frak x}  
+ \big(0, 0, \, {\Upsilon}^\bot(\tau, \frak x)[\widehat{\frak x} ] 
+ {\cal R}^\bot_{N}(\tau, \frak x)[\widehat{\frak x} ]  \big) \, ,
\end{equation}
$$
{\Upsilon}^\bot(\tau, \frak x) [\widehat{\frak x} ]  := 
 \sum_{k = 0}^{N } \lambda_{- k}(\tau, \frak x) \partial_x^{- k} \widehat w 
+ \sum_{k = 0}^{N } \eta_{- k}(\tau, \frak x)[\widehat{\frak x}] \partial_x^{- k } w \, ,
$$
with the following properties: there exist $s_N$, $\sigma_N \geq N$ so that for any $s \geq s_N$, 
there exist $0 < \delta \equiv \delta(s, N) < 1$ and $0 < \e_0 \equiv \e_0(s, N) < 1$ so that the following holds: for any $0 \le k \le N$ and $-1 \le \tau \le 1$,
$$
\begin{aligned}
& \lambda_{- k} \in C^\infty_b({\cal V}^{ \sigma_N}(\delta) \times [0, \e_0], \, \R),
\qquad \eta_{- k} \in C^\infty_b\big( {\cal V}^{ \sigma_N}(\delta) \times [0, \e_0],  \, {\cal B}(E_{ \sigma_N}, \R)\big)\,, \\
& {\cal R}^\bot_N \in C^\infty_b \big(  {\cal V}^s(\delta) \times [0, \e_0], \, {\cal B}(H^s_\bot(\T_1), H^{s + N +1}_\bot(\T_1)) \big),
\end{aligned}
$$
and  $\lambda_{- k}(\tau, \cdot)$, $\eta_{- k}(\tau, \cdot)$, and ${\cal R}^\bot_N(\tau, \cdot)$ are small of order $p - 1$. 
\end{lemma}
The following lemma can be proved arguing as in the proof of Lemma \ref{lemma coniugazioni}. 
\begin{lemma}\label{lemma coniugazioni multiplier}
Let $\Phi_\mathcal M(1, \frak x)$ denote the time one flow map considered in Lemma \ref{expansion flow multiplier},
corresponding to the vector field $\mathcal M = (0, 0, \, \mathcal M^\bot)$, with ${\cal M}^\bot \in  {\cal OF}^p(0, N)$ and $N$, $p \in \N$.
Then the following holds:\\
$(i)$ For any  $X := (0,0,  X^\bot)$ with $X^\bot \in {\cal OB}^q(n, N)$  and $q \geq 1$, $n \geq 0$, the pullback $\Phi_{\cal M}^* X$ 
of $X$ by $\Phi_\mathcal M(1, \cdot)$ admits an expansion of the form  
$$
\Phi_{\cal M}^* X(\frak x) = \big(0, 0, X^\bot(\frak x) + \Upsilon^\bot(\frak x) + \mathcal R^\bot_N(\frak x) \big) \, , \qquad
 \Upsilon^\bot \in {\cal OB}^{p + q - 1}(n, N), \quad \mathcal R^\bot_N \in {\cal OS}^{p + q - 1}(N).
 $$ 
\noindent
$(ii)$ For any $\mathcal M_1 = \big( 0, 0, {\cal M}_1^\bot \big)$ with ${\cal M}_1^\bot \in {\cal OF}^q(n, N)$ and $q \geq 1$, $n \geq 0$, 
the pullback $\Phi_{\cal M}^* \mathcal M_1$ 
of $\mathcal M_1$ by $\Phi_\mathcal M(1, \cdot)$
admits an expansion of the form 
$$
\Phi_{{\cal M}}^* \mathcal M_1(\frak x) = \big(0, 0, \mathcal M^\bot_1(\frak x) + \Upsilon^\bot(\frak x) + {\cal R}^\bot_N(\frak x) \big), \qquad
\Upsilon^\bot \in {\cal OF}^{p + q - 1}(n, N), \quad
{\cal R}^\bot_N \in {\cal OS}^{p + q - 1}(N).
$$ 
\noindent
$(iii)$ For any $X \in {\cal OS}^q(N)$, the pullback  $\Phi_{\cal M}^* X$ of $X$ by $\Phi_\mathcal M(1, \cdot)$ admits an expansion of the form 
$$
\Phi_{\cal M}^* X (\frak x)= X(\frak x) +  \big(0,0, \,  \Upsilon^\bot(\frak x) \big) + \mathcal R_N(\frak x)
$$ 
where $ \Upsilon^\bot \in {\cal OF}^{p + q - 1}(0, N)$ and $\mathcal R_N \in {\cal OS}^{p + q - 1}(N)$.

\end{lemma}
Next we consider ${\cal M} := (0, 0, {\cal M}^\bot)$ with $ {\cal M}^\bot \in {\cal OF}^2_{ww}(0, N)$ and $N \in \N$ (cf.  Definition \ref{paradiff omogenei di ordine 2}-$(ii2)$), i.e.,
${\cal M}^\bot(\frak x) = {\cal M}^\bot(\theta, w)[w]$ with $ {\cal M}^\bot(\theta, w) = \sum_{k=0}^N  \Lambda_{- k}(\theta)[w] \partial_x^{-k}$
where, for some integer $\sigma_N \ge 0$ and some $0 < \e_0 \equiv \e_0(N) < 1$,
\begin{equation}\label{mclauren}
 \Lambda_{ -k} : \T^{S_+} \times [0, \e_0] \to {\cal B}(H^{\sigma_N}_\bot(\T_1), \R), \, \theta \mapsto \Lambda_{-k} (\theta) \equiv \Lambda_{-k} (\theta, \e), \qquad \, 0 \le k \le N,
\end{equation} 
are $C^\infty-$smooth.
To obtain an expansion of the pullback $\Phi_{\cal M}^* X_\mathcal N$ of the vector field $X_\mathcal N$, 
defined in \eqref{forma campo vettoriale forma normale astratto},  by $\Phi_\mathcal M(1, \cdot)$, we first need to compute
the one of the commutator $[X_{\cal N}, \cal M]$.
\begin{lemma}\label{commutatore forma normale fourier multiplier}
The commutator $[X_{\cal N}, \cal M](\frak x)$ admits an expansion of the form
$$
[X_{\cal N}, {\cal M}](\frak x) = 
\big( 0,  0, \  \omega \cdot \partial_\theta ( {\cal M}^\bot(\theta, w)[w]) - {\cal M}^\bot(\theta, \ii \Omega_\bot  w)[w] + {\cal OF}^3(0, N) \big) \, .
$$
\end{lemma}
\begin{proof}
By \eqref{mclauren} the differential of $\mathcal M$ can be computed as 
$$
d {\cal M}(\frak x)[\widehat{\frak x}] = \big( 0, \, 0, \ 
{\cal M}^\bot(\theta, w)[\widehat w] + {\cal M}^\bot(\theta, \widehat w)[w] + d_\theta \big( {\cal M}(\theta, w)[w] \big)[\widehat \theta] \big)\,. 
$$
By \eqref{forma campo vettoriale forma normale astratto}, \eqref{forma campo vettoriale forma normale astratto differenziale}, 
the commutator $ [X_{\cal N}, \mathcal M](\frak x)  = d X_{\cal N}(y, w)[\mathcal M(\frak x)] - d \mathcal M(\frak x)[X_{\cal N}(y, w)]$ is given by
$$
\begin{aligned}
& [X_{\cal N}, {\cal M}](\frak x) = \big( 0, 0, \,  \ii \Omega_\bot {\cal M}^\bot(\theta, w)[w]  \big)\\
&  - 
\big( 0,  0, \, 
{\cal M}^\bot(\theta, w)[\ii \Omega_\bot  w] + {\cal M}^\bot(\theta, \ii \Omega_\bot  w)[w] 
- d_\theta \big( {\cal M}^\bot(\theta, w)[w] \big)[\omega + \e \widehat \omega + \nabla_y Q(y)] \big) \\
& = \big( 0, 0, \, 
[\ii \Omega_\bot, {\cal M}^\bot(\theta, w)]_{lin} w  - {\cal M}^\bot(\theta, \ii \Omega_\bot  w)[w] 
+ (\omega + \e \widehat \omega + \nabla_y Q(y)) \cdot \partial_\theta \big( {\cal M}^\bot(\theta, w)[w] \big) \big) \, .
\end{aligned}
$$
Since ${\cal M}^\bot(\theta, w)$ and $\ii \Omega_\bot$ are both Fourier multipliers, the linear commutator $[\ii \Omega_\bot, {\cal M}^\bot(\theta, w)]_{lin}$ vanishes.  
The lemma then follows in view of the fact that 
$$
(\e \widehat \omega + \nabla_y Q(y)) \cdot \partial_\theta \big( {\cal M}^\bot(\theta, w)[w] \big) =
 \sum_{k=0}^N (\e \widehat \omega + \nabla_y Q(y)) \cdot \partial_\theta)( \Lambda_{- k}(\theta)[w]) \partial_x^{-k} [w]
\in {\cal OF}^3(0, N).
$$ 
\end{proof}
\begin{lemma}\label{psuh forward multiplier X cal N}
The pullback $\Phi_{\cal M}^* X_\mathcal N$ of the vector field $X_\mathcal N$ by $\Phi_\mathcal M(1, \cdot)$ with ${\cal M}$ given by  \eqref{mclauren}
admits an expansion of the form  
$$
\Phi_{\cal M}^* X_{\cal N} (\frak x)=  \begin{pmatrix}
- \omega - \e \widehat \omega - \nabla_y Q(y) \\
0 \\
\ii \Omega_\bot w + \omega \cdot \partial_\theta ( {\cal M}^\bot(\theta, w)[w]) - {\cal M}^\bot(\theta, \ii \Omega_\bot w)[w] 
+ {\cal OF}^3(0, N)  + {\cal OS}^3(N) 
\end{pmatrix} .
$$
\end{lemma}
\begin{proof}
We argue as in the proof of Lemma \ref{corollario flusso con Omega bot}. By \eqref{espansione commutatori notazioni}, $ \Phi_{\cal M}^* X_{\cal N}$ can be expanded as
$$
\begin{aligned}
& \Phi_{\cal M}^* X_{\cal N}   = X_{\cal N} + [X_{\cal N}, {\cal M}] + Z, \qquad
   Z(\frak x) :=  \int_0^1 (1 - \tau) [d \Phi_{\cal M}(\tau, \frak x)]^{- 1}[[X_{\cal N}, {\cal M}], {\cal M}](\Phi_{\cal M}(\tau, \frak x))\, d \tau\,.
\end{aligned}
$$ 
The claimed statement then follows by applying 
Lemmata \ref{lemma commutatori Fourier multiplier}, \ref{expansion flow multiplier}, \ref{Corollario d Phi cal M}, \ref{commutatore forma normale fourier multiplier}.
\end{proof}

Finally, we consider smoothing vector fields.
Given a smoothing vector field ${\cal Q} \in {\cal OS}^p(N)$ with $N$, $p \in \N$ (cf. Definition \ref{def smoothing vector fields}),
we denote by $\Phi_{\cal Q} (\tau, \cdot)$ the flow corresponding to the vector field ${\cal Q}$. 
By the standard ODE theorem in Banach spaces, there exists $s_N \ge 0$ so that for $s \ge s_N$, there exist $0 < \delta \equiv \delta(s, N) < 1$ and 
$0 < \e_0 \equiv \e_0(s, N) \ll  \delta$, so that 
\begin{equation}\label{flusso smoothing vf 1}
\Phi_\mathcal Q(\tau, \cdot) \in C^\infty_b \big({\cal V}^s(\delta) \times [0, \e_0], \, {\cal V}^s(2 \delta) \big)\, ,  
\quad \Phi_\mathcal Q(\tau, \cdot) -  \text{Id} \quad \text{small of order } p,
\qquad   \forall \, -1  \le \tau \le 1 \, .
\end{equation}
\begin{lemma}\label{flusso smoothing vector fields}
Let ${\cal Q} \in {\cal OS}^p(N)$ with $N$, $p \in \N$. For any $-1 \le \tau \le 1$, the following holds. 

\noindent
$(i)$ The flow map $\Phi_\mathcal Q(\tau, \cdot)$ admits an expansion of the form
$$
\Phi_\mathcal Q(\tau, \frak x) = \frak x + {\cal R}_N(\tau, \frak x), \qquad {\cal R}_N(\tau, \cdot) \in {\cal OS}^p(N).
$$
\noindent
$(ii)$ The map $d \Phi_\mathcal S(\tau, \frak x)^{- 1}$ admits an expansion of the form 
$$
d \Phi_\mathcal Q(\tau, \frak x)^{- 1}[\widehat{\frak x}] = \widehat{\frak x} + {\cal R}_N(\tau, \frak x)[\widehat{\frak x}]
$$
where there exists $s_N \ge 0$  so that for any $s \geq s_N$ there are $0 < \delta \equiv \delta(s, N) < 1$ and $0 < \e_0 \equiv \e_0(s, N) < 1$ such that 
$$
{\cal R}_N(\tau, \cdot) \in C^\infty_b \big( {\cal V}^s(\delta)\times [0, \e_0], \, {\cal B}(E_s, E_{s + N +1}) \big)\,, \qquad \forall \, -1 \le \tau \le 1. 
$$
\end{lemma}
\begin{proof}
To prove item $(i)$ one uses the Volterra integral equation (cf. \eqref{integral equation}) and \eqref{flusso smoothing vf 1}
(cf. proof of Lemma \ref{expansion flow paradiff}).
To prove item $(ii)$, one argues as in the proof of Lemma \ref{Corollario d Phi F}, using the identity 
$d \Phi_\mathcal Q(\tau, \frak x)^{- 1} = d \Phi_\mathcal Q(- \tau, \Phi_\mathcal Q(\tau, \frak x))$, $-1 \le \tau \le 1$
(cf. Remark \ref{inverse of flow}). 
\end{proof}
\begin{lemma}\label{push forward smoothing 1}
For any ${\cal Q} \in {\cal OS}^p(N)$ with $N$, $p \in \N$, the following holds:

\noindent
$(i)$ For any $X := (0, 0, X^\bot)$ with $X^\bot \in {\cal OB}^q(m, N)$ and $m \in \Z$, $q \in \N$, the pullback $\Phi_{\cal Q}^* X$ of $X$ by $\Phi_{\cal Q}(1, \cdot)$
admits an expansion of the form
$$
\Phi_{\cal Q}^* X(\frak x) = \big(0,0, \, X^\bot(\frak x) +  \Upsilon^\bot(\frak x) + {\cal R}^\bot_N \big), \qquad
\Upsilon^\bot \in {\cal OB}^{p + q - 1}(m, N), \quad {\cal R}^\bot_N \in {\cal OS}^{p + q - 1}(N). 
$$
\noindent
$(ii)$ For any $\mathcal M := (0, 0, {\cal M}^\bot)$ with ${\cal M}^\bot \in {\cal OF}^q(m, N)$ and $m \in \Z$, $q \in \N$, 
the pullback $\Phi_{\cal Q}^* \mathcal M$ of $\mathcal M$ by $\Phi_{\cal Q}(1, \cdot)$
admits an expansion of the form
$$
\Phi_{\cal Q}^* \mathcal M(\frak x) = \big(0, 0, \mathcal M^\bot(\frak x) + \Upsilon^\bot(\frak x) + {\cal R}^\bot_N(\frak x) \big)\, , \qquad
\Upsilon^\bot \in {\cal OF}^{p + q - 1}(m, N), \quad {\cal R}^\bot_N \in {\cal OS}^{p + q - 1}(N) \, .
$$
\noindent
$(iii)$ For any ${\cal Q}_1 \in {\cal OS}^q(N)$ with $q \in \N$, 
the pullback $\Phi_{\cal Q}^* {\cal Q}_1$  of $\mathcal Q_1$ by  $\Phi_{\cal Q}(1, \cdot)$ admits an expansion of the form
 $\Phi_{\cal Q}^* {\cal Q}_1 = {\cal Q}_1 + {\cal OS}^{p + q - 1}(N)$. 
\end{lemma}
\begin{proof}
$(i)$ By \eqref{espansione commutatori notazioni}, $\Phi_{\cal Q}^* X(\frak x)$ can be expanded as
$$
\Phi_{\cal Q}^* X(\frak x) = X(\frak x) + Z, \qquad Z :=  \int_0^1 d \Phi_{\cal Q}(t, \frak x)^{- 1}[X, {\cal Q}](\Phi_\mathcal Q(t, \frak x))\, d t\,.
$$
By applying Lemma \ref{comm smoothing paradiff}, one gets that 
$$
[X, {\cal Q}] = \big(0, 0, \, \Upsilon^\bot + {\cal R}^\bot_{[X, {\cal Q}] } \big), 
\qquad \Upsilon^\bot \in {\cal OB}^{p + q - 1}(m, N), \qquad {\cal R}^\bot_{[X, {\cal Q}] } \in {\cal OS}^{p + q - 1}(N + m)\,. 
$$
Item $(i)$ then follows by the definition of $Z$, the property \eqref{flusso smoothing vf 1}, and Lemma \ref{flusso smoothing vector fields}.
Items $(ii)$ and $(iii)$ can be proved similarly,  using in addition Lemma \ref{Commutator of smoothing vector fields} and Lemma \ref{lemma commutatori Fourier multiplier}.  
\end{proof}
We now consider a smoothing vector field ${\cal Q} \in {\cal OS}(N)$, $N \in \N$, of the form $ \mathcal Q := \mathcal Q_0 + \mathcal Q_1$ where
\begin{equation}\label{smoothing vector field per forma normale w}
\begin{aligned}
&   \mathcal Q_0 : = (0, 0, {\cal Q}^\bot_0), \qquad \ \   {\cal Q}^\bot_0(\frak x) \equiv {\cal Q}^\bot_0(\theta, y)[w] \in {\cal OS}_{w}^2(N), \\
&  \mathcal Q_1 :=  ({\cal F}_1, 0,  {\cal Q}^\bot_1)\,, \qquad  {\cal Q}^\bot_1(\frak x) \equiv {\cal Q}^\bot_1(\theta)[w, w] \in {\cal OS}_{ww}^2(N),
\end{aligned}
\end{equation}
(cf. Definition \ref{paradiff omogenei di ordine 2}$(iii)$ for the definitions of ${\cal OS}_{w}^2(N)$ and ${\cal OS}_{ww}^2(N))$ and where 
for some $\sigma_N \ge 0$ and $0 < \e_0 \equiv \e_0(N) < 1$, ${\cal F}_1$ has the form 
\begin{equation}\label{addition smoothing vector field}
\begin{aligned}
& {\cal F}_1(\theta, w) := F_1(\theta)[w, w]\,, \qquad  F_1 \in C^\infty\big( \T^{S_+} \times [0, \e_0], \, {\cal B}_2(H^{\sigma_N}_\bot(\T_1), \R^{S_+}) \big) ,
\end{aligned}
\end{equation}
(cf.  \eqref{forme multilineari} for the definition ${\cal B}_2(H^{\sigma_N}_\bot(\T_1), \R^{S_+})$).
In the next lemma we compute an expansion of $\Phi_{\mathcal Q}^* X_{\cal N}$ 
where $X_{\cal N}$ is the normal form vector field defined in \eqref{forma campo vettoriale forma normale astratto}. 
\begin{lemma}\label{push forward smoothing X cal N}
For $\mathcal Q = \mathcal Q_0 + \mathcal Q_1$ as in \eqref{smoothing vector field per forma normale w}, the following holds.

\noindent
$(i)$ The commutator $[X_{\cal N}, \mathcal Q_0] \in {\cal OS}^2(N - 3)$ has the form $\Upsilon^{(1)} + {\cal OS}^3(N)$ where
$$\Upsilon^{(1)}(\frak x) = 
\Big(0, 0,  \big( [\ii \Omega_\bot , \, {\cal Q}^\bot_0(\theta, y)]_{lin} + \omega \cdot \partial_\theta {\cal Q}^\bot_0(\theta, y) \big)[w] \Big)\,. 
$$
\noindent
$(ii)$ The commutator $[X_{\cal N}, \mathcal Q_1] \in {\cal OS}^2(N - 3)$ has the form  $\Upsilon^{(2)} + {\cal OS}^3(N)$ where
$$
\begin{aligned}
& \Upsilon^{(2)}(\frak x) = \begin{pmatrix}
 \omega \cdot \partial_\theta F_1(\theta)[w, w] - F_1(\theta)[\ii \Omega_\bot w, w] - F_1(\theta)[w, \ii \Omega_\bot w] \\
0 \\
 \ii \Omega_\bot {\cal Q}^\bot_1(\theta)[w, w]- {\cal Q}^\bot_1(\theta)[\ii \Omega_\bot w, w] - {\cal Q}^\bot_1(\theta)[w, \ii \Omega_\bot w]  
 +  \omega \cdot \partial_\theta {\cal Q}^\bot_1(\theta)[w, w]
\end{pmatrix}
\end{aligned}
$$

\noindent
$(iii)$ The pullback $\Phi_{\mathcal Q}^* X_{\cal N}$ is of the form $X_{\cal N} +  \Upsilon^{(1)} +  \Upsilon^{(2)}  + {\cal OS}^3(N )$
with $ \Upsilon^{(1)}$ given by item (i) and $ \Upsilon^{(2)}$ given by item (ii).
\end{lemma}
\begin{proof} $(i)$
Arguing as in the proof of Lemma \ref{corollario coniugazione Omega bot}$(i)$ (cf. \eqref{formula generale commutatore paradiff forma normale}), 
one sees that $ [X_{\cal N}, \mathcal Q_0] (\frak x)$ is of the form $\big(0, 0,  [X_{\cal N}, \, \mathcal Q_0]^\bot (\frak x)\big)$ where
$$
 [X_{\cal N}, \mathcal Q_0]^\bot (\frak x) = 
\Big( \big[\ii \Omega_\bot , {\cal Q}^\bot_0(\theta, y)\big]_{lin} + (\omega + \e \widehat \omega) \cdot \partial_\theta {\cal Q}^\bot_0(\theta, y) 
 + \nabla_y Q(y) \cdot \partial_\theta {\cal Q}^\bot_0(\theta, y)\Big)[w]\, .
$$
One has
$$
 \omega \cdot \partial_\theta {\cal Q}^\bot_0(\theta, y) [w] \in {\cal OS}_{w}^2(N), \qquad
 \e \widehat \omega \cdot \partial_\theta {\cal Q}^\bot_0(\theta, y) [w]  \in {\cal OS}^3(N), \qquad
   \nabla_y Q(y) \cdot \partial_\theta {\cal Q}^\bot_0(\theta, y)[w] \in {\cal OS}^3(N),
 $$ 
 and since $\ii \Omega_\bot$ is a Fourier multiplier of order three, it follows that 
$\big[\ii \Omega_\bot , {\cal Q}^\bot_0(\theta, y)\big]_{lin} w \in {\cal OS}_{ww}^2(N - 3)$. The claimed statement then follows. 

\noindent
$(ii)$ Arguing as in the proof of Lemma \ref{corollario coniugazione Omega bot}$(ii)$ (cf. \eqref{formula generale commutatore paradiff forma normale 1}), and using that 
$F_1(\theta)[w, w]$ and ${\cal Q}^\bot_1(\theta)[w, w]$ are quadratic forms with respect to $w$, one sees that
$Y := [X_{\cal N}, \mathcal Q_1]$ is of the form  $Y = (Y^{(\theta)}, \, 0, \, Y^\bot)$ where 
$$
 Y^{(\theta)}(\frak x) = ( \omega + \e \widehat \omega) \cdot \partial_\theta F_1(\theta)[w, w] - F_1(\theta)[\ii \Omega_\bot w, w] - F_1(\theta)[w, \ii \Omega_\bot w]  
 + \nabla_y Q(y) \cdot \partial_\theta F_1(\theta)[w, w]
$$
$$
\begin{aligned}
Y^\bot(\frak x) & =  \ii \Omega_\bot {\cal Q}^\bot_1(\theta)[w, w]- {\cal Q}^\bot_1(\theta)[\ii \Omega_\bot w, w] - {\cal Q}^\bot_1(\theta)[w, \ii \Omega_\bot w] 
+ ( \omega  +  \e \widehat \omega) \cdot \partial_\theta {\cal Q}^\bot_1(\theta)[w, w]  \qquad \\
& \qquad + \nabla_y Q(y) \cdot \partial_\theta {\cal Q}^\bot_1(\theta)[w, w]\,. 
\end{aligned}
$$
By \eqref{addition smoothing vector field}, 
$ \omega \cdot \partial_\theta F_1(\theta)[w, w] $, $F_1(\theta)[\ii \Omega_\bot w, w]$, and $F_1(\theta)[w, \ii \Omega_\bot w]$ are smooth functions and small of order two, 
whereas $\e \widehat \omega \cdot \partial_\theta F_1(\theta)[w, w]$ and  $\nabla_y Q(y) \cdot \partial_\theta F_1(\theta)[w, w]$ are smooth functions and small of order three. 
(Here we used that by \eqref{cal N Q Omega bot},  $\nabla_y Q(y)$ is small of order one.)
Furthermore, by the definition of $ {\cal Q}^\bot_1$ one has
$ \omega \cdot \partial_\theta {\cal Q}^\bot_1(\theta)[w, w] \in {\cal OS}^2_{ww}(N)$, whereas 
$\e \widehat \omega \cdot \partial_\theta {\cal Q}^\bot_1(\theta)[w, w]$ and $\nabla_y Q(y) \cdot \partial_\theta {\cal Q}^\bot_1(\theta)[w, w]$ are in ${\cal OS}^3(N)$. 
Finally, since $\ii \Omega_\bot$ is a Fourier multiplier of order three, 
$$
\ii \Omega_\bot {\cal Q}^\bot_1(\theta)[w, w]- {\cal Q}^\bot_1(\theta)[\ii \Omega_\bot w, w] - {\cal Q}^\bot_1(\theta)[w, \ii \Omega_\bot w] \in {\cal OS}_{ww}^2(N - 3)\,.
$$ 
The claimed statement then follows.  

\noindent
$(iii)$  By \eqref{espansione commutatori notazioni}, $\Phi_{\cal Q}^* X_\mathcal N(\frak x)$ can be expanded as
$$
\Phi_{\mathcal Q}^* X_{\cal N} = X_{\cal N} + [X_{\cal N}, \mathcal Q] + Z , \quad Z(\frak x) := \int_0^1 (1 - t) d \Phi_\mathcal Q(t, \frak x)^{- 1}[[X_{\cal N}, \mathcal Q], \mathcal Q](\Phi(t, \frak x))\, d t\,. 
$$
By items $(i)$ and $(ii)$, the commutator $[X_{\cal N}, \mathcal Q]$ is in ${\cal OS}^2(N - 3)$, hence by Lemma \ref{Commutator of smoothing vector fields}, 
$[[X_{\cal N}, \mathcal Q], \mathcal Q] \in {\cal OS}^3(N - 3)$. By applying Lemma \ref{push forward smoothing 1}-$(iii)$, one then infers that 
$Z \in {\cal OS}^3(N - 3)$. The claimed expansion then follows by items $(i)$ and $(ii)$. 
\end{proof}

\smallskip

In Section \ref{forma normale smoothing standard}, we use  Hamiltonian vector fields $X_\mathcal F$, corresponding to  Hamiltonians $\mathcal F$, which are affine functions 
with respect to the normal component $w$. More precisely, $\mathcal F$ is assumed to be of the form
\begin{equation}\label{ham linear w general}
{\cal F}(\frak x) := {\cal F}_0(\theta, y) + \big\langle {\cal F}_1(\theta, y)\,,\, w \big\rangle
\end{equation}
where 
\begin{equation}\label{ham linear w general 2}
\begin{aligned}
& {\cal F}_0 \in C^\infty_b\big( \T^{S_+} \times B_{S_+}(\delta) \times  [0, \e_0] , \, \R \big)\,, \qquad 
{\cal F}_1 \in C^\infty_b\big(  \T^{S_+} \times B_{S_+}(\delta)\times [0, \e_0] , \, H^s_\bot(\T_1) \big), \quad \forall s \geq 0\,. 
\end{aligned}
\end{equation}
The Hamiltonian vector field generated by the Hamiltonian ${\cal F}$ is given by 
\begin{equation}\label{ham vector field general smoothing}
X_{\cal F}(\frak x) = \big( - \nabla_\theta {\cal F}(\frak x), \,  \nabla_y {\cal F}(\frak x), \, \partial_x {\cal F}_1(\theta, y) \big).
\end{equation}
The following lemma can be easily deduced by \eqref{ham linear w general}-\eqref{ham vector field general smoothing}.
\begin{lemma}\label{prop astratte campi vettoriali smoothing NF}
The vector field $X_{\cal F}$ is a smoothing vector field of arbitrary order, i.e.,  $X_{\cal F} \in {\cal OS}(N)$ for any $N \in \N$. 
Moreover, if in addition ${\cal F}_0$ is small of order $p$ and ${\cal F}_1$ is small of order $q$, then $\nabla_\theta {\cal F}$ is small of order ${\rm min}\{ p, q + 1\}$, $\nabla_y {\cal F}$ is small of order ${\rm min}\{p - 1, q\}$ and $\partial_x {\cal F}_1$ is small of order $q$. 
\end{lemma} 



\section{Reformulation of Theorem \ref{stability theorem}  and Normal Form Theorem}\label{normal form theorem}

The goal of this section is to describe the normal form coordinates provided by \cite[Theorem 1.1]{Kap-Mon-2},
specifically constructed to analyze perturbations of the KdV equations near finite gap solutions
and then to express equation \eqref{1.1} with respect to these coordinates. 
The main results of this section are Theorem \ref{long time ex action-angle}, which reformulates Theorem \ref{stability theorem} 
in these novel coordinates, and Theorem \ref{teorema totale forma normale} (Normal Form Theorem), which is the key ingredient into the proof of Theorem \ref{long time ex action-angle}.

We begin by rephrasing \cite[Theorem 1.1]{Kap-Mon-2} in a form, adapted to our needs. 
   Without further references, we use the notations introduced in Section \ref{introduzione paper}.
   \begin{theorem}\label{modified Birkhoff map}
Let $S_+ \subseteq \N$ be finite
and $\Xi \subset \R_{> 0}^{S_+}$ be compact.   
Then for $\delta > 0$ sufficiently small with $\Xi + B_{S_+}(\delta) \subset \R^{S_+}_{> 0}$ 
there exists  a $C^\infty$- smooth family of canonical diffeomorphisms
$$
\Psi_\mu : {\cal V}(\delta) \to \Psi_\mu({\cal V}(\delta)) \subseteq L^2_0(\T_1)\,, \,  \frak x \mapsto q,
$$ 
parametrized by $\mu \in \Xi$, with the property that for any $\mu \in \Xi$, $\Psi_\mu(\frak x)$ satisfies 
$$
\Psi_\mu(\theta, y , 0) = 
\Psi_{S_+}(\theta, \mu + y), 
\quad \forall (\theta, y, 0) \in {\cal V}(\delta) \, ,
$$
and is compatible with the scale of Sobolev spaces $H^s_0(\T_1)$, $s \in \Z_{\geq 0}$ (meaning that 
$\Psi_\mu \big( {\cal V}(\delta) \cap \mathcal E_s \big) \subseteq H^s_0(\T_1)$ and $\Psi_\mu : {\cal V}(\delta) \cap \mathcal E_s \to H^s_0(\T_1)$ 
is a $C^\infty$-diffeomorphism onto its image), so that the following holds:
\begin{description}
\item[({\bf AE1})] For any $N \in \N$,  $\mu \in \Xi$, and $\frak x = (\theta, y , w) \in \mathcal V(\delta)$, $\Psi (\frak x) \equiv \Psi_\mu(\frak x)$ 
has an expansion of the form,   
$$
\Psi(\frak x) = 
\Psi_{S_+}(\theta, \mu + y)
 +  w + 
\sum_{k = 1}^N a_{- k}(\frak x; \Psi) \, \partial_x^{- k} w 
+ {\cal R}_{N}(\frak x; \Psi) \, ,
$$ 
where  ${\cal R}_{N}(\theta, y , 0; \Psi) = 0$ and where
for any $s \in \Z_{\geq 0}$  and $1 \le k \le N$,
$$
 {\cal V}(\delta) \to H^s(\T_1),\,  \frak x \mapsto  a_{- k}(\frak x; \Psi), \qquad
  {\cal V}^s(\delta)   \to H^{s + N +1}(\T_1),\, \frak x \mapsto  {\cal R}_N(\frak x; \Psi) ,
$$
are $C^\infty$ maps (cf. \eqref{Vns} for the definition of ${\cal V}^s(\delta)$). 
\item[({\bf AE2})]
For any $\frak x = (\theta, y, w)  \in {\cal V}^1(\delta)$ and $\mu \in \Xi$, the transpose $d \Psi_\mu(\frak x)^\top$ 
(with respect to the standard inner products) of the differential
$d \Psi_\mu(\frak x) : E_1 \to H^1_0(\T_1)$ yields a bounded operator $d \Psi(\frak x )^\top \equiv d \Psi_\mu(\frak x )^\top : H^1_0(\T_1) \to E_1$. 
For any $\widehat q \in H^1_0(\T_1)$
and any integer $N \ge 1$, $d \Psi(\frak x )^\top [\widehat q]$ admits an expansion of the form 
$$
d \Psi(\frak x)^\top [\widehat q]= \Big( \, 0, \, 0, \, 
 \Pi_\bot \widehat q  +\Pi_\bot \sum_{k = 1}^N  a_{- k}(\frak x; {d \Psi^\top})  \partial_x^{- k}\widehat q \, 
+ \Pi_\bot  \sum_{k = 1}^N  ( \partial_x^{- k} w) \mathcal A_{- k}(\frak x; {d \Psi^\top})[\widehat q]  \, \Big)
 + {\cal R}_{N}(\frak x; {d \Psi^\top})[\widehat q]
$$
where for any $s \in \N$ and $1 \le  k \le N$, 
$$
{\cal V}^1(\delta)  \to H^s(\T_1)\,, \, \frak x \mapsto a_{- k}(\frak x; {d \Psi^\top})\,,  \qquad 
  {\cal V}^1(\delta)  \to  {\cal B}(H^1_0(\T_1), H^s(\T_1))\,, \, 
\frak x \mapsto \mathcal A_{- k}(\frak x; {d \Psi^\top})\,, 
$$
and
$$
  {\cal V}^s(\delta)  \to {\cal B}(H^s_0(\T_1), E_{s + N +1}), \,
\frak x \mapsto {\cal R}_N(\frak x; {d \Psi^\top})\, , 
$$
are $C^\infty$-smooth, bounded maps. 
 \item[({\bf AE3})] For any $\mu \in \Xi$, the Hamiltonian ${\cal H}_\mu^{kdv} := H^{kdv} \circ \Psi _\mu: {\cal V}^1(\delta) \to \R$ is in normal form up to order three. 
 More precisely,  for any $\frak x = (\theta, y, w) \in  {\cal V}^1(\delta),$ the Taylor expansion of ${\cal H}^{kdv} \equiv {\cal H}_\mu^{kdv}$ at $(\theta, 0 , 0)$ 
  with respect to $y$ and $w$ up to order three reads
\begin{equation}\label{hamiltoniana KdV}
{\cal H}^{kdv} (\frak x)  = e + \omega \cdot y + \frac12  \Omega_{S_+} [y] \cdot y  
+ \frac12\big\langle D^{- 1}_\bot \Omega_\bot w, w  \big\rangle  + {\cal P}^{kdv} (\frak x) \, ,
\end{equation}
where $e := {\cal H}_\mu^{kdv} (0, 0, 0)= H^{kdv}(
\Psi_{S_+}(0, \mu))
$, 
$$
\omega = (\omega_n^{kdv}(\mu, 0))_{n \in S_+}\, , \qquad
\Omega_{S_+} := (\partial_{I_j} \omega^{kdv}_k (\mu, 0))_{j, k \in S_+}\, ,
$$
and for any $w =  \sum_{n \in S^\bot}  w_n e^{\ii 2\pi n x}$,
$D_\bot^{- 1} w := \sum_{n \in S^\bot} \frac{1}{2 \pi n} w_n e^{\ii 2 \pi n x}$, and (cf. \eqref{def notation normal frequencies})
\be\label{Omega-normal}
 \Omega_\bot w: =  \sum_{n \in S^\bot } \Omega_n w_n e^{\ii 2\pi n x}\,, \qquad  
  \Omega_n := \omega_n^{kdv}(\mu , 0)\,, \quad \forall n \in S^\bot  \, .
\ee
Furthermore,
${\cal P}^{kdv} : {\cal V}^1(\delta) \to \R$ is $C^\infty$-smooth,
satisfies
$$
| {\cal P}^{kdv}(\frak x) | \lesssim ( |y| + \| w \|_1)^3 , \qquad \forall \, \frak x= (\theta, y, w) \in \mathcal V^1(\delta)\, , \ \forall \, \mu \in \Xi\, ,
$$
and has the following property: for any integer $N \ge 1$
there exists an integer $\sigma_N \ge N$ (loss of regularity) so that 
$\nabla {\cal P}^{kdv}(\frak x)= (\nabla_{\theta} {\cal P}^{kdv}(\frak x),  \nabla_{y} {\cal P}^{kdv}(\frak x), \nabla_\bot{\cal P}^{kdv}(\frak x))$ 
admits an expansion of the form 
$$
\nabla {\cal P}^{kdv}(\frak x) =
\big(\, 0, \, 0, \, \Pi_\bot  \sum_{k = 0}^N  T_{a_{- k}(\frak x; {\cal P}^{kdv})} \, \partial_x^{- k} w \, \big)   + {\cal R}_N(\frak x; {\cal P}^{kdv}) ,
$$
where there exist integers $s_N > 0$ and $\sigma_N > 0$ so that for any  $s \geq s_N$ and any $0 \le  k \le N$, 
$$
\begin{aligned}
&  {\cal V}^{s + \sigma_N}(\delta)  \to H^s(\T_1), \, \frak x \mapsto  a_{- k}(\frak x; {\cal P}^{kdv}) \, ,  
\qquad
{\cal V}^{s \lor \sigma_N}(\delta) \to \mathcal E_{s + N + 1}, \,  \frak x \mapsto  {\cal R}_N(\frak x; {\cal P}^{kdv}) 
\end{aligned}
$$
are $C^\infty$-smooth and satisfy for any $\theta \in \T_1^{S_+}$, $\mu \in \Xi$,
$$
a_{- k}(\theta, 0, 0; {\cal P}^{kdv}) = 0, \quad  {\cal R}_N(\theta, 0, 0; {\cal P}^{kdv}) = 0, \quad
\partial_y  {\cal R}_N(\theta, 0, 0; {\cal P}^{kdv}) = 0,  \quad d_\bot  {\cal R}_N(\theta, 0, 0; {\cal P}^{kdv}) = 0.
$$ 
\end{description}
\noindent
Here $T_{a_k(\frak x; {\cal P}^{kdv})}$ denotes the operator of para-multiplication with $a_k(\frak x; {\cal P}^{kdv})$ (cf. Definition \ref{definizione paraprodotto}).
\end{theorem}
\begin{remark}\label{remark D inv Omega bot}
Since $\Omega_{- n} = - \Omega_n$ for any $n \in S^\bot$ (cf. \eqref{kdv frequencies}, \eqref{def notation normal frequencies}), 
the Fourier multiplyer $\ii \Omega_\bot$ 
is a real operator. 
In view of the expansion \eqref{hamiltoniana KdV} and the identity $\partial_x D^{- 1} = \ii$, 
the component of the Hamiltonian vector field ${\cal H}^{kdv}_\mu$ in the normal direction is given by 
$$
\partial_x \nabla_\bot {\cal H}^{kdv}(\frak x) = \ii \Omega_\bot w + \partial_x \nabla_\bot {\cal P}^{kdv}(\frak x)\,. 
$$
\end{remark}

  Next, we want to express equation \eqref{1.1} in the normal form coordinates provided by Theorem \ref{modified Birkhoff map}. 
  To this end we write the nonlinear  vector field $F(u)$ in the coordinates $(\theta, y, z)$.  
  Recall that $F(u) =  \partial_x \nabla { P}_f(u)$ where ${P}_f(u) := \int_0^1 f(x, u(x))\, d x$ and $f$ is given by \eqref{f def intro}. 
  
\begin{proposition}\label{Ham vector field perturbation}
Let $N \in \N$. Then there exist integers $s_N> 0$, $\sigma_N > 0$   so that for any perturbation 
$P_f (u) = \int_0^1 f(x, u(x))\, d x$ with $f $ $C^{\infty}$-smooth, the following holds. 
For any $\mu \in \Xi$, the gradient of
\begin{equation}\label{perturbation in new coo}
{\cal P}_{f} \equiv {\cal P}_{f, \mu}  := P_f \circ \Psi_\mu : {\cal V}^1(\delta) \to \R
\end{equation} 
admits an expansion of the form 
$$
\nabla {\cal P}_f (\frak x) = \big(0, 0, \, \Pi_\bot  \sum_{k = 0}^N T_{a_{ - k}(\frak x; \nabla \mathcal P_f  )} \partial_x^{- k}w  \big) + {\cal R}_N(\frak x; \nabla \mathcal P_f) \, ,
$$
where for any $s \ge s_N$ and for any $0 \le k \le N$, the maps 
$$
{\cal V}^{s + \sigma_N}(\delta)  \to H^s(\T_1), \, \frak x \mapsto a_{- k}(\frak x;  \nabla \mathcal P_f)\, ,  \qquad
{\cal V}^s(\delta) \to E_{s + N + 1}, \, \frak x \mapsto {\cal R}_N(\frak x; \nabla \mathcal P_f) 
$$ 
are $C^\infty$-smooth.
\end{proposition}
\begin{proof}
One has
\begin{equation}\label{gradiente cal K u}
\begin{aligned}
& \nabla {P}_f(u)(x) = \partial_{\zeta} f(x, u(x))\,. 
\end{aligned}
\end{equation}
By the Bony para-linearization formula (cf. \cite[ Section 5.2.3]{Metivier}) for the composition operator, one gets that 
\begin{equation}\label{a0 a1 paralin}
\nabla {P}_f(u)(x) = \partial_{\zeta} f(x, u(x))= T_{\partial_{\zeta}^2 f(x, u(x))} u + {\cal R}_{f}(u)
\end{equation}
where there exists $s_N > N$ (large) so that for any integer $s \ge s_N$, 
the map $\mathcal R_f: H^s(\T_1) \to H^{s + N + 1}(\T_1)$ is $C^\infty$-smooth. 
Note that ${\cal R}_{f}(u)$ contains the zeroth order term $\partial_{\zeta} f(x, 0)$ of the Taylor expansion of $\partial_\zeta f(x, \zeta)$ at $\zeta =0$. 
By Theorem \ref{modified Birkhoff map}-${\bf (AE2)}$, $d \Psi(\frak x)^\top [\widehat q]$ has an expansion of the form
\begin{equation}\label{d Psi t trasp A}
\Big( \, 0, \, 0, \, \Pi_\bot [\widehat q]  + \Pi_\bot \sum_{k = 1}^N  a_{- k}(\frak x; d \Psi^\top)  \partial_x^{- k}\widehat q \, 
+ \Pi_\bot \sum_{k = 1}^N   (\partial_x^{- k} w) \mathcal A_{- k}(\frak x; d \Psi^\top)[\widehat q]  \Big)
+ {\cal R}_{N}(\frak x; d \Psi^\top)[\widehat q] \, ,
\end{equation}
where the maps ${\cal V}(\delta) \to H^s(\T_1), \, \frak x \mapsto a_k(\frak x; d \Psi^\top)$,
$$
{\cal V}^1(\delta) \to {\cal B}(H^1_0(\T_1), H^s(\T_1)), \, \frak x \mapsto \mathcal A_k(\frak x; d \Psi^\top), \qquad
{\cal V}^s(\delta) \to {\cal B}(H_0^s(\T_1), E_{s+ N + 1}), \, \frak x \mapsto {\cal R}_{N}(\frak x; d \Psi^\top),
$$ 
are $C^\infty$-smooth, bounded maps. 
Using the expansion of $\Psi(\frak x)$ provided by Theorem \ref{modified Birkhoff map}-{\bf (AE1)}, 
\begin{equation}\label{espansione Psi paradiff}
\Psi(\frak x) =  \Psi_{S_+}(\theta, \mu + y) + w + \sum_{k = 1}^N a_{- k}(\frak x; \Psi) \partial_x^{- k} w  + {\cal R}_N(\frak x; \Psi) 
\end{equation}
together with the para-product formula \eqref{paraprodotto} and Lemma \ref{lemma paraprodotto fine}, one obtains 
 \begin{equation}\label{nabla cal K Psi frak x}
\begin{aligned}
(\nabla { P}_f)(\Psi(\frak x)) &  =  
\sum_{k = 0}^N T_{a_{ - k}(\frak x; \nabla { P}_f \circ\Psi )} \partial_x^{- k} w + {\cal R}_N(\frak x;  \nabla { P}_f \circ\Psi ), 
\qquad a_{ 0}(\frak x; \nabla { P}_f \circ\Psi ) = \partial_{\zeta}^2 f(x, \Psi(\frak x)),
\end{aligned}
\end{equation}
where there exist integers $\sigma_N \ge  0$ and $s_N \ge 0$ so that for any $s \geq s_N$ and $0 \le k \le N$, 
the maps 
$$
{\cal V}^{s + \sigma_N} \to H^s(\T_1), \, \frak x \mapsto a_{- k}(\frak x; \nabla { P}_f \circ\Psi ), \qquad
{\cal V}^s(\delta) \to E_{s + N + 1}, \, \frak x \mapsto {\cal R}_N(\frak x; \nabla { P}_f \circ\Psi ) ,
$$ 
are $C^\infty$-smooth.
The expansion of $\nabla{\cal P}_f(\frak x) = d \Psi(\frak x)^\top (\nabla { P}_f)(\Psi(\frak x)) $ is then computed by using
the one of $d\Psi(\frak x)^\top$, provided by Theorem \ref{modified Birkhoff map}-{\bf (AE2)}. For any $1 \le k \le N$ ,
we thus need to compute the expansion of the sum
$\sum_{k=1}^N a_{- k}(\frak x; d \Psi^\top)  \partial_x^{- k} \nabla {P}_f(\Psi(\frak x)) + ( \partial_x^{- k} w) \mathcal A_{- k}(\frak x; d \Psi^\top)[\nabla { P}_f(\Psi(\frak x))] $. 
By \eqref{nabla cal K Psi frak x} and using the para-product formula \eqref{paraprodotto} one obtains
$$
\begin{aligned}
&\Pi_\bot  \sum_{k = 1}^N a_{ - k}(\frak x; d \Psi^\top)  \partial_x^{- k} \nabla P_f(\Psi(\frak x))  + (\partial_x^{-k}w)\mathcal A_{- k}(\frak x; d \Psi^\top)[\nabla P_f(\Psi(\frak x))] \\
& = \Pi_\bot   \sum_{k = 1}^N \Big( T_{a_{- k}(\frak x; d \Psi^\top)} \partial_x^{- k} \nabla P_f(\Psi(\frak x)) + 
T_{ \partial_x^{- k} \nabla P_f(\Psi(\frak x)) } a_{- k}(\frak x; d \Psi^\top) \Big)
+ {\cal R}^{(B)}\big(a_{- k}(\frak x; d \Psi^\top) ,\,  \partial_x^{- k} \nabla P_f(\Psi(\frak x)) \big) \\
& \quad +  \Pi_\bot   \sum_{k = 1}^N T_{\mathcal A_{- k}(\frak x; d \Psi^\top)[\nabla P_f(\Psi(\frak x))]} \partial_x^{- k} w 
+ T_{\partial_x^{- k}w} \mathcal A_{- k}(\frak x; d \Psi^\top)[\nabla P_f(\Psi(\frak x))]  \\
& \quad + \Pi_\bot   \sum_{k = 1}^N {\cal R}^{(B)}\big(\mathcal A_{- k}(\frak x; d \Psi^\top)[\nabla P_f(\Psi(\frak x))] ,\, \partial_x^{- k}  w \big) \\
& = \Pi_\bot  \sum_{k = 1}^N \Big( T_{a_{- k}(\frak x; d \Psi^\top)} \partial_x^{- k} \nabla P_f(\Psi(\frak x)) 
+ T_{\mathcal A_{- k}(\frak x; d \Psi^\top)[\nabla P_f(\Psi(\frak x))]} \partial_x^{- k} w \Big)  +  {\cal R}_N^{(1)}(\frak x)
\end{aligned}$$
where 
\begin{equation}\label{cal R N (1) def}
\begin{aligned}
{\cal R}_{N}^{(1)} &(\frak x) 
:= \Pi_\bot   \sum_{k = 1}^N  T_{ \partial_x^{- k} \nabla P_f(\Psi(\frak x)) } a_{- k}(\frak x; d \Psi^\top)  + T_{\partial_x^{- k}w } \mathcal A_{- k}(\frak x; d \Psi^\top)[\nabla P_f(\Psi(\frak x))]  \\
& \quad + \Pi_\bot  \sum_{k = 1}^N \Big( {\cal R}^{(B)}\big(a_{- k}(\frak x; d \Psi^\top)\,,\,  \partial_x^{- k} \nabla P_f(\Psi(\frak x)) \big)  
+ {\cal R}^{(B)}\big(\mathcal A_{- k}(\frak x; d \Psi^\top)[\nabla P_f(\Psi(\frak x))]\,,\, \partial_x^{- k} w  \big) \Big)\,.
\end{aligned}
\end{equation}
By applying Theorem \ref{modified Birkhoff map}-${\bf (AE1)}$,${\bf (AE2)}$, and Lemma \ref{lemma remainder paraprod}, one obtains, after increasing $s_N$ if needed,  
that for any $s \ge s_N$, the map ${\cal V}^s(\delta) \to E_{s + N + 1}$, $\frak x \mapsto {\cal R}^{(1)}_{N}(\frak x)$ is $C^\infty$-smooth. 
By the expansion given in \eqref{nabla cal K Psi frak x} and by applying Lemma \ref{lemma composizione nostri simboli} (composition of para-differential operators), one then gets 
the following identity for the normal component $(\nabla{\cal P}_f)^\bot$ of $\nabla{\cal P}_f$,
$$
\begin{aligned}
(\nabla{\cal P}_f)^\bot(\frak x) & = 
 \Pi_\bot [\nabla P_f(\Psi(\frak x))] + 
 \Pi_\bot  \sum_{k = 1}^N \Big( T_{a_{- k}(\frak x; d \Psi^\top)} \partial_x^{- k} \nabla P_f(\Psi(\frak x)) + 
T_{\mathcal A_{- k}(\frak x; d \Psi^\top)[\nabla P_f (\Psi(\frak x))]} \partial_x^{- k} w \Big)  +  {\cal R}_N^{(1)}(\frak x)\\
& =   \Pi_\bot\sum_{k = 0 }^N T_{a_{ - k}(\frak x; \nabla{\cal P}_f)} \partial_x^{- k} w + {\cal R}_N^{(2)}(\frak x)\, , 
\qquad \qquad a_{ 0}(\frak x; \nabla{\cal P}_f) = \partial_{\zeta}^2 f(x, w(x)) \, , 
\end{aligned}
$$
where there exist constants $s_N \geq N$ and $\sigma_N \ge N$ so that for any $s \ge s_N$ and any $0 \le k \le N$, 
the maps 
$$
{\cal V}^{s + \sigma_N}(\delta) \to H^s(\T_1),  \, \frak x \mapsto a_{- k}(\frak x; \nabla{\cal P}_f), \qquad
{\cal V}^s(\delta) \to H_\bot^{s + N + 1}(\T_1), \, \frak x \mapsto {\cal R}_N^{(2)}(\frak x)  ,
$$ 
are $C^\infty$-smooth. Altogether we obtain
$$
\nabla {\cal P}_f (\frak x) = = d \Psi(\frak x)^\top (\nabla { P}_f)(\Psi(\frak x)) =
 \big(0, 0, \, \Pi_\bot \sum_{k = 0}^N T_{a_{ - k}(\frak x; \nabla \mathcal P_f  )} \partial_x^{- k}w  \big) + {\cal R}_N(\frak x; \nabla \mathcal P_f) \, ,
$$
where
$$
{\cal R}_N(\frak x; \nabla {\cal P}_f) := (0, 0, {\cal R}_N^{(2)}(\frak x)) + {\cal R}_{N}(\frak x; d \Psi^\top)[\nabla P_f(\Psi(\frak x))].
$$
One verifies in a straightforward way that ${\cal R}_N(\frak x; \nabla {\cal P}_f)$ has the stated properties.
\end{proof}

Combining Theorem \ref{modified Birkhoff map} and Proposition \ref{Ham vector field perturbation} together with Lemma \ref{composizione paraprodotto e derivate} 
yields the following corollary.
\begin{corollary}[\bf Expansion of $\mathcal H_\mu$]\label{espansion hom}
For any $\mu \in \Xi$, $\mathcal H \equiv \mathcal H_\mu = (H^{kdv} + \e P_f) \circ \Phi_\mu $ can be written as 
\begin{equation}\label{kdv hamiltonian + perturbation}
\mathcal H(\frak x) = e +  {\cal N}(\frak x) + {\cal P}(\frak x) , \qquad {\cal P} (\frak x): =  {\cal P}^{kdv}(\frak x) + \e {\cal P}_{f}(\frak x), 
\end{equation}
where $e$, ${\cal N}$, and ${\cal P}^{kdv}$ are given by Theorem  \ref{modified Birkhoff map}-{\bf{(AE3)}} and  ${\cal P}_{f}$ by Proposition \ref{Ham vector field perturbation}.
More precisely, $e = {\cal H}_\mu^{kdv} (0, 0, 0)$ and for any $\frak x = (\theta, y, w) \in \mathcal V^1(\delta)$, 
\begin{equation}\label{hamiltoniana totale 1}
{\cal N}(y, w) = \omega \cdot y + \frac12  \Omega_{S_+} [y] \cdot y  + \frac12 \big\langle D_\bot^{- 1} \Omega_\bot w\,,\, w \big\rangle ,
\end{equation}
with
\begin{equation}\label{def D inv Omega bot}
D^{- 1}_\bot w (x) = \sum_{j \in S^\bot} \frac{1}{2 \pi n} w_n e^{\ii 2 \pi n x}, \qquad \Omega_\bot w(x) = \sum_{n \in S^\bot } \Omega_n w_n e^{\ii 2 \pi n x}\,.
\end{equation}
The perturbation $\mathcal P$ is of the form (cf. Proposition \ref{Ham vector field perturbation})
\begin{equation}\label{cal N cal H cal P}
 {\cal P}(\frak x) = \e{\cal P}_L(\frak x) + {\cal P}_e(\frak x)  \, , \quad \quad
 {\cal P}_L(\frak x) := {\cal P}_{00}(\theta ) + {\cal P}_{1 0}(\theta) \cdot y + \big\langle {\cal P}_{0 1}(\theta)\,,\, w \big\rangle\, ,
\end{equation}
with ${\cal P}_e$, ${\cal P}_{00}(\theta )$, ${\cal P}_{1 0}(\theta)$, and ${\cal P}_{0 1}(\theta)$ having the following properties:
there exist $0 < \delta < 1$, $0 < \e_0 < 1$, and an integer $\sigma > 0$ so that
\begin{equation}\label{proprieta ham primissima per forma normale}
\begin{aligned}
& {\cal P}_{00} \in C^\infty(\T^{S_+}, \, \R), \quad {\cal P}_{10} \in C^\infty(\T^{S_+}, \, \R^{S_+}), \quad {\cal P}_{0 1}  \in C^\infty(\T^{S_+}, \, H^s_\bot(\T_1)), 
\quad \forall s \geq 0\,, \\
& {\cal P}_e \in C^\infty({\cal V}^\sigma(\delta) \times [0, \e_0], \, \R) \quad \text{small of order three}, \\
& X_{{\cal P}_e} = (X_{{\cal P}_e}^{(\theta)}, X_{{\cal P}_e}^{(y)}, X_{{\cal P}_e}^\bot) 
= (- \nabla_y {\cal P}_e,  \nabla_\theta {\cal P}_e , \, \partial_x \nabla_\bot {\cal P}_e) \quad  \text{small of order two}, \\
& X_{{\cal P}_e}^\bot = \partial_x \nabla_\bot {\cal P}_e = {\cal OB}^2(1, N) + {\cal OS}^2(N), \quad \forall N \in \N\, ,
\end{aligned}
\end{equation}
(cf. Definition \ref{paradiff vector fields} and Definition \ref{def smoothing vector fields} for the classes of vector fields ${\cal OB}^2(1, N)$ and respectively,  ${\cal OS}^2(N)$).
\end{corollary}
\begin{remark}\label{irrelevant constant}
Since the constant $e$ in \eqref{kdv hamiltonian + perturbation} does not affect the Hamiltonian vector field $X_\mathcal H$, by notational convenience,
we will suppress it in the sequel. The same convention will be used for any Hamiltonian under consideration.
\end{remark}
We now reformulate Theorem \ref{stability theorem} in the coordinates, provided by Theorem \ref{modified Birkhoff map}.
 By Corollary \ref{espansion hom}, the one parameter family of Hamiltonians $\mathcal H \equiv \mathcal H_\mu = (H^{kdv} + \e P_f) \circ \Phi_\mu $, $\mu \in \Xi$, is given by
 \begin{equation}\label{hamiltoniana totale}
  {\cal H}(\frak x) =  {\cal N}(\frak x) + \e{\cal P}_L(\frak x) + {\cal P}_e(\frak x)
 \end{equation}
 with $\mathcal N$ defined by \eqref{hamiltoniana totale 1} and $\mathcal P_L$, $\mathcal P_e$ by \eqref{cal N cal H cal P} (cf. Remark \ref{irrelevant constant}).
 Using that $\partial_x D_\bot^{- 1} \Omega_\bot = \ii \Omega_\bot$, 
 the Hamiltonian vector field $X_{\cal H} = \big(- \nabla_y \mathcal H, \nabla_\theta \mathcal H, \partial_x\nabla_\bot \mathcal H \big)$ 
 can be computed as
\begin{equation}\label{very first ham vec field}
X_{\cal H}(\frak x) = \begin{pmatrix}
- \omega -  \Omega_{S_+} [y] -  \e {\cal P}_{1 0}(\theta) - \nabla_y {\cal P}_e(\frak x) \\
 \e \nabla_\theta {\cal P}_L(\frak x) + \nabla_\theta {\cal P}_e(\frak x) \\
\ii \Omega_\bot w + \e \partial_x {\cal P}_{0 1}(\theta) + \partial_x \nabla_\bot {\cal P}_e(\frak x)
\end{pmatrix}
\end{equation}
and the corresponding Hamiltonian equations are 
\begin{equation}\label{very first ham equations}
\begin{aligned}
\partial_t \theta & = - \omega -  \Omega_{S_+} y -  \e {\cal P}_{1 0}(\theta) - \nabla_y {\cal P}_e(\frak x), \\
 \partial_t y & =  \e \nabla_\theta {\cal P}_L(\frak x) +  \nabla_\theta {\cal P}_e(\frak x), \\
\partial_t w &= \ii \Omega_\bot w + \e \partial_x {\cal P}_{0 1}(\theta) + \partial_x \nabla_\bot {\cal P}_e(\frak x).
\end{aligned}
\end{equation}
Except for the measure estimate \eqref{main measure estimate}, Theorem \ref{stability theorem} is an immediate consequence of the following
theorem. (We refer to Section \ref{measure estimates} for a proof of \eqref{main measure estimate}.)
\begin{theorem}\label{long time ex action-angle}
Let $f \in C^{\infty}(\T_1 \times \R, \, \R)$, $S_+$ be a finite subset of $\N$, 
$\tau$ be a number with $\tau > |S_+| $ (cf. \eqref{condizioni forma normale}), and 
 $\mu = \mu(\omega)$ with $\omega \in \Pi_\gamma$, $0 < \gamma  < 1$.
Then for any integer $s$ sufficiently large, there exists $0 < \e_0 \equiv \e_0(s, \gamma) < 1$ with the following properties: 
for any $0 < \e  \le \e_0$ there exists $T \equiv T_{\e, s, \gamma} = O(\e^{-2})$,
so that for any initial data $\frak x_0 = (\theta_0, y_0, w_0) \in \T^{S_+} \times \R^{S_+} \times H^s_\bot(\T_1)$, satisfying 
 \begin{equation}\label{smallness initial datum action-angle}
 | y_0|\,,\, \| w_0 \|_s \leq \e \, ,
 \end{equation}
there exists a unique solution $t \mapsto \frak x(t) = (\theta(t), y(t), w(t))$ of \eqref{very first ham equations} with $\frak x(0) = \frak x_0$ and
$$
 \theta \in C^1([- T, T], \T^{S_+}), \quad   y \in C^1([- T, T], \R^{S_+}), \quad w \in C^0([- T, T], H^s_\bot(T_1)) \cap C^1([- T, T], H^{s - 3}_\bot(\T_1))\,.
 $$
 In addition, the solution satisfies $ | y(t)|\,,\, \| w(t)\|_s \lesssim_{s, \gamma} \e$ for any $ t \in [- T, T]$.
\end{theorem}
Theorem \ref{long time ex action-angle} is proved in Section \ref{conclusioni forma normale}. 
A key ingredient of its proof is the following result on normal forms. 
\begin{theorem}\label{teorema totale forma normale}{(\bf Normal Form Theorem)} 
Let $f \in C^{\infty}(\T_1 \times \R, \, \R)$, $S_+$ be a finite subset of $\N$, 
$\tau$ be a number with $\tau > |S_+| $ (cf. \eqref{condizioni forma normale}), and 
$\mu = \mu(\omega)$ with $\omega \in \Pi_\gamma$, $0 < \gamma < 1$.
Then there exists $\sigma_\ast  >0$ so that for any integer $s \ge \sigma_\ast$ the following holds:
there exist $0 < \delta \equiv \delta(s, \gamma) < 1$, $0 < \e_0\equiv \e_0(s, \gamma) \ll \delta$,
and $C_0 \equiv C_0(s, \gamma)  > 1$ 
 with the property that for any $0 < \e \le \e_0$ there exists an invertible map ${\mathtt \Phi}$ with inverse ${\mathtt \Phi}^{-1}$
(cf. Remark \ref{inverse of flow}),
\begin{equation}\label{trasformazione finale cal U}
 {\mathtt \Phi}^{\pm 1}\in {\cal C}^\infty_b({\cal V}^s(\delta), {\cal V}^s(C_0 \delta)),  \qquad \quad
 {\mathtt \Phi}^{\pm 1}(\frak x) - {\frak x} \ \  \text{small of order one}\, ,
\end{equation}
so that the pull back  $X  = (X^{(\theta)}, X^{(y)}, X^\bot) := {{\mathtt \Phi}}^* X_{\cal H_\mu}$ of the vector field $X_{\cal H_\mu}$ by ${\mathtt \Phi}$ has the form
\begin{equation}\label{forma campo vettoriale finale dopo NF}
\begin{aligned}
 X^{(\theta)}(\frak x) = & - \omega -  \e \widehat \omega + {\mathtt  N}^{(\theta)}(y, w) + {\cal O}_3^{(\theta)}(\frak x)\, , \qquad  X^{(y)}(\frak x) = {\cal O}_3^{(y)}(\frak x)\, ,  \qquad\\
& X^\bot (\frak x) =  \ii \Omega_\bot w + {\mathtt D}^\bot(\frak x)[w] + \Pi_\bot T_{a(\frak x)} \partial_x w + {\cal R}^\bot(\frak x) \, , \qquad
\end{aligned}
\end{equation}
where $\widehat \omega \in \R^{S_+}$ and
\begin{equation}\label{proprieta elementi campo vettoriale}
\begin{aligned}
& \mathtt N^{(\theta)} \in C^\infty_b \big(B_{S_+}(\delta) \times B^{\sigma_\ast}_\bot(\delta) \times [0, \e_0], \, \R^{S_+} \big)  
\quad \text{small of order one } (\text{and independent of } \theta), \qquad \qquad \qquad \\
& {\cal O}_3^{(\theta)}, \ {\cal O}_3^{(y)} \in C^\infty_b( {\cal V}^{\sigma_\ast}(\delta) \times [0, \e_0], \, \R^{S_+}) \quad \text{small of order three}, \\
& {\mathtt D}^\bot \in C^\infty_b \big( {\cal V}^{\sigma_\ast}(\delta) \times  [0, \e_0], \, {\cal B}(H^{s}_\bot(\T_1), H^{s - 1}_\bot(\T_1) ) \big)  \quad \text{small of order one,}  \\
& \quad  {\mathtt D}^\bot \text{ Fourier multiplier of the form }  {\mathtt D}^\bot(\frak x)[w] = \sum_{j \in S^\bot} d_j(\frak x) w_j e^{\ii 2 \pi j x}\text{ with the properties}\\ 
&\quad d_j \in C^\infty_b \big( {\cal V}^{\sigma_\ast}(\delta) \times  [0, \e_0], \, \R \big), \ \  \forall j \in S^\bot, 
 \qquad {\mathtt D}^\bot \text{ skew-adjoint: }  {\mathtt D}^\bot(\frak x)^\top = - {\mathtt D}^\bot(\frak x), \\
& a\in C^\infty_b \big(  {\cal V}^{s + \sigma_\ast}(\delta) \times [0, \e_0], \, H^s(\T_1) \big) \quad \text{small of order two}, \qquad \qquad \qquad \\
& {\cal R}^\bot \in C^\infty_b\big({\cal V}^s(\delta) \times [0, \e_0], \, H^s_\bot(\T_1) \big) \quad \text{small of order three.} \qquad \qquad \qquad
\end{aligned}
\end{equation}
\end{theorem}
The proof of Theorem \ref{teorema totale forma normale} is given in Section \ref{conclusioni forma normale}. The transformation $\mathtt \Phi$ is obtained as
the composition of several transformations, constructed in Section \ref{forma normale smoothing standard} - Section \ref{normalization II}.

\section{Smoothing normal form steps }\label{forma normale smoothing standard}

As part of the proof of Theorem \ref{teorema totale forma normale},
the aim of this section is to normalize terms in the Taylor expansion of the Hamiltonian ${\cal H}$ (cf. \eqref{hamiltoniana totale}),
which are affine with respect to the normal coordinate $w$ and homogeneous of order at most three
with respect to the coordinates $y, w$ and the parameter $\e$
 (cf. {\em Overview of the proof of Theorem  \ref{stability theorem}} in Section \ref{introduzione paper}). 
The main result of this section is the following one.
\begin{proposition}\label{prop ordine e 3}
Let $f \in C^{\infty}(\T_1 \times \R, \, \R)$, $S_+$ be a finite subset of $\N$, 
$\tau$ be a number with $\tau > |S_+| $ (cf. \eqref{condizioni forma normale}),
and $\mu = \mu(\omega)$ with $\omega \in \Pi_\gamma$, $0 < \gamma < 1$.
Then for any $N \in 	\N$, there exist integers $s_N  > 0$, $\sigma_N > 0$ so that for any $s \geq s_N$, there exist 
$0 < \delta \equiv \delta(s, \gamma, N) <1$ and $0 <  \e_0\equiv \e_0(s, \gamma, N)  \ll \delta$ 
with the following properties:  for any $0 < \e \le \e_0$ there exists an invertible symplectic transformation $\Phi$ with inverse $ \Phi^{-1}$ so that 
\begin{equation}\label{prop Phi finale passi smoothing}
 \Phi^{\pm 1}\in {\cal C}^\infty_b({\cal V}^s(\delta) \times [0, \e_0], {\cal V}^s(2 \delta)) \, , \qquad
 \Phi^{\pm 1}(\frak x) - {\frak x} \quad \text{small of order one}\, ,
\end{equation}
and so that the Hamiltonian ${\cal H}^{(3)} := {\cal H} \circ \Phi$ (cf. \eqref{irrelevant constant}) has the form 
\begin{equation}\label{rid termini lineari}
{\cal H}^{(3)}(\frak x) = {\cal N}^{(3)}(\frak x) + {\cal K}(\frak x) \, , \qquad 
{\cal N}^{(3)}(\frak x) :=   \omega \cdot y + \e \widehat \omega \cdot y + \frac12 \big\langle D_\bot^{- 1} \Omega_\bot w\,,\, w \big\rangle + Q(y)\, .
\end{equation}
Here $\widehat \omega \equiv \widehat \omega(\e) \in \R^{S_+}$ is an affine function of $\e$, $Q(y) \equiv Q(y, \e)$  is small of order two, 
a polynomial of degree three in $y$ and an affine function of $\e$, and the components of the Hamiltonian vector field 
$X_{\mathcal K} = (X^{(\theta)}_{\mathcal K}, X^{(y)}_{\mathcal K}, X^\bot_{\mathcal K}) =  (- \nabla_y \mathcal K, \, \nabla_\theta \mathcal K, \, \partial_x \nabla_\bot \mathcal K)$,
corresponding to the Hamiltonian ${\cal K}$, satisfy the following properties: 
$X^{(\theta)}_{\cal K}(\frak x)$ is of the form $\Upsilon_2^{(\theta)}(\theta)[w, w] + \Upsilon_3^{(\theta)}(\frak x)$ with
$$
\Upsilon_2^{(\theta)} \in C^\infty_b(\T^{S_+}, \, \mathcal B_2(H^{\sigma_N}_\bot(\T_1), \, \R^{S_+})), \qquad \qquad
\Upsilon_3^{(\theta)} \in C_b^\infty({\cal V}^{\sigma_N}(\delta) \times [0, \e_0], \, \R^{S_+}), \, \text{small of order three,}
$$
and
\begin{equation}\label{vector field cal K}
 X^{(y)}_{\cal K} \in C^\infty_b( {\cal V}^{\sigma_N}(\delta) \times [0, \e_0],  \, \R^{S_+}),  \,  \text{small of order three,} \qquad \qquad
 X^\bot_{\cal K}(\frak x) =   {\Upsilon}^\bot (\frak x)  + {\cal R}^\bot_N(\frak x),  \qquad
 \end{equation}
 where
 $$
 {\Upsilon}^\bot =  {\cal OB}^2_{w}(1, N) + {\cal OB}^2_{ww}(1, N) + {\cal OB}^3(1, N)\,, \qquad \quad
  {\cal R}^\bot_N = {\cal OS}^2_w(N) + {\cal OS}^2_{ww}(N) + {\cal OS}^3(N)\,.
$$
\end{proposition}
In the remaining part of this section we prove Proposition \ref{prop ordine e 3}. The transformation $\Phi$ is obtained as the composition 
$\Phi^{(1)} \circ \Phi^{(2)} \circ \Phi^{(3)}$ of three symplectic transformations $\Phi^{(j)}$, $1 \le j \le 3$. 

\medskip
\noindent
 {\bf Normalization of ${\cal P}_L$    up to $O(\e^2)$. }
The aim of this first step is to construct a symplectic transformation $\Phi^{(1)}$ so that 
${\cal P}_L(\frak x) \stackrel{\eqref{cal N cal H cal P}}{=} \e \big( {\cal P}_{00}(\theta) +  {\cal P}_{10}(\theta) \cdot y + \langle {\cal P}_{01}(\theta)\,,\,w \rangle \big),$ 
when expressed in the new coordinates, is in normal form up to order $\e^2$. We construct $\Phi^{(1)}$ as the time one flow of a Hamiltonian flow corresponding to a Hamiltonian of the form 
$$
\e {\cal F}^{(1)}(\frak x) =   \e{\cal F}_{00}^{(1)}(\theta) +  \e{\cal F}_{10}^{(1)}(\theta) \cdot y + \e \langle {\cal F}_{01}^{(1)}(\theta)\,,\,w \rangle
$$
where
\begin{equation}\label{properties cal F (1)}
\begin{aligned}
{\cal F}_{00}^{(1)} \in C^\infty \big(\T^{S_+}, \, \R \big), 
\qquad {\cal F}_{1 0}^{(1)} \in C^\infty\big(\T^{S_+}, \,  \R^{S_+} \big), 
\qquad {\cal F}_{01}^{(1)} \in C^\infty\big(\T^{S_+}, \, H^n_\bot(\T_1) \big), \ \  \forall \, n \geq 0 ,
\end{aligned}
\end{equation}
will be chosen to serve our needs. 
The Hamiltonian vector field corresponding to the Hamiltonian $\e {\cal F}^{(1)}(\frak x)$,
$$
X_{\e {\cal F}^{(1)}} (\frak x) = 
\Big(  - \e {\cal F}_{10}^{(1)}(\theta), \
  \e \big( \nabla_\theta {\cal F}_{00}^{(1)}(\theta) +  \nabla_\theta {\cal F}_{10}^{(1)}(\theta)\cdot y + \nabla_\theta \langle {\cal F}_{01}^{(1)}(\theta)\,,\,w \rangle  \big)\,, \
   \e \partial_x {\cal F}_{01}^{(1)}(\theta) \Big) \, ,
$$
is small of order one and by Lemma \ref{prop astratte campi vettoriali smoothing NF} arbitrarily smoothing. It means that $X_{\e {\cal F}^{(1)}} \in {\cal OS}^1(N)$ for any integer $N \geq 1$ 
(cf. Definition \ref{def smoothing vector fields}). 
Denote by  $\Phi^{(1)}(\tau, \cdot) \equiv \Phi_{\e {\cal F}^{(1)}}(\tau, \cdot)$ the flow of $X_{\e {\cal F}^{(1)}}$. 
For any given $N \in \N,$ there exists an integer $s_N >0$
with the property that for any $s \ge s_N$, there exist
$0 < \delta \equiv \delta(s, \gamma, N) < 1$ and $0 < \e_0 \equiv \e_0(s, \gamma, N) < 1$ (small), so that $\Phi^{(1)}(\tau, \cdot) \in C^\infty_b(\mathcal V^s(\delta) \times [0, \e_0],  \mathcal V^s(2\delta))$ for any $-1 \le \tau \le 1$. 
The inverse of the time one flow map $\Phi^{(1)} := \Phi^{(1)}(1, \cdot )$ is then given by $(\Phi^{(1)})^{- 1} = \Phi^{(1)}(- 1, \cdot)$ (cf. Remark \ref{inverse of flow})
and by Lemma \ref{flusso smoothing vector fields}, 
\begin{equation}\label{Phi (1) - Id}
\Phi^{(1)}(\tau, \cdot) (\frak x) - {\frak x} \in {\cal OS}^1(N)\, , \qquad \forall \, -1 \le \tau \le 1 \,. 
\end{equation} 
We now compute ${\cal H}^{(1)} := {\cal H} \circ \Phi^{(1)}$ by separately expanding the terms appearing in \eqref{cal N cal H cal P}. 
By \eqref{Lie expansion Hamiltonian} (Lie expansion), \eqref{Phi (1) - Id} (properties of $\Phi^{(1)}$) and \eqref{def poisson action-angle} (Poisson bracket) one has
$$
\begin{aligned}
& {\cal N}  \circ \Phi^{(1)}= {\cal N} + \e \{ {\cal N} ,\, {\cal F}^{(1)} \} + \e^2   \int_0^1 (1 - \tau)\{ \{{\cal N} \,, \, {\cal F}^{(1)} \}\,,\,{\cal F}^{(1)}\} \circ \Phi^{(1)}(\tau, \cdot )\, d \tau\,, \\ 
& \{ {\cal N} ,\, {\cal F}^{(1)} \} =   \omega \cdot \partial_\theta  {\cal F}_{00}^{(1)}(\theta) 
+ \big(\, \omega \cdot \partial_\theta {\cal F}_{10}^{(1)}(\theta) + \Omega_{S_+} [\nabla_\theta{\cal F}_{00}^{(1)}(\theta)] \big) \cdot  y  
+ \big\langle \big(\omega \cdot \partial_\theta + \ii \Omega_\bot\big){\cal F}_{0 1}^{(1)}(\theta),\, w\big\rangle    \\
& \qquad  \qquad \qquad +  ( \Omega_{S_+} [y] \cdot \partial_\theta)({\cal F}_{10}^{(1)}(\theta) \cdot y) +  \big\langle (\Omega_{S_+} [y] \cdot \partial_\theta)  {\cal F}_{0 1}^{(1)}(\theta), w  \big\rangle
\end{aligned}
$$
and by \eqref{Lie expansion Hamiltonian} (Lie expansion) and \eqref{proprieta ham primissima per forma normale} (properties of $ {\cal P}_e$)
$$
 \e {\cal P}_L \circ \Phi^{(1)} = \e {\cal P}_L + \e^2 \int_0^1 \{ {\cal P}_L, {\cal F}^{(1)} \} \circ \Phi^{(1)}(\tau, \cdot)\, d \tau\,, \qquad
 {\cal P}_e \circ \Phi^{(1)} \quad C^\infty-\text{smooth, small of order three}.
$$
Altogether, one obtains 
\begin{equation}\label{cal H reg 1}
\begin{aligned}
{\cal H}^{(1)} &=  {\cal N}+  \e \big( \, \omega \cdot \partial_\theta {\cal F}_{00}^{(1)}(\theta) + {\cal P}_{00}(\theta) \big) 
+ \e \big(\, \omega \cdot \partial_\theta {\cal F}_{10}^{(1)}(\theta) 
+ {\cal P}_{10}(\theta) + \Omega_{S_+}[\nabla_\theta{\cal F}_{00}^{(1)}(\theta)] \big) \cdot y  \\
& + \e \big\langle \big(\omega \cdot \partial_\theta + \ii \Omega_\bot\big){\cal F}_{0 1}^{(1)} + {\cal P}_{0 1}\,,\, w\big\rangle + {\cal P}^{(1)} \, , \\  
\end{aligned}
\end{equation}
\begin{equation}\label{def cal P1}
\begin{aligned}
{\cal P}^{(1)} & :=  \e^2  \int_0^1 (1 - \tau)\{ \{{\cal N} \,, \, {\cal F}^{(1)} \}\,,\,{\cal F}^{(1)}\} \circ \Phi^{(1)}(\tau, \cdot)\, d \tau  + \e^2 \int_0^1 \{ {\cal P}_L, {\cal F}^{(1)} \} \circ \Phi^{(1)}(\tau, \cdot)\, d \tau  \\
& \quad + \e ( \Omega_{S_+} [y] \cdot \partial_\theta)({\cal F}_{10}^{(1)}(\theta) \cdot y) + \e \big\langle (\Omega_{S_+}[y] \cdot \partial_\theta) {\cal F}_{0 1}^{(1)}(\theta), w  \big\rangle + {\cal P}_e \circ \Phi^{(1)}.
\end{aligned}
\end{equation}
Since the terms appearing in the second line of \eqref{def cal P1} are small of order three, the Hamiltonian ${\cal P}^{(1)}$ admits an expansion of the form
\begin{equation}\label{espansione hamiltoniana cal P (1)}
{\cal P}^{(1)} (\frak x) = \e^2 {\cal P}_{00}^{(1)}(\theta) + {\cal P}^{(1)}_e\, ,
\end{equation}
where  ${\cal P}_{00}^{(1)} \in C^\infty(\T^{S_+}, \R)$ and ${\cal P}^{(1)}_e$ is small of order three. In view of \eqref{cal H reg 1} and since
 $\Omega_{S_+} [\nabla_\theta{\cal F}^{(1)}_{00} ]$ has zero average in $\theta$, we consider the following system of homological equations 
 for ${\cal F}_{00}^{(1)}$, ${\cal F}_{1 0}^{(1)}$, ${\cal F}_{0 1}^{(1)}$,
\begin{equation}\label{homological equation 1}
\begin{cases}
\omega \cdot \partial_\theta {\cal F}_{00}^{(1)} + {\cal P}_{00} = \langle {\cal P}_{00}\rangle_\theta\,, \\
\omega \cdot \partial_\theta {\cal F}_{10}^{(1)} + {\cal P}_{10} + \Omega_{S_+}[ \nabla_\theta{\cal F}^{(1)}_{00}] = \langle {\cal P}_{10} \rangle_\theta\,, \\
\big(\omega \cdot \partial_\theta + \ii \Omega_\bot\big){\cal F}_{0 1}^{(1)} + {\cal P}_{0 1} = 0\, .
\end{cases}
\end{equation}
Since by assumption $\omega \in \Pi_\gamma$, $0 < \gamma < 1$, (cf. \eqref{non resonant set tot}), we can apply Lemmata \ref{eq omo zero melnikov abstract}, \ref{eq homo astratta prime melnikov},
to conclude that  the  system \eqref{homological equation 1} has a unique solution
${\cal F}_{00}^{(1)}, {\cal F}_{1 0}^{(1)}, {\cal F}_{0 1}^{(1)}$ satisfying \eqref{properties cal F (1)} and $\langle {\cal F}_{00}^{(1)}\rangle_\theta = 0$, $\langle {\cal F}_{1 0}^{(1)}\rangle_\theta = 0$. 
The Hamiltonian ${\cal H}^{(1)}$, defined in \eqref{cal H reg 1},  then reads
\begin{equation}\label{cal H reg 1 b}
 {\cal H}^{(1)} = {\cal N} + \e \widehat{\cal N}_1 + \e^2 {\cal P}_{00}^{(1)}(\theta) +  {\cal P}^{(1)}_e\,, \qquad
 \widehat{\cal N}_1(y) := \langle {\cal P}_{00}\rangle_\theta +  \langle {\cal P}_{10} \rangle_\theta \cdot y \,.
\end{equation}
Since ${\cal P}_e^{(1)}$ is small of order three, its Hamiltonian vector field $X_{{\cal P}_e^{(1)}}$ is small of order two. 
For later use we discuss the normal component $X_{{\cal P}^{(1)}_e}^\bot $ of the vector field $X_{{\cal P}^{(1)}_e}$. 
Since $X_{\e {\cal F}^{(1)}} \in {\cal OS}^1(N)$, 
and $ X^\bot_{ {\cal P}_e} = {\cal OB}^2(1, N) + {\cal OS}^2(N)$ (cf. \eqref{proprieta ham primissima per forma normale})
it follows from Lemma \ref{push forward smoothing 1} that $X^\bot_{\mathcal P_e \circ \Phi^{(1)} } = {\cal OB}^2(1, N) + {\cal OS}^2(N)$.
Arguing similarly for all the other terms in the definition of ${\cal P}_e^{(1)}$ (cf. \eqref{def cal P1}, \eqref{espansione hamiltoniana cal P (1)})
one can show that
\begin{equation}\label{prop X cal P (1) e}
X_{{\cal P}^{(1)}_e}^\bot = \partial_x \nabla_\bot {\cal P}_e^{(1)} = {\cal OB}^2(1, N) + {\cal OS}^2(N).
\end{equation}

\medskip
\noindent
{\bf Normalization of $\e^2 {\cal P}_{00}^{(1)}(\theta)$.}
The aim of this second step is to normalize the term $\e^2 {\cal P}_{00}^{(1)}(\theta)$ (small of order $2$) in \eqref{cal H reg 1 b}. 
To this end we construct a symplectic transformation $\Phi^{(2)}$, given again by the time one flow of a Hamiltonian flow, corresponding to a Hamiltonian
of the form $\e^2 {\cal F}^{(2)}(\theta)$ with
\begin{equation}\label{prop cal F (2) NF}
{\cal F}^{(2)} \in C^\infty(\T^{S_+}, \R)
\end{equation}
being a function to be determined. The Hamiltonian vector field corresponding to the Hamiltonian $\e^2 {\cal F}^{(2)}(\theta)$,
$$
X_{\e^2 {\cal F}^{(2)}}(\frak x) = \big( 0, \,  \e^2 \nabla_\theta {\cal F}^{(2)}(\theta), \, 0 \big) \,. 
$$
is small of order two and by Lemma \ref{prop astratte campi vettoriali smoothing NF} arbitrarily smoothing. It means that $X_{\e^2 {\cal F}^{(2)}} \in {\cal OS}^2(N)$ for any integer $N \geq 1$
 (cf. Definition \ref{def smoothing vector fields}). 
Denote by  $\Phi^{(2)}(\tau, \cdot) \equiv \Phi_{\e^2 {\cal F}^{(2)}}(\tau, \cdot)$ the flow of $X_{\e^2 {\cal F}^{(2)}}$. 
For any given $N \in \N,$ there exists an integer $s_N >0$ 
with the property that for any $s \ge s_N$, there exist
$0 < \delta \equiv \delta(s, \gamma, N) < 1$ and $0 < \e_0 \equiv \e_0(s, \gamma, N) < 1$ (small), so that
$\Phi^{(2)}(\tau, \cdot) \in C^\infty_b(\mathcal V^s(\delta) \times [0, \e_0],  \mathcal V^s(2\delta))$ for any $-1 \le \tau \le 1$. 
The inverse of the time one flow map $\Phi^{(2)} := \Phi^{(2)}(1, \cdot )$ is then given by $(\Phi^{(2)})^{- 1} = \Phi^{(2)}(- 1, \cdot)$ (cf. Remark \ref{inverse of flow})
and by Lemma \ref{flusso smoothing vector fields},
\begin{equation}\label{Phi (2) - Id}
\Phi^{(2)}(\tau, \cdot) (\frak x) - {\frak x} \in {\cal OS}^2(N)\, , \qquad \forall \, -1 \le \tau \le 1 \,. 
\end{equation} 
We now compute ${\cal H}^{(2)} := {\cal H}^{(1)} \circ \Phi^{(2)}$ by separately expanding the terms in \eqref{cal H reg 1 b}. 
By \eqref{Lie expansion Hamiltonian} (Lie expansion), \eqref{Phi (2) - Id} (properties of $\Phi^{(2)}$) and \eqref{def poisson action-angle} (Poisson bracket) one has
$$
\begin{aligned}
& {\cal N}  \circ \Phi^{(2)} = {\cal N} + \e^2 \{ {\cal N},  \, {\cal F}^{(2)} \} + \e^4   \int_0^1 (1 - \tau)\{ \{{\cal N} \,, \, {\cal F}^{(2)} \}\,,\,{\cal F}^{(2)}\} \circ \Phi^{(2)}(\tau, \cdot )\, d \tau   \\ 
& \qquad \qquad  = {\cal N} + \e^2   \omega \cdot \partial_\theta  {\cal F}^{(2)}(\theta) + \e^4   \int_0^1 (1 - \tau)\{ \{{\cal N} \,, \, {\cal F}^{(2)} \}\,,\,{\cal F}^{(2)}\} \circ \Phi^{(2)}(\tau, \cdot )\, d \tau \,, \\
& \e \widehat{\cal N}_1 \circ \Phi^{(2)} = \e \widehat{\cal N}_1 + \e^3 \int_0^1 \{ \widehat{\cal N}_1, {\cal F}^{(2)} \} \circ \Phi^{(2)}(\tau, \cdot)\, d \tau\,, \\
& \e^2 {\cal P}_{00}^{(1)} \circ \Phi^{(2)} = \e^2 {\cal P}_{00}^{(1)}(\theta) + \e^4 \int_0^1 \{ {\cal P}_{00}^{(1)}, {\cal F}^{(2)} \} \circ \Phi^{(2)}(\tau, \cdot)\, d \tau \\
& {\cal P}_e^{(1)} \circ \Phi^{(2)} \quad C^\infty-\text{smooth,   small of order three.}
\end{aligned}
$$
Altogether, one obtains 
\begin{equation}\label{cal H reg 2}
\begin{aligned}
{\cal H}^{(2)} &= {\cal H}^{(1)} \circ \Phi^{(2)}  =  {\cal N}+ \e \widehat{\cal N}_1 +  \e^2 \big( \omega \cdot \partial_\theta {\cal F}^{(2)}(\theta) + {\cal P}_{00}^{(1)}(\theta)  \big) + {\cal P}^{(2)}\,, \\
{\cal P}^{(2)} & := \e^4   \int_0^1 (1 - \tau)\{ \{{\cal N} \,, \, {\cal F}^{(2)} \}\,,\,{\cal F}^{(2)}\} \circ \Phi^{(2)}(\tau, \cdot )\, d \tau + \e^3 \int_0^1 \{ \widehat{\cal N}_1, {\cal F}^{(2)} \} \circ \Phi^{(2)}(\tau, \cdot)\, d \tau \\
& \quad + \e^4 \int_0^1 \{ {\cal P}_{00}^{(1)}, {\cal F}^{(2)} \} \circ \Phi^{(2)}(\tau, \cdot)\, d \tau + {\cal P}_e^{(1)} \circ \Phi^{(2)}\,. 
\end{aligned}
\end{equation}
Since ${\cal P}_e^{(1)}$ is $C^\infty$-smooth and small of order three, so is ${\cal P}^{(2)}$. 
In view of the formula for ${\cal H}^{(2)}$ in \eqref{cal H reg 2} we consider the following homological equation for ${\cal F}^{(2)}$,
\begin{equation}\label{homological equation 2} 
\omega \cdot \partial_\theta {\cal F}^{(2)}(\theta)+ {\cal P}_{00}^{(1)}(\theta)= \langle {\cal P}_{00}^{(1)}\rangle_\theta\,.
\end{equation}
Since by assumption $\omega \in \Pi_\gamma$, $0 < \gamma < 1$, (cf. \eqref{non resonant set tot}), we can apply Lemmata \ref{eq omo zero melnikov abstract}, \ref{eq homo astratta prime melnikov},
to conclude that \eqref{homological equation 2} has a unique solution ${\cal F}^{(2)} \in C^\infty(\T^{S_+}, \R)$ with $\langle {\cal F}^{(2)} \rangle_\theta = 0$.
The Hamiltonian ${\cal H}^{(2)}$ in \eqref{cal H reg 2}  then reads
\begin{equation}\label{forma finale hamiltoniana}
 {\cal H}^{(2)} = {\cal N} + \e \widehat{\cal N}_2 + {\cal P}^{(2)}\,,  \qquad
 \widehat{\cal N}_2 := \widehat{\cal N}_1 + \e \langle {\cal P}_{00}^{(1)} \rangle_\theta \stackrel{\eqref{cal H reg 1 b}}{=} \langle {\cal P}_{00}\rangle_\theta 
 +  \langle {\cal P}_{10} \rangle_\theta \cdot y + \e \langle {\cal P}_{00}^{(1)} \rangle_\theta \,.
\end{equation}
Since  ${\cal P}^{(2)}$ is small of order three, its Hamiltonian vector field $X_{{\cal P}^{(2)}}$ is small of order two. 
For later use, we again discuss the normal component $X_{{\cal P}^{(2)}}^\bot $ of the vector field $X_{{\cal P}^{(2)}}$. 
Since $X_{\e^2 {\cal F}^{(2)}} \in {\cal OS}^2(N)$, 
and $ X^\bot_{ {\cal P}^{(1)}_e} = {\cal OB}^2(1, N) + {\cal OS}^2(N)$ (cf. \eqref{prop X cal P (1) e})
it follows from Lemma \ref{push forward smoothing 1} that $X^\bot_{\mathcal P^{(1)}_e \circ \Phi^{(2)} } = {\cal OB}^2(1, N) + {\cal OS}^2(N)$.
Arguing similarly for all the other terms in ${\cal P}^{(2)}$ (cf. \eqref{cal H reg 2}
\eqref{def cal P1}, \eqref{espansione hamiltoniana cal P (1)}, \eqref{cal H reg 1 b})
one shows that
\begin{equation}\label{prop X cal P (2)}
X_{{\cal P}^{(2)}}^\bot = \partial_x \nabla_\bot {\cal P}^{(2)} = {\cal OB}^2(1, N) + {\cal OS}^2(N).
\end{equation}


\medskip
\noindent
{\bf Normalization of terms affine in $w$.}
The aim of this third step is to construct a symplectic coordinate transformation $\Phi^{(3)}$, normalizing the terms in the Taylor expansion of $ {\cal P}^{(2)}$ (cf. \eqref{forma finale hamiltoniana})
with respect to $y$, $w$ at $(y, w) = (0, 0)$, which are homogeneous  in $y, w, \e$ of order three, of degree at most one in $w$,  and of degree at most two in $\e$.
 The Taylor expansion of ${\cal P}^{(2)}$ in $y$, $w$, $\e$ up to order four reads
$$
\begin{aligned}
 {\cal P}^{(2)}(\frak x) = &\e^3 {\cal P}_{00}^{(2)}(\theta) + \e^2 \big( {\cal P}_{10}^{(2)}(\theta) \cdot y + \langle {\cal P}_{0 1}^{(2)}(\theta), w \rangle \big)  
 + \e \big( {\cal P}_{2 0}^{(2)}(\theta)[y, y]  +  \langle {\cal P}_{11}^{(2)}(\theta)[y], w \rangle\big)  + \langle {\cal P}_{0 2}^{(2)}(\theta, y)[w], w \rangle \\
&  + {\cal P}_{3 0}^{(2)}(\theta)[y,y,y] + \langle {\cal P}_{2 1}^{(2)}(\theta)[y, y], w \rangle  + {\cal P}_{0 3}^{(2)}(\theta)[w,w,w] 
 + {\cal O}_4(\frak x),
\end{aligned}
$$
where for any $n \ge 0$,
\begin{equation}\label{prop espansione cal P (2)}
\begin{aligned}
& {\cal P}_{00}^{(2)}  \in C^\infty(\T^{S_+}, \, \R) \,,
\qquad \qquad  {\cal P}_{10}^{(2)} \in C^\infty(\T^{S_+}, \, \R^{S_+}), \qquad \qquad \qquad \ \  {\cal P}_{0 1}^{(2)} \in C^\infty(\T^{S_+}, \, H^n_\bot(\T_1)), \\
&{\cal P}_{2 0}^{(2)} \in C^\infty(\T^{S_+}, \, {\cal B}_2(\R^{S_+})), \quad \ \ {\cal P}_{11}^{(2)} \in C^\infty\big(\T^{S_+}, \, {\cal B}(\R^{S_+}, H^n_\bot(\T_1)) \big), \\
&  {\cal P}_{3 0}^{(2)} \in C^\infty(\T^{S_+}, \, {\cal B}_3(\R^{S_+}))\, , 
 \quad \, {\cal P}_{2 1}^{(2)} \in C^\infty(\T^{S_+}, \, {\cal B}_2(\R^{S_+}, H^n_\bot(\T_1))), \quad  {\cal P}_{0 3}^{(2)} \in C^\infty(\T^{S_+}, \, {\cal B}_3(H^n_\bot(\T_1)), \\
 & {\cal P}_{0 2}^{(2)} \in C^\infty\big( \T^{S_+} \times \R^{S_+} \times \R, \, {\cal B}(H^n_\bot(\T_1)) \big),  \qquad {\cal O}_4(\frak x) \ \  C^\infty{\text{-smooth,  small of order four.}}\\
\end{aligned}
\end{equation}
\begin{remark}\label{mathcal P^(2) (0 2)}
 In the above Taylor expansion of  ${\cal P}^{(2)}$, we combined the terms
which are of the order $(0 \, 2)$ and $(1 \,2)$ in the variables $y$, $w$ and for notational convenience, denoted the combined term by $\langle {\cal P}_{0 2}^{(2)}(\theta, y)[w], w \rangle$.
The map ${\cal P}_{0 2}^{(2)}: (\theta, y, \e) \mapsto {\cal P}_{0 2}^{(2)}(\theta, y) \equiv {\cal P}_{0 2}^{(2)}(\theta, y, \e)$ is linear in $y, \e$.
\end{remark}
We split ${\cal P}^{(2)}$ as ${\cal P}^{(2)} = {\cal P}^{(2)}_1 + {\cal P}^{(2)}_2 + {\cal O}_4$ where 
\begin{equation}\label{splitting cal P 2 normalizzazione}
\begin{aligned}
{\cal P}^{(2)}_1 & := \e^2 {\cal P}^{(2)}_{10}(\theta) \cdot y +  \e {\cal P}^{(2)}_{2 0}(\theta)[y, y] + {\cal P}^{(2)}_{3 0}(\theta)[y,y,y]  \\
& \quad + \e^2 \big\langle {\cal P}^{(2)}_{0 1}(\theta), w \big\rangle + \e \big\langle {\cal P}^{(2)}_{11}(\theta)[y], w \big\rangle + \big\langle {\cal P}^{(2)}_{2 1}(\theta)[y, y], w \big\rangle \\  
{\cal P}^{(2)}_2 & :=  \e^3 {\cal P}^{(2)}_{00}(\theta) +  \langle {\cal P}^{(2)}_{0 2}(\theta, y)[w],  w \rangle  +  {\cal P}^{(2)}_{0 3}(\theta)[w,w,w] \, .
\end{aligned}
\end{equation}
Note that ${\cal P}^{(2)}_1$ is affine in $w$ and that
the Hamiltonian vector field corresponding to the term $ \e^3 {\cal P}^{(2)}_{00}(\theta)$ is small of order three.
The transformation $\Phi^{(3)}$ is then defined as the time one flow of the Hamiltonian vector field $X_{\mathcal F^{(3)}}$ with a Hamiltonian $\mathcal F^{(3)}$ of the form
\begin{equation}\label{def cal F (3)}
\begin{aligned}
& {\cal F}^{(3)}(\frak x) :=  \e^2 {\cal F}^{(3)}_{10}(\theta) \cdot y +  \e  {\cal F}^{(3)}_{2 0}(\theta)[y, y] + {\cal F}^{(3)}_{3 0}(\theta)[y,y,y] \\
& \qquad + \e^2 \big\langle {\cal F}^{(3)}_{0 1}(\theta), w \big\rangle + \e \big\langle {\cal F}^{(3)}_{11}(\theta)[y], w \big\rangle + \big\langle {\cal F}^{(3)}_{2 1}(\theta)[y, y], w \big\rangle
\end{aligned}
\end{equation}
satisfying for any $n \ge 0$,
\begin{equation}\label{proprieta cal F (3)}
\begin{aligned}
& {\cal F}^{(3)}_{10} \in C^\infty(\T^{S_+}, \R^{S_+}), \qquad \ {\cal F}^{(3)}_{2 0} \in C^\infty(\T^{S_+}, {\cal B}_2(\R^{S_+})), 
\qquad  \qquad \ {\cal F}^{(3)}_{3 0} \in C^\infty(\T^{S_+}, {\cal B}_3(\R^{S_+})),  \\
&{\cal F}^{(3)}_{0 1} \in C^\infty(\T^{S_+}, H^n_\bot(\T_1)), \ \  {\cal F}^{(3)}_{11} \in C^\infty(\T^{S_+}, {\cal B}(\R^{S_+}, H^n_\bot(\T_1))), 
\quad  {\cal F}^{(3)}_{2 1} \in C^\infty(\T^{S_+}, {\cal B}_2(\R^{S_+}, H^n_\bot(\T_1))).
\end{aligned}
\end{equation}
The functions $\mathcal F^{(3)}_{i j}$ will be chosen according to our needs.
 By \eqref{def cal F (3)}, \eqref{proprieta cal F (3)}, the Hamiltonian vector field $X_{{\cal F}^{(3)}}$
is small of order two and by Lemma \ref{prop astratte campi vettoriali smoothing NF} arbitrarily smoothing. It means that $X_{{\cal F}^{(3)}} \in {\cal OS}^2(N)$ for any integer $N \geq 1$ 
(cf. Definition \ref{def smoothing vector fields}). 
Denote by  $\Phi^{(3)}(\tau, \cdot) \equiv \Phi_{ {\cal F}^{(3)}}(\tau, \cdot)$ the flow of $X_{ {\cal F}^{(3)}}$. 
For any given $N \in \N,$ there exists an integer $s_N >0$
 with the property that for any $s \ge s_N$, there exist
$0 < \delta \equiv \delta(s, \gamma, N) < 1$ and $0 < \e_0 \equiv \e_0(s, \gamma, N) < 1$ (small), so that
$\Phi^{(3)}(\tau, \cdot) \in C^\infty_b(\mathcal V^s(\delta) \times [0, \e_0],  \mathcal V^s(2\delta))$ for any $-1 \le \tau \le 1$. 
The inverse of the time one flow map $\Phi^{(3)} := \Phi^{(3)}(1, \cdot )$ is then given by $(\Phi^{(3)})^{- 1} = \Phi^{(3)}(- 1, \cdot)$ and by Lemma \ref{flusso smoothing vector fields},
\begin{equation}\label{Phi (3) - Id}
\Phi^{(3)}(\tau, \cdot) (\frak x) - {\frak x} \in {\cal OS}^2(N)\, , \qquad \forall \, -1 \le \tau \le 1 \,. 
\end{equation} 
We now compute ${\cal H}^{(3)} := {\cal H}^{(2)} \circ \Phi^{(3)}$ by expanding separately the terms in \eqref{forma finale hamiltoniana}. 
By \eqref{Lie expansion Hamiltonian} (Lie expansion), \eqref{Phi (3) - Id} (properties of $\Phi^{(3)}$),  
\eqref{splitting cal P 2 normalizzazione} (splitting of $\mathcal P^{(2)}$),  \eqref{def cal F (3)} - \eqref{proprieta cal F (3)} (properties of $\mathcal F^{(3)}$), 
and \eqref{def poisson action-angle} (Poisson bracket) 
$$
{\cal N} \circ \Phi^{(3)}  = {\cal N} + \{ {\cal N}\,,\, {\cal F}^{(3)} \} +  \int_0^1 ( 1- \tau) \{ \{{\cal N}, {\cal F}^{(3)} \}, {\cal F}^{(3)} \} \circ \Phi^{(3)}(\tau, \cdot)\, d \tau
$$ 
can be expanded as
\begin{equation}
\begin{aligned}
{\cal N} \circ \Phi^{(3)} &
 = {\cal N} + \e^2 (\omega \cdot \partial_\theta) {\cal F}_{1 0}^{(3)}(\theta) \cdot y + \e ( \omega \cdot \partial_\theta) {\cal F}_{2 0}^{(3)}(\theta)[y, y] 
 + (\omega \cdot \partial_\theta) {\cal F}_{3 0}^{(3)}(\theta)[y,y,y] \\
& + \e^2 \big\langle (\omega \cdot \partial_\theta + \ii \Omega_\bot) {\cal F}_{0 1}^{(3)}(\theta),w\big\rangle 
+ \e \big\langle (\omega \cdot \partial_\theta + \ii \Omega_\bot) {\cal F}_{1 1}^{(3)}(\theta)[y],w\big\rangle  
+ \big\langle (\omega \cdot \partial_\theta  + \ii \Omega_\bot) {\cal F}_{2 1}^{(3)}(\theta)[y, y],w\big\rangle\\
& +(\Omega_{S_+} [y] \cdot \partial_\theta) {\cal F}^{(3)}  
 +  \int_0^1 ( 1- \tau) \{ \{{\cal N}, {\cal F}^{(3)} \}, {\cal F}^{(3)} \} \circ \Phi^{(3)}(\tau, \cdot)\, d \tau \,, 
 \end{aligned}
 \end{equation}
$$
 \widehat{\cal N}_2 \circ \Phi^{(3)}   = \widehat{\cal N}_2 + \int_0^1 \{ \widehat{\cal N}_2\,,\, {\cal F}^{(3)} \} \circ \Phi^{(3)}(\tau, \cdot)\, d \tau, \quad 
{\cal P}^{(2)} \circ \Phi^{(3)}  = {\cal P}_1^{(2)} + {\cal P}_2^{(2)} + \int_0^1 \{ {\cal P}^{(2)}, {\cal F}^{(3)} \} \circ \Phi^{(3)}(\tau, \cdot)\, d \tau \,. 
$$
Since ${\cal P}^{(2)}$ (cf. \eqref{forma finale hamiltoniana}),
${\cal F}^{(3)}$ (cf. \eqref{def cal F (3)}) are small of order three and in view of the definition of $\mathcal N$, $\widehat{\cal N}_2$ (cf. \eqref{forma finale hamiltoniana}), 
$\{ \{ {\cal N}, \, {\cal F}^{(3)} \}$, ${\cal F}^{(3)} \}$, $\e \{ \widehat{\cal N}_2, {\cal F}^{(3)} \}$, 
and $\{{\cal P}^{(2)}, {\cal F}^{(3)} \}$ are small of order four. Hence the Hamiltonian ${\cal H}^{(3)}$ takes the form 
\begin{equation}\label{hamiltonian cal H3}
\begin{aligned}
{\cal H}^{(3)} & = {\cal N} + \e \widehat{\cal N}_2
+ \e^2\Big( \omega \cdot \partial_\theta {\cal F}_{1 0}^{(3)}(\theta) + {\cal P}_{1 0}^{(2)}(\theta) \Big)\cdot y 
+ \e  \Big(\omega \cdot \partial_\theta {\cal F}_{2 0}^{(3)}(\theta) + {\cal P}_{2 0}^{(2)}(\theta)\Big)[y, y] \\
& + \Big(\omega \cdot \partial_\theta {\cal F}_{3 0}^{(3)}(\theta) + {\cal P}_{3 0}^{(2)}(\theta) \Big)[y,y,y] 
 + \e^2 \big\langle (\omega \cdot \partial_\theta + \ii \Omega_\bot) {\cal F}_{0 1}^{(3)}(\theta)  + {\cal P}_{0 1}^{(2)}(\theta), \, w\big\rangle \\
 & + \e \big\langle (\omega \cdot \partial_\theta + \ii \Omega_\bot) {\cal F}_{1 1}^{(3)}(\theta)[y]  + {\cal P}_{11}^{(2)}(\theta)[y],w\big\rangle  
  + \big\langle (\omega \cdot \partial_\theta + \ii \Omega_\bot) {\cal F}_{2 1}^{(3)}(\theta)[y, y] + {\cal P}_{21}^{(2)}(\theta)[y, y],w\big\rangle  \\
  &+ {\cal P}_2^{(2)} + {\cal O}_4  
\end{aligned}
\end{equation}
where ${\cal O}_4$ comprises all the terms which are small of order four.
In view of \eqref{hamiltonian cal H3},  we consider the following system of homological equations for $\mathcal F^{(3)}_{i j}$,
\begin{equation}\label{eq omologiche ordine 3}
\begin{aligned}
& \omega \cdot \partial_\theta {\cal F}_{j 0}^{(3)}(\theta) + {\cal P}_{j 0}^{(2)}(\theta) = \big\langle {\cal P}_{j 0}^{(2)} \big\rangle_\theta, \quad 1 \le j \le 3,  \\
&  (\omega \cdot \partial_\theta + \ii \Omega_\bot) {\cal F}^{(3)}_{0 1}(\theta) + {\cal P}^{(2)}_{0 1}(\theta) = 0\,,  \qquad
 (\omega \cdot \partial_\theta + \ii \Omega_\bot) {\cal F}^{(3)}_{1 1}(\theta) + {\cal P}^{(2)}_{11}(\theta) = 0\,, \\
& (\omega \cdot \partial_\theta + \ii \Omega_\bot) {\cal F}^{(3)}_{2 1}(\theta) + {\cal P}^{(2)}_{21}(\theta) = 0.
 \end{aligned}
\end{equation}
Since by assumption $\omega \in \Pi_\gamma$, $0 < \gamma < 1$ (cf. \eqref{non resonant set tot}), we can apply Lemmata \ref{eq omo zero melnikov abstract}, \ref{eq homo astratta prime melnikov},
to conclude that  the  system \eqref{eq omologiche ordine 3} has a unique solution
${\cal F}_{ij}^{(3)}$, satisfying the properties \eqref{proprieta cal F (3)}.
The Hamiltonian ${\cal H}^{(3)}$ in \eqref{hamiltonian cal H3} then reads
\begin{equation}\label{forma finale cal H3}
\begin{aligned}
& {\cal H}^{(3)}  = {\cal N}^{(3)}+ {\cal K}, \qquad
{\cal N}^{(3)}  :=  \omega \cdot y + \e \widehat \omega \cdot y + \frac12 \big\langle D^{- 1} \Omega_\bot w\,,\, w \big\rangle + Q(y) \,, 
\qquad {\cal K}  := {\cal P}_2^{(2)} + {\cal O}_4 \,.  \\
& \widehat \omega  := \langle {\cal P}_{10} \rangle_\theta  + \e \langle {\cal P}_{1 0}^{(2)} \rangle_\theta  \,, \qquad
Q(y) := \frac12 \Omega_{S_+} y \cdot y + \e \langle {\cal P}_{2 0}^{(2)} \rangle_\theta [y, y] + \langle {\cal P}_{3 0}^{(2)} \rangle_\theta[y, y, y]  \, .
\end{aligned}
\end{equation}
Here we dropped the irrelevant constant term $\e \langle {\cal P}_{00} \rangle_\theta + \e^2 \langle {\cal P}_{00}^{(1)} \rangle_\theta$ from the Hamiltonan ${\cal H}^{(3)}$ 
(cf. Remark \ref{irrelevant constant}). By \eqref{splitting cal P 2 normalizzazione}, \eqref{forma finale cal H3}, the components of the Hamiltonian vector field 
$X_{{\cal H}^{(3)}} = (X_{{\cal H}^{(3)}}^{(\theta)}, X_{{\cal H}^{(3)}}^{(y)}, X_{{\cal H}^{(3)}}^\bot)$ read
\begin{equation}\label{ham vec field cal H3}
\begin{aligned}
  X_{\mathcal H^{(3)}}^{(\theta)}(\frak x) = - \omega - \e \widehat \omega - & \nabla_y Q(y) - \nabla_y {\cal P}_2^{(2)}(\frak x) - \nabla_y {\cal O}_4(\frak x)  \,,  \qquad
  X_{{\cal H}^{(3)}}^{(y)}(\frak x) =  \nabla_\theta {\cal P}_2^{(2)}(\frak x)  +  \nabla_\theta {\cal O}_4(\frak x)\,, \\
& X_{{\cal H}^{(3)}}^\bot (\frak x) =  \ii \Omega_\bot w + \partial_x \nabla_\bot {\cal P}_2^{(2)}(\frak x) + \partial_x \nabla_\bot {\cal O}_4(\frak x) \, . 
\end{aligned}
\end{equation}
Since ${\cal P}_2^{(2)}$ is a $C^\infty-$smooth and small of order three and ${\cal O}_4$ is small of order four, $\nabla_\theta {\cal P}_2^{(2)}$ 
is small of order three and $\nabla_\theta {\cal O}_4$ is small of order four, implying that
\begin{equation}\label{X cal H3 (y)}
X_{{\cal H}^{(3)}}^{(y)} \in C^\infty_b \big( {\cal V}^{\sigma_N}(\delta) \times  [0, \e_0], \, \R^{S_+} \big) \quad \text{small of order three}
\end{equation} 
for some $\sigma_N > 0$. 
Towards $X_{{\cal H}^{(3)}}^{(\theta)}$, note that $\nabla_y {\cal O}_4$ is small of order three and that
$\nabla_y {\cal P}_2^{(2)}$ (cf. \eqref{splitting cal P 2 normalizzazione}) is small of order two and  has the additional property of being at least quadratic with respect to $w$. Therefore 
\begin{equation}\label{animal 0}
 \nabla_y {\cal P}_2^{(2)}(\frak x)+ \nabla_y {\cal O}_4(\frak x)  = \Upsilon_2^{(\theta)}(\theta)[w, w] + \Upsilon_3^{(\theta)}(\frak x)\, ,
\end{equation} 
where
$$
 \Upsilon_2^{(\theta)} \in C^\infty \big( \T^{S_+}, \, {\cal B}_2(H^{\sigma_N}_\bot(\T_1), \, \R^{S_+})\big),  \qquad
  \Upsilon_3^{(\theta)} \in C^\infty\big({\cal V}^{\sigma_N}(\delta) \times  [0, \e_0], \, \R^{S_+} \big) \ \  \text{small of order three}
$$
for some $\sigma_N > 0$. 
For later use, we  discuss the normal component $X_{\cal K}^\bot $ of the vector field $X_{\cal K}$.
Since by \eqref{splitting cal P 2 normalizzazione}, 
${\cal P}^{(2)}_2 =  \e^3 {\cal P}^{(2)}_{00}(\theta) +  \langle {\cal P}^{(2)}_{0 2}(\theta, y)[w],  w \rangle +  {\cal P}^{(2)}_{0 3}(\theta)[w,w,w]$ (cf. Remark \ref{mathcal P^(2) (0 2)})
one infers that
\begin{equation}\label{formula X mathbb K bot}
X_{\mathcal K}^\bot (\frak x) = \partial_x \nabla_\bot {\cal P}_2^{(2)}(\frak x) + \partial_x \nabla_\bot {\cal O}_4(\frak x) 
= 2 \partial_x  {\cal P}^{(2)}_{0 2}(\theta, y)[w] + \Upsilon^\bot_2 (\theta)[w, w] + \Upsilon^\bot_3(\frak x)
\end{equation}
where $\Upsilon^\bot_3(\frak x)$ is small of order three.
Since $X_{{\cal F}^{(3)}} \in {\cal OS}^2(N)$ and $\partial_x \nabla_\bot {\cal P}^{(2)} = {\cal OB}^2(1, N) + {\cal OS}^2(N)$ (cf. \ref{forma finale hamiltoniana}, \ref{prop X cal P (2)})
and in view of the definition of $\mathcal O_4$ (cf. \eqref{hamiltonian cal H3})
it then follows from Lemma \ref{push forward smoothing 1} that
\begin{equation}\label{prop X cal K}
\begin{aligned}
& \qquad  \qquad \qquad \qquad  \partial_x  {\cal P}^{(2)}_{0 2}(\theta, y)[w] = {\cal OB}^2_w(1, N) + {\cal OS}^2_w(1, N)\,, \\
&  \Upsilon^\bot_2 (\theta)[w, w]  = {\cal OB}^2_{ww}(1, N) + {\cal OS}^2_{ww}(N)\,, \qquad
 \Upsilon^\bot_3(\frak x) = {\cal OB}^3(1, N) + {\cal OS}^3( N)\, .
\end{aligned}
\end{equation}

\bigskip
\noindent
{\em{Proof of Proposition \ref{prop ordine e 3}.}} 
We define $\Phi := \Phi^{(1)} \circ \Phi^{(2)} \circ \Phi^{(3)}$ where $\Phi^{(1)}, \Phi^{(2)}, \Phi^{(3)}$ are the symplectic coordinate transformations,
given in the paragraphs above. 
Using the properties \eqref{Phi (1) - Id}, \eqref{Phi (2) - Id}, \eqref{Phi (3) - Id} of $\Phi^{(1)}$, $\Phi^{(2)}$, and $\Phi^{(3)}$,  respectively
one shows that  there exists an integer $s_N > 0$ with the property that for any $s \ge s_N$ 
there exist $0 < \delta \equiv \delta(s, \gamma, N) < 1$ and $0 < \e_0 \equiv \e_0(s, \gamma, N) < 1$ so that  \eqref{prop Phi finale passi smoothing} holds, 
$$
 \Phi^{\pm 1}\in {\cal C}^\infty_b({\cal V}^s(\delta) \times [0, \e_0], \, {\cal V}^s(2 \delta)) \, , \qquad
 \Phi^{\pm 1}(\frak x) - {\frak x} \quad \text{small of order one}\, .
$$
Since ${\cal K} = {\cal P}_2^{(2)} + {\cal O}_4$,
the remaining statements of Proposition \ref{prop ordine e 3} then follow by \eqref{X cal H3 (y)} - \eqref{prop X cal K}.
\hfill $\square$

\section{Normalization steps by para-differential calculus}\label{normalization II}

The goal of this section is to normalize terms in the vector field $X_\mathcal K$, which are linear or quadratic in the variable $w$,
where $X_\mathcal K$ denotes the Hamiltonian vector field of the Hamiltonian $\mathcal K$ of Proposition \ref{prop ordine e 3}.
This is achieved in three steps, described in the following three subsections, by using para-differential calculus.

\subsection{Normalization of terms linear or quadratic in $w$}\label{sec regolarizzazione w w2}
%
The aim of this subsection is to reduce to constant coefficients the terms in the normal component $ X^\bot \equiv X_{{\cal H}_3}^\bot$ of the vector field  $ X \equiv X_{{\cal H}_3}$, which are linear and quadratic in $w$.
Recall that such a reduction is needed since  $\Pi_\gamma^{(3)}$ (cf. \eqref{condizioni forma normale}) allows for a loss of derivatives in space.

By Proposition \ref{prop ordine e 3}, $X^\bot$ is of the form
$$
 X^\bot(\frak x) = X_{{\cal H}_3}^\bot(\frak x) \stackrel{\eqref{vector field cal K}}{=} \ii \Omega_\bot w + X_{\cal K}^\bot(\frak x).
 $$
 Since $\Omega_\bot $ is a diagonal Fourier multiplier with constant real coefficients (cf. \eqref{def notation normal frequencies}, \eqref{definition Omega bot}),
it remains to normalize $X_{\cal K}^\bot(\frak x )$ in the above sense.

By Proposition \ref{prop ordine e 3}, $X_{\cal K}^\bot(\frak x )$ admits an expansion of the form
\begin{equation}\label{expansion of X bot}
\begin{aligned}
& X_{\cal K}^\bot (\frak x) =   X^\bot_1(\theta, y)[w] + X^\bot_2(\theta)[w, w] + {\cal OB}^3(1, N) + {\cal OS}^3(1, N)\,, \\
& X^\bot_1(\theta, y)[w] := \Upsilon^\bot_1(\theta, y)[w]  + {\cal R}^\bot_{N,1}(\theta, y)[w], \qquad  
X^\bot_2(\theta)[w, w] := \Upsilon^\bot_2(\theta, w)[w]  + {\cal R}^\bot_{N, 2}(\theta)[w, w]\,,  
\end{aligned}
\end{equation}
where
\begin{equation}\label{expansion for X_1 + X_2}
\begin{aligned}
&   \Upsilon^\bot_1(\theta, y)[w] = \Pi_\bot \sum_{k = 0}^{N+1} T_{a_{1 - k}(\theta, y)} \partial_x^{1 - k} w \in {\cal OB}^2_w(1, N), 
\qquad {\cal R}^\bot_{N, 1}(\theta, y)[w] \in {\cal OS}^2_w(N)\,, \\
&  \Upsilon^\bot_2(\theta, w)[w] =   \Pi_\bot \sum_{k = 0}^{N+1} T_{A_{1 - k}(\theta)[w]} \partial_x^{1 - k} w \in {\cal OB}^2_{ww}(1, N), 
\quad {\cal R}^\bot_{N, 2}(\theta)[w, w] \in {\cal OS}^2_{ww}(N)\,. 
\end{aligned}
\end{equation}
By Definition \ref{paradiff omogenei di ordine 2}, for any given $N \in \N$, there are integers $s_N $, $\sigma_N > 0$ (large) 
with the property that for any $s \geq s_N$ 
there exist $0 < \delta = \delta(s, \gamma, N) < 1$ and $0 < \e_0 \equiv \e_0(s, \gamma, N) < 1$ so that  for any $0 \le k \le N+1$
\begin{equation}\label{exp a 1 - k for reg}
\begin{aligned}
& a_{1 - k} \in C^\infty_b \big( \T^{S_+} \times B_{S_+}(\delta) \times [0, \e_0] , \, H^s(\T_1) \big) \ \  \text{small of order one},  \\
& A_{1 - k} \in C^\infty\big(\T^{S_+} \times [0, \e_0],  \, {\cal B}(H^{s + \sigma_N}_\bot(\T_1), H^s(\T_1)) \big)\,. 
\end{aligned}
\end{equation}
Note that $X^\bot_1(\theta, y)[w]$ is a vector field small of order $2$ and linear in $w$, whereas  $X^\bot_2(\theta)[w, w]$ is small of order $2$, but quadratic in $w$. 
Since the vector field $X^\bot_{{\cal K}}$ is Hamiltonian,  every term in the expansion \eqref{expansion of X bot}, which is homogeneous in the coordinates $y, w$, is a Hamiltonian vector field as well. 
In particular, $ X^\bot_1(\theta, y)[w]$ is such a vector field.

\medskip

\noindent
{\bf Preliminary analysis of the vector field $X^\bot_1(\theta, y)[w]$. } 
Since $ X^\bot_1(\theta, y)[w]$ is a Hamiltonian vector field which is linear in $w$, 
 \eqref{diag purely imaginary} in Appendix \ref{Appendix A} implies that the diagonal operator 
\begin{equation}\label{parte diagonal X1}
{\rm diag}_{j \in S^\bot} [X^\bot_1(\theta, y)]_j^j
\end{equation}
is skew-adjoint,
\begin{equation}\label{eredita hamiltoniana}
[X^\bot_1(\theta, y)]_j^j = - \overline{[X^\bot_1(\theta, y)]_j^j}, \qquad j \in S^\bot\,.
\end{equation} 
We will show that the normal form transformations, constructed in this and the following subsection, preserve this property of  $X^\bot_1(\theta, y)[w]$. 
Since this is the only property of  the transformed vector field $X^\bot_1(\theta, y)[w]$ which is needed in the energy estimates in Section \ref{conclusioni forma normale}
we can allow for normal form transformations, which are not necessarily symplectic, as long as they preserve \eqref{eredita hamiltoniana}. 

\smallskip

Our aim is to construct iteratively a coordinate transformation on ${\cal V}^s(\delta)$ so that when expressed in the new coordinates, the vector field $X^\bot_1(\theta, y)[w] + X^\bot_2(\theta)[w, w]$ 
is again of the form \eqref{expansion for X_1 + X_2} 
and that the coefficients $a_{1 - k}(\theta, y) + A_{1 - k}(\theta)[w]$, $0 \le k \le N+1$, are independent of $x$. 
At the $(n+1)$th step, $n \geq 0$, we deal with a vector field $X_n = (X_n^{(\theta)}, X_n^{(y)}, X_n^\bot)$, 
defined as the pull back of $X$ by the composition of the transformations up to the nth step,  of the form
\begin{equation}\label{Xn bot}
\begin{aligned}
X_n^{(\theta)}(\frak x) & = - \omega -  \e \widehat \omega - \nabla_y Q(y) - \Upsilon_2^{(\theta)}(\theta)[w, w] + {\cal O}_3^{(\theta)}(\frak x) \,, \qquad X_n^{(y)}(\frak x)  = {\cal O}_3^{(y)}(\frak x)\,, \\
X_n^\bot (\frak x) & =\ii  \Omega_\bot w + {\cal D}^\bot_{n, 1}(\theta, y)[w] + {\cal D}^\bot_{n, 2}(\theta, w)[w] + X^\bot_{n, 1}(\theta, y)[w] + X^\bot_{n, 2}(\theta, w)[w]  \\
& \qquad + {\cal R}^{\bot}_{N, 1}(\theta, y)[w] +  {\cal R}^{\bot}_{N, 2}(\theta, w)[w] + {\cal OB}^3(1, N) + {\cal OS}^3(N)
\end{aligned}
\end{equation}
where for notational convenience, we write $ {\cal R}^{\bot}_{N, j} \equiv  {\cal R}^{\bot}_{n, N, j}$ for $j=1,2$, and where
\begin{equation}\label{cal Mn cal Pn cal Rn}
\begin{aligned}
& {\cal D}^\bot_{n, 1}(\theta, y)[w] \in {\cal OF}^2_{w}(1, N), \qquad \qquad {\cal D}^\bot_{n, 2}(\theta, w)[w] \in {\cal OF}^2_{ww}(1, N), \\
& X^\bot_{n, 1}(\theta, y)[w] \in {\cal OB}_{w}^2(1 - n, N),  \qquad X^\bot_{n, 2}(\theta, w)[w] \in {\cal OB}_{ww}^2(1 - n, N)\\ 
&{\cal R}^{\bot}_{N,1}(\theta, y)[w] \in {\cal OS}_{w}^2(N), \qquad \qquad \ \  {\cal R}^\bot_{N, 2}(\theta, w)[w] \in {\cal OS}^2_{ww}(N)\,, \\
& {\cal O}_3^{(\theta)}, {\cal O}_3^{(y)} \in C^\infty_b \big( {\cal V}^{\sigma_N}(\delta) \times [0, \e_0], \, \R^{S_+} \big) \ \  \text{small of order three} 
\end{aligned}
\end{equation}
for some $\sigma_N > 0$.
Moreover 
\begin{equation}\label{antisimmetria}
\begin{aligned}
& {\cal D}^\bot_{n, 1}(\theta, y) = - {\cal D}^\bot_{n, 1}(\theta, y)^\top, \quad  [X^\bot_{n, 1}(\theta, y)]_j^j = - \overline{[X^\bot_{n, 1}(\theta, y)]_j^j}, \quad
 [{\cal R}^\bot_{N, 1}(\theta, y)]_j^j = - \overline{[{\cal R}^\bot_{N,1}(\theta, y)]_j^j}, \quad  \forall j \in S^\bot . 
\end{aligned}
\end{equation}
Our goal at the (n+1)th step  is to construct a transformation so that when expressed in the new coordinates, 
the vector field $X^\bot_{n, 1}(\theta, y)[w] + X^\bot_{n, 2}(\theta, w)[w]$ is of order 
$1 -(n+1) = - n $. 
Since $X^\bot_{n, 1}(\theta, y)[w] \in {\cal OB}^2_{w}(1 - n, N)$ and $X^\bot_{n, 2}(\theta, w)[w] \in {\cal OB}_{ww}^2(1 - n, N)$ we can write 
\begin{equation}\label{espansione cal P n (0)}
\begin{aligned}
& X^\bot_{n, 1} (\theta, y)[w]   = \Pi_\bot T_{a_{1 - n}(\theta, y)} \partial_x^{1 - n} w + {\cal OB}^2_{w}(- n, N)\,, \\
& X^\bot_{n, 1} (\theta, w)[w] = \Pi_\bot T_{A_{1 - n}(\theta)[w]} \partial_x^{1 - n} w + {\cal OB}^2_{ww}(- n, N)
\end{aligned}
\end{equation}
with the property that there are integers $s_N > 0$, $\sigma_N  \ge 0$ 
so that for any $s \geq s_N$ 
there exist $0 < \delta \equiv  \delta(s, \gamma, N) <1$ and $0 < \e_0 \equiv \e_0(s, \gamma, N) < 1$ so that 
\begin{equation}\label{proprieta alpha - n beta - n step 1}
\begin{aligned}
& a_{1 - n} \in C^\infty_b \big(\T^{S_+} \times B_{S_+}(\delta) \times [0, \e_0], \, H^s(\T_1) \big) \ \  \text{small of order one}, \\
& A_{1 - n}\in C^\infty\big(\T^{S_+} \times [0, \e_0], \, {\cal B}(H^{s + \sigma_N}_\bot(\T_1), \, H^s(\T_1)) \big)\,. 
\end{aligned}
\end{equation}
Hence we need to normalize the vector field $\Pi_\bot T_{a_{1 - n}(\theta, y) + A_{1 - n}(\theta)[w]} \partial_x^{1 - n} w $. 
In order to achieve this, we consider a para-differential vector field of the form 
\begin{equation}\label{hamiltoniana Gn forma normale}
\begin{aligned}
& Y^\bot_n (\theta, y, w) =  Y^\bot_{n, 1}(\theta, y)[w] + Y^\bot_{n, 2}(\theta, w)[w]\,, \\
& Y^\bot_{n, 1}(\theta,y)[w] := \Pi_\bot T_{b_n(\theta, y)} \partial_x^{- n - 1} w \in {\cal OB}^2_{w}(- n - 1, N)\,, \\
& Y^\bot_{n, 2}(\theta,w)[w] := \Pi_\bot T_{B_n(\theta)[w]} \partial_x^{- n - 1} w \in {\cal OB}^2_{ww}(- n - 1, N)\,,
\end{aligned}
\end{equation}
and make the ansatz that $b_n(\theta, y)$, $B_n(\theta)[w]$ are smooth functions (satisfying conditions as in \eqref{proprieta alpha - n beta - n step 1}) and  
$$
\langle b_n(\theta, y)\rangle_x = 0,  \qquad   \langle B_n(\theta)[w]\rangle_x= 0\, .
$$ 
To determine $b_n$ and $B_n$, we compute the pullback $X_{n + 1} := \Phi_{Y_n}^* X_n$ of $X_n$ by the time one flow map $ \Phi_{Y_n}$. corresponding to the vector field $Y_n$ . 
By Lemmata \ref{lemma coniugazioni}, \ref{corollario flusso con Omega bot}, \ref{lemma coniugazioni multiplier} 
and the induction hypothesis \eqref{antisimmetria}, one infers that the components of $X_{n + 1} = (X_{n + 1}^{(\theta)}, X_{n + 1}^{(y)}, X_{n + 1}^\bot)$ satisfy
\begin{equation}\label{def X n + 1 bot}
\begin{aligned}
X_{n + 1}^{(\theta)}(\frak x) & = - \omega - \e \widehat \omega - \nabla_y Q(y) - \Upsilon_2^{(\theta)}(\theta)[w,w] + {\cal O}_3^{(\theta)}(\frak x)\,, 
\qquad X_{n + 1}^{(y)}(\frak x)  = {\cal O}_3^{(y)}(\frak x)\,, \\
X_{n + 1}^\bot (\frak x) &=\ii  \Omega_\bot w + {\cal D}^\bot_{n, 1}(\theta, y)[w] + {\cal D}^\bot_{n, 2}(\theta, w)[w]  +  \Pi_\bot T_{- 3 \partial_x b_n(\theta, y) + a_{1 - n}(\theta, y)} \partial_x^{1 - n} w \\
&+  \Pi_\bot T_{- 3 \partial_x B_n(\theta)[w]+ A_{1 - n}(\theta)[w]} \partial_x^{1 - n} w 
 + X^\bot_{n + 1, 1}(\theta, y)[w] + {\cal R}^\bot_{N, 1}(\theta, y)[w] \\
&+ X^\bot_{n + 1, 2}(\theta, w)[w] + {\cal R}^\bot_{N, 2}(\theta)[w, w]   + {\cal OB}^3(1, N)  + {\cal OS}^3(N)
\end{aligned}
\end{equation}
where 
\begin{equation}\label{cal Mn cal P n + 1 cal R n + 1}
\begin{aligned}
&X^\bot_{n + 1, 1}(\theta, y)[w]\in {\cal OB}_{w}^2(- n, N),  \qquad X^\bot_{n + 1, 2}(\theta, w)[w] \in {\cal OB}_{ww}^2( - n, N)\\ 
&{\cal R}^\bot_{N, 1}(\theta, y)[w] \in {\cal OS}_{w}^2(N), \qquad \qquad \ \  {\cal R}^\bot_{N, 2}(\theta, w)[w] \in {\cal OS}^2_{ww}(N)\,, \\
& {\cal O}_3^{(\theta)}, {\cal O}_3^{(y)} \in C^\infty_b \big( {\cal V}^\sigma(\delta) \times [0, \e_0], \ \R^{S_+} \big) \quad \text{small of order three} 
\end{aligned}
\end{equation}
 and the diagonal matrix elements of the operators $X^\bot_{n + 1, 1}(\theta, y)$, ${\cal R}^\bot_{N,1}(\theta, y)$ are purely imaginary, namely 
\begin{equation}
[ X^\bot_{n + 1, 1}(\theta, y)]_j^j,  \  [ {\cal R}^\bot_{N,1}(\theta, y)]_j^j \in \ii \R\,, \quad \forall j \in S^\bot\,. 
\end{equation} 
We then choose $b_n(\theta, y)$ and $B_n(\theta)[w]$ to be solutions of
\begin{equation}\label{omologica lin in w}
\begin{aligned}
& - 3 \partial_x b_n(\theta, y) + a_{1 - n}(\theta, y) = \langle a_{1 - n}(\theta, y) \rangle_x\, , \\
& - 3 \partial_x B_n(\theta)[w] + A_{1 - n}(\theta)[w] = \langle A_{1 - n}(\theta)[w] \rangle_x\, .
\end{aligned}
\end{equation}
More precisely, we define
\begin{equation}\label{def gn}
\begin{aligned}
& b_n(\theta, y) := \frac13 \partial_x^{- 1}\big( a_{1 - n}(\theta, y) - \langle a_{1 - n}(\theta, y) \rangle_x   \big)\,, \\
& B_n(\theta)[w] := \frac13 \partial_x^{- 1}\big( A_{1 - n}(\theta)[w] - \langle  A_{1 - n}(\theta)[w] \rangle_x  \big) \, .
\end{aligned}
\end{equation}
Since $\Pi_\bot T_{\langle a_{1 - n}(\theta, y) \rangle_x} \partial_x^{1 - n} w = 
\langle a_{1 - n}(\theta, y) \rangle_x \partial_x^{1 - n} w$ and $\Pi_\bot T_{\langle A_{1 - n}(\theta)[w] \rangle_x} \partial_x^{1 - n} w 
= \langle A_{1 - n}(\theta)[w] \rangle_x \partial_x^{1 - n} w$ one infers from \eqref{proprieta alpha - n beta - n step 1} that
\begin{equation}\label{def cal M n + 1}
\begin{aligned}
& {\cal D}^\bot_{n + 1, 1}(\theta, y)[w] := {\cal D}^\bot_{n, 1}(\theta, y)[w] + \langle a_{1 - n}(\theta, y) \rangle_x \partial_x^{1 - n} w \in {\cal OF}^2_{w}(1, N), \\
& {\cal D}^\bot_{n+ 1, 2}(\theta, w)[w] := {\cal D}^\bot_{n, 2}(\theta, w)[w] + \langle A_{1 - n}(\theta)[w] \rangle_x \partial_x^{1 - n} w \in {\cal OF}^2_{ww}(1, N)\,. 
\end{aligned}
\end{equation}
Since $a_{1 - n }(\theta, y)$ is real valued, the Fourier multiplier $\langle a_{1 - n }(\theta, y) \rangle_x \partial^{1 - n}$ is skew-adjoint if $n $ is even. 
Futhermore, by the induction hypothesis \eqref{antisimmetria} and Lemmata \ref{lemma parte diagonale}, \ref{lemma 2 campi lineari} in Appendix \ref{Appendix A}, one has 
$$
\begin{aligned}
\langle a_{1 - n}(\theta, y) \rangle_x = 0 \quad \text{if} \  n \ \text{is odd.}
\end{aligned}
$$
Hence the Fourier multiplier ${\cal D}^\bot_{n + 1, 1}(\theta, y)$ is skew-adjoint. Altogether we showed that the vector field $X_{n + 1}^\bot$ is of the form 
\begin{equation}\label{Xn bot}
\begin{aligned}
X_{n + 1}^\bot (\frak x) & =\ii \Omega_\bot w + {\cal D}^\bot_{n + 1, 1}(\theta, y)[w] + {\cal D}^\bot_{n + 1, 2}(\theta, w)[w]   + X^\bot_{n + 1, 1}(\theta, y)[w] \\
&+ X^\bot_{n + 1, 2}(\theta, w)[w] + {\cal R}^\bot_{N, 1}(\theta, y)[w] + {\cal R}^\bot_{N, 2}(\theta, w)[w] 
 + {\cal OB}^3(1, N) + {\cal OS}^3(N) \, .
\end{aligned}
\end{equation}
We thus have proved the following
\begin{proposition}\label{prop total regularization}
For any $N \in 	\N$, there exist $s_N$, $\sigma_N > 0$
 with the following property: 
for any $s \geq s_N$ there exist $0 < \delta \equiv \delta(s, \gamma, N) < 1$, $0 < \e_0 \equiv \e_0(s, \gamma, N) < 1$ 
so that the following holds: 
there exists a transformation $\Psi^{(1)}$ with inverse $(\Psi^{(1)})^{-1}$ (cf. Remark \ref{inverse of flow}),
\begin{equation}\label{prop Psi (1)}
(\Psi^{(1)})^{\pm 1} \in {\cal C}^\infty_b({\cal V}^s(\delta) \times [0, \e_0],  \, {\cal V}^s(2 \delta)),  \quad \forall \, s \ge s_N, \qquad 
(\Psi^{(1)})^{\pm 1} (\frak x) - \frak x \ \  \text{small of order two}\, ,
\end{equation}
so that the transformed vector field $X_4 := (\Psi^{(1)})^* X_{{\cal H}_3} = (X_4^{(\theta)}, X_4^{(y)}, X_4^\bot)$ has the following properties:
\begin{equation}\label{regularized vector field 1}
\begin{aligned}
 X_4^{(\theta)}(\frak x)   = &- \omega - \e \widehat \omega - \nabla_y Q(y) - \Upsilon_2^{(\theta)}(\theta)[w, w] + {\cal O}_3^{(\theta)}(\frak x)\, , \qquad
X_4^{(y)}(\frak x)  = {\cal O}_3^{(y)}(\frak x) \, ,\\
X^\bot_4(\frak x)  = & \ii \Omega_\bot w + {\cal D}^\bot_{4, 1}(\theta, y)[w] + {\cal D}^\bot_{4, 2}(\theta, w)[w]   + {\cal R}^\bot_{N, 1}(\theta, y)[w] + {\cal R}^\bot_{N,  2}(\theta, w)[w]  \\
& +   {\cal OB}^3(1, N)+ {\cal OS}^3(N) 
\end{aligned}
\end{equation}
 where 
 \begin{equation}\label{termini campo X4}
 \begin{aligned}
 & {\cal D}^\bot_{4, 1}(\theta, y)[w]  \in {\cal OF}_{w}^2(1, N), \qquad {\cal D}^\bot_{4, 2}(\theta, w)[w] \in {\cal OF}_{ww}^2(1, N)\,, \\
 & {\cal R}^\bot_{N, 1}(\theta, y)[w] \in {\cal OS}_{w}^2(N), \qquad  \quad {\cal R}^\bot_{N, 2}(\theta, w)[w] \in {\cal OS}_{ww}^2(N)\,, \\
 & {\cal O}_3^{(\theta)}, {\cal O}_3^{(y)} \in C^\infty_b \big( {\cal V}^{\sigma_N}(\delta) \times [0, \e_0], \, \R^{S_+} \big) \quad \text{small of order three.} 
   \end{aligned}
 \end{equation}
 Moreover 
 \begin{equation}\label{prop autoaggiuntezza cal M cal R 1}
 \quad {\cal D}^\bot_{4, 1}(\theta, y) = - ({\cal D}^\bot_{4, 1}(\theta, y))^\top ,  \qquad \quad [{\cal R}^\bot_{N, 1}(\theta, y)]_j^j \in \ii \R, \quad \forall j \in S^\bot\,. 
 \end{equation}
 \end{proposition}
 

 \subsection{Normalization of Fourier multiplier quadratic in $w$.}\label{normalizzazione Fourier multipliers}
 The goal of this subsection is to normalize the vector field ${\cal D}^\bot_{4, 2}(\theta, w)[w]$ in \eqref{regularized vector field 1}. 
 According to Proposition \ref{prop total regularization} and Definitions \eqref{classe fourier multipliers}, \eqref{paradiff omogenei di ordine 2},
 \begin{equation}\label{forma espansione cal D 4 (1)}
 \begin{aligned}
& {\cal D}^\bot_{4, 2}(\theta, w)[w] = \Lambda^\bot_1(\theta)[w] \partial_x w + \widetilde {\cal D}^\bot_{4, 2}(\theta, w)[w],  \\
& \widetilde {\cal D}^\bot_{4, 2}(\theta, w)[w] := \sum_{k = 1}^{N+1} \Lambda^\bot_{1 - k}(\theta)[w] \partial_x^{1 - k} w \in {\cal OF}^2_{ww}(0, N),
 \end{aligned}
 \end{equation}
 where for any  $0 \le k \le N + 1$,  $\Lambda^\bot_{1 - k} \in C^\infty(\T^{S_+}, {\cal B}(H^{\sigma_N}_\bot, \R))$  for some $\sigma_N > 0$ (large). 
 Since  $\Lambda_1(\theta)[w]$ is real valued, the leading order operator $\Lambda_1(\theta)[w] \partial_x$ is a skew-adjoint Fourier multiplier 
 and hence has the property needed for the energy estimates in Section \ref{conclusioni forma normale}. 
 This however is not true for  $\widetilde {\cal D}^\bot_{4, 2}(\theta, w)[w]$. The goal of this section is to eliminate it.
  To this end, we consider a vector field of the form 
 \begin{equation}\label{forma multiplier per normalizzare}
 {\cal M} (\frak x) := \big( 0, 0, \,  {\cal M}^\bot(\theta, w)[w]\big)\, ,
 \qquad {\cal M}^\bot(\theta, w)[w] = \sum_{k = 1}^{N+1} \Xi^\bot_{1 - k}(\theta)[w] \partial_x^{1 - k}w \in {\cal OF}^2_{ww}(0, N),
 \end{equation}
 where $ \Xi^\bot_{1 - k}(\theta)$ will be chosen so that  the time one flow map $\Phi_{\cal M}$, generated by the vector field $X_{\cal M}$,  
 is a coordinate transformation serving our needs. In more detail, consider the pullback $X_5 := \Phi_{\cal M}^* X_4 = (X_5^{(\theta)}, X_5^{(y)}, X_5^\bot)$ 
 of the vector field $X_4$ of Proposition \ref{prop total regularization}  by $\Phi_{\cal M}$. 
 By Lemmata \ref{lemma coniugazioni multiplier}, \ref{psuh forward multiplier X cal N}, one has
 \begin{equation}\label{expansion X5 bot}
 \begin{aligned}
 X_5^{(\theta)}(\frak x) & = - \omega -  \e \widehat \omega -  \nabla_y Q(y) - \Upsilon^{(\theta)}_2(\theta)[w, w] + {\cal O}_3^{(\theta)}(\frak x) \, , \qquad 
 X_5^{(y)}  = {\cal O}_3^{(y)}(\frak x) \, ,\\
 X_5^\bot(\frak x) & =\ii \Omega_\bot w + {\cal D}_{4,1}^\bot(\theta, y)[w] +\Lambda_1(\theta)[w] \partial_x w    + {\cal R}^\bot_{4, 1}(\theta, y)[w] 
 + {\cal R}^\bot_{4,2}(\theta, w)[w]  \\
 & \qquad + \omega \cdot \partial_\theta {\cal M}^\bot(\theta, w)[w] - {\cal M}^\bot(\theta,\ii  \Omega_\bot w) [w]+ 
 \widetilde{\cal D}_{4, 2}(\theta, w)[w]  + {\cal OB}^3(1, N) + {\cal OS}^3(N)
\end{aligned}
 \end{equation}
 where for some integer $\sigma_N > 0$, 
 ${\cal O}_3^{(\theta)}$, ${\cal O}_3^{(y)}$ are in $C^\infty_b \big({\cal V}^{\sigma_N}(\delta) \times [0,\e_0], \, \R^{S_+}\big)$ and small of order three. 
 The vector field ${\cal M}^\bot(\theta, w)[w]$ is chosen to be a solution the following homological equation
 \begin{equation}\label{equazione omologica cal M w2 0}
 \omega \cdot \partial_\theta {\cal M}^\bot(\theta, w)[w] - {\cal M}^\bot(\theta, \ii  \Omega_\bot w) [w]+ \widetilde {\cal D}^\bot_{4, 2}(\theta, w)[w] = 0 ,  
 \end{equation}
 or in view of \eqref{forma espansione cal D 4 (1)}, \eqref{forma multiplier per normalizzare} equivalently, that for any $1 \le k \le  N + 1$, $ \Xi^\bot_{1 - k}(\theta)[w]$ is a solution of
 \begin{equation}\label{equazione omologica cal M w2 1}
 \omega \cdot \partial_\theta  \, \Xi^\bot_{1 - k}(\theta)[w]   -  \Xi^\bot_{1 - k}(\theta)[ \ii \Omega_\bot w] + \Lambda^\bot_{1 - k}(\theta)[w] = 0\,.
 \end{equation}
 Since $\Lambda^\bot_{1 - k}$, $\Xi^\bot_{1 - k} \in C^\infty(\T^{S_+} \times [0, \e_0], \, {\cal B}(H^{\sigma_N}_\bot(\T_1), \R))$, there exist uniquely determined maps
 $a_{\Lambda^\bot_{1 - k}}$, $a_{\Xi^\bot_{1 - k}}$ in $C^\infty(\T^{S_+} \times [0, \e_0], \, H^{- \sigma_N}_\bot(\T_1))$ so that 
 $$
 \Lambda^\bot_{1 - k}(\theta)[w] = \big\langle a_{\Lambda^\bot_{1 - k}}(\theta)\,,\, w \big\rangle\,, \qquad  
 \Xi^\bot_{1 - k}(\theta)[w]  = \big\langle a_{\Xi^\bot_{1 - k}}(\theta)\,,\, w \big\rangle \, .
 $$
Equation \eqref{equazione omologica cal M w2 1} then reads
 \begin{equation}\label{equazione omologica cal M w2 2}
 \big\langle \omega \cdot \partial_\theta \, a_{\Xi^\bot_{1 - k}}(\theta), \, w \big\rangle  
 - \big\langle a_{\Xi^\bot_{1 - k}}(\theta), \, \ii \Omega_\bot w \big\rangle + \big\langle a_{\Lambda^\bot_{1 - k}}(\theta), \, w \big\rangle  
 = 0\,.
 \end{equation}
 Since $\ii \Omega_\bot$ is skew-adjoint, one has $- \big\langle a_{\Xi^\bot_{1 - k}}(\theta), \, \ii \Omega_\bot w \big\rangle 
 = \big\langle \ii  \Omega_\bot a_{\Xi^\bot_{1 - k}}(\theta),  w \big\rangle$.  We choose $a_{\Xi^\bot_{1 - k}}$  as the solution of 
 \begin{equation}\label{equazione omologica cal M w2 3}
 \big( \omega \cdot \partial_\theta + \ii  \Omega_\bot \big)a_{\Xi^\bot_{1 - k}}(\theta)  + a_{\Lambda^\bot_{1 - k}}(\theta) = 0\,.
 \end{equation}
This equation can be solved by expanding $a_{\Xi^\bot_{1 - k}}(\theta)$ and $a_{\Lambda^\bot_{1 - k}}(\theta)$  in Fourier series with respect to $\theta$ and $x$, 
$$
a_{\Xi^\bot_{1 - k}}(\theta) = \sum_{(\ell, j) \in \Z^{S_+} \times S^\bot} \widehat a_{\Xi^\bot_{1 - k}}(\ell, j) e^{\ii \ell \cdot \theta} e^{\ii 2 \pi j x}, \qquad 
a_{\Lambda^\bot_{1 - k}}(\theta) = \sum_{(\ell, j) \in \Z^{S_+} \times S^\bot} \widehat a_{\Lambda^\bot_{1 - k}}(\ell, j) e^{\ii \ell \cdot \theta} e^{\ii 2 \pi j x}.
$$
Since by assumption $\omega \in \Pi^{(1)}_\gamma$, $0 < \gamma < 1$, (cf. \eqref{condizioni forma normale}), equation \eqref{equazione omologica cal M w2 3} can be solved.
The solution  $a_{\Xi^\bot_{1 - k}} (\theta)$ is given by
\begin{equation}\label{equazione omologica cal M w2 4}
a_{\Xi^\bot_{1 - k}}  = - (\omega \cdot \partial_\theta + \ii  \Omega_\bot)^{- 1} a_{\Lambda^\bot_{1 - k}} =  -
\sum_{(\ell, j) \in \Z^{S_+} \times S^\bot}\dfrac{\widehat a_{\Lambda^\bot_{1 - k}}(\ell, j)}{\ii (\omega \cdot \ell + \Omega_j)} e^{\ii \ell \cdot \theta} e^{\ii 2 \pi j x}\,. 
\end{equation}
Since $a_{\Lambda^\bot_{1 - k}} \in C^\infty(\T^{S_+}, H^{- \sigma_N}_\bot(\T_1))$, one infers that
$a_{\Xi^\bot_{1 - k}} \in C^\infty(\T^{S_+}, H^{- \sigma_N - \tau}_\bot(\T_1))$ and 
therefore \eqref{forma multiplier per normalizzare} is verified and equation \eqref{equazione omologica cal M w2 0} is solved. 
Finally, the vector field $X_5^\bot$ is of the form 
\begin{equation}\label{formal finale campo vett X5 bot}
\begin{aligned}
& X_5^\bot(\frak x) = \ii \Omega_\bot w + {\cal D}_5(\frak x)[w] + {\cal R}^\bot_{N, 1}(\theta, y)[w] + {\cal R}^\bot_{N, 2}(\theta, w)[w] + {\cal OB}^3(1, N) + {\cal OS}^3(N)\,, \\
& {\cal D}^\bot_5(\frak x)[w] :=  {\cal D}^\bot_{4, 1}(\theta, y)[w] +\Lambda^\bot_1(\theta)[w] \partial_x w \in {\cal OF}^2(1, N),
\end{aligned}
\end{equation} 
where the remainders ${\cal R}^\bot_{N, 1}(\theta, y)[w]$, ${\cal R}^\bot_{N, 2}(\theta, w)[w]$ are given in Proposition \ref{prop total regularization}.
Furthermore, ${\cal D}^\bot_5(\frak x)$ is skew-adjoint,
\begin{equation}\label{cal D5 skew symmetric}
{\cal D}^\bot_5(\frak x)^\top = - {\cal D}^\bot_5(\frak x)\,. 
\end{equation}
We summarize our findings of this subsection as follows.
\begin{proposition}\label{prop regularization fourier multiplier}
For any $N \in 	\N$, there exists an integer $s_N  > 0$
with the property that for any $s \ge s_N$ there exist
$0 < \delta \equiv \delta(s, \gamma, N) < 1$ and $0 < \e_0 \equiv \e_0(s, \gamma, N) < 1$ so that the following holds: there exists a transformation $\Psi^{(2)}$ 
with inverse $(\Psi^{(2)})^{-1}$ (cf. Remark \ref{inverse of flow}),
\begin{equation}\label{prop Psi (2)}
 (\Psi^{(2)})^{\pm 1} \in {\cal C}^\infty_b({\cal V}^s(\delta) \times [0, \e_0], \, {\cal V}^s(2 \delta)),  \quad \forall \, s \ge s_N, \qquad
 (\Psi^{(2)})^{\pm 1}(\frak x) - \frak x \quad \text{small of order two},
\end{equation}
so that the transformed vector field $X_5 := (\Psi^{(2)})^* X_{4} = (X_5^{(\theta)}, X_5^{(y)}, X_5^\bot)$ has the form
\begin{equation}\label{regularized vector field X5}
\begin{aligned}
X_5^{(\theta)}(\frak x) & = - \omega -  \e \widehat \omega -  \nabla_y Q(y)  - \Upsilon^{(\theta)}_2(\theta)[w, w] + {\cal O}_3(\frak x) \, , \qquad 
X_5^{(y)}(\frak x)  = {\cal O}_3(\frak x) \, , \\
X^\bot_5(\frak x) & = \ii \Omega_\bot w + {\cal D}^\bot_5(\frak x)[w]    + {\cal R}^\bot_{N, 1}(\theta, y)[w] + {\cal R}^\bot_{N, 2}(\theta, w)[w]  
  + {\cal OB}^3(1, N) + {\cal OS}^3(N)
\end{aligned}
\end{equation}
 where 
 \begin{equation}\label{prop autoaggiuntezza cal D5}
 \begin{aligned}
 & {\cal D}^\bot_5(\frak x)[w] \in {\cal OF}^2(1, N),   \qquad  {\cal D}^\bot_5(\frak x) = - {\cal D}^\bot_5(\frak x)^\top\,  
   \end{aligned}
 \end{equation}
 and the smoothing remainders ${\cal R}^\bot_{N,1}(\theta, y)[w]$, ${\cal R}^\bot_{N, 2}(\theta, w)[w]$ are given by Proposition \ref{prop total regularization}. 
  \end{proposition}


\subsection{Normalization of the smoothing remainders}\label{section smoothing remainders}
In this subsection, we normalize the vector field 
$$
\big(  \Upsilon^{(\theta)}_2(\theta)[w, w], \, 0, \,  {\cal R}^\bot_{N, 1}(\theta, y)[w] + {\cal R}^\bot_{N, 2}(\theta)[w, w] \big) \, ,
$$
which is part of the vector field $X_5$ defined in \eqref{regularized vector field X5}. 
Note that all the terms are either linear or quadratic in the variable $w$.
We consider a smoothing vector field of the form 
$$
{\cal S}(\frak x) := \big( \, \mathcal S^{(\theta)}(\theta)[w, w], \ 0,  \ {\cal S}^\bot_1(\theta, y)[w] + {\cal S}^\bot_2(\theta)[w, w] \big) 
$$
where we make the ansatz that for some $\sigma_N  > 0$, 
$\mathcal S^{(\theta)} \in C^\infty\big( \T^{S_+} \times [0, \e_0], \, {\cal B}_2(H^{\sigma_N}_\bot(\T_1), \R^{S_+}) \big)$ and
\begin{equation}\label{ansatz ultima normalizzazione}
{\cal S}^\bot_1(\theta, y)[w] \in {\cal OS}_{w}^2(N - 1), \qquad {\cal S}^\bot_2(\theta)[w, w] \in {\cal OS}_{ww}^2(N - 5)\,.
\end{equation}
We then consider the time one flow map $\Phi_{\cal S}$, associated to the vector field ${{\cal S}}$, and 
compute the pullback $X_6 := \Phi_{\cal S}^* X_5 = (X_6^{(\theta)}, X_6^{(y)}, X_6^\bot)$ of the vector field $X_5$ by $\Phi_\mathcal S$. 
By Lemmata \ref{push forward smoothing 1}, \ref{push forward smoothing X cal N} 
and in view of Remark \eqref{remark fourier multiplier paradiff}, $X_6$ is of the form 
\begin{equation}\label{prima expansione X6 bot}
\begin{aligned}
& X_6^{(\theta)}(\frak x)  = - \omega - \e \widehat \omega - \nabla_y Q(y) 
+ \omega \cdot \partial_\theta \,  \mathcal S^{(\theta)}(\theta)[w, w]  - 
\mathcal S^{(\theta)}(\theta)[\ii \Omega_\bot w, w] - \mathcal S^{(\theta)}(\theta)[w, \ii \Omega_\bot w] \\
& \qquad \qquad - \Upsilon^{(\theta)}_2(\theta)[w, w] + {\cal O}^{(\theta)}_3(\frak x) , \\
& X_6^{(y)}(\frak x)  = {\cal O}^{(y)}_3(\frak x) , \\
& X_6^\bot (\frak x)  =\ii  \Omega_\bot w + {\cal D}^\bot_5(\frak x)[w] + 
\Big( \omega \cdot \partial_\theta {\cal S}^\bot_1(\theta, y) + [\ii \Omega_\bot , \, {\cal S}^\bot_1(\theta, y)]_{lin} + {\cal R}^\bot_{N, 1}(\theta, y) \Big)[w] \\
& \qquad +  \omega \cdot \partial_\theta \, {\cal S}^\bot_2(\theta)[w, w] +
\ii \Omega_\bot {\cal S}^\bot_2(\theta)[w, w] - {\cal S}^\bot_2(\theta)[\ii  \Omega_\bot w, w] - {\cal S}^\bot_2(\theta)[w, \ii \Omega_\bot w] + {\cal R}^\bot_{N, 2}(\theta)[ w, w] \\
& \qquad + {\cal OB}^3(1, N ) + {\cal OS}^3(N - 6)
\end{aligned}
\end{equation}
where ${\cal O}^{(\theta)}_3$, ${\cal O}^{(y)}_3$  denote terms which are small of order three. The components
$\mathcal S^{(\theta)}$ and  $\mathcal S^\bot_1$, $\mathcal S^\bot_2$ are now chosen as the solutions of the following homological equations,
\begin{equation}\label{eq omologiche smoothing terms}
\begin{aligned}
&  \omega \cdot \partial_\theta \,  \mathcal S^{(\theta)}(\theta)[w , w] - \mathcal S^{(\theta)}(\theta)[\ii \Omega_\bot w, w] - \mathcal S^{(\theta)}(\theta)[w, \ii \Omega_\bot w] 
- \Upsilon^{(\theta)}_2(\theta)[w, w] = - \mathcal Z^{(\theta)}[w, w],  \\
& \qquad \quad \mathcal Z^{(\theta)} [w, w]  := \sum_{j \in S^\bot} \, w_j w_{- j} \,  \langle  \Upsilon_2^{(\theta)}(\theta)[e^{\ii 2 \pi j x}, e^{- \ii 2 \pi j x}]  \rangle_\theta  \, , \\
&  \omega \cdot \partial_\theta \,  {\cal S}^\bot_1(\theta, y) + [\ii \Omega_\bot , {\cal S}^\bot_1(\theta, y)]_{lin} + {\cal R}^\bot_{N, 1}(\theta, y) = {\cal Z}^\bot( y)\,, \qquad
 {\cal Z}^\bot( y) := {\rm diag}_{j \in S^\bot} [\widehat{\cal R}^\bot_{N, 1}(0, y)]_j^j\,, \\
&  \omega \cdot \partial_\theta \, {\cal S}^\bot_2(\theta)[w, w] + \ii \Omega_\bot {\cal S}^\bot_2(\theta)[w, w] 
- {\cal S}^\bot_2(\theta)[\ii  \Omega_\bot w, w] - {\cal S}^\bot_2(\theta)[w,\ii  \Omega_\bot w] + {\cal R}^\bot_{N, 2}(\theta)[ w, w]= 0\,.
\end{aligned}
\end{equation}
Homological equations of this form can be solved by applying the following lemma.
\begin{lemma}\label{lemma equazioni omologiche smoothing}
Let $N \in \N$.
$(i)$ Let $\mathcal M^{(\theta)} \in C^\infty(\T^{S_+} \times [0, \e_0], \, {\cal B}_2(H^\sigma_\bot(\T_1), \, \R^{S_+}))$ for some $\sigma > 0$ and 
assume that $\omega \in \Pi_\gamma$, $0 < \gamma < 1$ (cf. \eqref{condizioni forma normale}). 
Then there exists $\mathcal S^{(\theta)} \in C^\infty(\T^{S_+} \times [0, \e_0], \, {\cal B}_2(H^{\sigma + 1}_\bot(T_1), \, \R^{S_+}))$ solving
\begin{equation}\label{eq omologica smoothing astratta 0}
\begin{aligned}
& \omega \cdot \partial_\theta \, \mathcal S^{(\theta)}(\theta)[w , w] - \mathcal S^{(\theta)}(\theta)[\ii \Omega_\bot w, w] - 
\mathcal S^{(\theta)}(\theta)[w, \ii \Omega_\bot w] - \mathcal M^{(\theta)}(\theta)[w, w] = - \mathcal Z^{(\theta)}[w, w], \\
& \mathcal Z^{(\theta)}[w, w] := \sum_{j \in S^\bot}  \, w_j w_{- j} \, \langle  \mathcal M^{(\theta)}(\theta)[e^{\ii 2 \pi j x}, e^{- \ii 2 \pi j x}]  \rangle_\theta \,. 
\end{aligned}
\end{equation}
$(ii)$ Let ${\cal R}^\bot_{N, 1}(\theta, y)[w] \in {\cal OS}_{w}^2(N)$ and $\omega \in \Pi_\gamma$, $0 < \gamma < 1$ (cf. \eqref{condizioni forma normale}). 
Then there exists ${\cal S}^\bot_1(\theta, y)[w] \in {\cal OS}_{w}^2(N - 1)$ which solves the equation
\begin{equation}\label{eq omologica smoothing astratta 1}
  \omega \cdot \partial_\theta \, {\cal S}^\bot_1(\theta, y) + [\ii \Omega_\bot , \,  {\cal S}^\bot_1(\theta, y)]_{lin} + {\cal R}^\bot_{N, 1}(\theta, y) = {\cal Z}^\bot( y)\,, \qquad
 {\cal Z}^\bot( y) := {\rm diag}_{j \in S^\bot} [\widehat{\cal R}^\bot_{N, 1}(0, y)]_j^j\,. 
\end{equation}
$(iii)$ Let ${\cal R}^\bot_{N, 2}(\theta)[w, w]  \in {\cal OS}_{ww}^2(N)$ and assume that $\omega \in \Pi_\gamma$, $0 < \gamma < 1$ (cf. \eqref{condizioni forma normale}). 
Then there exists ${\cal S}^\bot_2(\theta)[w, w] \in {\cal OS}_{ww}^2(N - 5)$ which solves the equation
\begin{equation}\label{eq omologica smoothing astratta 2}
\begin{aligned}
&  \omega \cdot \partial_\theta \, {\cal S}^\bot_2(\theta)[w,  w] +\ii \Omega_\bot {\cal S}^\bot_2(\theta)[w, w] - {\cal S}^\bot_2(\theta)[\ii  \Omega_\bot w, w] 
- {\cal S}^\bot_2(\theta)[w, \ii \Omega_\bot w]+ {\cal R}^\bot_{N, 2}(\theta)[w, w] = 0\,. 
\end{aligned}
\end{equation}
\end{lemma}
\begin{proof}
Since items $(i)$, $(ii)$ can be proved by arguments similar to the ones used in the proof of item $(iii)$, we only prove the latter. 
By assumption, $\mathcal R^\bot_{N, 2}(\theta, w)[w] \equiv {\mathcal R^\bot_{N, 2}}(\theta)[w, w] \in {\cal OS}_{ww}^2(N)$.
Hence there exists an integer $s_N> 0$ with the property that for any $s \geq s_N$, there exists $0 < \e_0 \equiv \e_0(s) < 1$ so that  
$$
\mathcal R^\bot_{N, 2} : \T^{S_+} \times [0, \e_0] \to  {\cal B}_{2, s, N} , \,  (\theta, \e) \mapsto \mathcal R^\bot_{N, 2}(\theta) \equiv  \mathcal R^\bot_{N, 2}(\theta, \e) \, ,
\qquad {\cal B}_{2, s, N}: = \mathcal B_2(H^{s}(\T_1), H^{s + N + 1}(\T_1)),
$$ 
is  $C^\infty$-smooth and bounded (cf. \eqref{forme multilineari}, Defintion \ref{paradiff omogenei di ordine 2}).
A a consequence, for any multi-index $\alpha \in \Z_{\ge 0}^{S_+}$, 
\begin{equation}\label{china 100}
\| \partial_\theta^\alpha \mathcal R^\bot_{N, 2}(\theta) \|_{{\cal B}_{2, s, N}} \lesssim_{\alpha, s} 1\,. 
\end{equation}
Expanding $\mathcal R^\bot_{N, 2}(\theta)$  in its Fourier  series, 
$\mathcal R^\bot_{N, 2}(\theta) = \sum_{\ell \in \Z^{S_+}} \widehat{\mathcal R^\bot_{N, 2}}(\ell) e^{\ii \ell \cdot \theta}$,
the latter estimates imply
\begin{equation}\label{china 101}
\| \widehat{\mathcal R^\bot_{N, 2}}(\ell) \|_{{\cal B}_{2, s, N}} \lesssim_{\alpha, s} \langle \ell \rangle^{- |\alpha|}, \qquad \forall \,  \alpha \in \Z_{\ge 0}^{S_+}\,, \ \  \forall \,  \ell \in \Z^{S_+}\,.  
\end{equation}
Since for any $\ell \in \Z^{S_+}$, $\widehat{\mathcal R^\bot_{N, 2}}(\ell) \in {\cal B}_{2, s, N}$, one has for any $w, v \in H^s_\bot(\T_1)$
\begin{equation}\label{forma bilineare cal R ell}
 \widehat{\mathcal R^\bot_{N, 2}}(\ell)[w, v] = \sum_{j, j' \in S^\bot}  w_j v_{j'} \widehat{\mathcal R^\bot_{N, 2}}(\ell)_{j j'} \,, \qquad
 \widehat{\mathcal R^\bot_{N, 2}}(\ell)_{j j'}( x) := \widehat{\mathcal  R^\bot_{N, 2}}(\ell)[e^{\ii 2 \pi j  x}, e^{\ii 2 \pi j'  x}]\,.
\end{equation}
In particular, for $w = e^{\ii 2 \pi j x}$, $v = e^{\ii 2 \pi j' x}$, one infers from \eqref{china 101} that
\begin{equation}\label{china 102}
\| \widehat{\mathcal R^\bot_{N, 2}}(\ell)_{j j'} \|_{s + N + 1} \lesssim_{\alpha, s} 
\langle \ell \rangle^{- |\alpha|}\langle j \rangle^s \langle j' \rangle^s\,, \qquad
\forall \, \alpha \in \Z_{\ge 0}^{S_+}, \ \ell \in \Z^{S_+}, \  j , j' \in S^\bot  .
\end{equation}
Expanding also $\mathcal S^\bot_2(\theta)$  in its Fourier  series, 
$\mathcal S^\bot_2(\theta) = \sum_{\ell \in \Z^{S_+}} \widehat{\mathcal S^\bot_2}(\ell) e^{\ii \ell \cdot \theta}$,
one has for any $w, v \in H^s_\bot(\T_1)$,
\begin{equation}\label{forma bilineare cal S ell}
 \widehat{\mathcal S^\bot_2}(\ell)[w, v] = \sum_{j, j' \in S^\bot}  w_j v_{j'} \widehat{\mathcal S^\bot_2}(\ell)_{j j'} \,, \qquad
 \widehat{\mathcal S^\bot_2}(\ell)_{j j'}( x) := \widehat{\mathcal  S^\bot_2}(\ell)[e^{\ii 2 \pi j  x}, e^{\ii 2 \pi j'  x}]\,.
\end{equation}
By expanding $ \widehat{\mathcal S^\bot_2}(\ell)_{j j'}( x) $ and $ \widehat{\mathcal R^\bot_2}(\ell)_{j j'}( x) $ with respect to the variable $x \in \T_1$ in Fourier series,
\begin{equation}\label{china 80}
 \widehat{\mathcal S^\bot_2}(\ell)_{j j'}( x)  = \sum_{n \in S^\bot}  \widehat{\mathcal S^\bot_2}(\ell, n)_{j j'}  e^{\ii 2 \pi n x}, \qquad 
  \widehat{\mathcal R^\bot_{N, 2}}(\ell)_{j j'}( x)  = \sum_{n \in S^\bot}  \widehat{\mathcal R^\bot_{N, 2}}(\ell, n)_{j j'}  e^{\ii 2 \pi n x}\,,
\end{equation}
the homological equation \eqref{eq omologica smoothing astratta 2} yields the following equations for the coefficients  $\widehat{\mathcal S^\bot_2}(\ell, n)_{j j'}$ of $\widehat{\mathcal S^\bot_2}(\ell)$,
\begin{equation}\label{china 103}
\ii \big( \omega \cdot \ell + \Omega_n - \Omega_j - \Omega_{j'}\big) \widehat{\mathcal S^\bot_2}(\ell, n)_{j j'} + \widehat{\mathcal R^\bot_{N, 2}}(\ell, n)_{j j'} = 0\,. 
\end{equation}
Since $\omega \in \Pi_\gamma^{(3)}$, $0 < \gamma < 1$ (cf. \eqref{condizioni forma normale}), the latter equations admit solutions. They are given by
\begin{equation}\label{china 104}
\widehat{\mathcal S^\bot_2}(\ell, n)_{j j'} = - \dfrac{\widehat{\mathcal R^\bot_{N, 2}}(\ell, n)_{j j'}}{\ii \big(\omega \cdot \ell + \Omega_n - \Omega_j - \Omega_{j'}\big)}\, , 
\qquad \forall \, \ell \in \Z^{S_+}, \ n, \,  j, \,  j' \in S^\bot \, ,
\end{equation}
and satisfy the estimate
$|\widehat{\mathcal S^\bot_2}(\ell, n)_{j j'}| \leq \langle \ell \rangle^\tau \langle j \rangle^2 \langle j' \rangle^2 \langle n \rangle^2 \gamma^{- 1} |\widehat{\mathcal R^\bot_{N, 2}}(\ell, n)_{j j'}|$
 (cf. \eqref{condizioni forma normale}).
By \eqref{china 80}, one has 
$\| \widehat{\mathcal  S^\bot_2}(\ell)_{j j'} \|_{s + N - 1}  
= \big( \sum_{n \in S^\bot} \langle n \rangle^{2 (s + N - 1)} |\widehat{\mathcal S^\bot_2}(\ell, n)_{j j'}|^2 \big)^{\frac12} $
and hence
\begin{equation}\label{china 106}
\begin{aligned}
\| \widehat{\mathcal  S^\bot_2}(\ell)_{j j'} \|_{s + N - 1} &  \leq \langle \ell \rangle^\tau \langle j \rangle^2 \langle j' \rangle^2 \gamma^{- 1}
\Big( \sum_{n \in S^\bot} \langle n \rangle^{2 (s + N - 1)} \langle n \rangle^4 | \widehat{\mathcal R^\bot_{N, 2}} (\ell, n)_{j j'}|^2 \Big)^{\frac12}  \\
& = \langle \ell \rangle^\tau \langle j \rangle^2 \langle j' \rangle^2 \gamma^{- 1} \| \widehat{\mathcal R^\bot_{N, 2}}(\ell)_{j j'} \|_{s + N + 1} 
\stackrel{\eqref{china 102}}{\lesssim_{\alpha, s}} 
\langle \ell \rangle^{\tau - |\alpha|} \langle j \rangle^{s + 2} \langle j' \rangle^{s + 2} \gamma^{- 1}\,. 
\end{aligned}
\end{equation}
For any $w, v \in H^{s+3}_\bot(\T_1)$, one then obtains by the Cauchy-Schwarz inequality,
\begin{equation}\label{china 107}
\begin{aligned}
\| \widehat{\mathcal S^\bot_2}(\ell)[w, v] \|_{s + N - 1} & {\leq} \sum_{j, j' \in S^\bot} |w_j | |v_{j'}|  \,\| \widehat{\mathcal S^\bot_2}(\ell)_{j j'} \|_{s + N - 1} 
\stackrel{\eqref{china 106}}{\lesssim_{\alpha, s}} 
\langle \ell \rangle^{\tau - |\alpha|} \gamma^{- 1} \sum_{j, j' \in S^\bot} \langle j \rangle^{s + 2} |w_j| \langle j' \rangle^{s + 2} |v_{j'}|  \\
& \lesssim_{\alpha, s} \langle \ell \rangle^{\tau - |\alpha|} \gamma^{- 1} \| w \|_{s + 3} \| v \|_{s + 3} \, .
\end{aligned}
\end{equation}
Writing $s$ for $s + 3$, we thus have proved that there exists $s_N > 0$ (large) so that
$$
\| \widehat{\mathcal S^\bot_2}(\ell) \|_{{\cal B}_{2, s, N - 4}} 
\lesssim_{\alpha, s}
 \langle \ell \rangle^{\tau - |\alpha|}\gamma^{- 1}\, , \qquad \forall \, \alpha \in \Z_{\ge 0}^{S_+}, \ s \ge s_N\, .
$$
implying that ${\mathcal S^\bot_2} \in C^\infty(\T^{S_+} \times [0, \e_0], \, {\cal B}_{2, s, (N - 5) + 1})$ for any $s \geq s_N$.  Hence 
${\mathcal S^\bot_2}(\theta)[w, w] \in {\cal OS}_{ww}^2(N - 5)$.  
\end{proof}

\smallskip
By Lemma \ref{lemma equazioni omologiche smoothing} and in view of  \eqref{prima expansione X6 bot}, \eqref{eq omologiche smoothing terms}, 
the vector field $X_6 = (X_6^{(\theta)}, X_6^{(y)}, X_6^\bot)$ takes the form
\begin{equation}\label{forma finale X6 bot}
\begin{aligned}
X_6^{(\theta)}(\frak x) & = - \omega -  \e \widehat \omega - \nabla_y Q(y) - \mathcal Z^{(\theta)}[w, w] + {\cal O}_3^{(\theta)}(\frak x)\, , \qquad \quad  \
X_6^{(y)}(\frak x)  = {\cal O}_3^{(y)}(\frak x) \, ,\\
X_6^\bot (\frak x) & = \ii  \Omega_\bot w + {\cal D}^\bot_6(\frak x)[w] + {\cal OB}^3(1, N) + {\cal OS}^3(N - 6)\,, \qquad
{\cal D}^\bot_6(\frak x)   := {\cal D}^\bot_5(\frak x) + {\cal Z}^\bot(y)\,, \quad\\
& {\cal O}_3^{(\theta)}, {\cal O}_3^{(y)} \in C^\infty_b([0, \e_0] \times {\cal V}^{\sigma_N}(\delta), \, \R^{S_+})  \quad \text{ terms small of order three}
\end{aligned}
\end{equation} 
for some $\sigma_N > 0$.
Since by \eqref{prop autoaggiuntezza cal M cal R 1}, $[{\mathcal R^\bot_{N, 1}}(\theta, y)]_j^j \in \ii \R$, $j \in S^\bot$,
and ${\cal Z}^\bot(y) = {\rm diag}_{j \in S^\bot} [\widehat{\mathcal R^\bot_1}(0, y)]_j^j$, the operator ${\cal Z}^\bot(y)$ is a skew-adjoint Fourier multiplier and hence
by \eqref{prop autoaggiuntezza cal D5} so is ${\cal D}^\bot_6(\frak x)$. We summarize our findings as follows. 
\begin{proposition}\label{proposizione X6 bot}
For any $N \in 	\Z_{\ge 6}$, there exists an integer $s_N > N$
 with the property that for any $s \ge s_N$, there exist $0 < \delta \equiv \delta(s, \gamma, N) <1$ and $0 < \e_0 \equiv \e_0(s, \gamma, N) < 1$  so that the following holds. 
There exists a map $\Psi^{(3)}$ with inverse $(\Psi^{(3)})^{-1}$ (cf. Remark \ref{inverse of flow}), 
\begin{equation}\label{prop Psi (3)}
 (\Psi^{(3)})^{\pm 1} \in {\cal C}^\infty_b({\cal V}^s(\delta)  \times [0, \e_0], \, {\cal V}^s(2 \delta)), \quad \forall s \geq s_N\,,  \qquad
 (\Psi^{(3)})^{\pm 1}(\frak x ) - \frak x \quad \text{small of order two},
\end{equation}
so that the transformed vector field $X_6 := (\Psi^{(3)})^* X_{5} = (X_6^{(\theta)}, \, X_6^{(y)}, \, X_6^\bot)$ has the form
\begin{equation}\label{regularized vector field X6}
\begin{aligned}
X^{(\theta)}_6(\frak x)  = - \omega - & \e \widehat \omega - \nabla_y Q(y)  - \mathcal Z^{(\theta)}[w, w] + {\cal O}_3^{(\theta)}(\frak x) , \qquad \qquad
X^{(y)}_6(\frak x)  = {\cal O}_3^{(y)}(\frak x) \, , \\
X^\bot_6(\frak x) & = \ii \Omega_\bot w + {\cal D}^\bot_6(\frak x)[w]    + {\cal OB}^3(1, N) + {\cal OS}^3(N - 6) \, ,
\end{aligned}
\end{equation}
 where ${\cal D}^\bot_6(\frak x)$ is a Fourier multiplier of order one given by \eqref{forma finale X6 bot} 
 and satisfies ${\cal D}^\bot_6(\frak x) = - {\cal D}^\bot_6(\frak x)^\top$, where
 \begin{equation}\label{Z ww}
 \mathcal Z^{(\theta)} \in {\cal B}_2(H^{\sigma_N}_\bot, \R^{S_+}),   \qquad  \mathcal Z^{(\theta)}[w, w] 
 = \sum_{j \in S^\bot} w_j w_{- j} \, \langle  \Upsilon_2^{(\theta)}(\theta)[e^{\ii 2 \pi j x}, e^{- \ii 2 \pi j x}]  \rangle_\theta , \ \ \forall \, w \in H^{\sigma_N}_\bot(\T_1),
 \end{equation}
 for some $\sigma_N > 0$, and where ${\cal O}_3^{(\theta)}$, ${\cal O}_3^{(y)}$ comprises terms which are small of order three. 
\end{proposition}

\section{Proofs of Theorem  \ref{long time ex action-angle} and Theorem \ref{teorema totale forma normale}}\label{conclusioni forma normale}

First we prove Theorem \ref{teorema totale forma normale}.

\smallskip

\noindent
{\em Proof of Theorem \ref{teorema totale forma normale}.} We apply Propositions \ref{prop ordine e 3}, \ref{prop total regularization}, \ref{prop regularization fourier multiplier}, \ref{proposizione X6 bot}. Choose $N = 6$ and define 
\begin{equation}\label{def cal U}
{\mathtt \Phi} := \Phi \circ \Psi^{(1)} \circ \Psi^{(2)} \circ \Psi^{(3)}\,. 
\end{equation}
By \eqref{prop Phi finale passi smoothing}, \eqref{prop Psi (1)}, \eqref{prop Psi (2)}, \eqref{prop Psi (3)}, ${\mathtt \Phi}$ satisfies property \eqref{trasformazione finale cal U}. 
Moreover $X = X_6 = {\mathtt \Phi}^* X_{\cal H}$ is given in \eqref{regularized vector field X6} with $N = 6$. Hence by setting 
$$
{\mathtt D}^\bot := {\cal D}^\bot_6\, , \qquad \mathtt N^{(\theta)}(y, w) := - \nabla_y Q(y) - \mathcal Z^{(\theta)}[w, w]\, ,
$$ 
one has that ${\mathtt D}^\bot$, $\mathtt N^{(\theta)}$, ${\cal O}_3^{(\theta)}$, ${\cal O}_3^{(y)}$ satisfy the properties stated in \eqref{proprieta elementi campo vettoriale}. 
Since $N = 6$ , the remainder term ${\cal OB}^3(1, 6) + {\cal OS}^3(0)$ in the expansion of $X^\bot(\frak x) = X^\bot_6(\frak x)$ in  \eqref{regularized vector field X6}
has the form (cf. Definitions \ref{paradiff vector fields}, \ref{def smoothing vector fields} )
$$
\Pi_\bot \sum_{k = 0}^7 T_{a_{1 - k}(\frak x)} \partial_x^{1 - k} w + {\cal R}^\bot_0(\frak x)
$$
with the following property: there are integers $s_\ast $, $\sigma > 0$ so that for any $s \geq s_\ast$
 there exist $0 < \delta \equiv \delta(s, \gamma) < 1$ and $0 < \e_0 \equiv \e_0(s, \gamma) < 1$ 
so that 
\begin{equation}\label{schifido 0}
\begin{aligned}
& a_{1 - k} \in C^\infty_b \big(  {\cal V}^{s + \sigma}(\delta) \times [0, \e_0], \, H^s(\T_1) \big) \quad \text{small of order two}, \qquad \forall \, 0 \le k \le 7 \, ,\\
& {\cal R}^\bot_0 \in C^\infty_b \big( {\cal V}^{s}(\delta) \times [0, \e_0], \, H^s_\bot(\T_1) \big) \quad \text{small of order three. }
\end{aligned}
\end{equation}
We then define 
$$
a (\frak x) := a_1(\frak x), \qquad {\cal R}^\bot(\frak x) := \Pi_\bot \sum_{k = 0}^6 T_{a_{- k}(\frak x)} \partial_x^{- k} w + {\cal R}^\bot_0(\frak x)\, .
$$
One shows that ${\cal R}^\bot \in C^\infty_b( {\cal V}^s(\delta) \times [0, \e_0], \, H^s_\bot(\T_1))$ for any $s \ge s_\ast +\sigma$ and that ${\cal R}^\bot$ is small of order three.
Indeed, by  \eqref{schifido 0}  and the estimate \eqref{stima elementare paraproduct} (paraproduct), it follows that for any $\frak x \in {\cal V}^{s}(\delta)$,
$$
\begin{aligned}
& \| {\cal R}^\bot(\frak x) \|_s \lesssim_{s, \gamma} {\rm max}_{0 \le k \le 7} \| a_{1 - k}(\frak x) \|_1 \| w \|_s + ( \e + \| y \| + \| w \|_s )^3 \\
& \qquad \qquad \lesssim_{s, \gamma} {\rm max}_{0 \le k \le 7} \| a_{1 - k}(\frak x) \|_{s_\ast} \| w \|_s + ( \e + \| y \| + \| w \|_s )^3 \\
& \qquad \qquad \lesssim_{s, \gamma} ( \e + \| y \| + \| w \|_{s_\ast +  \sigma} )^3  + ( \e + \| y \| + \| w \|_s )^3 .
\end{aligned}
$$
Hence we proved that for any $s \ge  s_\ast +  \sigma$,
$$
 \| {\cal R}^\bot(\frak x) \|_s \lesssim_{s, \gamma}  ( \e + \| y \| + \| w \|_s )^3.
$$
Theorem \ref{teorema totale forma normale} then follows by choosing $\sigma_\ast:= s_\ast + \sigma$.
\hfill $\square$

\medskip


Let us now turn to the proof of Theorem \ref{long time ex action-angle}.
It is based on energy estimates for the solutions of the equation $\partial_t \frak x = X(\frak x)$ where $X$ 
is the vector field provided by Theorem \ref{teorema totale forma normale}
(cf. \eqref{forma campo vettoriale finale dopo NF}, \eqref{proprieta elementi campo vettoriale})
\begin{equation}\label{PDE NF energy estimates}
\begin{cases}
\partial_t \theta (t) = - \omega - \e \widehat \omega + \mathtt N^{(\theta)}(y, w) + {\cal O}_3^{(\theta)}(\frak x) \\
\partial_t y (t)  = {\cal O}_3^{(y)}(\frak x) \\
\partial_t w (t) = \ii \Omega_\bot w + {\mathtt  D}^\bot(\frak x)[w] + \Pi_\bot T_{a(\frak x)} \partial_x w + {\cal R}^\bot(\frak x).
\end{cases}
\end{equation}
Choose $\sigma_\ast > 0$ and for any $s \geq \sigma_\ast$, $0 < \delta \equiv \delta(s, \gamma) < 1 $, 
$0 < \e_0 \equiv \e_0(s, \gamma) \ll \delta$ 
as in Theorem \ref{teorema totale forma normale}. For any $s \ge \sigma_\ast$ and $0 < \e \le \e_0(s, \gamma)$ 
we then consider the Cauchy problem of \eqref{PDE NF energy estimates} 
with small  initial data $\frak x_0 = (\theta_0, y_0, w_0) \in \T^{S_+} \times \R^{S_+} \times H^s_\bot(\T_1)$, 
\begin{equation}\label{dati iniziali cauchy prob}
 | y_0| , \, \| w_0 \|_s \leq \e\,. 
\end{equation}
Increasing $\s_*$ and decreasing $\e_0$, if needed, it follows from Proposition \ref{local existence theorem}
that for any $s \ge \s_*$ and $0 < \e \le \e_0$ there exists $T\equiv T_{\e, s, \gamma}> 0 $ so that
the Cauchy problem of \eqref{PDE NF energy estimates}  for any initial data $\frak x_0 = (\theta_0, y_0, w_0)$ satisfying \eqref{dati iniziali cauchy prob} 
has a unique solution $t \mapsto \frak x(t) = (\theta(t), y(t), w(t))$ with
\begin{equation}\label{local solution}
 \theta \in C^1([- T, T], \T^{S_+}),  \quad y \in C^1([- T, T], \R^{S_+}), \quad w \in C^0([- T, T], H^s_\bot(\T_1)) \cap C^1([- T, T], H^{s - 3}_\bot(\T_1)). 
\end{equation}
In addition, by Proposition \ref{local existence theorem} there exists $C_\ast \equiv C_\ast(\gamma) > 1$ so that
\begin{equation}\label{sol locali sistema NF}
 | y(t) |, \, \| w(t) \|_s\, , \,  |\Theta(t)| \,  \leq C_\ast  \e \, , \qquad \forall t \in [- T, T]\,, 
\end{equation}
where
\begin{equation}\label{def Theta}
 \Theta(t) := \theta(t) - \theta_0 + (\omega + \e \widehat \omega) t - \int_0^t \mathtt N^{\theta}(y(\tau), w(\tau))\, d \tau, 
 \qquad t \in [- T, T].
\end{equation}
We now prove that the time $T$  of existence of the solution can be chosen to be of size $\e^{-2}$.
\begin{proposition}\label{prop tempo di esistenza}
Let $\sigma_\ast$ and $0 < \e_0 \equiv \e_0(s, \gamma) < 1$, $s \ge \sigma_\ast$ be given as above.
Then for any $s \ge \sigma_\ast$ there exists a constant 
$C_{**} \equiv C_{**}(s, \gamma) > 0$ so that for any $0 < \e \le \e_0$, 
the time of existence $T$ of the solution $\frak x(t)$ can be chosen as $T_{\e, s, \gamma} := C_{**} \e^{- 2}$. 
\end{proposition} 
To prove the latter proposition, we first need to make some preliminary considerations. Let $s \geq  \sigma_\ast$ and $0 < \e \le \e_0$.
 By \eqref{proprieta elementi campo vettoriale}, 
 $a$ is small of order two and ${\cal R}^\bot$, $ {\cal O}_3^{(\theta)}$, ${\cal O}_3^{(y)}$ are small of order three, 
and by applying the estimates \eqref{sol locali sistema NF}, one has 
\begin{equation}\label{stime epsilon quantita campo vett}
\begin{aligned}
& |{\cal O}_3^{(\theta)}(\frak x(t))|\,,  \,  |{\cal O}_3^{(y)}(\frak x(t))| \lesssim_\gamma \e^3, \qquad 
 \| a(\frak x(t)) \|_{\sigma_\ast} \lesssim_\gamma \e^2, \qquad \| {\cal R}^\bot(\frak x(t)) \|_s \lesssim_{s, \gamma} \e^3, \qquad \forall \, t \in [- T, T]\,.
\end{aligned}
\end{equation}
First we prove the following lemma.
\begin{lemma}\label{lemma eps T soluzione}
Given any $s \ge  \sigma_\ast$, there exists a constant $K_0 \equiv K_0(s, \gamma) > 0$
(large) so that the solutions \eqref{local solution} satisfy
\begin{equation}\label{estimate local solution}
 |\Theta(t)| \leq K_0 T \e^3\,,\qquad |y(t)|, \, \| w(t) \|_s  \leq \e + K_0 \e^3 T, \qquad \qquad  \forall  \, t \in [- T, T]\,. 
\end{equation}
As a consequence, for any $T > 0$ satisfying $T \leq \frac{1}{K_0} \e^{-2}$, one has
\begin{equation}\label{estimate local solution 1}
|\Theta(t)| \leq  \e, \qquad |y(t)|\,,\, \| w(t) \|_s  \leq 2\e, \quad \forall t \in [- T, T]\,.  
\end{equation}
\end{lemma}
\noindent
{\em Proof of Lemma \ref{lemma eps T soluzione}.} 
Let $s \ge  \sigma_\ast$.
First we prove the claimed estimates for  $\Theta(t)$ and $y(t)$. By the definition \eqref{def Theta} of $\Theta$ 
and \eqref{PDE NF energy estimates} (Hamiltonian equations), one has
$$
\Theta(0) = 0,  \qquad  \quad \partial_t \Theta(t) = {\cal O}_3^{(\theta)}(\frak x(t)),
$$
implying that 
$$
\Theta(t) = \int_0^t {\cal O}_3^{(\theta)}(\frak x(\tau))\, d \tau\,. 
$$
Moreover by \eqref{PDE NF energy estimates}, 
$$
y(t) = y_0 + \int_0^t {\cal O}_3^{(y)}(\frak x(\tau))\, d \tau\,. 
$$
By \eqref{dati iniziali cauchy prob} and \eqref{stime epsilon quantita campo vett}, one then concludes that there exists a constant $C_1 \equiv C_1(s, \gamma) > 0$ so that
\begin{equation}\label{stima Theta (t) y (t)}
| \Theta(t)| \leq C_1 T \e^3, \qquad |y(t)| \leq \e  + C_1 T \e^3, \qquad \qquad  \forall  \, t \in [- T, T] \, .
\end{equation}

\smallskip

\noindent
It remains to estimate the $H^s$-norm of $w(t)$. To this end recall that for any $w \in H^s_\bot(\T_1)$, 
$$
\| w \|_s = \big(\sum_{j \in S^\bot} |j|^{2 s} |w_j|^2 \big)^{\frac12} =  \| \partial_x^s w \| \, ,
$$
where $\| \partial_x^s w \|$ denotes the $L^2$-norm of $\partial_x^s w $. Then 
\begin{equation}\label{energy estimate 0}
\begin{aligned}
\partial_t \| \partial_x^s w(t) \|^2 & 
= \big\langle \partial_x^s \partial_t w(t)\,,\, \partial_x^s w(t) \big\rangle + \big\langle \partial_x^s w(t)\,,\, \partial_x^s \partial_t w(t) \big\rangle \\
& \stackrel{\eqref{PDE NF energy estimates}}{=} 
\big\langle \partial_x^s\big( \ii \Omega_\bot w + {\mathtt D}^\bot(\frak x)[w] + \Pi_\bot T_{a(\frak x)} \partial_x w + {\cal R}^\bot(\frak x) \big)\,,\, \partial_x^s w \big\rangle \\
& \quad + \big\langle \partial_x^s w \,,\, \partial_x^s \big(\ii \Omega_\bot w + {\mathtt D}^\bot(\frak x)[w] + \Pi_\bot T_{a(\frak x)} \partial_x w + {\cal R}^\bot(\frak x) \big)\big\rangle\,. 
\end{aligned}
\end{equation}
Since $\Omega_\bot$ and ${\mathtt D}^\bot(\frak x)$ are both Fourier multipliers, the linear commutators with the Fourier multiplier $\partial_x^s$ vanish,
$$
[\partial_x^s, \Omega_\bot]_{lin}= 0 \, , \qquad  [\partial_x^s, {\mathtt D}^\bot(\frak x)]_{lin} = 0\, . 
$$
Using in addition that ${\mathtt D}^\bot(\frak x)$ is skew-adjoint (cf. \eqref{proprieta elementi campo vettoriale}) and hence
$(\ii \Omega_\bot + {\mathtt D}^\bot(\frak x))^\top = - \ii \Omega_\bot - {\mathtt D}^\bot(\frak x)$, one infers
\begin{equation}\label{energy estimate 1}
\begin{aligned}
& \big\langle \partial_x^s \big( \ii \Omega_\bot w + {\mathtt D}^\bot(\frak x)[w] \big)\,,\, \partial_x^s w \big\rangle 
+ \big\langle \partial_x^s w \,,\, \partial_x^s \big(\ii \Omega_\bot w + {\mathtt D}^\bot(\frak x)[w] \big)\big\rangle \\
& \qquad = \big\langle \big(\ii \Omega_\bot  + {\mathtt D}^\bot(\frak x) \big)  \partial_x^s w\,,\, \partial_x^s w \big\rangle 
+ \big\langle \partial_x^s w \,,\, \big(\ii \Omega_\bot  + {\mathtt D}^\bot(\frak x) \big)  \partial_x^s w \big\rangle \\
& \qquad = \big\langle \big(\ii \Omega_\bot  + {\mathtt D}^\bot(\frak x) \big)  \partial_x^s w\,,\, \partial_x^s w \big\rangle 
+ \big\langle \big(\ii \Omega_\bot  + {\mathtt D}^\bot(\frak x) \big)^\top  \partial_x^s w\,,\, \partial_x^s w \big\rangle = 0\,. 
\end{aligned}
\end{equation}
Moreover 
\begin{equation}\label{energy estimate 2}
\begin{aligned}
& \big\langle \partial_x^s  T_{a(\frak x)} \partial_x w\,,\, \partial_x^s w \big\rangle + \big\langle \partial_x^s w \,,\, \partial_x^s  T_{a(\frak x)} \partial_x w\big\rangle \\
& = \big\langle   T_{a(\frak x)} \partial_x \partial_x^s w\,,\, \partial_x^s w \big\rangle + \big\langle \partial_x^s w \,,\,   T_{a(\frak x)} \partial_x \partial_x^sw\big\rangle 
 + \big\langle   [\partial_x^s, T_{a(\frak x)} \partial_x]  w\,,\, \partial_x^s w \big\rangle + \big\langle \partial_x^s w \,,\,   [\partial_x^s, T_{a(\frak x)} \partial_x] w\big\rangle \\
& = \big\langle   \big( T_{a(\frak x)} \partial_x + (T_{a(\frak x)} \partial_x)^\top \big) \partial_x^s w\,,\, \partial_x^s w \big\rangle  
 + \big\langle   [\partial_x^s, T_{a(\frak x)} \partial_x]  w\,,\, \partial_x^s w \big\rangle + \big\langle \partial_x^s w \,,\,   [\partial_x^s, T_{a(\frak x)} \partial_x] w\big\rangle\,.
\end{aligned}
\end{equation}
By increasing $\sigma_*$ if needed one gets by Corollary \ref{aggiunto nostri operatori} (with $N=1$, $m=1$)
$$
\|  \Pi_\bot T_{a(\frak x)} \partial_x + \Pi_\bot (T_{a(\frak x)} \partial_x)^\top\|_{{\cal B}(L^2_\bot)} \lesssim  \| a(\frak x) \|_{\sigma_\ast} 
\stackrel{\eqref{stime epsilon quantita campo vett}}{\lesssim_\gamma} \e^2 
$$
and hence by the Cauchy-Schwarz inequality,
\begin{equation}\label{energy estimate 2}
\begin{aligned}
|\big\langle   \big( T_{a(\frak x)} \partial_x + (T_{a(\frak x)} \partial_x)^\top \big) \partial_x^s w\,,\, \partial_x^s w \big\rangle |  
\lesssim_\gamma \e^2 \| \partial_x^s w \| \lesssim_\gamma \e^2 \| w \|_s^2 \,. 
\end{aligned}
\end{equation}
Moreover, arguing as  in \cite[Lemma A.1]{BM}, one has 
$$
\| [\partial_x^s, T_{a(\frak x)} \partial_x]  w\|_{L^2} \lesssim_{s} \| a(\frak x) \|_2 \| w \|_s 
\stackrel{\sigma_\ast \geq 2, \eqref{stime epsilon quantita campo vett}}{\lesssim_{s, \gamma}} \e^2 \| w \|_s\,.
$$
The latter estimate, together with the Cauchy-Schwarz inequality, imply that 
\begin{equation}\label{energy estimate 3}
\begin{aligned}
|\big\langle   [\partial_x^s, T_{a(\frak x)} \partial_x]  w\,,\, \partial_x^s w \big\rangle + \big\langle \partial_x^s w \,,\,   [\partial_x^s, T_{a(\frak x)} \partial_x] w\big\rangle| \
\lesssim_{s, \gamma} \e^2 \| w \|_s^2\,. 
\end{aligned}
\end{equation}
Finally, by using the Cauchy-Schwarz inequality once more and the estimate \eqref{stime epsilon quantita campo vett} for ${\cal R}^\bot$, one gets 
\begin{equation}\label{energy estimate 4}
\begin{aligned} 
& |\big\langle \partial_x^s {\cal R}^\bot(\frak x)\,,\, \partial_x^s w \big\rangle + \big\langle  \partial_x^s w\,,\,\partial_x^s {\cal R}^\bot(\frak x) \big\rangle| 
\lesssim \| {\cal R}^\bot(\frak x) \|_s \| w \|_s  \lesssim_{s, \gamma} \e^3 \| w \|_s\,. 
\end{aligned}
\end{equation}
Thus, collecting \eqref{energy estimate 0}-\eqref{energy estimate 4}, and 
since by \eqref{sol locali sistema NF}, $\| w(t) \|_s \leq C_\ast \e$ for any $t \in [- T, T]$, 
one gets 
$$
|\partial_t  \, \| \partial_x^s w(t) \|^2| \lesssim_{s, \gamma} \e^4\, , \qquad  \forall \,  t \in [- T, T].
$$
We then conclude that there exists a constant  $C_2 \equiv C_2(s, \gamma) > 0$ so that
\begin{equation}\label{energy estimate 5}
\| w(t) \|_s \leq (\| w_0 \|_s^2 + C_2 T \e^4 )^{1/2} \leq \e (1 +  C_2 T \e^2)^{1/2} \leq \e + C_2 T \e^3 , \quad \forall t \in [- T, T] \, .
\end{equation}
The claimed statement then follows with $K_0(s, \gamma)  := \max\{ C_1(s, \gamma), C_2(s, \gamma) \}$.
\hfill $\square$

\bigskip

\noindent
{\em Proof of Proposition \ref{prop tempo di esistenza}.}
For any given $s \ge  \sigma_\ast$, $0 < \e \le \e_0$, 
and initial data satisfying \eqref{dati iniziali cauchy prob},  
consider the solution $t \mapsto \frak x(t)$ in \eqref{local solution} of \eqref{PDE NF energy estimates}.
It satisfies the estimates \eqref{estimate local solution} - \eqref{estimate local solution 1} of Lemma \ref{lemma eps T soluzione}.
Let
$$
 \check T := {\rm sup}\{  \, 0 < T < \frac{1}{K_0} \e^{-2} : 
 \, 2 |\Theta(t)|, \,  |y(t)|, \, \| w(t) \|_s \, \leq 2 \e, \ \forall  t \in [- T, T] \} \, ,
 $$
where $K_0 \equiv K_0(s, \gamma)$ is given by Lemma \ref{lemma eps T soluzione}, and define 
$$
M(T) := \max_{ \, |t| \le T}\{2 |\Theta( t)|,  |y( t)|, \| w( t) \|_s \}, \quad T \in [0,  \check T)\,.
$$
Assume that  
$\check T \leq \frac 12 \frac{1}{K_0} \e^{- 2}$.
By the definition of $\check T$ and Proposition \ref{local existence theorem}
it then follows that $\sup_{T < \check T} M(T) = 2 \e$. On the other hand, 
from Lemma \ref{lemma eps T soluzione} one infers that
$$
\begin{aligned}
M( \check T) \leq \e + K_0 \e^3 \check T \leq \e(1 + 1/2 ) \le \frac{3}{2} \e\, .
\end{aligned}
$$ 
Hence we obtained a contradiction and thus conclude that  $\check T = O(\e^{- 2})$. 
 \hfill $\square$
 
 \medskip
 
 \noindent
{\em Proof of Theorem \ref{long time ex action-angle}.}
 Let $ t \mapsto \frak x(t) = (\theta(t), y(t), w(t))$ be a curve satisfying \eqref{dati iniziali cauchy prob} -  \eqref{sol locali sistema NF}.
By Theorem \ref{teorema totale forma normale} (Normal Form Theorem),  
$\frak x(t) = (\theta(t), y(t), w(t))$ is a solution of \eqref{PDE NF energy estimates} if and only if 
$$
\frak x'(t) = (\theta'(t), y'(t), w'(t)) := {\mathtt \Phi}(\frak x(t))
$$
is a solution of \eqref{very first ham equations} with initial data $\frak x'_0 = {\mathtt \Phi}(\frak x_0)$. 

By \eqref{trasformazione finale cal U} (properties of the transformation $\mathtt \Phi$), 
for any $\frak x$ in $\mathcal V^s(\delta)$ with $\frak x' := {\mathtt \Phi}(\frak x) \in \mathcal V^s(\delta)$ one has
$ \frak x = {\mathtt \Phi}^{- 1}(\frak x')$
and
$$
| y'|, \| w' \|_s \leq C(s, \gamma)\big( \e + | y | + \| w\|_s \big)\, ,  \qquad  |y|, \| w \|_s \leq C(s, \gamma) \big( \e + | y' | + \| w'\|_s \big) 
$$
for some constant $C(s, \gamma) > 0$.
Hence, if $\frak x(t)$ satisfies \eqref{dati iniziali cauchy prob} - \eqref{sol locali sistema NF}, then $ \frak x_0' \in \T^{S_+} \times \R^{S_+} \times H^s_\bot(\T_1)$ with
$| y_0' |$, $\| w_0' \|_s \leq C(s, \gamma) \e$
 and
$$
\theta' \in C^1([- T, T], \T^{S_+}), \quad y' \in C^1([- T, T], \R^{S_+}), 
\quad w' \in C^0([- T, T], H^s_\bot(\T_1)) \cap C^1([- T, T], H^{s - 3}_\bot(\T_1))
$$
with
$$
 | y'(t) |, \, \| w'(t) \|_s\, \leq 2 C(s, \gamma)  \e\, ,  \quad \forall t \in [- T, T]\,.
 $$
By Proposition \ref{prop tempo di esistenza},  $T$ can be chosen as $T_{\e, s, \gamma} = O(\e^{- 2})$. This proves Theorem \ref{long time ex action-angle} .
\hfill $\square$

\bigskip

\section{Measure estimates}\label{measure estimates}

In this section we prove the measure estimate \eqref{main measure estimate} of the set $\Pi_\gamma$ defined in \eqref{non resonant set tot}, \eqref{condizioni forma normale}. 
More precisely we show the following 
\begin{proposition}\label{prop measure total} 
There exists $\mathtt a \in (0, 1)$ so that for any $0 \le j \le 3$ and any $0 < \gamma < 1$,
$ | \Pi \setminus \Pi_\gamma^{(j)} | \lesssim \gamma^{\mathtt a}$.
\end{proposition}  
We will concentrate on the proof of the claimed measure estimate of $\Pi_\gamma^{(3)}$. The ones of 
$\Pi_\gamma^{(0)}$, $\Pi_\gamma^{(1)}$, and $\Pi_\gamma^{(2)}$ can be obtained in a similar way and are in fact a bit easier to prove. 
Recall that 
\begin{equation}\label{Pi gamma (3) est}
\begin{aligned}
\Pi_\gamma^{(3)} & = \Big\{ \omega \in \Pi : |\omega \cdot \ell + \Omega_{j_1} (\omega) + \Omega_{j_2}(\omega) + \Omega_{j_3}(\omega)| \geq \frac{\gamma}{\langle \ell \rangle^\tau \langle j_1 \rangle^2 \langle j_2 \rangle^2 \langle j_3 \rangle^2}, \\
&  \qquad \forall (\ell, j_1, j_2, j_3) \in \Z^{S_+}  \times S^\bot \times S^\bot \times S^\bot \quad \text{with}\quad j_k + j_m \neq 0, \quad \forall  k, m \in \{ 1,2,3\}\Big\}\,
\end{aligned}
\end{equation}
where for any $ j \in S^\bot$, $\Omega_j(\omega) := \omega^{kdv}_j(\mu(\omega), 0)$.
One has 
$
\Pi \setminus \Pi_\gamma^{(3)}  \subset  
\bigcup_{\begin{subarray}{c}
\ell \in \Z^{S_+}, j_1, j_2, j_3 \in S^\bot\\
 j_k + j_m \neq 0, \forall  k, m \in \{ 1,2,3\}
\end{subarray}} 
R_{\ell j_1 j_2 j_3}(\gamma)\,, $ where
$$
R_{\ell j_1 j_2 j_3}(\gamma)  = \Big\{ \omega \in \Pi : |\omega \cdot \ell + 
\Omega_{j_1}(\omega) + \Omega_{j_2}(\omega) + \Omega_{j_3}(\omega)| 
< \frac{\gamma}{\langle \ell \rangle^\tau \langle j_1 \rangle^2 \langle j_2 \rangle^2 \langle j_3 \rangle^2} \Big\}\,. 
$$
First we need to establish the following regularity properties and asymptotics for the normal frequencies $\Omega_j(\omega)$, $j \in S^\bot$.
\begin{lemma}\label{properties Omega}
The map
$$
\Omega^*: \Pi \to \ell^\infty(S^\bot, \mathbb R), \, \omega \mapsto (\Omega^*_j(\omega))_{j \in S^\bot}\, , \qquad
\Omega^*_j(\omega) := j \big( \Omega_j(\omega) - (2\pi j)^3 \big)\, ,
$$
is real analytic. Furthermore, uniformly on a complex neighborhood of $\Pi$ in $\C^{S_+}$,
\begin{equation}\label{asintotica Omega j}
\Omega_j(\omega) = (2 \pi j)^3 + O(j^{- 1}) \,  \quad \text{as }  j \to \pm \infty\, .
\end{equation}
\end{lemma}
\begin{proof}
Since by \cite[Theorem 1.2 (i)]{KM},
$\Xi \to \ell^\infty(S^\bot_+, \R),  \, I \mapsto  (\omega^{kdv}_j(I, 0))_{j \in S_+^\bot}$
is real analytic and since by \cite[Theorem 1.2 (iii)]{KM}
$$
\Xi \to \ell^\infty(S^\bot_+, \R), \, I \mapsto  \big(j (\omega^{kdv}_j(I, 0) - (2\pi j)^3) \big)_{j \in S_+^\bot}
$$
is locally bounded in a complex neighborhood of $\Pi$ in $\C^{S_+}$, it follows from \cite[Theorem A.3]{KP} that the latter map is real analytic.
Furthermore, by  \cite[Theorem 15.4]{KP}, the action to frequency map
$$
\Xi \to \Pi, \, I = (I_j)_{j \in S_+} \mapsto (\omega_j^{kdv}(I, 0))_{j \in S_+}
$$
is real analytic and by the definition of $\Xi$ and $\Pi$, it is a diffeomorphism. Hence its inverse $\mu: \Pi \to \Xi, \omega \mapsto \mu(\omega)$ 
is also a real analytic diffeomorphism. 
Since for any $\omega \in \Pi$ and $j \in S^\bot$, $\Omega_j(\omega) = \omega^{kdv}_j(\mu(\omega), 0)$ and  $\Omega_j(\omega) = - \Omega_{- j}(\omega)$
we altogether have proved that the composition 
$$
\Omega^* : \Pi \to \ell^\infty(S^+, \mathbb R), \, \omega \mapsto (j (\omega^{kdv}_j(\mu(\omega), 0) - (2\pi j)^3)_{j \in S^\bot}
$$
is real analytic. Since $\Pi \subset \R^{S_+}$ is compact, $\Omega^*$ is actually bounded on a complex neighborhood of $\Pi$ in $\C^{S_+}$
and hence the claimed asymptotics hold.
\end{proof}

\begin{lemma}\label{lemma easy resonant sets}
There exist constants $C_0 > 0$ and $C_1 > 0$ so that for any $j_1, j_2, j_3 \in S^\bot$ and any $\ell \in \Z^{S_+}$ with $|\ell| \geq C_1$
$$
|R_{\ell j_1 j_2 j_3}(\gamma)| \le C_0 \frac{\gamma}{\langle \ell \rangle^\tau \langle j_1 \rangle^2 \langle j_2 \rangle^2 \langle j_3 \rangle^2}\,. 
$$
\end{lemma}
\begin{proof}
Let $\ell \in \Z^{S_+} \setminus \{ 0 \}$. Choose $v \in \R^{S_+}$ with $v \cdot \ell = 0$ and introduce 
$s \mapsto  \omega(s)  := s \frac{\ell}{|\ell|} + v$. Then $\ell \cdot \omega(s) = s |\ell |$ and hence for any $j_1, j_2, j_3 \in S^\bot$
and any $s \in \R$ with $\omega(s) \in \Pi,$
$$
\vphi (s) := \ell \cdot \omega(s) +  \Omega_{j_1}(\omega(s)) +  \Omega_{j_2}(\omega(s)) +  \Omega_{j_3}(\omega(s))
=  s |\ell | +  \Omega_{j_1}(\omega(s)) +  \Omega_{j_2}(\omega(s)) +  \Omega_{j_3}(\omega(s)).
$$
By Lemma \ref{properties Omega} and Cauchy's theorem there exists $C >0$, independent of $j_1, j_2, j_3 \in S^\bot$, so that 
$$
\big| \frac{d}{ds} \big( \Omega_{j_1}(\omega(s)) +  \Omega_{j_2}(\omega(s)) +  \Omega_{j_3}(\omega(s)) \big) \big|  \le C\, .
$$
It then follows that $|\vphi'(s)| \geq 1$ for any $|\ell | \geq C_1:= C+1$. This implies the claimed estimate. 
\end{proof}

\begin{lemma}\label{Lemma three wave res}
There exist constants $C_0 > 0$, $C_2 > 0$ so that for $j_1, j_2, j_3 \in S^\bot$ with $\min \{ |j_1|, |j_2|, |j_3| \} \geq C_2$ 
one has 
\begin{equation}\label{measure easy}
R_{0 j_1 j_2 j_3} (\gamma) = \emptyset\, , \qquad 
|R_{\ell j_1 j_2 j_3}(\gamma)| \le C_0 \frac{\gamma}{\langle \ell \rangle^\tau \langle j_1 \rangle^2 \langle j_2 \rangle^2 \langle j_3 \rangle^2}\, ,
\quad  \forall \, \ell \in \Z^{S_+}\setminus\{0\}\, .
\end{equation}
\end{lemma}
\begin{proof}
First we consider the case $\ell = 0$.
By the asymptotics \eqref{asintotica Omega j} it follows that for any $j_1, j_2, j_3 \in S^\bot$, 
$$
|\Omega_{j_1} + \Omega_{j_2} + \Omega_{j_3}| \geq 8\pi^3 |j_1^3 + j_2^3 + j_3^3| - \frac{C}{\min \{|j_1|, |j_2|, |j_3| \}}
$$
for some constant $C > 0$. By the case $n=3$ of Fermat's Last Theorem (cf. \cite{Euler})
$$
|j_1^3 + j_2^3 + j_3^3| \geq 1\, .
$$
Requesting that ${\rm min}\{|j_1|, |j_2|, |j_3| \} \geq C_2:= 2C$, one gets 
$|\Omega_{j_1} + \Omega_{j_2} + \Omega_{j_3}|  \geq 4\pi^3$ and hence
$R_{0 j_1 j_2 j_3}(\gamma) = \emptyset$ for any such $j_1, j_2 , j_3$ in $S^\bot$.

Now let us consider the case $\ell \in \Z^{S_+} \setminus \{ 0 \}$. For  any given $j_1, j_2, j_3 \in S^\bot$, 
define $s \mapsto \varphi(s)$ as in the proof of Lemma \ref{lemma easy resonant sets},
$$
\vphi (s) := |\ell | s + \Omega_{j_1}(\omega(s)) + \Omega_{j_2}(\omega(s)) + \Omega_{j_3}(\omega(s))\,. 
$$
By Lemma \ref{properties Omega} there exists $C >0$, independent of $j_1, j_2, j_3 \in S^\bot$, so that 
$$
\big| \frac{d}{ds} j_k \Omega_{j_k}(\omega(s)) \big|  \le C\, , \qquad \forall \, 1 \le k \le 3\, .
$$
By increasing $C_2$ if needed, it follows that for $j_1, j_2, j_3 \in S^\bot$ satisfying $\min\{ |j_1|, |j_2|, |j_3|\} \geq C_2$,
$$
|\vphi'(s)| \geq |\ell| - \frac{3C}{\min \{ |j_1|, |j_2|, |j_3|\}} \geq \frac12 \, .
$$
This implies the claimed measure estimate \eqref{measure easy}. 
\end{proof}
\begin{lemma}\label{lemma measure est case 2}
There exists a constant $C_3 \ge \max \{ C_2, C_1 \}$, where $C_2$ is the constant of Lemma \ref{Lemma three wave res} and $C_1$ 
the constant of Lemma \ref{lemma easy resonant sets}, so that 
$$
R_{\ell j_1 j_2 j_3} (\gamma) = \emptyset\, \qquad 
\forall \ell  \in \Z^{S_+} \text{ with } |\ell | < C_1\, \text{ and } \, \, \forall \, j_1, j_2, j_3 \in S^\bot \text{ satisfying } (*)
$$ 
where
$$  
(*) \qquad  j_k + j_m \neq 0, \quad 
\forall \,  k, m  \in \{ 1,2,3 \}\, , \qquad  \min \{ |j_1|, |j_2|, |j_3| \} < C_2 \, , \quad  \max \{ |j_1|, |j_2|, |j_3| \} \ge C_3 \, .
 $$
\end{lemma}
\begin{proof}
Let $\ell \in \Z^{S_+}$ with  $|\ell| \leq C_1$ and  $j_1, j_2, j_3 \in S^\bot$ with $\min \{ |j_1|, |j_2|, |j_3| \} \leq C_2$ 
and $j_k + j_m \neq 0$ for any $k, m \in \{ 1,2,3 \}$. 
First consider the case where $|j_2|, |j_3| < C_2$. By Lemma \ref{properties Omega} one then has
for $|j_1| \geq C_3$ with $C_3 > 0$ chosen large enough,
$$
|\omega \cdot \ell + \Omega_{j_1} + \Omega_{j_2} + \Omega_{j_3}| \, \geq
\, 8\pi^3 ( |j_1|^3 - |j_2|^3 - |j_3|^3) - C - |\omega| C_1 \,  \geq \, C_3^3 - 2 C_2^3 - C - |\omega| C_1 \geq 1 \, ,
$$
implying that $R_{\ell j_1 j_2 j_3}(\gamma) = \emptyset$. 

Let us now turn to the case where $|j_1| , |j_2| \geq C_3$ and $|j_3| \leq C_2$. If $j_1$ and $j_2$ have the same sign, then 
one concludes again that
$$
|\omega \cdot \ell + \Omega_{j_1} + \Omega_{j_2} + \Omega_{j_3}| \geq 
8 \pi^3(|j_1|^3 + |j_2|^3  - |j_3|^3) - C - |\omega| C_1 \geq 2 C_3^3 - C_2^3 - C - |\omega| C_1 \geq 1
$$
by increasing $C_3$ if needed. Hence again $R_{\ell j_1 j_2 j_3}(\gamma) = \emptyset$. Now assume that $j_1$ and $j_2$ do not have the same sign. 
Since by assumption,  
$j_1 + j_2 \neq 0$, one has $|j_1| - |j_2| \ne 0 $ and it then follows that 
$$
\begin{aligned}
|\omega \cdot \ell + \Omega_{j_1} + \Omega_{j_2} + \Omega_{j_3} | & \geq ||j_1|^3 - |j_2|^3| - |j_3|^3 - C - C_1 |\omega| \\
& \geq |(|j_1| - |j_2|)|(|j_1|^2 + |j_1||j_2| + |j_2|^2) - C_2^3 - C - C_1 |\omega|  \\
& \geq 3 C_3^2 - C_2^3 - C - C_1 |\omega| \geq 1
\end{aligned}
$$
by increasing $C_3$ once more if needed. We conclude that also in this case $R_{\ell j_1 j_2 j_3}(\gamma) =\emptyset$. 
\end{proof}

\noindent
{\em Proof of Proposition \ref{prop measure total}.}
As already mentioned, we concentrate on the proof of the claimed estimate for $| \Pi \setminus \Pi_\gamma^{(3)} | $.
In view of Lemma \ref{lemma easy resonant sets} -  Lemma \ref{lemma measure est case 2}, 
it remains to estimate the measure of the finite union  
$$
\bigcup_{\begin{subarray}{c}
|\ell| \leq C_1 \\
 |j_1|, |j_2|, |j_3| \leq C_3
 \end{subarray}} R_{\ell j_1 j_2 j_3}(\gamma)
$$
where $C_1 > 0$ is given by Lemma \ref{lemma easy resonant sets} and $C_3 > 0$ by Lemma \ref{lemma measure est case 2}.
By Lemma \ref{properties Omega},  for any $\ell \in \Z^{S_+}$, $j_1, j_2, j_3 \in S^\bot$ with $|\ell| \leq C_1$
and  $|j_1|, |j_2|, |j_3| \leq C_3$, the function
$$
\omega \mapsto \omega \cdot \ell + \Omega_{j_1}(\omega) + \Omega_{j_2}(\omega) + \Omega_{j_3}(\omega)
$$ 
is real analytic and by \cite[Proposition 15.5]{KP}, does not vanish identically. Hence by the Weierstrass preparation Theorem 
(cf. \cite[Lemma 9.7]{BKM}, \cite[Proposition 3.1]{BC}), for any given $C > 0$ there exists $\mathtt a \in (0, 1)$ so that
 $$
\big|  \bigcup_{\begin{subarray}{c}
|\ell| \leq C_1 \\
 |j_1|, |j_2|, |j_3| \leq C_3
 \end{subarray}} \big\{ \omega \in \Pi : |\omega \cdot \ell + \Omega_{j_1}(\omega) + \Omega_{j_2}(\omega) + \Omega_{j_3}(\omega)| \leq C \gamma \big\} \big|
 \lesssim \gamma^{\mathtt a} 
 $$
and the claimed estimate for $| \Pi \setminus \Pi_\gamma^{(3)} | $ follows.
\hfill $\square$

\begin{remark}\label{remark on Pi_gamma^4}
Note that there exist (many) non-trivial  solutions of the diophantine equation
\begin{equation}\label{diophantine}
j_1^3 + j_2^3 + j_3^3 + j_4^3 = 0 
\ee
where $(j_1, j_2, j_3, j_4) \in \Z^4$ is said to be a trivial solution if there exist $1 \le \alpha < \beta \le 4$ so that $j_\alpha = - j_\beta$. 
The following example was suggested by Michela Procesi,
$$
(10)^3 + 9^3 + (-1)^3 + (-12)^3 = 0 \, .
$$
We therefore expect that Lemma \ref{Lemma three wave res} does not extend to the sets $R_{\ell j_1 j_2 j_3 j_4} (\gamma)$, defined as
$$
R_{\ell j_1 j_2 j_3 j_4}(\gamma)  : =
 \Big\{ \omega \in \Pi : \big| \omega \cdot \ell + \sum_{k=1}^4\Omega_{j_k}(\omega) \big|
<  \frac{\gamma}{\langle \ell \rangle^\tau \langle j_1 \rangle^2 \langle j_2 \rangle^2 \langle j_3 \rangle^2 \langle j_4 \rangle^2} \Big\}
$$
and hence that an estimate for $|\Pi \setminus \Pi^{(4)}_\gamma|$ of the type as in Proposition \ref{prop measure total} for $|\Pi \setminus \Pi^{(3)}_\gamma|$   
does not hold. Here $ \Pi_\gamma^{(4)}$ is defined as
$$
\begin{aligned}
 \Pi_\gamma^{(4)} & := \big\{ \omega \in \Pi  \ : \ |\omega  \cdot \ell + \sum_{k = 1}^4 \Omega_{j_k}(\omega) | 
 \geq \frac{\gamma}{\langle \ell \rangle^\tau \langle j_1 \rangle^2 \langle j_2 \rangle^2 \langle j_3 \rangle^2  \langle j_4 \rangle^2}  
 \nonumber \\
& \qquad \forall (\ell, j_1, j_2, j_3, j_4) \in \Z^{S_+} \times (S^\bot)^4  \ \text{with} \  j_k + j_m \neq 0 \ \  \forall k, m \in \{1,2,3, 4\}  \big\} \, .
\end{aligned}
$$
\end{remark}


\appendix

\section{Linear vector fields on $H^s_\bot(\T_1)$}\label{Appendix A}
In this appendix we discuss properties of linear vector fields on  $H^s_\bot(\T_1)$, used throughout the main body of the paper.
Let $X$ be an unbounded linear vector field on $H^s_\bot(\T_1)$, $s \in \N$, with domain $H^{s + 1}_\bot(\T_1)$,
$$
X : H^{s}_\bot(\T_1) \to H^{s-1}_\bot(\T_1)\, ,
$$
which admits an expansion of order $N \in \N$,
\begin{equation}\label{vec field 1 appendix}
X [w] =  \sum_{k = 0}^{N+1} \lambda_{1 - k} \partial_x^{1 - k} w + {\cal R}_N [w]\, , \qquad  \lambda_{1 - k} \in \R, \quad   \forall \, 0 \le k \le N + 1\, ,
\end{equation}
where the remainder ${\cal R}_N$ is $(N +1)$-regularizing,  ${\cal R}_N \in {\cal B}(H^s_\bot(\T_1), H^{s + N +1}_\bot(\T_1))$.
If in addition, $X$ is a Hamiltonian linear vector field on $H^s_\bot(\T_1)$, 
$$
X[w] = \partial_x \nabla H[w]\, , \qquad H(w) := \frac12\int_0^1 A[w] \cdot w d x  \, , \quad \forall w \in H^s_\bot(\T_1),
$$
where $A: H^s_\bot(\T_1) \to H^s_\bot(\T_1)$ is a symmetric, bounded linear operator,
then the diagonal matrix elements $X_j^j$ of $X$ satisfy
\begin{equation}\label{diag purely imaginary}
X_j^j = \int_0^1 \partial_x A [e^{\ii 2\pi j x}] \cdot e^{-\ii 2\pi j x} d x \, \in  \,\ii \R, \qquad \forall j \in S^\bot . 
\end{equation}
\begin{lemma}\label{lemma parte diagonale}
Let $X$ be a vector field as in \eqref{vec field 1 appendix} and assume that its diagonal matrix elements satisfy $X_j^j \in \ii \R$ for any $j \in S^\bot$. 
Then $\lambda_{1 - k} = 0$ for any  $0 \le k \le N + 1$ with $1 - k$ even and $({\cal R}_N)_j^j \in \ii \R$ for any $j \in S^\bot$.
\end{lemma}
\begin{proof}
It follows from the assumptions that for any $j \in S^\bot$, 
$$
X_j^j = - \overline X_j^j \, , \qquad
X_j^j = \sum_{k = 0}^{N+1} \lambda_{1 - k} (\ii 2 \pi j)^{1 - k} + ({\cal R}_N)_j^j  \quad \text{with} \quad \lambda_{1-k} \in \R \, , \quad   ({\cal R}_N)_j^j = O(j^{- N - 1})\, .
$$ 
One thus concludes that 
$$
 \sum_{k = 0}^{N+1} \lambda_{1 - k} (\ii  2\pi  j)^{1 - k} + O(j^{- N -1}) = - \sum_{k = 0}^{N+1} \lambda_{1 - k}(- 1)^{1 - k} (\ii 2\pi j)^{1 - k} + O(j^{- N -1})
$$
and hence $\lambda_{1 - k} = 0$ for any $0 \le k \le N +1$ with $1 - k$ even. This implies that 
$$
({\cal R}_N)_j^j  = X^j_j - \sum_{k = 0}^{N+1} \lambda_{1 - k} (\ii  2\pi  j)^{1 - k} \in \ii \R\, , \qquad \forall \, j \in S^\bot  .
$$ 
\end{proof}
Consider a vector field $X : H^{s}_\bot(\T_1) \to H^{s-1}_\bot(\T_1)$, admitting an expansion of order $N$ of the form
\begin{equation}\label{vec field 2 appendix}
X[w]= \Pi_\bot \sum_{k = 0}^{N +1} T_{a_{1  - k}} \partial_x^{1 - k} w + {\cal R}_N[w], 
\qquad a_{1 - k}\in  H^s(\T_1)\,, \quad   \forall \, 0 \le k \le N + 1\, ,
\end{equation}
where the remainder ${\cal R}_N$ is $(N+1)$-regularizing,  ${\cal R}_N \in {\cal B}(H^s_\bot(\T_1), H^{s + N + 1}_\bot(\T_1))$.
\begin{lemma}\label{lemma 2 campi lineari}
Let $X$ be a vector field as in \eqref{vec field 2 appendix} and assume that $X_j^j \in \ii \R$ for any $j \in S^\bot$. 
Then $\langle a_{1 - k } \rangle_x = 0$ for any $0 \le k \le N + 1$ with $1-k$ even and $({\cal R}_N)_j^j \in \ii \R$ for any $j \in S^\bot$. 
\end{lemma}
\begin{proof}
For any $j \in S^\bot$, a direct calculation shows that 
$$
X_j^j = \sum_{k = 0}^{N+1}\lambda_{1 - k}  (\ii 2 \pi j)^{1 - k} + ({\cal R}_N)_j^j, \qquad  \lambda_{1 - k} := \langle a_{1 - k} \rangle_x \in \R, \quad \forall 0 \le k \le N + 1\,. 
$$
Since by assumption $X_j^j$ is purely imaginary, the claimed results then follow from Lemma \ref{lemma parte diagonale}. 
\end{proof}

\section{Standard results on homological equations}\label{Appendix B}

In this appendix we record two standard results on homological equations, used in our normal form procedure. 
Without further reference, we use the notations introduced in the paragraph {\em Notations and terminology} in Section \ref{introduzione paper}.
\begin{lemma}\label{eq omo zero melnikov abstract}
Let $\gamma \in (0, 1)$, $\tau > 0$, and $\omega \in \R^{S_+}$. Assume that 
$$
|\omega \cdot \ell | \geq \frac{\gamma}{|\ell|^\tau}, \qquad \forall \ell \in \Z^{S_+} \setminus \{ 0 \} ,
$$
and that ${\cal P} \in C^\infty(\T^{S_+}, B)$ where  $B$ is a Banach space with norm $\| \cdot \|_B$. 
Then there exists a unique solution ${\cal F} \in C^\infty (\T^{S_+}, B)$  with zero average of
$$
\omega \cdot \partial_\theta \, {\cal F}(\theta) + {\cal P}(\theta) = \langle {\cal P} \rangle_\theta\,, \qquad
\langle {\cal P} \rangle_\theta := \int_{\T^{S_+}} \mathcal F(\theta) d \theta = 0 \, .
$$
It is denoted by 
${\cal F}(\theta)  = - 
(\omega \cdot \partial_\theta)^{- 1}\big( {\cal P}(\theta) - \langle {\cal P} \rangle_\theta \big)$. 
\end{lemma}
\begin{lemma}\label{eq homo astratta prime melnikov}
Let $\Omega_\bot :  L^2_\bot(\T_1) \to L^2_\bot(\T_1)$
be a (possibly unbounded) Fourier multiplier of diagonal form, 
$\Omega_\bot [w] := \sum_{n \in S^\bot} \Omega_n w_n e^{\ii 2\pi n x}$,
and let $0 < \gamma < 1$, $\tau > 0$, and $\omega \in \R^{S_+}$.  Assume that 
$$
|\omega \cdot \ell + \Omega_n| \geq \frac{\gamma}{\langle \ell \rangle^\tau}, \qquad \forall \, (\ell, n) \in \Z^{S_+} \times S^\bot ,
$$
and that ${\cal P} \in C^\infty(\T^{S_+}, H^s_\bot(\T_1))$ for any $s \geq 0$. 
Then there exists a unique solution ${\cal F} \in C^\infty(\T^{S_+}, H^0_\bot(\T_1))$ of the equation 
$$
\big( \omega \cdot \partial_\theta + \ii \Omega_\bot \big) {\cal F}(\theta) + {\cal P}(\theta) = 0\,. 
$$
Furthermore, ${\cal F} \in C^\infty(\T^{S_+}, H^s_\bot(\T_1))$ for any $s \geq 0$. 
\end{lemma}


\section{A local existence result for $\partial_t \frak x = X(\frak x)$}\label{appendix loc well posed}

The goal of this appendix is to state a local existence result for the equation 
$\partial_t \frak x = X(\frak x)$ where $X$ is the vector field, introduced in Theorem \ref{teorema totale forma normale}
(cf. \eqref{forma campo vettoriale finale dopo NF}, \eqref{proprieta elementi campo vettoriale}),
\begin{equation}\label{PDE NF energy estimates C}
\begin{cases}
\partial_t \theta  = - \omega - \e \widehat \omega + \mathtt N^{(\theta)}(y, w) + {\cal O}_3^{(\theta)}(\frak x) \\
\partial_t y = {\cal O}_3^{(y)}(\frak x) \\
\partial_t w= \ii \Omega_\bot w + {\mathtt  D}^\bot(\frak x)[w] + \Pi_\bot T_{a(\frak x)} \partial_x w + {\cal R}^\bot(\frak x)
\end{cases}
\end{equation}
where we assume that the assumptions of Theorem \ref{teorema totale forma normale} are satisfied. In particular, 
$\omega \in \Pi_\gamma$, $0 < \gamma < 1$.
 This local existence result is used in Section \ref{conclusioni forma normale}. It reads as follows. 
\begin{proposition}\label{local existence theorem}
There exists $\sigma_\ast > 0$ (large) so that for any integer $s \ge \s_*$, there exist $0 < \e_0 \equiv \e_0(s, \gamma) < 1$ (small) 
and $C_* = C_*(s, \gamma) > 1$ (large)  with the following property: for any $0 < \e \le \e_0$, there exists 
$T= T_{\e, s, \gamma} > 0 $ so that for any initial data $\frak x_0 = (\theta_0, y_0, w_0) \in \T^{S_+} \times \R^{S_+} \times H^s_\bot(\T_1)$ 
with 
\begin{equation}\label{dati iniziali cauchy prob C}
 | y_0| \le  \e \,  , \quad   \| w_0 \|_s \leq \e \, ,
\end{equation}
there exists a unique solution 
$\frak x(t) = (\theta(t), y(t), w(t))$, $t \in [-T, T]$, of \eqref{PDE NF energy estimates C}
with $\frak x(0) = \frak x_0$ satisfying 
\begin{equation}\label{sol locali sistema NF fin C}
 \theta \in C^1([- T, T], \T^{S_+}),  \quad y \in C^1([- T, T], \R^{S_+}), \quad w \in C^0([- T, T], H^s_\bot(\T_1)) \cap C^1([- T, T], H^{s - 3}_\bot(\T_1)). 
 \end{equation}
 Furthermore, 
 \begin{equation}\label{estimate solution C}
 | y(t) |, \, \| w (t) \|_s\, , \,  |\Theta(t)| \,  \leq C_*  \e \qquad \forall t \in [- T, T]\,, 
\end{equation}
where
\begin{equation}\label{def Theta fin C}
\begin{aligned}
&  \Theta (t) := \theta (t) - \theta_0 + (\omega + \e \widehat \omega) t - \int_0^t \mathtt N^{(\theta)}(y(\tau), w(\tau))\, d \tau.  
 \end{aligned}
\end{equation}
\end{proposition}
The rest of this appendix is devoted to the proof of Proposition \ref{local existence theorem}, which is based on an iterative scheme. 
For any given $\frak x_0$ satisfying \eqref{dati iniziali cauchy prob C},
define inductively a sequence $ \frak x^{(n)}(t) = (\theta^{(n)}(t), y^{(n)}(t), w^{(n)}(t))$, $n \geq 0,$ as follows:
\begin{equation}\label{theta y w 0 (t) C}
\frak x^{(0)}(t) = ( \theta^{(0)}(t),  y^{(0)}(t), w^{(0)}(t)) : = \frak x_0 = (\theta_0, y_0,  w_0)
\end{equation}
whereas for $n \ge 1$, $\frak x^{(n)}(t) = (\theta^{(n )}(t), y^{(n )}(t), w^{(n )}(t))$ is defined to be the solution (cf. Lemma \ref{proprieta yn wn C} below) of 
\begin{equation}\label{PDE NF energy estimates approx C}
\begin{aligned}
\begin{cases}
\partial_t \theta^{(n )} = - \omega - \e \widehat \omega + \mathtt N^{(\theta)}(y^{(n )}, w^{(n )}) + {\cal O}_3^{(\theta)}({\frak x}^{(n-1)}),  \\
\partial_t y^{(n )}  = {\cal O}_3^{(y)}({\frak x}^{(n -1)}), \\
\partial_t w^{(n)} = \ii \Omega_\bot w^{(n )} + {\mathtt  D}^\bot({\frak x}^{(n-1)})[w^{(n )}] + \Pi_\bot T_{a({\frak x}^{(n-1)})} \partial_x w^{(n )} + {\cal R}^\bot({\frak x}^{(n-1)}),
\end{cases} 
\end{aligned}
\end{equation}
with initial data  $\frak x^{(n)}(0) = \frak  x_0$.
The following lemma holds. 
\begin{lemma}\label{proprieta yn wn C}
There exists $\sigma_\ast > 0$ (large) 
so that for any integer $s \ge \s_*$, there exist 
$\e_0 \equiv \e_0(s, \gamma) > 0$ (small) and $C_\ast \equiv C_*(s, \gamma) > 1$ (large) with the following property: for any
$0 < \e \le \e_0(s, \gamma)$, there exists $T= T_{\e, s, \gamma} > 0 $
 so that for any initial data $\frak x_0 = (\theta_0, y_0, w_0) \in \T^{S_+} \times \R^{S_+} \times H^s_\bot(\T_1)$ 
satisfying \eqref{dati iniziali cauchy prob C} and for any integer $n \geq 0$, the system \eqref{PDE NF energy estimates approx C}
admits a unique solution, satisfying $ \theta^{(n)} \in C^1([- T, T], \, \T^{S_+}),$ $ y^{(n)} \in C^1([- T, T], \, \R^{S_+})$, and 
\begin{equation}\label{sol locali sistema NF C}
 w^{(n)} \in C^0([- T, T], \, H^s_\bot(\T_1)) \cap C^1([- T, T], \, H^{s - 3}_\bot(\T_1)).
\end{equation}
Furthermore,
\begin{equation}\label{sol locali sistema NF C 1}
 | y^{(n)}(t) |, \ \| w^{(n)} (t) \|_s\, , \  |\Theta^{(n)}(t)| \,  \leq C_\ast  \e \, , \qquad \forall t \in [- T, T]\,, 
\end{equation}
where $ \Theta^{(0)}(t) := 0$ and
\begin{equation}\label{def Theta C}
  \Theta^{(n)} (t) := \theta^{(n)} (t) - \theta_0 + (\omega + \e \widehat \omega) t - \int_0^t \mathtt N^{(\theta)}(y^{(n - 1)}(\tau), w^{(n - 1)}(\tau))\, d \tau, 
 \quad n \geq 1\,.
\end{equation}
\end{lemma}
\begin{proof}
 We prove the claimed results by induction on $n$. For $n = 0$, by the definition \eqref{theta y w 0 (t) C} of $\frak x^{(0)}(t)$, the claimed statement holds 
 with $T = 1$ and with $ \sigma_\ast$, $\e_0$ given as in Theorem \ref{teorema totale forma normale}.
 Now assume that the claimed statement holds at the step $n\ge 0$ of the induction 
 and let us prove it at the step $n + 1$. 
We first need to make some preliminary considerations.
Let $s \geq  \sigma_\ast$ and $0 < \e \le \e_0$.
 Since by \eqref{proprieta elementi campo vettoriale}, $a$ is small of order two and 
 ${\cal R}^\bot$, $ {\cal O}_3^{(\theta)}$, ${\cal O}_3^{(y)}$ are small of order three, 
it follows from Theorem \ref{teorema totale forma normale} and the estimates \eqref{sol locali sistema NF C 1}, which hold by the induction hypothesis, 
that there exists a constant $C_s \equiv  C_s( \gamma) > 0$, independent of $n$, so that for any $t \in [- T, T]$
\begin{equation}\label{stime epsilon quantita campo vett C}
\begin{aligned}
 \quad  |{\cal O}_3^{(\theta)}(\frak x^{(n)}(t))|\,,  \  |{\cal O}_3^{(y)}(\frak x^{(n)}(t))| \le C_s  \e^3, \qquad 
 \| a(\frak x^{(n)}(t)) \|_{\sigma_\ast} \le C_s \e^2, 
  \qquad \| {\cal R}^\bot(\frak x^{(n)}(t)) \|_s \le C_s \e^3 \,
\end{aligned}
\end{equation}
By the second equation in \eqref{PDE NF energy estimates approx C}, one has
$$
y^{(n + 1)}(t) = y_0 + \int_0^t {\cal O}_3^{(y)}({\frak x}^{(n)}(\tau))\, d \tau,
$$
implying that  
\begin{equation}\label{prop y n + 1 C}
y^{(n + 1)} \in C^1([- T, T], \, \R^{S_+}),   \qquad   
|y^{(n + 1)}(t)| \leq \e + T C_s \e^3  \leq  \,  C_* \e, \quad \forall \, t \in [-T, T] \, ,
\end{equation}
where we have chosen $T>0$ so that $T C_s \e^2 \le 1$ small enough. 
By \eqref{stime epsilon quantita campo vett C} it then also follows that
\begin{equation}\label{T a frak x n 1 C}
T \| a({\frak x}^{(n)}) \|_{\sigma_*} \leq T C_s \e^2  \le1\,.
\end{equation}
To solve the equation for $w^{(n + 1)}$ in \eqref{PDE NF energy estimates approx C}, 
we apply Lemma \ref{lemma positura lineare f} in Appendix \ref{Appendix C} with ${\cal D}(t) = \ii \Omega_\bot + {\mathtt D}^\bot({\frak x}^{(n)}(t))$,
 $a = a({\frak x}^{(n)}(t))$, and $f = {\cal R}^\bot({\frak x}^{(n)}(t))$ to conclude that there exists a unique solution $w^{(n + 1)}$ of
$$
\begin{cases}
\partial_t w^{(n + 1)} = \ii \Omega_\bot w^{(n + 1)} + {\mathtt  D}^\bot({\frak x}^{(n)})[w^{(n + 1)}] + 
\Pi_\bot T_{a({\frak x}^{(n)})} \partial_x w^{(n + 1)} + {\cal R}^\bot({\frak x}^{(n)}) \\
 w^{(n + 1)}(0) = w_0
\end{cases}
$$ 
in  $C^0([- T, T], \, H^s_\bot(\T_1)) \cap C^1([- T, T], \, H^{s - 3}_\bot(\T_1))$ and that $w^{(n+1)}$ satisfies
\begin{equation}\label{prop w n + 1}
\| w^{(n + 1)}(t) \|_s , \  \| \partial_t w^{(n + 1)}(t) \|_{s - 3} \lesssim_{s, \gamma}  
\e + T \| {\cal R}^\bot({\frak x}^{(n)}) \|_s \stackrel{ \eqref{stime epsilon quantita campo vett C}}{\lesssim_{s, \gamma} } \e + T  C_s \e^3 \leq C_* \e 
\end{equation}
since $T C_s\e^2 \le 1$. We then define 
\begin{equation}\label{def Theta n + 1 C}
 \Theta^{(n + 1)} (t) := \theta^{(n + 1)} (t) - \theta_0 + (\omega + \e \widehat \omega) t - \int_0^t \mathtt N^{(\theta)}(y^{(n + 1)}(\tau), w^{(n + 1)}(\tau))\, d \tau .
 \qquad t \in [- T, T].
\end{equation}
By the first equation in \eqref{PDE NF energy estimates approx C}, one gets 
$\Theta^{(n + 1)}(t) = \int_0^t {\cal O}_3^{(\theta)}({\frak x}^{(n)}(\tau))\, d \tau$
and hence, using again \eqref{stime epsilon quantita campo vett C}, 

\begin{equation}\label{prop theta n + 1 C}
\theta^{(n + 1)} \in C^1 ([- T, T], \T^{S_+}), \qquad |\Theta^{(n + 1)}(t)| \leq C_* \e, \quad \forall t \in [- T, T]\,. 
\end{equation}
This concludes the proof of the lemma. 
\end{proof}
In order to prove the convergence of the sequence $(\frak x_n(t))_{n \ge 0}$, constructed in Lemma \ref{proprieta yn wn C}, we prove 
\begin{lemma}\label{lemma n + 1 - n C}
Under the assumptions of Lemma \ref{proprieta yn wn C}, for any $n \geq 1$,
$$
\| {\frak x}^{(n )}(\cdot) - {\frak x}^{(n - 1)}(\cdot) \|_{{\cal C}^0_t E_{s - 1}}\,, \
\|\partial_t ({\frak x}^{(n )}(\cdot) - {\frak x}^{(n - 1)}(\cdot)) \|_{{\cal C}^0_t E_{s - 4}}  \leq  \, 2^{- n}\,. 
$$
\end{lemma}
\begin{proof} 
By \eqref{PDE NF energy estimates approx C},  $\widehat{\frak x}^{(n)}(t) = 
(\widehat\theta^{(n)}(t), \widehat y^{(n)}(t), \widehat w^{(n)}(t)) := 
{\frak x}^{(n)}(t) - {\frak x}^{(n -1)}(t)
$
satisfies 
\begin{equation}\label{equazione differenze conv C}
 \begin{cases}
\partial_t \widehat \theta^{(n)} = f^{(\theta, n)}\, , \\
\partial_t \widehat y^{(n)} = f^{(y, n)}\, ,  \\
\partial_t \widehat w^{(n)} = \ii \Omega_\bot \widehat w^{(n)} + {\mathtt D}^\bot({\frak x}^{(n)}) \widehat w^{(n)} + \Pi_\bot T_{a({\frak x}^{(n)})} \partial_x \widehat w^{(n)}+ f^{(\bot, n)}, 
\end{cases} 
\end{equation} 
with $\widehat{\frak x}^{(n)}(0) = (0, 0, 0)$, where 
\begin{equation}\label{fn eq difference C}
\begin{aligned}
& f^{(\theta, n)} := \mathtt N^{(\theta)}(y^{(n)}, w^{(n)}) - \mathtt N^{(\theta)}(y^{(n-1)}, w^{(n-1)}) + {\cal O}_3^{(\theta)}({\frak x}^{(n)}) - {\cal O}_3^{(\theta)}({\frak x}^{(n - 1)}) \,, \\
& f^{(y, n)} := {\cal O}_3^{(y)}({\frak x}^{(n)}) - {\cal O}_3^{(y)}({\frak x}^{(n - 1)})\,, \\
& f^{(\bot, n)} := \Big( {\mathtt D}^\bot({\frak x}^{(n)}) - {\mathtt D}^\bot({\frak x}^{(n - 1)}) \Big)[w^{(n - 1)}] + 
\Pi_\bot T_{a({\frak x}^{(n)}) - a({\frak x}^{(n - 1)})} \partial_x w^{(n-1)}+ {\cal R}^\bot({\frak x}^{(n)}) - {\cal R}^\bot({\frak x}^{(n - 1)}) \, .
\end{aligned}
\end{equation}
By the properties stated in \eqref{proprieta elementi campo vettoriale} and by the mean value theorem, for some $\sigma > 0$ large enough and $s \geq \sigma$, one can show that 
\begin{equation}\label{prop fn differenze C}
\begin{aligned}
& f^{(\theta, n)}, f^{(y, n)} \in C^0([- T, T], \R^{S_+}),  \qquad |f^{(\theta, n)}| \lesssim 
\| {\frak x}^{(n)} - {\frak x}^{(n - 1)} \|_{{C}^0_t E_\sigma}, \quad |f^{(y, n)}| \lesssim \e^2 \| {\frak x}^{(n)} - {\frak x}^{(n-1)} \|_{{C}^0_t E_\sigma}  \,, \\
& f^{(\bot, n)} \in C^0([- T, T], H^{s - 1}_\bot(\T_1)), \qquad \| f^{(\bot, n)} \|_{C^0_t H^{s - 1}_x} \lesssim_s 
\e \| {\frak x}^{(n)} - {\frak x}^{(n - 1)} \|_{{C}^0_t E_{ s - 1}}\,.
\end{aligned}
\end{equation}
Hence we immediately conclude that for any $t \in [- T, T]$,
\begin{equation}\label{stima widehat theta n y n C}
|\widehat\theta^{(n)} (t)| \lesssim T \| {\frak x}^{(n)} - {\frak x}^{(n - 1)} \|_{{C}^0_t E_\sigma}, \qquad 
|\widehat y^{(n)}(t)| \lesssim T \e^2 \| {\frak x}^{(n)} - {\frak x}^{(n-1)} \|_{{C}^0_t E_\sigma} . 
\end{equation}
Furthermore, by applying Lemma \ref{lemma positura lineare f}, with ${\cal D}(t) := \ii \Omega_\bot + {\mathtt D}^\bot({\frak x}^{(n)})$, $a = a({\frak x}^{(n)})$, 
$f = f^{(\bot, n)}$, and by the estimate \eqref{prop fn differenze C} for $f^{(\bot, n)}$, one also deduces that 
\begin{equation}\label{stima widehat wn C}
\| \widehat w^{(n)}(t) \|_{s - 1} \lesssim_s \e T \| {\frak x}^{(n)} - {\frak x}^{(n-1)} \|_{{\cal C}^0_t E_{s - 1}}, \quad \forall t \in [- T, T]\,. 
\end{equation}
Therefore, collecting \eqref{stima widehat theta n y n C}, \eqref{stima widehat wn C}, using the induction hypothesis, and by taking $T$ small enough, 
one gets $\| {\frak x}^{(n + 1)} - {\frak x}^{(n)} \|_{{\cal C}^0_t E_{s - 1}} \leq 2^{- (n + 1)}$ which is one of the two claimed estimates at the step $n + 1$. 
The estimate for $\partial_t ({\frak x}^{(n+1)} - {\frak x}^{(n)})$ can be proved in a similar fashion. 
\end{proof}
By Lemma \ref{lemma n + 1 - n C} and by a standard telescoping argument, one obtains
$$
\theta^{(n)} \to \theta, \quad  y^{(n)} \to y\,,\qquad \partial_t \theta^{(n)} \to \partial_t \theta\,,\quad  \partial_t y^{(n)} \to \partial_t y \quad \text{uniformly for} \quad - T \le t \le T \, .
$$
By the estimates \eqref{sol locali sistema NF C} and by passing to the limit as $n \to + \infty$, one then obtains the bounds \eqref{sol locali sistema NF fin C} for $\Theta(t)$ and $y(t)$.

\noindent
Furthermore, 
$$
w^{(n)} \to w \quad \text{in} \quad C^0([- T, T], \, H^{s - 1}_\bot) \cap C^1([- T, T], \, H^{s - 4}_\bot(\T_1))\,. 
$$
and $(\theta(t), y(t), w(t))$ is a smooth solution of \eqref{PDE NF energy estimates C}. Furthermore, arguing as at the end of the proof of Lemma \ref{lemma positura lineare}, 
one shows that 
$$
w \in C^0([- T, T], \, H^s_\bot(\T_1)) 
$$
and in turn, using the equation, that $\partial_t w \in C^0([- T, T], H^{s - 3}_\bot(\T_1))$. 
One also shows that $w(t)$ satisfies the claimed bound \eqref{sol locali sistema NF fin C} 
by using the bounds on $w^{(n)}$ in \eqref{sol locali sistema NF C 1}. To prove the uniqueness, take two smooth solutions ${\frak x}_1, {\frak x}_2$ 
satisfying the same initial condition ${\frak x}_1(0) = {\frak x}_0 = {\frak x}_2(0)$. 
Then write the equation for the difference ${\frak x}_1 - {\frak x}_2$ and argue as in the proof of Lemma \ref{lemma n + 1 - n C} to conclude that
$$
\| {\frak x}_1(t) - {\frak x}_2(t) \|_{E_{\sigma}} \lesssim \int_0^t \| {\frak x}_1(\tau) - {\frak x}_2(\tau) \|_{E_{\sigma}}\, d \tau, \quad \forall t \in [- T, T]
$$
for some $\sigma > 0$ (large). By the Gronwall Lemma,  ${\frak x}_1 = {\frak x}_2$. 
This concludes the proof of Proposition \ref{local existence theorem}.


\section{On a class of linear para-differential equations}\label{Appendix C}
In this appendix we discuss a well-posedness result for a linear para-differential equation of the form 
\begin{equation}\label{PDE lineare Hs bot con f}
\partial_t w = {\cal D}(t) [w] + \Pi_\bot T_a \partial_x w + f\,, \qquad x \in \T_1, \ t \in [-T, T] , 
\end{equation}
in the Sobolev space  $H^s_\bot(\T_1)$ for some integer $s \ge \sigma$ with $\sigma> 0$ sufficiently large.
Here the linear operator ${\cal D}(t)$ is a time-dependent Fourier multiplier of order $m \ge 1$, ${\cal D}(t) w( x) = \sum_{n \in S^\bot} d_n(t)w_n e^{\ii 2\pi n x}$
with
\begin{equation}\label{ipotesi cal D (t)} 
 {\cal D} \in C^0\big([- T, T], \, {\cal B}(H^{s }_\bot(\T_1), H^{s-m}_\bot(\T_1))  \big), 
\qquad  {\cal D}(t) = - {\cal D}(t)^\top, \quad \forall t \in [- T, T]\, ,
\end{equation}
and the coefficient $a(t, x)$ of the operator $T_a$ of para-multiplication by $a$ and the forcing term $f(t, x)$ satisfy 
\begin{equation}\label{ipotesi a pde lin}
a \in C^0\big( [- T, T], H^\sigma_\bot(\T_1) \big), \qquad 
f \in C^0 \big( [- T, T], H^s_\bot(\T_1) \big) .
\end{equation}
The main result of this appendix is Lemma \ref{lemma positura lineare f} which is used in the proof of Proposition \ref{local existence theorem}. 

\noindent
First we consider the initial value problem for equation \eqref{PDE lineare Hs bot con f} with vanishing forcing term, 
\begin{equation}\label{PDE lineare Hs bot}
\partial_t w = {\cal D}(t) [w] + \Pi_\bot T_a \partial_x w\,, \qquad w(\tau, \cdot) = w_0(\cdot) \, ,
\end{equation}
where the initial time $\tau$ is in $[- T, T]$. 
\begin{lemma}\label{lemma positura lineare}
There exists $\s \ge m $ (large) with the following property: Assume that for some  $0 < T \le 1$ and any $s \ge \s$, 
 \eqref{ipotesi cal D (t)} - \eqref{ipotesi a pde lin} hold and $\| a \|_{{\cal C}^0_t H^\sigma_x} \leq 1$. Then for any $w_0 \in H^s_\bot(\T_1)$, 
 there exists a unique solution $w$ of \eqref{PDE lineare Hs bot} in
$C^0([- T, T], H^s_\bot(\T_1)) \cap C^1([- T, T], H^{s - m}_\bot(\T_1))$. For any $t \in [- T, T]$, it satisfies the estimate 
\begin{equation}\label{stima PDE lin hom}
\begin{aligned}
& \| w(t) \|_s\,,\, \| \partial_t w(t) \|_{s - m} \lesssim_s   \| w_0 \|_s\,.
\end{aligned}
\end{equation}
\end{lemma}

\begin{proof}
The lemma is proved by constructing a sequence of approximating solutions. To this end we introduce
 for any integer $N \ge 1$ the finite dimensional subspace  $H_N$ of $L^2_\bot(\T_1)$,
\begin{equation}\label{HN}
H_N := \big\{ u \in L^2_\bot(\T_1) :  \, u(x) = \sum_{j \in S^\bot_N}  u_n e^{\ii 2 \pi n x}\big\}\, , \qquad  S^\bot_N:= S^\bot \cap [-N, N]\, ,
\end{equation}
and denote by $\Pi_N$ the corresponding $L^2-$orthogonal projector $\Pi_N : L^2_\bot(\T_1) \to H_N$. 
We consider the truncated equation
\begin{equation}\label{flow-propagator N}
 \partial_t w =  \Pi_N \big( {\cal D}(t) [w] + \Pi_\bot T_a \partial_x w  \big), \qquad   w(\tau, \cdot) = \Pi_N w_0\, ,
\end{equation}
where $w(t, x) = \sum_{n \in S^\bot_N} w_n(t) e^{\ii 2 \pi n x} \in H_N$.
The equation in \eqref{flow-propagator N}  is a linear non-autonomous ODE on the finite dimensional space $H_N$ 
and hence it admits a unique solution $w^{(N)} \in C^1([- T, T], H_N)$. 
We will show that the sequence $(w^{(N)})_{N \ge 1}$ admits a limit, which is the solution 
of \eqref{PDE lineare Hs bot} with the claimed properties. To this end, in a first step, we prove estimates for the Sobolev norm $ \| w^{(N)}(t) \|_{s}$.

\medskip

\noindent
{\sc Bound of $ \| w^{(N)}(t) \|_{s}$.} Note that $ \| w^{(N)}(t) \|_{s} =  \| \partial_x^s w^{(N)}(t) \| $.  
Since ${\cal D}(t)$ is a Fourier multiplier, the commutator $ [\partial_x^s, \, {\cal D}(t)] $ vanishes and since
 for any $v \in L^2_\bot(\T_1)$, 
$$
\big\langle \Pi_\bot u,\, v \big\rangle = \big\langle u,\, v \big\rangle \, , \ \ \forall \, u \in L^2(\T_1) \, , \qquad 
\big\langle \Pi_N v , \, g \big\rangle = \big\langle v,\, g \big\rangle\, , \ \ \forall \, g \in H_N, 
$$
one concludes that
\begin{align}
\partial_t \| \partial_x^s w^{(N)} \| & = 
 \big\langle \partial_x^s \big( {\cal D}(t) [w^{(N)}] + \Pi_N \Pi_\bot T_a \partial_x w^{(N)}  \big) , \, \partial_x^s w^{(N)} \big\rangle 
  + \big\langle \partial_x^s  w^{(N)} , \, \partial_x^s \big( {\cal D}(t) [w^{(N)}] + \Pi_N \Pi_\bot T_a \partial_x w^{(N)} \big)\big\rangle \nonumber  \\ 
 & = \big\langle  {\cal D}(t) \partial_x^s w^{(N)} ,  \, \partial_x^s w^{(N)} \big\rangle  + 
 \big\langle \partial_x^s w^{(N)}, \, {\cal D}(t)  \partial_x^s  w^{(N)}  \big\rangle \label{termine-prin} \\
 & \quad + \big\langle  \partial_x^s  T_a \partial_x w^{(N)} ,\, \partial_x^s w^{(N)} \big\rangle + 
 \big\langle  \partial_x^s w^{(N)} ,\,  \partial_x^s  T_a \partial_x w^{(N)} \big\rangle\,. \label{termine-prin-2}
\end{align}

\noindent
{\em Analysis of the terms in \eqref{termine-prin}.} Since by assumption ${\cal D}(t)^\top = - {\cal D}(t)$, one has
\begin{equation}\label{pezzo cal D (t)}
\begin{aligned}
& \big\langle  {\cal D}(t) \partial_x^s w^{(N)}, \, \partial_x^s w^{(N)} \big\rangle  + \big\langle \partial_x^s w^{(N)},  \, {\cal D}(t)  \partial_x^s  w^{(N)}  \big\rangle 
& = \big\langle \big( {\cal D}(t) + {\cal D}(t)^\top\big) \partial_x^s w^{(N)}, \, \partial_x^s w^{(N)} \big\rangle = 0\,. 
\end{aligned}
\end{equation}
{\em Analysis of the terms in \eqref{termine-prin-2}.} One computes
\begin{equation}\label{energy estimate 2}
\begin{aligned}
& \big\langle \partial_x^s  T_{a} \partial_x w^{(N)}\,,\, \partial_x^s w^{(N)} \big\rangle + \big\langle \partial_x^s w^{(N)} \,,\, \partial_x^s  T_{a} \partial_x w^{(N)}\big\rangle \\
& = \big\langle   T_{a} \partial_x \partial_x^s w^{(N)}\,,\, \partial_x^s w^{(N)} \big\rangle + \big\langle \partial_x^s w^{(N)} \,,\,   T_{a} \partial_x \partial_x^s w^{(N)}\big\rangle 
 + \big\langle   [\partial_x^s, T_{a} \partial_x]  w^{(N)}\,,\, \partial_x^s w^{(N)} \big\rangle + \big\langle \partial_x^s w^{(N)} \,,\,   [\partial_x^s, T_{a} \partial_x] w^{(N)}\big\rangle \\
& = \big\langle   \big( T_{a} \partial_x + (T_{a} \partial_x)^\top \big) \partial_x^s w^{(N)}\,,\, \partial_x^s w^{(N)} \big\rangle  
 + \big\langle   [\partial_x^s, T_{a} \partial_x]  w^{(N)}\,,\, \partial_x^s w^{(N)} \big\rangle + \big\langle \partial_x^s w^{(N)} \,,\,   [\partial_x^s, T_{a} \partial_x] w^{(N)}\big\rangle\,.
\end{aligned}
\end{equation}
By Corollary \ref{aggiunto nostri operatori} (with $N=1$, $m=1$) there exists an integer $\sigma \ge 1$ so that
$$
\|  T_{a} \partial_x + (T_{a} \partial_x)^\top\|_{{\cal B}(L^2)} \lesssim  \| a \|_{\sigma} 
$$
and hence by the Cauchy-Schwarz inequality, 
\begin{equation}\label{energy estimate 2a}
\begin{aligned}
|\big\langle   \big( T_{a} \partial_x + (T_{a} \partial_x)^\top \big) \partial_x^s w^{(N)}\,,\, \partial_x^s w^{(N)} \big\rangle | \lesssim \| a \|_\sigma \,  \| \partial_x^s w^{(N)} \|^2 \,. 
\end{aligned}
\end{equation}
Moreover, arguing as  in \cite[Lemma A.1]{BM}, one has 
$$
\| [\partial_x^s, \, T_{a} \partial_x]  w^{(N)}\| \lesssim_s \| a \|_2 \| w^{(N)} \|_s \stackrel{\sigma \geq 2}{\lesssim_s}  \| a \|_\sigma \,  \| \partial_x^s w^{(N)} \|\,.
$$
The latter estimate, together with the Cauchy-Schwarz inequality, imply that 
\begin{equation}\label{energy estimate 3}
\begin{aligned}
|\big\langle   [\partial_x^s, \, T_{a} \partial_x]  w^{(N)}, \, \partial_x^s w^{(N)} \big\rangle + \big\langle \partial_x^s w^{(N)} \,,\,  
 [\partial_x^s, \, T_{a} \partial_x] w^{(N)}\big\rangle| \lesssim_s \| a \|_\sigma \,  \| \partial_x^s w^{(N)} \|^2\,. 
\end{aligned}
\end{equation}
Using \eqref{energy estimate 2a}-\eqref{energy estimate 3}, one then infers from \eqref{energy estimate 2}
\begin{equation}\label{energy estimate 3b}
|\big\langle \partial_x^s  T_{a} \partial_x w^{(N)}, \, \partial_x^s w^{(N)} \big\rangle + \big\langle \partial_x^s w^{(N)} ,\, \partial_x^s  T_{a} \partial_x w^{(N)}\big\rangle| 
\lesssim_s \| a \|_\sigma \| \partial_x^s w^{(N)} \|^2 \, .
\end{equation}
Combining \eqref{termine-prin}, \eqref{termine-prin-2}, \eqref{pezzo cal D (t)}, \eqref{energy estimate 3b}, yields the estimate
\begin{equation}
|\, \partial_t \| \partial_x^s w^{(N)} \|^2| \lesssim_s  \| a \|_\sigma \| \partial_x^s w^{(N)} \|^2 , 
\end{equation}
which implies that 
\begin{equation}\label{est for gronwall}
\begin{aligned}
\| \partial_x^s w^{(N)} (t)\|_{L^2}^2 & \leq \| w_0 \|_{s}^2 + C(s) \Big| \int_\tau^t  \| a(t' ) \|_\sigma \| \partial_x^s w^{(N)}(t') \|^2\, d t' \Big| \\
& {\leq} \| w_0 \|_s^2  + C(s) \| a \|_{{\cal C}^0_t H^\sigma_x}  \Big| \int_\tau^t \| \partial_x^s w^{(N)}(t') \|^2\, d t'  \Big|
\end{aligned}
\end{equation}
for some constant $C(s) > 0$. The Gronwall Lemma (recall that $- T \le t,  \tau \le T$) then implies that 
$$
\| w^{(N)}(t) \|_s^2= \| \partial_x^s w^{(N)} \|^2 \le {\rm exp} \big(C(s) \| a \|_{{\cal C}^0_t H^\sigma_x} T \big) \| w_0 \|_s^2, \quad \forall t \in [- T, T]\,. 
$$
Since by assumption $0 < T \le 1$ and $ \| a \|_{{\cal C}^0_t H^\sigma_x}  \le 1$, it then follows that
\begin{equation}\label{uniform bound wN}
\| w^{(N)}(t) \|_s^2= \| \partial_x^s w^{(N)} \|^2 \le {\rm exp} (C(s)) \| w_0 \|_s^2, \quad \forall t \in [- T, T]\,. 
\end{equation}

 \medskip
 
 \noindent
 {\sc Convergence.} 
 Now we  pass to the limit $ N \to + \infty $. 
By \eqref{uniform bound wN} the sequence of functions $w^{(N)}$ is bounded in 
$C^0([- T, T], \, H^s_\bot(\T_1)) \subseteq L^\infty([- T, T], \, H^s_\bot(\T_1))$ and,  
 up to subsequences,  
 \begin{equation}\label{disuguaglianza liminf u uN}
 w^{(N)} \stackrel{w^*}{\rightharpoonup} w \ \  \text{in } \  L^\infty([- T, T], H^s_\bot(\T_1))  \, , 
\qquad  \| w \|_{L^\infty_{ t} H^s_x} \leq \liminf_{N \to + \infty} \| w^{(N)} \|_{L^\infty_{ t} H^s_x}\,.
 \end{equation} 
 {\em Claim:}  $( w^{(N)})_{N \ge 1}$ converges to $w$ in $C^0([- T, T],  \, H^{s}_\bot(\T_1))  \cap C^1([- T, T], \, H^{ s - m}_\bot(\T_1))  $, 
 and $ w$ solves \eqref{PDE lineare Hs bot}. 
  \\[1mm]
We first prove that $w^{(N)}$ is a Cauchy sequence in $C^0([- T, T], L^2_\bot(\T_1) )$. Indeed, by \eqref{flow-propagator N},  
    the difference $h^{(N)} := w^{(N + 1)} - w^{(N)}$ solves 
    $$
    \partial_t h^{(N)} = {\cal D}(t) h^{(N)}  + \Pi_{N+ 1} ( \Pi_\bot T_a \partial_x h^{(N)} ) +(\Pi_{N + 1} - \Pi_N) \Pi_\bot T_a \partial_x  w^{(N)}\,, 
    \qquad  h^{(N)}(\tau) = (\Pi_{N + 1} - \Pi_N) w_0 \, , 
    $$ 
    and therefore
    \begin{align}
    \partial_t \| h^{(N)}(t)\|^2 & = \big\langle \partial_t h^{(N)}, \,h^{(N)} \big\rangle+ \big\langle  h^{(N)}, \,\partial_t h^{(N)} \big\rangle  \nonumber\\
    & =  \big\langle {\cal D}(t) h^{(N)}, \, h^{(N)} \big\rangle + \big\langle h^{(N)}, \, {\cal D}(t) h^{(N)} \big\rangle  
     + \big\langle  T_a \partial_x h^{(N)} ,\, h^{(N)} \big\rangle + \big\langle  h^{(N)},\,  T_a \partial_x h^{(N)} ) \big\rangle \nonumber\\
    & \quad + \big\langle  (\Pi_{N + 1} - \Pi_N)\Pi_\bot T_a \partial_x  w^{(N)}, \, h^{(N)} \big\rangle 
    + \big\langle h^{(N)}, \, (\Pi_{N + 1} - \Pi_N) \Pi_\bot T_a \partial_x  w^{(N)} \big\rangle \,. \label{coccodrillo 0}
    \end{align}
    Arguing as in \eqref{pezzo cal D (t)}, \eqref{energy estimate 2} - \eqref{energy estimate 3b}, one gets 
    \begin{equation}\label{coccodrillo 100}
    \begin{aligned}
    & \big\langle {\cal D}(t) h^{(N)}, \, h^{(N)} \big\rangle + \big\langle h^{(N)}, \, {\cal D}(t) h^{(N)} \big\rangle = 0\,, \\
    & | \big\langle  T_a \partial_x h^{(N)} ,\, h^{(N)} \big\rangle + \big\langle  h^{(N)} , \,  T_a \partial_x h^{(N)} )|  \lesssim \| a \|_\sigma \| h^{(N)} \|^2\,. 
    \end{aligned}
    \end{equation}
  Moreover 
  \begin{align}
  &| \big\langle  (\Pi_{N + 1} - \Pi_N)\Pi_\bot T_a \partial_x  w^{(N)}, h^{(N)} \big\rangle + \big\langle h^{(N)}, (\Pi_{N + 1} - \Pi_N) \Pi_\bot T_a \partial_x  w^{(N)} \big\rangle|  \nonumber\\
  & \lesssim  \| (\Pi_{N + 1} - \Pi_N) \Pi_\bot T_a \partial_x  w^{(N)}  \| \| h^{(N)}\|  
  \lesssim  
  \| h^{(N)}\|^2 + \| (\Pi_{N + 1} - \Pi_N) \Pi_\bot T_a \partial_x  w^{(N)} \|^2  \nonumber\\
  & \lesssim   \| h^{(N)}\|^2 + \big( N^{-2} \|  T_a \partial_x  w^{(N)}  \|_2 \big)^2  
  \stackrel{\eqref{stima elementare paraproduct}, \eqref{uniform bound wN}}{\lesssim} \| h^{(N)}\|^2  + \big( N^{-2} \| w_0 \|_3 \big)^2 
\stackrel{\s \geq 3}{\lesssim}   \| h^{(N)}\|^2  + \big( N^{-2} \| w_0  \|_{\sigma} \big)^2\,.  \label{coccodrillo 2}
  \end{align}
  Hence \eqref{coccodrillo 0}-\eqref{coccodrillo 2} imply that 
  $$  
  \partial_t \| h^{(N)}(t)\|^2 \lesssim  \| h^{(N)}(t)\|^2 + N^{-4}\| w_0\|_{\sigma}^2 
  $$ 
 and, since $ \| h^{(N)}(\tau)\| \leq N^{- 2} \| w_0\|_{2}$, we deduce from the Gronwall Lemma that
$$
\| w^{(N + 1)} - w^{(N)}\|_{{\mathcal C}^0_{ t} L^2_x} \lesssim   N^{-2} \| w_0\|_\sigma  \, {\rm exp}(C T)^{\frac12}
 $$
 for some constant $C > 0$. The above inequality, together with a standard telescoping argument implies that $ w^{(N)} $ is a Cauchy sequence in 
 $ C^0([- T, T],  \, L^2_\bot(\T_1)) $. 
Hence $ w^{(N)} \to \tilde w \in C^0([- T, T],  \, L^2_\bot(\T_1))  $. By \eqref{disuguaglianza liminf u uN} 
we have $ \tilde w = w  \in C^0([- T, T],  \, L^2_\bot(\T_1))  \cap L^\infty([- T, T],  \, H^s_\bot(\T_1)) $. 
 Next, for any $ \bar s \in [0, s) $ one has by the interpolation inequality 
  $$
  \| w^{(N)} - w \|_{L^\infty_{t} H^{\bar s}_x } \leq 
  \| w^{(N)} - w \|_{L^\infty_{ t} L^2_x }^{1 - \lambda}  \, \| w^{(N)} - w \|_{L^\infty_{ t} H^{ s}_x }^\lambda \, , \qquad \lambda:=  \bar s / s, 
$$
and, since  $ w^{(N)} $ is bounded in $ L^\infty([- T, T], \, H^s_\bot(\T_1)) $ (see \eqref{uniform bound wN}),
$ w  \in L^\infty([- T, T], \, H^s_\bot(\T_1)) $, 
and $ w^{(N)} \to w \in C^0([- T, T],  \, L^2_\bot(\T_1))$, we deduce that  $w^{(N)} \to w$ in $ C^0([- T, T], \, H^{\bar s}_\bot(\T_1))$.
  Moreover  we deduce 
$$
  \partial_t w^{(N)} = \Pi_N \big( {\cal D}(t) [w^{(N)}] + \Pi_\bot T_a \partial_x w^{(N)}  \big) \to {\cal D}(t) [w] 
  + \Pi_\bot T_a \partial_x w  \quad  \text{in}\ \  C^0([- T, T], \,  H^{\bar s - m}_\bot(\T_1) ) 
  \, , \quad \forall \bar s \in [0, s) \, . 
$$
 As a consequence $ w \in 
 C^1([- T, T], \,  H^{\bar s - m}_\bot(\T_1) )$ 
 and $ \pa_t w = {\cal D}(t) [w] + \Pi_\bot T_a \partial_x w $  
  solves \eqref{PDE lineare Hs bot}. 
 
Finally,  
arguing as in \cite{Tay}, Proposition 5.1.D, 
it follows that the function $t \to \| w(t)\|_{s}^2$ is Lipschitz.  
Furthermore, one can show that  if $ t_n \to t $ then 
$ w(t_n ) \rightharpoonup w(t) $ weakly in $  H^s_\bot(\T_1) $, because 
 $w(t_n) \to w(t)$ in $H^{\bar s}_\bot(\T_1)$ for any $\bar s \in [0, s)$.
 As a consequence the sequence $w(t_n) \to w(t) $ strongly in $H^s_\bot(\T_1)$. 
 This proves that  $w \in C^0([- T, T], H^s_\bot(\T_1))$ and therefore $\partial_t w \in C^0([- T, T], H^{s - m}_\bot(\T_1)) $. 
 
 \medskip
 \noindent
  {\sc Uniqueness.} If $w_1, w_2 \in C^0([- T, T], H^{s}_\bot(\T_1)) \cap  C^1([- T, T], H^{s - m}_\bot(\T_1)) $, $s \geq \s $,  
  are solutions of \eqref{PDE lineare Hs bot} with $w_1(\tau) = w_2(\tau) \in H^s_\bot(\T_1)$, then $h: = w_1 - w_2$  solves 
  $$
  \partial_t h = {\cal D}(t) h + \Pi_\bot T_a \partial_x h\,, \qquad h(\tau) = 0\,.
  $$
  Arguing as in the proofs of the previous energy estimates, we deduce the energy inequality 
  $  \partial_t \| h (t)\|^2 \leq C \| h(t)\|^2 $. Since $ h(\tau)= 0 $, 
the Gronwall Lemma implies that $\| h(t)\|^2 = 0$, for any $t \in [-T, T]$. This shows the uniqueness. 

\smallskip

The estimate for $\| w \|_{s }$ in  \eqref{stima PDE lin hom} then follows by \eqref{uniform bound wN} - \eqref{disuguaglianza liminf u uN} 
and the one of $\| \partial_t w \|_{s - m}$ in  \eqref{stima PDE lin hom} by using the equation. 
\end{proof}
In the next lemma we consider the inhomogeneous equation \eqref{PDE lineare Hs bot con f}.  
\begin{lemma}\label{lemma positura lineare f}
Let $\sigma \ge m$ and $m$ be given as in Lemma \ref{lemma positura lineare} and 
 assume that for some $0 < T \le 1$ and $s \ge \s$,
 \eqref{ipotesi cal D (t)} - \eqref{ipotesi a pde lin} hold
and $\| a \|_{{\cal C}^0_t H^\sigma_x} \leq 1$.
Then for any $w_0 \in H^s_\bot(\T_1)$, 
 there exists a unique solution $t \mapsto w(t)$ of \eqref{PDE lineare Hs bot con f} in $C^0([- T, T], H^s_\bot(\T_1)) \cap C^1([- T, T], H^{s - m}_\bot(\T_1))$, 
with $w(0) = w_0$. For any $t \in [- T, T]$ it satisfies, 
\begin{equation}\label{stima PDE lin inomo}
\begin{aligned}
& \| w(t) \|_s \lesssim_s \| w_0 \|_s + \int_0^t \| f(\tau) \|_s \,d \tau \lesssim_s \| w_0 \|_s + T \| f \|_{{\cal C}^0_t H^s_x}\,, \\
&  \| \partial_t w(t) \|_{s - m} \lesssim_s   \| w_0 \|_s + T \| f \|_{{\cal C}^0_t H^s_x}\,, \quad \forall t \in [- T, T].
\end{aligned}
\end{equation}
\end{lemma}
\begin{proof}
For any $t, \tau \in [-T, T]$, denote  by $\Phi(\tau, t)$ the flow map of the para-differential equation \eqref{PDE lineare Hs bot},
$$
\partial_t w = {\cal D}(t) [w] + \Pi_\bot T_a \partial_x w\,, \qquad w(\tau, \cdot) = w_0(\cdot) \, .
$$
By Lemma \ref{lemma positura lineare}, $\Phi(\tau, t)$ is a bounded linear operator $H^s_\bot(\T_1) \to  H^s_\bot(\T_1)$ for any $s \ge \sigma$. 
The estimate \eqref{stima PDE lin hom} implies that
$$
\| \Phi(\tau, t) w_0 \|_s \lesssim_s \| w_0 \|_s , \qquad \| \partial_t \Phi(\tau, t) w_0 \|_{s - m} \lesssim_s \| w_0 \|_s\,. 
$$
The unique solution of the equation \eqref{PDE lineare Hs bot con f} in $C^0([- T, T], H^s_\bot(\T_1)) \cap C^1([- T, T], H^{s - m}_\bot(\T_1))$
with initial data $w(0) = w_0$ is then given by the Duhamel formula 
$w(t) = \Phi(0, t) w_0 + \int_0^t \Phi(\tau, t) f(\tau)\, d \tau$
and the claimed estimates  easily follow. 
\end{proof}

\medskip

\noindent
T. Kappeler, 
Institut f\"ur Mathematik, 
Universit\"at Z\"urich, Winterthurerstrasse 190, CH-8057 Z\"urich;\\
${}\qquad$  email: thomas.kappeler@math.uzh.ch 

\noindent
R. Montalto, 
University of Milan, Via Saldini 50, 20133, Milan, Italy;\\
${}\qquad$ email: riccardo.montalto@unimi.it

\end{document}